\documentclass[12pt]{amsart}
\usepackage{amsfonts,amssymb,amscd,xy}
\usepackage[leqno]{amsmath}
\usepackage{mathrsfs}
\usepackage{amssymb}
\usepackage{amsmath}
\usepackage{amsxtra}
\usepackage{amsthm}
\DeclareMathAlphabet{\mathbbmsl}{U}{bbm}{m}{sl}

\usepackage[usernames,dvipsnames]{xcolor}
\usepackage{tabularx}
\usepackage{epsfig,amssymb}
\usepackage{bm} 
\usepackage{ulem}
\usepackage{hyperref}

\makeatletter
\def\bm#1{\mathpalette\bmstyle{#1}}
\def\bmstyle#1#2{\mbox{\boldmath$#1#2$}}

\DeclareMathAlphabet{\mathcalligra}{T1}{calligra}{m}{n}

\newcommand{\grn}{{\mathrm{Gr}}_{\be\le n}^{R}}
\newcommand{\grm}{{\mathrm{Gr}}_{\lbe m}^{\lbe R}}
\newcommand{\grnl}{{\mathrm{Gr}}_{n+1}^{\lbe R}}
\newcommand{\grni}{{\mathrm{Gr}}_{n+i}^{\lbe R}}
\newcommand{\grbe}{{\uwave{\mathrm{Gr}}}}
\newcommand{\pgrn}{{\mathbf{Gr}}_{\be\le n}^{R}}

\newcommand{\gra}{{\mathrm{Gr}}^{\le\mathfrak R}\lbe}

\newcommand{\Fhr}{{\m h}^{\lbe R}}
\newcommand{\hrn}{h_{\lle n}^{\lbe R}\lbe}

\newcommand{\hra}{h^{\lbe \mathfrak R}\be\le}

\newcommand{\sep}{{\rm sep}}

\newcommand{\sh}{\kern -.4em\phantom{a}^{\mathbf{\sim}}}

\newcommand{\qc}{\rm{qc}}
\newcommand{\alg}{\rm{alg}}
\newcommand{\nr}{\le{\rm nr}}

\newcommand{\lra}{\longrightarrow}

\newcommand{\HH}{\mathrm{H}}
\newcommand{\et}{{\rm {\acute et}}}

\newcommand{\fpqc}{{\rm fpqc}}
\newcommand{\fppf}{{\rm fppf}}

\newcommand{\dimn}{{\rm dim}_{\e k}\le}
\newcommand{\m}{\mathfrak }

\newcommand{\red}{{\rm red}}
\def\ogs{\omega_{G\lbe/\lbe S}^{1}}
\def\ogsp{\omega_{G^{\le\prime}\be/\lbe S^{\lle\prime}}^{1}}
\newcommand{\cO}{{\s O}}
\newcommand{\kbar}{\overline{k}}
\def\be{\kern -.1em}
\def\bbe{\kern -.07em}
\def\le{\kern 0.03em}
\def\lle{\kern 0.015em}
\def\lbe{\kern -.025em}
\def\llbe{\kern -.03em}
\def\Re{{\rm Res} }

\newcommand{\ari}{\bar{e}}
\newcommand{\est}{a}

\newcommand{\A}{{\mathbb A}}

\newcommand{\N}{{\mathbb N}}
\newcommand{\G}{{\mathbb G}}
\newcommand{\gr}{\mathrm{Gr}^{R}\lbe}
\newcommand{\pgr}{\mathbf{Gr}^{\lbe R}}

\newcommand{\spec}{\mathrm{ Spec}\,}
\newcommand{\spf}{\mathrm{Spf}\e}

\newcommand{\krn}{\mathrm{Ker}\e}
\newcommand{\img}{\mathrm{Im}\e}
\newcommand{\cok}{\mathrm{Coker}\,}
\newcommand{\Hom}{{\mathrm{Hom}}}

\newcommand{\Rnr}{R^{\le\rm{nr}}}
\newcommand{\Rnnr}{R^{\le\rm{nr}}_{\le n}}

\newcommand{\Snr}{S^{\e\rm{nr}}}
\newcommand{\Snnr}{S^{\e\rm{nr}}_{\lbe n}}

\def\e{\kern 0.08em}

\newcommand{\overbar}[1]{{ \mkern 8mu\overline{\mkern-8mu#1 }}}
\newcommand{\overbarr}[1]{ {\mkern 10mu\overline{\mkern-10mu#1 }}}
\newcommand{\s}{\mathscr }
\newcommand{\mm}{\mathfrak m}
\newcommand{\mr}{\mathrm }

\newcommand{\pf}{{\rm pf}}

\numberwithin{equation}{section}

\newtheorem{lemma}[equation]{Lemma} 
\newtheorem{theorem}[equation]{Theorem}
\newtheorem{proposition-definition}[equation]{Proposition-Definition}
\newtheorem{corollary}[equation]{Corollary}
\newtheorem{proposition}[equation]{Proposition}
\theoremstyle{definition}
\newtheorem{definition}[equation]{Definition}

\theoremstyle{remark}
\newtheorem{remark}[equation]{Remark}
\newtheorem{remarks}[equation]{Remarks}
\newtheorem{example}[equation]{Example}
\newtheorem{examples}[equation]{Examples}

\definecolor{labelkey}{rgb}{1,0,0}

\begin{document}

\input xy     
\xyoption{all}

\title[The  Greenberg functor revisited]{The  Greenberg functor revisited}

\author{Alessandra Bertapelle}
\address{Universit\`a degli Studi di Padova, Dipartimento di Matematica, via Trieste 63, I-35121 Padova}
\email{alessandra.bertapelle@unipd.it}

\author{Cristian D. Gonz\'alez-Avil\'es}
\address{Departamento de Matem\'aticas, Universidad de La Serena, Cisternas 1200, La Serena 1700000, Chile}
\email{cgonzalez@userena.cl}
\thanks{C.\,G.-A. was partially supported by Fondecyt grants 1120003 and 1160004.}
\date{\today}

\subjclass[2010]{Primary 11G25, 14G20}
\keywords{Greenberg realization, formal schemes, Weil restriction}

\topmargin -1cm

\smallskip

\begin{abstract} We extend Greenberg's original construction to arbitrary (in particular, non-reduced) schemes over (certain types of) local artinian rings. We then establish a number of basic properties of the extended functor and determine, for example, its behavior under Weil restriction.  We also discuss a formal analog of the functor.  
\end{abstract}

\maketitle

\tableofcontents

\topmargin -1cm

\section{Introduction}

The Greenberg functor, originally introduced in \cite{gre1}, has played and continues to play an important role in arithmetic and algebraic geometry, most recently in \cite{bt, ns,ns2}. See also \cite{secft,beg,blr,cgp,lip}. Unfortunately, Greenberg's construction is difficult to understand since it uses, in part, a pre-Grothendieck language to describe the key construction of {\it Greenberg algebras}  \cite[\S1]{gre1}. In this paper we revisit Greenberg's construction using a modern scheme-theoretic language and generalize it in various ways, removing in particular certain unnecessary reducedness and finiteness conditions in \cite{gre1,gre2}. Further, we establish a number of refinements of known properties of the (classical) Greenberg functor and establish new results.  We also clarify the relation that exists between the Greenberg algebra associated to a local artinian ring $\mathfrak R$ (of a certain type) and the Greenberg module associated to an ideal $\mathfrak I$ of $\mathfrak R$. See Remark \ref{gr-max}. The new insight  thus gained led us to a better understanding of the change of ring morphisms \eqref{tr0} and therefore also of the change of level morphism \eqref{clm0}.  In addition, we describe the kernel of the change of level morphism \eqref{clm0}  for possibly non-smooth and non-commutative group schemes. See Proposition \ref{vker}. These morphisms (which play a role in certain important limit constructions) seem to have been previously discussed only in the smooth and commutative case. Among the main results of this paper the reader will find the complete determination of the behavior of the Greenberg functor under Weil restriction. See Theorem \ref{wr-gr} below. To our knowledge, only a very specific instance of this result has appeared in print before (within the context of formal geometry), namely \cite[Theorem 4.1]{ns}.

\smallskip

We now describe the contents of the paper.

\smallskip

The extended preliminary Section \ref{pre} consists of eight subsections. Subsections \ref{not}\,--\,\ref{for-sch}\, introduce notation and recall basic material on vector bundles,  Witt vectors, groups of components, the connected-\'etale sequence and formal schemes. In subsection \ref{wres} we discuss the Weil restriction functor $\Re_{S^{\le\prime}\be/S}$ and show in particular that the hypotheses in the basic existence theorem \cite[\S7.6, Theorem 4, p.~194]{blr} can be weakened. We also record here the fundamental fact that the Weil restriction of a scheme along a finite and locally free {\it universal homeomorphism} always exists. Subsection \ref{ftop} (which is relevant for Section \ref{gp-sch}) presents certain results on the fpqc topology which we have been unable to find in the literature in the precise formulation in which we need them. 
Section \ref{ga} contains a general discussion of Greenberg modules/algebras associated to finite $W_{\be m}(k)$-modules/algebras, where $m\geq 1$ and (the field) $k$ is assumed to be perfect and of positive characteristic if $m>1$.  Readers who are familiar with Greenberg's original construction will have noticed that this author encountered a number of technical difficulties that forced him to replace, depending on the situation, a module variety with a homeomorphic one. See Remark \ref{gr-max} for a full discussion of this issue. In this paper we correctly identify the ideal subscheme \eqref{kbar} of the relevant Greenberg algebra that must be chosen in order to circumvent all such technical difficulties.
In Section \ref{truc} we specialize the discussion of Section \ref{ga} to truncated discrete valuation rings, using as our starting point the careful presentation of Nicaise and Sebag in \cite[pp.~1591-94]{ns}. Incidentally, these authors seem to have been the first to notice that a certain formula involving Greenberg algebras which appears in \cite[p.~276, line~-18]{blr} is incorrect (in Remark \ref{rems1}(d) we explain why the indicated error is inconsequential when working with the tower of Greenberg algebras). Section \ref{ram} discusses Greenberg algebras and ramification. In Section \ref{gad} we use the results of Section \ref{truc} and a standard limit process to discuss Greenberg algebras of discrete valuation rings. Section \ref{gr-art} introduces (at long last!) the Greenberg functor in the general setting of this paper. The constructions of Section \ref{gr-art} are then specialized to truncated discrete valuation rings in Section \ref{gberg}. Section \ref{s-cr} discusses the all-important change of rings morphism, which specializes to the change of level morphism \eqref{clm0} of Section \ref{s-clm}. For example, we show that the morphism \eqref{clm0} associated to a scheme $Z$ over a truncated discrete valuation ring $R_{n}$ is surjective if $Z$ is formally smooth over $R_{n}$. When $Z$ is smooth and of finite type over $R_{n}$, this result was obtained by Greenberg \cite[Corollary 2, p.~262]{gre2} using a method which differs from the one used here. Section \ref{bas} presents a number of basic results on the Greenberg functor, some of which do not seem to have been noticed before.
For example, we show that the Greenberg functor preserves quasi-projective schemes (see Proposition \ref{q-proj}). This result is new in the unequal characteristics case (in the equal characteristic case the Greenberg functor of level $n$ coincides with the Weil restriction functor $\Re_{R_{\le n}/k}$ and the corresponding result is a particular case of \cite[Proposition A.5.8]{cgp}).
 In Section \ref{gstr} we extend Greenberg's structure theorem \cite[p.~ 263]{gre2}, showing in particular that (the original version of) the indicated result is unaffected by Greenberg's occasional replacement of certain Greenberg modules by homeomorphic group varieties.
Section \ref{wrbe} contains the already noted description of the behavior of the Greenberg functor under Weil restriction. In Section \ref{gp-sch} we describe the kernel of the change of level morphism introduced in Section \ref{s-clm}. 
 In particular, we show in Example \ref{cex-beg} that \cite[Lemma 4.1.1(2)\le]{beg} is false in general. In spite of the above, the main results of \cite{beg} are fortunately valid, as we explain in Remark \ref{begueri}.  In Section \ref{pfgf}, relying on \cite{bga}, we discuss the perfect Greenberg functor. Section \ref{fgr} contains information on the Greenberg realization of a finite group scheme, which may not itself be finite over $k$ (see Remark \ref{rm.grf}(a)). We now note that Sebag, in his thesis \cite[\S3]{seb} (see also \cite[\S2.3]{ls} and \cite{ns}), defined the Greenberg realization of a separated formal scheme of topologically finite type. In Section \ref{forgr} we extend his construction to the larger category of adic formal schemes and determine the behavior of the new functor under Weil restriction. In particular, we generalize \cite[Theorem 4.1]{ns}. The constructions of Section \ref{forgr} are then applied in Section \ref{grsh} to study the Greenberg realization of an $R$-scheme, where $R$ is a complete discrete valuation ring.
Section \ref{comgr} studies the Greenberg realization of a flat, commutative and separated $R$-group scheme $F$, where $R$ is as above, using a smooth resolution of $F$ when one exists (this is the case if $F$ is finite over $R$). Finally, Section \ref{eqch} presents a generalization of the equal characteristic case discussed previously in the text (where the relevant ring may no longer be a discrete valuation ring).

\begin{center}
Acknowledgements
\end{center}
We thank James Borger and Maurizio Candilera for helpful comments on Witt vectors, Brian Conrad and Angelo Vistoli for enlightening comments on Weil restriction
and the fpqc topology and Johannes Nicaise for helpful comments on Greenberg approximation. We also thank Michel Raynaud for answering some of our questions regarding \cite{beg}.

\section{Preliminaries}\label{pre}
\bigskip

\subsection{Generalities}\label{not}
If $x$ is a real number, $\lfloor x\rfloor$ will denote the largest integer that is less than or equal to $x$ and $\lceil x\rceil$ will denote the smallest integer that is larger than or equal to $x$. Note that $\lceil x\rceil\geq 1$ if $x>0$. 
\smallskip

All rings considered in this paper are commutative and unital (with unity 1). If $A$ is a ring and $f\in A$, $A_{f}$ will denote the localization of $A$ with respect to the multiplicative set $\{f^{\le r}\}_{r\geq\le 0}$, where $f^{\e 0}=1$. 
\smallskip

If $X$ is an object of a category, the identity morphism of $X$ will be denoted by $1_{\lbe X}$.
\smallskip

If $X$ is a scheme, $|X|$ will denote the underlying topological space of $X$.  Further, if $X$ and $Y$ are $S$-schemes, where $S=\spec A$ is an affine scheme, then $X\!\times_{\be S}\be Y$ will be denoted by $X\!\times_{\be A}\be Y$. 

If $k$ is a field and $X\to\spec k$ is a finite morphism of schemes, then $|X|$ is a finite set \cite[II, Corollary 6.1.7]{ega}. Thus, since $X(k)$ may be identified with a subset of $|X|$ \cite[(3.5.5), p.~243]{ega1}, $X(k)$ is also a finite set. 

If $S$ is the spectrum of a local ring with residue field $k$ and $X$ is an $S$-scheme, we will write 
\[
X_{\mr s}=X\times_{S}\spec k
\]
for the special fiber of the structural morphism $X\to S$. If $f\colon X\to Y$ is a morphism of  $S$-schemes,  $f\times_{S}1_{\spec k}$ will be written  $f_{\mr s}\colon X_{\mr s}\to Y_{\mr s}$.
\smallskip

Given a diagram of morphisms of schemes
\[
\xymatrix{T^{\e\prime}\ar[d]& X\ar[d]^(.45){f}\\
T\ar[r]^(.5){g}& S,}
\]
we will write $X\be\times_{f,\le S,\e g}\be T$ for the fiber product of $f$ and $g$. When $f$ and $g$ are not relevant in a particular discussion, we will write  $X\times_{S}T$ for $X\be\times_{f,\e S,\le g}\be T$. We will make the identifications
\[
\begin{array}{rcl}
T\times_{S}S &=& S\times_{S}T=T\\
X\times_{S}T &=& T\times_{S}X\\
(X\times_{S} T\e)\times_{T} T^{\e\prime} &=& X\times_{S} T^{\e\prime}.
\end{array}
\]
Note that, if $f\colon X\to Y$ is an $S$-morphism of schemes, then $f_{\lbe T}=f\times_{S}1_{\lbe T}\colon X_{\lbe T}\to Y_{ T}$ is a $T$-morphism of schemes. If $T=\spec B$ is affine, $X_{\lbe T}$ and $f_{T}$ will be denoted by $X_{\be B}$ and $f_{\be B}$, respectively.

Let $X\to T\to S$ be a pair of morphisms of schemes. If $Y$ is an $S$-scheme and $\mr{pr}_{\lbe Y}\colon Y\be\times_{S}\be T \to Y$ is the first projection
then, by the universal property of the fiber product, the map
\begin{equation}\label{bc}
\Hom_{\e T}(X,Y\be\times_{S}\be T\e)\stackrel{\!\sim}{\to}\Hom_{\le S}(X,Y),\,\, g\mapsto \mr{pr}_{\lbe Y}\circ g,
\end{equation}
is bijective. 

\begin{lemma} \label{sp} Let $k$ be a perfect field of positive characteristic $p$, let $A$ be a $k$-algebra and let $f\in A$.
\begin{enumerate}
\item[(i)] If $A=A^{\le p}$, then $A_{\lbe f}=(A_{\lbe f})^{\le p}$.
\item[(ii)] There exist a $k$-algebra $B$ with $B=B^{\le p}$ and $B_{f}=(B_{f})^{\le p}$ and injective homomorphisms of $k$-algebras $A\hookrightarrow B$ and $A_{f}\hookrightarrow B_{f}$.  
\end{enumerate}
\end{lemma}
\begin{proof} If $A=A^{\le p}$ and $a/f^{\le n}\in A_{f}$, choose $b,g\in A$ such that $b^{\e p}=a$ and $g^{\le p}=f$. Then $(\le g^{\le n(\e p-1)}b/f^{\le n})^{\le p}=f^{n(\le p-1)}a/f^{np}=a/f^{\le n}$, whence (i) follows. Now, by \cite[Lemma 0.1, p.~18]{lip}, there exist a $k$-algebra $B$ satisfying $B=B^{\le p}$ (and therefore also $B_{f}=(B_{f})^{\le p}$, by (i)) and a faithfully flat ring homomorphism $A\to B$. Since the latter map is injective by \cite[(4.C)(i), p.~28]{mat}, to complete the proof of (ii) it remains only to check that $A_{f}\to B_{f}$ is injective. Since
$A_{f}$ is flat over $A$, the map $A\hookrightarrow B$ induces an injection $A_{f}\hookrightarrow B\otimes_{A}\! A_{f}$. Composing the latter map with the isomorphism $B\otimes_{A}A_{f}\overset{\!\sim}{\to} B_{f}$ of \cite[Proposition 3.5, p.~39]{am}, we deduce the injectivity of $A_{f}\to B_{f}$.
\end{proof}

Following \cite[p.~191]{blr}, we will say that a morphism of schemes is {\it finite and locally free} if it is finite, flat and of finite presentation. The class of finite and locally free morphisms is stable under base change. A morphism of schemes $S^{\e\prime}\to S$ is called a {\it universal homeomorphism} if, for every base change $T\to S$, the induced morphism $S^{\e\prime}_{\le T}\to T$ is a homeomorphism. The class of universal homeomorphisms is stable under base change. Further, by \cite[$\text{IV}_{4}$, Corollary 18.12.11]{ega}, a morphism of schemes is a universal homeomorphism if, and only if, it is integral, surjective and radicial. In particular, a universal homeomorphism is affine. If $k^{\e\prime}\be/k$ is a purely inseparable extension of fields, the associated morphism of schemes $\spec k^{\e\prime}\to\spec k$ is a universal homeomorphism. Recall now that a {\it nilimmersion} is a surjective immersion or, equivalently, a closed immersion defined by a nilideal. Such a morphism is a homeomorphism. The class of nilimmersions is stable under base change. Consequently, a nilimmersion is a universal homeomorphism. In particular, a {\it nilpotent immersion}, i.e., a closed immersion defined by a nilpotent ideal, is a universal homeomorphism.
 
\begin{lemma}\label{uh0} Let $k$ be a field and let $B$ be a finite and local $k$-algebra with residue field $k^{\e\prime}$. Assume that the
associated field extension $k^{\e\prime}\be/k$ is purely inseparable. Then $\spec B\to\spec k$ is a finite and locally free universal homeomorphism.
\end{lemma}
\begin{proof} The canonical morphism $\spec B\to\spec k$ is clearly finite and locally free (i.e., finite and flat). Now observe that $B$ is an artinian local ring (whence $\spec B$ is a one-point scheme) and $k^{\e\prime}$ is a finite extension of $k$ by \cite[Corollary 7.10, p.~82, and Exercise 3, p.~92]{am}. Since $k^{\e\prime}\be/k$ is purely inseparable, the composite morphism $\spec k^{\e\prime}\to\spec B\to\spec k$ is a universal homeomorphism. Thus, since $\spec k^{\e\prime}\to\spec B$ is surjective, $\spec B\to\spec k$ is a universal homeomorphism as well by \cite[Proposition 3.8.2(iv), p.~249]{ega1}.
\end{proof}

If $S$ is a scheme, $\Lambda$ is a directed set and $(X_{\lbe\lambda},u_{\lambda,\e\mu}\e ; \lambda,\e\mu\in\Lambda)$ is a projective system of $S$-schemes with affine transition morphisms, then $X=\varprojlim X_{\lbe\lambda}$ exists in the category of $S$-schemes  \cite[$\text{IV}_{3}$, Proposition 8.2.3]{ega}. More precisely, $X$ is isomorphic to the spectrum of the $\s O_{\be X_{0}}$-algebra $\varinjlim_{\lambda\e\geq\e 0} u_{\lambda,\e 0\e *}\s O_{\be X_{\lbe\lambda}}$, where $0$ is any fixed element of $\Lambda$. Thus, for every $S$-scheme $Z$, there exists a canonical bijection 
\begin{equation}\label{plim}
\mathrm{Hom}_{\le S}\big(Z,\varprojlim X_{\lbe\lambda}\big)=\varprojlim\mathrm{Hom}_{\le S}(Z,X_{\lbe\lambda}).
\end{equation}

\begin{lemma} \label{paux} Let $S$ be a scheme and let $(X_{\lbe n},u_{\le m,\e n}\e ; m\geq n\in\N\e)$ be a projective system of $S$-schemes with affine transition morphisms and index set $\N$. Set $X=\varprojlim X_{n}$ and assume that there exist two strictly increasing sequences $\{r_{\be n}\}_{n}$ and $\{s_{n}\}_{n}$ in $\N$ such that $r_{\be n}\geq s_{n}$ for every $n$. Then $\varprojlim u_{\e r_{\be n},\e s_{n}}\colon X\to X$ is the identity morphism of $X$.
\end{lemma}
\begin{proof} For every $n\in\N$, let $p_{n}\colon X\to X_{n}$ be the canonical projection morphism. Now let $h\colon Y\to X$ be an arbitrary morphism of $S$-schemes and write $h_{n}=p_{n}\circ h\colon Y\to X_{n}$. We claim that
$(\varprojlim u_{\e r_{\be n},\e s_{n}})\circ h=h$. Clearly, it suffices to check that $p_{s_{n}}\circ(\varprojlim u_{\e r_{\be n},\e s_{n}})\circ h=h_{s_{n}}$ for every $n$. Now $p_{s_{n}}\circ (\varprojlim u_{\e r_{\be n},\e s_{n}})\circ h= u_{\e r_{\be n},\e s_{n}}\circ p_{\le r_{\be n}}\circ h=u_{\le r_{\be n}, \le s_{n}}\circ h_{r_{\be n}}=h_{s_{n}}$, which completes the proof.  
\end{proof}

\subsection{Vector bundles} \label{vb}
Let $S$ be a scheme, $\s E$ a quasi-coherent $\cO_{\be S}$-module and $\mathbb V(\s E)$ the $S$-vector bundle associated to $\s E$ \cite[II, (1.7.8)]{ega}. By definition, for every $S$-scheme $f\colon X\to S$, there exists a canonical bijection
\begin{equation}\label{vb2}
\mr{Hom}_{S}(X, \mathbb V(\s E))=\mr{Hom}_{\e\cO_{\be S}}(\s E, f_{*}\cO_{\! X}).
\end{equation}
The $S$-scheme $\mathbb O_{S}=\mathbb V(\cO_{\be S})\,$\footnote{We adopt the notation introduced in \cite[I, 4.3.3]{sga3}. In \cite[II, (1.7.13)]{ega}, $\mathbb O_{\lbe S}$ is denoted by $S\le[\le T\le]$.} has a canonical $S$-ring scheme structure and $\mathbb V(\s E\le)$ is canonically endowed with an $S$-module scheme structure over $\mathbb O_{S}$. Note that the $S$-scheme (respectively, $S$-group scheme) underlying $\mathbb O_{S}$ is $\A^{\be 1}_{S}=\spec \cO_{\be S}[\e T\e]$ (respectively, $\G_{a, \e S}$). If $S=\spec A$ is affine, we will write $\mathbb O_{\be A}$ for $\mathbb O_{S}$.

Let $G$ be an $S$-group scheme, $\varepsilon\colon S\to G$ the unit section of $G$ and set $\ogs=\varepsilon^{\le*}\e\Omega_{G/S}^{\le 1}$, which is a quasi-coherent $\cO_{\be S}$-module. For every $S$-scheme $f\colon S^{\e\prime}\to S$, set $G^{\e\prime}=G_{\lbe S^{\lle\prime}}$ and 
$w_{G^{\le\prime}\be/S^{\le\prime}}^{1}=\varGamma\lbe\big(S^{\e\prime}\be,\e\ogsp\big)$.
We have $\mathbb V(\ogs)\!\times_{S}\! S^{\e\prime}=\mathbb V\big(\ogsp\be\big)$ and
\begin{equation}\label{v}
\mathbb V\big(\ogsp\be\big)(S^{\e\prime}\e)=\mathbb V(\ogs)(S^{\e\prime}\e)=\Hom_{\e\cO_{\be S^{\prime}}}\lbe(\ogsp,\cO_{\be S^{\prime}})=\Hom_{\e\cO_{\be
S}}\lbe(\ogs,f_{*}\lbe\cO_{\be S^{\prime}}), 
\end{equation}
by \eqref{vb2} and \cite[Proposition 9.4.11(iv), p.~374]{ega1}.
Further, by \cite[II, 4.11]{sga3}, $\mathbb V(\ogs)$ represents the
functor $\underline{\rm{Lie}}(G/S)$ which assigns to an $S$-scheme $S^{\e\prime}$ the Lie $\Gamma(S^{\e\prime},\cO_{S^{\prime}})$-algebra of right-invariant derivations of $G^{\e\prime}$ over $S^{\e\prime}$.
If $G$ is smooth over $S$, $\ogs$ is locally free of finite type and $\mathbb V(\ogs)$ is also smooth over $S$ by \cite[${\rm{IV}}_{4}$, Propositions 17.2.3(i) and 17.3.8]{ega}.

Now, if  $S^{\e\prime}=\spec B^{\e\prime}$ is affine, we will write $\ogsp=\omega_{G^{\le\prime}\be/B^{\le\prime}}^{1}$ and $w_{G^{\le\prime}\be/S^{\le\prime}}^{1}=w_{G^{\le\prime}\be/B^{\le\prime}}^{1}$. Note that, by \cite[Corollary 1.4.2, p.~ 206]{ega1},  $\omega_{G^{\le\prime}\be/B^{\le\prime}}^{1}=\widetilde{w}
_{G^{\le\prime}\be/B^{\le\prime}}^{\e 1}$ is the $\cO_{\spec B^{\le\prime}}\e$-\e module associated to the $B^{\e\prime}$-module $w_{G^{\le\prime}\be/B^{\le\prime}}^{1}$. If $f\colon \spec B^{\e\prime}\to \spec B$ is the morphism associated to a ring homomorphism $B\to B^{\e\prime}$ then, by \cite[Corollary 1.7.4, p.~213]{ega1}, \eqref{v} is equivalent to the identities
\begin{equation}\label{v2}
\mathbb V(\omega_{G^{\e\prime}\be/B^{\le\prime}}^{1})(B^{\e\prime})=   
\Hom_{\e B^{\le\prime}\text{-mod}}(w_{G^{\e\prime}\be/B^{\le\prime}}^{1},B^{\e\prime}\e)=
\Hom_{\le B\text{-mod}}(w_{G\lbe/\lbe B}^{1},B^{\e\prime}\e).
\end{equation}

Assume now that $S$ is the spectrum of a local ring  with residue field $k$ and that $G$ is locally of finite type over $S$. Then, by \cite[II, (4.11.3)]{sga3}, $\omega_{G_{\mr s}/k}^{1}$ is a free $\cO_{\spec k}$-module of rank 
\begin{equation}\label{d}
d=\dimn\le{\rm Lie}(G_{\mr s}).
\end{equation}
Thus there exists a (non-canonical) isomorphism of $k$-group schemes
\begin{equation}\label{pc}
\hskip -1.3cm\mathbb V\big(\omega_{\le G_{\mr s}\lbe/k}^{1}\big)\simeq\mathbb G_{a,\e k}^{d}.
\end{equation}
Note that $d\geq\dim G_{\mr s}$, with equality if, and only if, $G_{\mr s}$ is smooth over $k$ \cite[II, \S5, 1.3 and Theorem 2.1, pp.~235 and 238]{dg}.

\begin{lemma}\label{pc0} Let $S$ be a local scheme and let $G$ be a smooth $S$-group scheme. Then there exists a (non-canonical) isomorphism of $S$-group schemes
\[
\mathbb V\big(\ogs)\simeq\mathbb G_{\lbe a,\e S}^{d}\,,
\]
where $d=\dim G_{\mr s}$ is the dimension of the special fiber of $G$.
\end{lemma}
\begin{proof} Since $S$ is local and $\ogs$ is locally free over $S$, $\ogs$ is free over $S$ of rank $d$. The lemma is now clear.
\end{proof}

\subsection{Witt vectors} 
Let $p$ be a prime number. If $A$ is a ring, let $W(\lbe A\lbe )$ denote the ring of $p$-typical Witt vectors on $A$. By definition, $W(A)$ is the set $A^{\N_{0}}$ endowed with laws of composition defined by certain polynomials. See \cite[II, \S6]{self} or \cite[\S1]{ill} for more details. The map $V\colon W(A)\to W(A), (a_{0},a_{1},\dots)\mapsto (0,a_{0},a_{1},\dots),$ is an additive operator called the {\it Verschiebung}. For every integer $n\geq 1$,  the $n$-th truncation $W_{\be n}(A)\simeq W(A)/\le V^{n}W(A)$ is the ring of Witt vectors of length $n$. Now, if $n>s\geq 1$ are integers, consider both the injective homomorphism of abelian groups $V_{n-s,\le n}\colon W_{\be n-s}(A)\to W_{\be n}(A), (a_{0},\dots, a_{\le n-s-1})\mapsto (0,\dots,0,a_{0},\dots, a_{\le n-s-1})$ ($s$ zeroes) and the surjective homomorphism of rings $R_{\e n,\le s}\colon W_{\! n}(A)\twoheadrightarrow W_{\!s}\lbe(A), (a_{0},\dots, a_{\le n-1})\mapsto (a_{0},\dots, a_{\le s-1})$.
Clearly, the sequence
\begin{equation}\label{vr}
0\to W_{\!n-s}(A)\overset{V_{n\lbe-\lbe s,n}}{\lra} W_{\!n}(A)\overset{\!\be R_{\le n,s}}{\lra}
 W_{\! s}(A)\to 0
\end{equation}
is exact. To conform to standard notation, we will write
\begin{equation}\label{sta}
V^{s}W_{\!n-s}(A)=V_{n-s,\le n}W_{\!n-s}(A)\subset W_{\!n}(A).
\end{equation}
We now observe that $W_{\lbe 1}(A)=A$ and $W(A)=\varprojlim W_{\!n}(A)$, where the corresponding transition maps are the maps $R_{\e n+1,n}$. Further, if $a\in A=W_{\lbe 1}(A)$, then 
\begin{equation*}
V^{s}\be(a)\overset{{\rm def.}}{=}V_{1,\e s+1}(a)=(0,\dots,0,a)\in W_{\!s+1}(A)
\end{equation*}
($s$ zeroes). Next, assume that $A$ is a ring of characteristic $p$, i.e., an $\mathbb F_{p}$-algebra. Then the Frobenius endomorphism of $W\be(\lbe A)$ defined in \cite[\S1.3]{ill} agrees with the map $F\colon  W\be(\lbe A)\to W\be(\lbe A),(a_{0},a_{1},\dots)\mapsto (a_{0}^{\le p},a_{1}^{\le p},\dots)$, which is surjective if $A=A^{p}$ \cite[(1.3.3) and (1.3.5), p.~507]{ill}. By \cite[(1.3.7) and (1.3.8), p.~507]{ill}, we have 
\begin{equation}\label{ot}
p=VF=FV,
\end{equation}
whence 
\begin{equation}\label{fvp}
p\colon W\be(\lbe A)\to W\be(\lbe A),(a_{0},a_{1},\dots)\mapsto(0,a_{0}^{\le p},a_{1}^{\le p},\dots).
\end{equation}
Further, the following holds: for every pair $m,n\in\N$ and $x,y\in W\be(\lbe A)$,
\begin{equation}\label{vv}
V^{m}x\cdot V^{n}y=V^{\le m+n}(F^{\le n}x\cdot F^{m}y)
\end{equation}
\cite[1.3.12, p.~508]{ill},  and similar identities hold for the various truncations of $W\be(\lbe A)$.  Now, if $n\geq 1$, then the truncation homomorphism $W(A)\to W_{\!n}(A), (a_{0},a_{1},\dots)\mapsto (a_{0},a_{1},\dots,a_{n-1})$, induces a surjective ring homomorphism
\begin{equation}\label{tru}
t_{\le n}^{A}\colon W(A)/(\e p^{\le n})\twoheadrightarrow W_{\!n}(A)
\end{equation}
which is an isomorphism if $A=A^{p}$, i.e., 
\begin{equation}\label{tru1}
W(A)/(\e p^{\le n})=W_{\!n}(A) \quad\text{if $A=A^{p}\e$}.
\end{equation}
In this case, the inverse $(t_{\le n}^{\e A})^{-1}\colon  W_{\!n}(A)\to W(A)/(\e p^{\le n})$ is given by $(a_{0},a_{1},\dots,a_{n-1})\mapsto(a_{0},a_{1},\dots,a_{n-1},0,\dots)+ (\e p^{\le n})$.
We now observe that \eqref{fvp} induces a map
\begin{equation*} 
p\colon W_{\! n+1}(A)\to W_{\! n+1}(A), (a_{0},\dots,a_{\le n})\mapsto (0, a_{0}^{p},\dots,a_{\le n-1}^{p}),
\end{equation*}
such that
\begin{equation*} 
p^{n}(a_{0},\dots,a_{\le n})=\big(0,\dots,0,a_{0}^{\le p^{n}}\big)=V^{\le n}\be\big(a_{0}^{\le p^{n}}\big)
\end{equation*}
($n$ zeroes).  Consequently,
\begin{equation}\label{opr}
p^{\le n}W_{n+1}(A)=V^{\le n}W_{1}(A) \qquad\text{if $A=A^{p}\e$}.
\end{equation}

Now let $k$ be a perfect field of positive characteristic $p$. Then $W(k)$ is an absolutely unramified complete discrete valuation ring with maximal ideal $(\e p)=p\le W(k)$ and residue field $k$.  For every $n\geq 1$, $W_{\!n}\lbe(\lbe k)\simeq W\lbe(\lbe k)/(\e p^{\le n})$ is an artinian local ring with maximal ideal $p\le W_{\!n}\lbe(\lbe k)$ and residue field $k$. Let $A$ be a $k$-algebra and recall the homomorphism $t_{\lbe n}^{\le k}$ \eqref{tru}. Composing $(t_{\lbe n}^{\le k})^{-1}\!\otimes\be 1_{W\lbe(\lbe A)}\colon W_{\!n}(k)\otimes_{\le W(k)}W(A)\overset{\!\sim}{\to} W(k)/(\e p^{\le n})\otimes_{\e W(k)}W(A)$ with the canonical isomorphism $W(k)/(\e p^{\le n})\otimes_{\e W(k)}W(A)\simeq W(A)/(\, p^{\e n})$, we obtain an isomorphism $W_{\!n}(k)\otimes_{\le W(k)}W(A)\simeq W(A)/(\, p^{\e n})$ which, when composed with \eqref{tru}, yields a surjective homomorphism of $W_{\! n}(k)$-algebras
\begin{equation}\label{wepi}
W_{\!n}(k)\otimes_{\e W(k)}\be W(A)\twoheadrightarrow W_{\!n}(A).
\end{equation}
If $A=A^{p}$, the preceding map is an isomorphism, i.e.,
\begin{equation}\label{wepi1}
\hskip 6em W_{\!n}(k)\otimes_{\e W(k)}\be W(A)\simeq W_{\!n}(A)   \quad\text{if $A=A^{p}\e$}.
\end{equation}
Explicitly, \eqref{wepi} is induced by the assignment
\[
(x_{0},x_{1},\dots, x_{n-1})\otimes (a_{0},a_{1},\dots)\mapsto (b_{0},b_{1},\dots, b_{n-1}),
\]
where each $x_{i}\in k$ and $(b_{0},b_{1},\dots)=(x_{0},x_{1},\dots, x_{n-1},0,\dots)\cdot(a_{0},a_{1},\dots)\in W(A)$.

\begin{remark} \label{f-gr} For every integer $n\geq 1$, there exists a canonical isomorphism of $W\lbe(k)/(\e p^{\e n+1})$-modules
\[
W\lbe(k)/(\e p^{\le n})\overset{\!\sim}{\to} p\e (W\lbe(k)/(\e p^{\le n+1})),\e a+(\, p^{\le n})\mapsto p\e a+(\, p^{\le n+1}) \quad (a\in W\lbe(k)\e).
\]
Under the truncation isomorphisms $t_{\lbe n}^{\le k}$ and $t_{\lbe n+1}^{\le k}$\eqref{tru}, the above isomorphism corresponds to the isomorphism of $W_{\!n+1}(k)$-modules $W_{\!n}(k) \overset{\!\sim}{\to} p\e W_{\!n+1}(k)$ which maps  $(a_{0},\dots, a_{n-1})$ to $p\le(a_{0},\dots,a_{n-1},0)=(0,a_{0}^{\le p},\dots,a_{n-1}^{\le p})$.
\end{remark}

\smallskip

If $f\in A$ and $n\geq 1$, we will write $[\e f\e]=(f,0,\dots,0)\in W_{\!n}(A)$. The same notation will be used for $(f,0,\dots)\in W(A)$ when there is no risk of confusion. Let $W_{\be n}(A)_{[\e f\e]}$ denote the localization of $W_{\be n}(A)$ with respect to the multiplicative set $\{[\e f\e]^{r}\}_{r\geq\le 0}$, where $[\e f\e]^{0}=1_{n}=(1,0,\dots,0)\in W_{\!n}(A)$. Then there exists a canonical  isomorphism of $W_{\!n}(A)$-algebras
\begin{equation}\label{wloc}
W_{\! n}(A)_{[\e f\e]}\overset{\!\sim}{\to} W_{\!n}(A_{\lbe f}\lbe)
\end{equation}
which maps $(a_{0}, \dots,a_{n-1})/\le[\e f\e]^{\le r}$ to $(a_{0}, \dots,a_{n-1})\cdot [\e 1/\be f^{\le r}\e]\in W_{\!n}(A_{\lbe f}\lbe)$. See \cite[(1.1.9), p.~505, and (1.5.3), p.~512]{ill}.
 
\begin{lemma}\label{flatw} Let $A$ be a reduced $k$-algebra. Then $W\be(\be A)$ is flat over $W\be(k)$.
\end{lemma}
\begin{proof} By \cite[Corollary 2.14, p.~11]{liu}, it suffices to check that $W(A)$ is $W(k)$-torsion free, i.e., that $p$ is not a zero divisor in $W(A)$. If $(a_{0},a_{1},\dots)\in W(A)$, then $p\le(a_{0},a_{1},\dots)=(0,a_{0}^{\le p},a_{1}^{\le p},\dots)$, which is zero if, and only if, $(a_{0},a_{1},\dots)=0$.
\end{proof}
 
\smallskip

Let $\mathbb W_{\be n}$ (respectively, $\mathbb W\e$) denote the $k$-ring scheme of  Witt vectors of length $n\geq 1$ (respectively, of infinite length) defined in, e.g., \cite[V, \S1]{dg}. The $k$-scheme underlying $\mathbb  W_{\be n}$ is $\A_{k}^{\!n}$. Further, for every $k$-algebra $A$, 
\[
\mathbb  W_{\be n}(\spec A)=\Hom_{k}(\spec A,\mathbb  W_{\be n})=W_{\!n}(A).
\]
Similarly, $\mathbb W(\spec A)=W(A)$.

If $Y$ is a $k$-scheme, we let
$W_{\!n}(Y)=(\e|Y|, W_{\!n}(\cO_{Y}\be))$ be the $W_{\!n}(k)$-scheme defined in \cite[\S 1.5]{ill}. Here 
$W_{\!n}(\cO_{Y}\be)$ is the (Zariski) sheaf $U\mapsto
W_{\!n}(\cO_{Y}(U\le))$ on $Y$. We have $W_{\be 1}(Y)=Y$ and, for every $k$-algebra
$A$, $W_{\!n}(\spec A)=\spec W_{n}(A)$. The infinite-length variant
of this construction will be denoted by $W(Y)$.

\begin{lemma}\label{ww} Let $f=a_{\le r}+\dots+a_{1}T^{\e r-1}+T^{\le r}\in W\lbe(k)[\le T\le]$ be an Eisenstein polynomial, i.e., $p\!\mid\! a_{i}$ for all $i$ and $p^{2}\!\nmid\! a_{\le r}$. Then, for every $k$-algebra $A$ such that $A=A^{p}$, there exists a canonical isomorphism of $W\lbe(\be A)$-algebras
\[
\varprojlim  W\lbe(\lbe A)[\le T\le]/(f, T^{\e n})\simeq W\lbe(\lbe A)[\le T\le]/(f\le).
\]
\end{lemma}
\begin{proof} Since $f-T^{\e r}\in p\le (W(k)[[\le T\le]])^{\times}$, we have $(f,T^{\le n\le r})=(f,p^{\le n})\subseteq W(k)[[\le T\le]]$ for every $n\geq 1$. Thus, by \eqref{tru1},
\[
\begin{array}{rcl}
W\lbe(\lbe A)[\e T\e]/(f, T^{\e n\le r})=
W\lbe(\lbe A)[[ \e T\e]]/(f, T^{\e n\le r})&=& 
W\lbe(\lbe A)[[\e T\e]]/ (f, p^{\e n})\\
&=& W\lbe(\lbe A)[\e T\e]/(f, p^{\e n})\simeq W_{\be n}(A)[\le T\le]/(f\e).
\end{array}
\]
On the other hand, since the maps $f\e W_{\be n+1}(A)[T]\to f\e W_{\be n}(A)[T]$ are surjective, we have
\[ 
\varprojlim  W\lbe(\lbe A)[\le T\le]/(f, T^{\e n})\simeq  W\lbe(\lbe A)\langle T\rangle/(f\le),
\]
where $W\lbe(\lbe A)\langle\e T\e\rangle=\varprojlim  W_{\be n+1}\lbe(\be A)[\e T\e]\subset W\lbe(\lbe A)[[\le T\le]]$ is the algebra of restricted power series over $W\lbe(\be A)$. Now, by \cite[Theorem 11, p.~406]{sal}, the inclusion $W\lbe(\lbe A)[\le T\le]\to W\lbe(\be A)\langle\e T\e\rangle$ induces an isomorphism of $W\lbe(\lbe A)$-algebras $ W\lbe(\lbe A)\langle\e T\e\rangle/(f\le)\simeq W\lbe(\lbe A)[\le T\le]/(f\le)$, whence the lemma follows.
\end{proof}

\subsection{Groups of components} 
If $G$ is a group scheme locally of finite type over an artinian local ring $A$, we will write $G^{\e 0}$ for the identity component of $G$ as defined in \cite[${\rm{VI}}_{\rm A}$, \S2.3]{sga3}. Thus $G^{\e 0}$ is a normal, open and closed subgroup scheme of $G$. Further, if $G$ is flat over $A$, then the quotient fppf sheaf
\begin{equation}\label{pio}
\pi_{0}(G\le)=G^{\e 0}\e\backslash G 
\end{equation}
is represented by an \'etale $A$-group scheme. Further, the canonical morphism 
\begin{equation}\label{pim}
p_{\le G}\colon G\to \pi_{0}(G\le)
\end{equation}
is faithfully flat and locally of finite presentation. In particular, if $A=k$ is a field and $\kbar$ is an algebraic closure of $k$, then
$p_{\le G}\be\big(\e\kbar\e\big)\colon G\big(\e\kbar\e\big)\to \pi_{0}(G\le)\big(\e\kbar\e\big)$ is surjective, i.e., 
$\pi_{0}(G\le)\big(\e\kbar\e\big)=G^{\e 0}\lbe\big(\e\kbar\e\big)\e\backslash G\big(\e\kbar\e\big)$. We also note that, if $G$ is smooth over $A$, then $G^{\e 0}$ is smooth as well and (consequently) \eqref{pim} is a smooth morphism.  See \cite[${\rm{VI}}_{\rm A}$, \S2, and ${\rm{VI}}_{\rm B}$, \S3]{sga3} for more details.

\begin{lemma} 
Let $A\to B$ be a homomorphism of artinian local rings and let $G$ be a flat $A$-group scheme locally of finite type. Then there exists a canonical isomorphism of \'etale $B$-group schemes $\pi_{0}(G_{\be B}\le)\simeq\pi_{0}(G\le)_{B}$.
\end{lemma}
\begin{proof} Note that $\pi_{0}(G_{\be B}\le)$ is an \'etale $B$-group scheme since $G_{\be B}$ is flat and locally of finite type over $B$. Now, if $\varepsilon_{A}\colon \spec A\to \pi_{0}(G\le)$ denotes the unit section of $\pi_{0}(G\le)$, then the unit section of $\pi_{0}(G\le)_{B}$ is $(\varepsilon_{A})_{B}\colon (\spec A)_{B}\to \pi_{0}(G\le)_{B}$. Further, $(p_{\le G})_{\be B}\colon G_{\be B}\to \pi_{0}(G\le)_{B}$ is faithfully flat and locally of finite presentation and its kernel equals $G_{\be B}\times_{\pi_{0}(G\le)_{B}}(\spec A)_{B}=(G^{\e 0})_{\be B}$. On the other hand, there exists a canonical isomorphism of $B$-group schemes $(G^{\e 0})_{\lbe B}=(G_{\be B})^{0}$ by \cite[$\text{VI}_{\text{B}}$, Proposition 3.3]{sga3}. Thus there exists a canonical sequence of $B$-group schemes
\[
1\to (G_{\be B})^{0}\to G_{\be B}\to \pi_{0}(G\le)_{\lbe B}\to 1
\]
which is exact for the fppf topology on $B$. The lemma follows.
\end{proof}

Henceforth, we will make the identification
\[
\pi_{0}(G_{\be B}\le)=\pi_{0}(G\le)_{B}
\]
via the isomorphism of the lemma.

Now let $S$ be {\it any} scheme and let $G$ be an $S$-group scheme. For each $s\in S$, let 
$|G_{\lbe s}\le|^{0}$ denote the identity component of the $k(s)$-group scheme $G_{\lbe s}=G\times_{S}\spec k(s)$ and consider the functor defined by
\begin{equation}\label{idcomp2}
G^{\e 0}(T)=\left\{u\in G(T)\colon \forall s\in S, u_{\le s}(|\le T_{s}\le|)\subseteq |G_{\lbe s}|^{0}\right\},
\end{equation}
where $T$ is an $S$-scheme. If $G$ is smooth, then \eqref{idcomp2} is represented by an open and smooth subgroup scheme $G^{\e 0}$ of $G$ whose fibers are geometrically connected by \cite[${\rm{VI}}_{\rm A}$, Proposition 2.1.1, and ${\rm{VI}}_{\rm B}$, Proposition 3.3 and Theorem 3.10]{sga3}. Note that, in general, $G^{\e 0}$ is not a closed subgroup scheme of $G$ and $G^{\e 0}\e\backslash G $ is not represented by a scheme

\subsection{The connected-\'etale sequence}\label{c-et}
Let $R$ be a henselian local ring with residue field $k$ and let $F$ be a finite and flat $R$-group scheme. Then $F=\spec A$, where $A$ is a finite $R$-algebra. By \cite[I, \S1, Propositions 1 and 3]{ray}, there exists a canonical isomorphism of rings $A=\prod_{\le i=1}^{\le r}A_{\le i}$, where $r\geq 1$ is an integer and each $A_{\le i}$ is a local ring. Consequently, $F=\coprod_{\le i=1}^{\le r}\spec A_{\le i}$. Now, since $\spec R$ is connected, there exists a unique $i_{0}\in\{1,\dots,r\le\}$ such that the unit section $\spec R\to F$ factors as 
$\spec R\to\spec A_{\le i_{0}}\to F$, where the second morphism is induced by the projection $\prod_{\le i=1}^{\le r}A_{\le i}\to A_{\le i_{0}}$. It is shown in \cite[pp.~ 139-140]{ta} that $\spec A_{\le i_{0}}$ is a flat, normal, open and closed $R$-subgroup scheme of $F$. We will use the notation $F^{\e\circ}=\spec A_{\le i_{0}}$ (rather than the standard $F^{\e 0}=\spec A_{\le i_{0}}$) since there do exist examples where $F^{\e\circ}$ has a disconnected generic fiber and therefore does not represent the functor  \eqref{idcomp2} introduced above. Now, by \cite[p.~140]{ta}, the quotient $F^{\e\et}=F/F^{\le\circ}$ is an \'etale $R$-group scheme. The induced sequence of $R$-group schemes 
\begin{equation}\label{c-e-k}
1\to F^{\le\circ}\to F\to F^{\e \et}\to 1
\end{equation}
is exact for both the fppf and fpqc topologies on $(\mr{Sch}/R\le)$ \cite[Lemma 2.3]{bga}. Note that the special fiber of the preceding sequence is the canonical sequence of $k$-group schemes 
\[
1\to F^{\e 0}_{\be k}\to F_{\be k}\to F_{\be k}^{\le \et}\to 1.
\]
See \cite[II, \S5, no. 1, Proposition 1.8, p.~237]{dg}.

\subsection{Formal schemes}\label{for-sch}
In this paper we need to consider (certain types of) non-noetherian formal schemes. Since the standard reference for the theory of formal schemes, namely \cite[\S10]{ega1}, is not entirely satisfactory in a non-noetherian setting, we will instead rely on \cite[Chapter 2]{ab} and \cite[Chapter I]{fk}. Unfortunately, these two equally-useful references attach different meanings to the term ``adic formal scheme". In order to avoid confusion, we will follow exclusively the terminology of \cite[Chapter I]{fk}, which is compatible with that of \cite[\S10]{ega1}. For the convenience of the reader, references to \cite[Chapter 2]{ab} below will be accompanied by a reference to the appropriate entry from the following dictionary

\begin{remarks}\label{term}\indent
\begin{enumerate}
\item[(a)] In \cite[Chapter I, \S1.8]{ab}, the adic rings of \cite[Definition 7.1.9, p.~172]{ega1} are called ``preadic, complete and separated''.
\item[(b)] In \cite[Chapter 2]{ab}, the adic formal schemes of \cite[Definition 10.4.2, p.~407]{ega1} are called ``preadic formal schemes'' \cite[Definition 2.1.16, p.~121]{ab}.
\item[(c)] In \cite[Chapter 2]{ab}, an ``adic formal scheme" is an adic formal scheme in the sense of \cite[Definition 10.4.2, p.~407]{ega1} with the additional property that it has a finitely generated ideal of definition. See \cite[Definition 2.1.24, p.~123]{ab} and (b) above. Thus the ``adic formal schemes" of \cite[Chapter 2]{ab} correspond to the objects we call below {\it adic formal schemes globally of finite ideal type} (see Definition \ref{gfit}).
\end{enumerate}
\end{remarks}

\smallskip

Unadorned limits in this Subsection are indexed by $\N$.

\smallskip

An adic formal scheme \cite[Definition 10.4.2, p.~407]{ega1} $\m X$ is said to be {\it of finite ideal type} if there exists an open covering $\m X=\bigcup_{\le\alpha}{\m U}_{\e\alpha}$ where each $\m U_{\e\alpha}$ is 
isomorphic to a formal spectrum $\spf A_{\le\alpha}$ for some adic ring $A_{\le\alpha}$ which has a finitely generated ideal of definition\,\footnote{See \cite[Definition 1.1.16, p.~260]{fk}. Perhaps it would be more appropriate to call these schemes ``adic formal schemes {\it locally} of finite ideal type" but, as stated above, we will adopt the terminology introduced in \cite{fk}.}.
Clearly, any locally noetherian formal scheme is an adic formal scheme of finite ideal type.

\begin{definition}\label{gfit} An adic formal scheme $\m X$ is said to be {\it globally of finite ideal type} if it has an ideal of definition of finite type $\s I\subseteq\s O_{\m X}$.
\end{definition}

By the discussion following \cite[Proposition 1.1.19, p.~261]{fk} and \cite[Proposition 2.1.11, p.~119]{ab} (see also Remark \ref{term}(a)), any adic formal scheme which is globally of finite ideal type is of finite ideal type. 

By \cite[Corollary 3.7.13, p.~309]{fk}, an adic affine formal scheme $\m X\simeq\spf A$ is of finite ideal type if, and only if, $A$ is an adic ring which has a finitely generated ideal of definition. If this is the case, and if $I$ is a finitely generated ideal of definition of $A$, then $\s I\overset{\rm def.}{=}I^{\Delta}\subseteq\s O_{\m X}$ is an ideal of definition of finite type of $\m X$ \cite[Proposition 2.1.11, p.~119]{ab} (see also Remark \ref{term}(a)). Thus, any adic affine formal scheme of finite ideal type is globally of finite ideal type. More generally, any quasi-compact and quasi-separated adic formal scheme of finite ideal type is globally of finite ideal type\,\footnote{See \cite[Definitions 1.6.1 and 1.6.5, pp.~276-277; Corollary 3.7.12, p.~309, Definition 1.6.6, p.~277, and comment after this definition]{fk} for the fact that any affine formal scheme is quasi-compact and quasi-separated.}. Further, by \cite[Proposition
10.5.4, p.~410]{ega1}, any locally noetherian formal scheme is, in fact, {\it globally} of finite ideal type.

\smallskip

A morphism $u\colon\m X\to \m S$ of adic formal schemes globally of finite ideal type is said to be {\it adic} if there exists an ideal of definition of finite type $\s I$ of $\m S$ such that $u^{*}\lbe(\be\s I\lbe)\s O_{\m X}$ is an ideal of definition (clearly of finite type) of $\m X$ (see \cite[(4.3.5), p.~98]{ega1} for the definition of $u^{*}\lbe(\be\s I\lbe)\s O_{\m X}$). We then say that $\m X$ is an {\it adic formal $\m S$-scheme} or that $\m X$ is {\it adic over $\m S$}. See \cite[Definition 2.2.7, p.~128]{ab} (see also Remark \ref{term}(c)) and \cite[comment after Definition 1.3.1, p.~266]{fk}. If $\m X$ and $\m Y$ are two adic formal $\m S$-schemes, then every $\m S$-morphism $\m X\to\m Y$ is adic (cf. \cite[end of (10.12.1), p.~437]{ega1}). For any $\m S$ as above, we will write $(\text{Ad-For}/\m S)$ for the category of adic formal $\m S$-schemes.

\smallskip

\begin{remark}\label{ns-cat} If $\m S$ is a locally noetherian formal scheme, $(\text{Ad-For}/\m S)$ contains (as a full subcategory) the category of locally noetherian adic formal $\m S$-schemes considered in \cite[10.12]{ega1}. The latter category contains, in turn, the category of formal $\m S$-schemes which are locally of topologically finite type, as follows from \cite[Proposition 7.5.2(ii), p.~181]{ega1}. We conclude that the categories of formal schemes considered in \cite{ns} and \cite{bert} are full subcategories of $(\text{Ad-For}/\m S)$.
\end{remark}

\smallskip

Every adic formal scheme globally of finite ideal type can be represented as an inductive limit of schemes. Indeed, let $\m X$ be an adic formal scheme globally of finite ideal type and let $\s I$ be an ideal of definition of finite type of $\m X$. By \cite[Corollary 1.1.24, p.~263]{fk},  $\{\s I^{n}\}_{n\le\in\le\N}$ is a fundamental system of ideals of definition of finite type of $\m X$. Then, by \cite[Proposition 10.6.2, p.~412]{ega1}, $\m X$ is the inductive limit, in the category of formal schemes, of the schemes $\m X_{n}=(\m X,\s O_{\m X}/\be\s I^{n})$ as $n$ runs over $\N$, where the transition morphisms in the indicated limit are the canonical closed immersions.

\smallskip

Let $\m S$ be an adic formal scheme globally of finite ideal type and let $\s K$ be an ideal of definition of finite type of $\m S$. For every $n\in\N$, set $\m S_{n}=(\m S,\s O_{\m S}/\s K^{n})$, which is a scheme. An inductive system $(X_{n})$ of $\m S_{n}$-schemes is said to be an {\it adic inductive $(\m S_{n})$-system} if the structural morphisms $u_{n}\colon X_{n}\to \m S_{n}$ are such that, for every $m\leq n$, the square
\begin{equation}\label{cart0}
\xymatrix{X_{m}\ar[d]\ar[r]^{u_{m}}& \m S_{m}\ar[d]\\
X_{n}\ar[r]^{u_{n}}& \m S_{n},}
\end{equation}
is cartesian (whence $X_{m}$ can be identified with $X_{n}\!\times_{\m S_{n}}\!\m S_{m}$). The adic inductive $(\m S_{n})$-systems form a category  denoted by $\mathrm{Ad}\e$-$\e\mr{Ind}(\m S)$: a morphism between such systems $(X_{n})\to (Y_{n})$ is an inductive system of $\m S_{n}$-morphisms $f_{n}\colon X_{n}\to Y_{n}$ such that 
$f_{m}=f_{n}\be\times_{\m S_{n}}\be\m S_{m}$ for every $m\leq n$. The latter category is canonically equivalent to the category
$(\text{Ad-For}/\m S)$ of adic formal $\m S$-schemes. The equivalence is obtained as follows. To each object $u\colon\m X\to\m S$ of $(\text{Ad-For}/\m S)$, we associate the inductive system of the $\m S_{n}$-schemes  $\m X_{n}=(\m X,\s O_{\m X}/\!\s J^{\be n})$, where $\s J=u^{*}(\be\s K\lbe)\s O_{\m X}$ and the structural morphism $u_{n}\colon \m X_{n}\to \m S_{n}$ is determined by $u$ via \cite[Proposition 10.5.6(i), p.~410]{ega1}. Note that each of the transition morphisms of the system $(\m X_{n})$ is a nilpotent immersion and therefore a universal homeomorphism. If $v\colon \m Y\to\m S$ is another object of $(\text{Ad-For}/\m S)$, set $\s I=v^{*}(\be\s K\le)\s O_{\m Y}$ and $\m Y_{n}=(\m Y,\s O_{\m Y}/\s I^{n})$ for every $n\in\N$. Then to each $\m S$-morphism $f\colon\m X\to \m Y$ there corresponds a morphism $(\m X_{n})\to (\m Y_{n})$ of adic inductive $(\m S_{n})$-systems. Conversely, given an adic inductive $(\m S_{n})$-system $(X_{n})$ with associated sequence of structural morphisms $(u_{n})$, there exists an adic formal scheme $\m X$ such that $X_{n}=\m X_{n}=(\m X,\s O_{\m X}/\be\s J^{\lbe n})$, where $\!\s J\!$  is an ideal of definition of finite type of $\m X$, and the sequence $(u_{n})$ defines a morphism $u\colon \m X\to \m S$ which satisfies $\s J=u^{*}(\be\s K\le)\s O_{\m X}$ (whence $\m X$ is adic over $\m S$). Further, to a morphism of adic inductive $(\m S_{n})$-systems $(X_{n})\to (Y_{n})$ there corresponds an $\m S$-morphism $\m X\to\m Y$. The preceding equivalence yields, for two adic formal $\m S$-schemes $\m X$ and $\m Y$, a canonical bijection
\begin{equation}\label{for-hom}
\Hom_{\m S}(\m X,\m
Y)\simeq\displaystyle\varprojlim \Hom_{\e
S_{n}}\be(\m X_{n},\m Y_{n}),
\end{equation}
where the transition maps in the projective limit are $v_{n}\mapsto v_{n}\times_{S_{n}}S_{m}$ for $m\leq n$ (cf. \cite[(10.12.3.2), p.~438]{ega1}). See \cite[proof of Proposition 2.2.14, p.~130]{ab} for more details.

\smallskip

\begin{lemma} \label{for-prop} Let $\m S$ be an adic formal scheme globally of finite ideal type, let $\s K$ be an ideal of definition of finite type of $\m S$ and set $\m S_{n}=(\m S,\s O_{\m S}/\s K^{n})$, where $n\in\N$. Let $(f_{n})\colon (\m X_{n})\to (\m Y_{n})$ be a morphism of adic inductive $(\m S_{n})$-systems and let $f\colon \m X\to\m Y$ be the corresponding morphism of adic formal $\m S$-schemes, as above. Consider, for a morphism of formal schemes, the property of being:
\begin{enumerate}
\item[(i)] quasi-compact;
\item[(ii)] quasi-separated;
\item[(iii)] separated;
\item[(iv)] an open immersion;
\item[(v)] a closed immersion;
\item[(vi)] affine.
\end{enumerate}
If $\bm{P}$ denotes one of the above properties, then $f\colon \m X\to\m Y$ has property $\bm{P}$ if, and only if, $f_{n}\colon \m X_{n}\to \m Y_{n}$ has property $\bm{P}$ for every  $n\in\N$.
\end{lemma}
\begin{proof} For properties (i) and (ii), see \cite[Proposition 1.6.9, p.~279]{fk}. For property (iii), see \cite[Proposition 4.6.9, p.~326]{fk}. For properties (iv) and (v), see \cite[Proposition 4.4.2, p.~321]{fk}. For property (vi), see \cite[Proposition 4.1.12, p.~314]{fk}.
\end{proof}

A morphism of formal schemes $f\colon\m X\to\m Y$ is said to be {\it formally \'etale} \cite[Definition 2.4.1, p.~139]{ab} if, for every affine scheme $Y$, nilpotent immersion $i\colon Y_{0}\to Y$ and morphism of formal schemes $Y\to\m Y\le$\,\footnote{ \,Here $Y$ and $Y_{0}$ are being regarded as adic formal schemes with ideal of definition $0$. See \cite[Remark 1.1.15, p.~260]{fk}.}, the map
\[
\Hom_{\e\m Y}(Y,\m X)\to \Hom_{\e\m Y}(Y_{0},\m X),
\]
induced by $i$, is a bijection.

\begin{lemma} \label{for-et} Let $\m S$ be an adic formal scheme globally of finite ideal type, let $\s K$ be an ideal of definition of finite type of $\m S$ and set  $\m S_{n}=(\m S,\s O_{\m S}/\s K^{n})$, where $n\in\N$. Let $(f_{n})\colon (\m X_{n})\to (\m Y_{n}\lbe)$ be a morphism of adic inductive $(\m S_{n})$-systems and let $f\colon \m X\to\m Y$ be the corresponding morphism of adic formal $\m S$-schemes. Then $f$ is formally \'etale if, and only if, $f_{n}\colon \m X_{n}\to \m Y_{n}$ is formally \'etale for every $n\in\N$.
\end{lemma}
\begin{proof} The proof is similar to the proof of \cite[Proposition 2.4.8, p.~140]{ab}.
\end{proof}

\bigskip

Let $S$ be a scheme and let $Z$ be a closed subscheme of $S$ defined by a quasi-coherent ideal of finite type $\s I\subseteq\s O_{S}$. Let $\widehat{S}= S_{/\le Z}$ be the formal completion of $S$ along $Z$. Then $\widehat{S}$ is an adic formal scheme globally of finite ideal type and $\widehat{\s I}=\s I_{/ Z}\subseteq \s O_{\widehat{S}}$  is an ideal of definition of finite type of $\widehat{S}$ \cite[Proposition 2.5.2(i), p.~145]{ab} (see also Remark \ref{term}(c)). 
Now let $f\colon X\to S$ be an $S$-scheme. Then $X\times_{S}Z$ is canonically isomorphic to the inverse image $f^{-1}(Z)$ of $Z$ by $f$. The latter is the closed subscheme of $X$ defined by the quasi-coherent ideal of finite type $f^{*}\lbe(\be\s I)\le\s O_{X}$. Thus $\widehat{X}=X_{/\le f^{-1}(Z)}$ is an adic formal scheme globally of finite ideal type which has $(f^{*}\lbe(\be\s I)\s O_{X})_{/\le f^{-1}(Z)}$ as an ideal of definition of finite type. Further, the morphism $\widehat{f}\colon \widehat{X}\to \widehat{S}$ induced by $f$ is adic. Indeed, $\big(\widehat{f}\,\,\big)^{\be *}\lbe\big(\lbe\widehat{\s I}\le\,\,\big)\s O_{\widehat{X}}=(f^{*}\lbe(\be\s I)\s O_{X})_{/\le f^{-1}(Z)} $ by \cite[2.5.10, p.~149, lines 9-10]{ab}. Consequently, if $\m S=\widehat{S}$, then  $\widehat{X}$ is an object of 
$(\text{Ad-For}/\m S)$. Further, by \cite[Corollary 10.9.9, p.~426]{ega1}, there exists a canonical isomorphism of adic formal $\m S$-schemes
\begin{equation}\label{hat1}
\widehat{X}\,=X\times_{S}\widehat{S}.
\end{equation}

Assume now that $S=\spec A$ is an affine scheme and $Z$ is defined by $\s I=\widetilde{I}$, where $I$ is a finitely generated ideal of $A$. Then $\widehat{S}=\spf\widehat{A}$, where  $\widehat{A}=\varprojlim A/I^{\le n}$ is the $I$-adic completion of $A$ \cite[Proposition 10.8.3, p.~419]{ega1}. Further, by \cite[Proposition 2.5.2(i), p.~145]{ab}, for every $n\in\N$ there exists a canonical isomorphism of schemes $\big(\widehat{S}\,\big)_{\be n}=S_{n}$, where $S_{n}=\spec(A/I^{\le n})$. Consequently, $\widehat{S}$ is canonically isomorphic to $\varinjlim S_{n}$ and \eqref{hat1} yields an isomorphism of adic formal $\m S$-schemes $\widehat{X}\,=\varinjlim\,(X\times_{S}S_{n})$.

\subsection{Weil restriction}\label{wres}

Let $f\colon S^{\e\prime}\to S$ be a morphism of schemes and let $X^{\prime}$ be an $S^{\le\prime}$-scheme. We will say that {\it the Weil restriction of $X^{\prime}$ along $f$ exists} if the contravariant functor
$(\mathrm{Sch}/S)\to(\mathrm{Sets}), T\mapsto\Hom_{
S^{\le\prime}}(T\times_{S}S^{\le\prime},X^{\le\prime}\le)$, is  representable, i.e., if there exists a pair $(\Re_{S^{\le\prime}\be/S}(X^{\prime}\e), q_{\le X^{\prime}\!,\e S^{\le\prime}\be/S})$, where $\Re_{S^{\le\prime}\be/S}(X^{\prime}\e)$ is an $S$-scheme and 
\begin{equation*} 
q_{\le X^{\prime}\be,\e S^{\le\prime}\be/S}\colon \Re_{S^{\le\prime}\be/S}(X^{\prime}\e)_{S^{\le\prime}}\to X^{\prime}
\end{equation*}
is an $S^{\e\prime}$-morphism of schemes, such that the map
\begin{equation}\label{wr}
\Hom_{\le S}\e(T,\Re_{S^{\le\prime}\be/S}(X^{\le\prime}\e))\overset{\!\sim}{\to}\Hom_{
S^{\le\prime}}(\e T\!\times_{S}\!S^{\e\prime},X^{\le\prime}\e), \quad g\mapsto q_{\le X^{\prime}\be,\e S^{\le\prime}\be/S}\circ g_{\le S^{\le\prime}} 
\end{equation}
is a bijection. The pair $(\Re_{S^{\le\prime}\be/S}(X^{\prime}\e), q_{\le X^{\prime}\!,\e S^{\le\prime}\be/S})$ (or, more concisely, the scheme $\Re_{S^{\le\prime}\be/S}(X^{\prime}\e)$) is called the {\it Weil restriction of $X^{\prime}$ along $f$}. If $S^{\le\prime}=\spec B$ and $S=\spec A$ are affine, we will write $(\Re_{B/A}(X^{\prime}\le),q_{\le X^{\prime}\!,\e B\lbe/A})$ for $(\Re_{S^{\le\prime}\be/S}(X^{\prime}\le),q_{\le X^{\prime}\!,\e S^{\le\prime}\lbe/S})$.

It follows from the above definition that $\Re_{S^{\le\prime}\be/S}$ is compatible with fiber products. In particular, if $X^{\prime}$ is an $S^{\e\prime}$-group scheme such that $\Re_{S^{\le\prime}\be/S}(X^{\prime}\le)$ exists, then $\Re_{S^{\le\prime}\be/S}(X^{\prime}\le)$ is an $S$-group scheme. On the other hand, if
$\Re_{S^{\le\prime}\be/S}(X^{\prime}\le)$ exists and $T\to
S$ is a morphism of schemes, then $\Re_{\le S^{\le\prime}_{T}\lbe/ T}\le(\e X^{\prime}\!\times_{S^{\le\prime}}\be S^{\e\prime}_{\le T})$ exists as well and \eqref{bc} and \eqref{wr} yield a canonical isomorphism of $T$-schemes 
\begin{equation}\label{wrbc}
\Re_{S^{\le\prime}\be/S}(X^{\prime}\le)\times_{S}T\overset{\!\sim}{\to}\Re_{\le S^{\le\prime}_{T}\lbe/ T}\le(\e X^{\prime}\!\times_{S^{\prime}}\! S^{\e\prime}_{\le T}).
\end{equation}
Moreover, if $S^{\le\prime\prime}\to S^{\le\prime}\to S$ are morphisms of schemes and  $X^{\prime\prime}$ is an $S^{\le\prime\prime}$-scheme such that $\Re_{\e S^{\prime\prime}\be/S^{\le\prime}\be}\e(X^{\prime\prime}\e)$ exists, then $\Re_{S^{\prime}\be/S}(\Re_{\e S^{\prime\prime}\be/S^{\le\prime}\be}\e(X^{\prime\prime}\e))$ exists if, and only if,  $\Re_{\e S^{\prime\prime}\!/S} (X^{\prime\prime}\e)$ exists. If this is the case, then there exists a canonical isomorphism of $S$-schemes 
\begin{equation}\label{wrcomp}
\Re_{S^{\le\prime}\be/S}(\le\Re_{\e S^{\prime\prime}\be/S^{\le\prime}\be}\e(X^{\prime\prime}\e))\overset{\!\lbe\sim}{\to}\Re_{\e S^{\prime\prime}\be/S} (X^{\prime\prime}\e)
\end{equation}

We now discuss existence results.

Let $f\colon S^{\e\prime}\to S$ be a finite and locally free morphism of schemes. Since $f$ is affine, there exists an isomorphism of $S$-schemes $S^{\e\prime}=\spec\s A\lbe(\lbe S^{\e\prime})$, where the quasi-coherent $\mathcal O_{\lbe S}\le$-algebra $\s A\lbe(\lbe S^{\e\prime}\le)=f_{*}\mathcal O_{\lbe S^{\prime}}$ is a locally free $\mathcal O_{\lbe S}\le$-module of finite rank. For every $s\in S$, let $n(f;s)$ denote the rank of the free $\mathcal O_{\lbe S,s}\e$-\e module $\s A\lbe(\lbe S^{\e\prime})_{s}$ and let $j\colon \spec k(s)\to S$ denote the canonical morphism. By \cite[Corollary 9.1.9, p.~356]{ega1}, if $s\in S$ is such that $f^{-1}(s)\neq\emptyset$ and $A(s)$ is the finite $k(s)$-algebra $j^{*}\be\s A\lbe(\lbe S^{\e\prime})$, then $S^{\e\prime}\times_{S}\spec k(s)=\spec A(s)$. Note that $A(s)$ is an artinian ring \cite[\S8, Exercise 3, p.~92]{am}. As a $k(s)$-module, $A(s)$ is isomorphic to $\s A\lbe(\lbe S^{\e\prime})_{\lbe s}\otimes_{\e\mathcal O_{\lbe S, s}}k(s)$ \cite[(2.5.8), p.~225]{ega1}, whence ${\rm dim}_{\e k(s)}\e A(s)=n(f;s)$. Now let $\overline{k(s)}$ be a fixed algebraic closure of $k(s)$, write $\overline{A(s)}=A(s)\otimes_{\e k(s)}\overline{k(s)}$ and let
\begin{equation}\label{gfs}
\gamma(f;s)=\#\big(S^{\e\prime}\times_{S}\spec \overline{k(s)}\,\big)=\#\big(\spec\overline{A(s)}\,\big).
\end{equation}
Thus $\gamma(f;s)$ is the cardinality of the geometric fiber of $f$ over $s$. By \cite[$\text{IV}_{2}$, Proposition 4.5.1]{ega}, $\gamma(f;s)$ is independent of the choice of $\overline{k(s)}$. Note that, since $\#\big(\le\spec\overline{A(s)}\e\big)=\#\big(\le\spec\overline{A(s)}_{\text{red}}\e\big)$ and
$\overline{A(s)}_{\text{red}}$ is isomorphic to a finite product of copies of $\overline{k(s)}$ (cf. \cite[proof of Theorem 8.7, p.~90]{am}), we have 
$\overline{A(s)}_{\text{red}}\simeq\prod_{\e i=1}^{\e\gamma(f;s)}\,\overline{k(s)}$. Consequently
\begin{equation}\label{gfs2}
\gamma(f;s)={\rm dim}_{\e \overline{k(s)}}\,\overline{A(s)}_{\text{red}}\leq {\rm dim}_{\e \overline{k(s)}}\,\overline{A(s)}=n(f;s).
\end{equation}
Clearly, $\gamma(f;s)=n(f;s)$ if $\overline{A(s)}$ is reduced.

If $S$ is a one-point scheme (e.g., $S=\spec B$, where $B$ is an artinian local ring) we will write $\gamma(f)$ for $\gamma(f;s)$, where $s\in S$ is the unique point of $S$.

\begin{remarks}\label{gpts} Let $f\colon S^{\e\prime}\to S$ be a finite and locally free morphism of schemes.
\begin{enumerate}
\item[(a)] If $s\in S$ and $K$ is an extension of $k(s)$, then $\#\big(S^{\e\prime}\times_{S}\spec K\,\big)\leq \gamma(f;s)$ by \cite[$\text{IV}_{2}$,  Proposition 4.5.1]{ega}.
\item[(b)] If $k$ is a field, $A$ is a finite \'etale $k$-algebra and $f\colon\spec A\to\spec k$ is the corresponding morphism of schemes, then $\gamma(f\e)={\rm dim}_{\e k} A$. This follows from  \eqref{gfs2} using the fact that, if $\kbar$ is an algebraic closure of $k$, then $A\be\otimes_{\e k}\be\kbar$ is reduced by \cite[V, \S6, no.7, Theorem 4, p.~A.V.34]{bou2}. In particular, if $k^{\e\prime}\be/k$ is a finite separable extension of fields and $f\colon\spec k^{\e\prime}\to\spec k$ is the corresponding morphism of schemes, then $\gamma(f\e)=[\le k^{\e\prime}\colon\be k\e]$.
\item[(c)] Let $g\colon T\to S$ be a morphism of schemes and consider the finite and locally free morphism $f\times_{S}T\colon S^{\e\prime}\times_{S}T\to T$. Let $t\in T$ and set $s=g(t)$. Then $\gamma(f\times_{S}T;t)=\#\spec \be\big(\,\overline{A(s)}\otimes_{\e\overline{k(s)}}\overline{k(t)}\,\big)$. Now, by \cite[Theorem 8.7, p.~90]{am}, we may write $\overline{A(s)}\e\otimes_{\e\overline{k(s)}}\,\overline{k(t)}=\prod_{\e i=1}^{\e \gamma(f;s)}\be\big( A_{\le i}\e\otimes_{\e\overline{k(s)}}\e\overline{k(t)}\,\big)$, 
where each $A_{\le i}$ is a local finite $\overline{k(s)}$-algebra. Since each morphism $\spec\be\big(A_{\le i}\e\otimes_{\e\overline{k(s)}}\e\overline{k(t)}\,\big)\to\spec \e\overline{k(t)}$ is a (universal) homeomorphism by Lemma \ref{uh0}, we conclude that
$\gamma(f\times_{S}T;t)=\gamma(f;s)$.
\item[(d)] Let $s\in S$ and let $g\colon T^{\e\prime}\to S^{\e\prime}$ be a universal homeomorphism such that $h=f\circ g\colon T^{\e\prime}\to S$ is finite and locally free. Then $T^{\e\prime}\times_{S}\spec \overline{k(s)}\to S^{\e\prime}\times_{S}\spec \overline{k(s)}$ is a (universal) homeomorphism. Consequently $\gamma(h\e ;s)=\gamma(f;s)$.

\item[(e)] Let $k^{\e\prime}\be/k$ be a finite extension of fields and let $k^{\e\prime}_{\sep}$ denote the separable closure of $k$ inside $k^{\e\prime}$. Note that, since   $k^{\e\prime}/k^{\e\prime}_{\sep}$ is purely inseparable, $g\colon\spec k^{\e\prime}\to\spec k^{\e\prime}_{\sep}$ is a universal homeomorphism. Now $h\colon\spec k^{\e\prime}\to \spec k$ factors as $f\circ g$, where $f\colon \spec k^{\e\prime}_{\sep}\to \spec k$ corresponds to the finite separable extension $k^{\e\prime}_{\sep}\le/\le k$. Thus, by remarks (b) and (d) above, we have $\gamma(h)=[k^{\e\prime}_{\sep}\colon \be k\e]=[k^{\e\prime}\colon\be k\e]_{\le\sep}$.
\end{enumerate}
\end{remarks}

\begin{definition}\label{adm} Let $f\colon S^{\e\prime}\to S$ be a finite and locally free morphism of schemes. An $S^{\le\prime}$-scheme $X^{\prime}$ is called {\it admissible relative to $f\e$} if, for every point $s\in S$, every collection of $\gamma(f;s\le)$ points in $X^{\le\prime}\times_{S}\e\spec k(s)$ is contained in an affine open subscheme of $X^{\le\prime}$, where $\gamma(f;s\le)$ is the integer \eqref{gfs}.
\end{definition}
If $S^{\e\prime}=\spec A$ and $S=\spec B$ are affine, we will also say that $X^{\prime}$ is {\it admissible relative to $B/A$}.

\begin{remarks}\label{rems-adm}\indent
\begin{enumerate}
\item[(a)] By \cite[II, Definition 5.3.1 and Corollary 4.5.4]{ega}, a quasi-projective $S^{\e\prime}$-scheme is admissible relative to an arbitrary finite and locally free morphism $S^{\e\prime}\to S$.
\item[(b)] If $k^{\e\prime}\be/k$ is a finite separable extension of fields and $f\colon\spec k^{\e\prime}\to\spec k$ is the corresponding finite and locally free morphism, then a $k^{\e\prime}$-scheme $X^{\prime}$ is admissible relative to $f$ if, and only if, every collection of $[\le k^{\e\prime}\colon\be k\e]$ points of $X^{\prime}$ is contained in an open affine subscheme. See Remark \ref{gpts}(b).
\item[(c)] If the geometric fibers of $f\colon S^{\e\prime}\to S$ are one-point schemes, then $\gamma(f;s\le)=1$ for every $s\in S$. Consequently, {\it every} $S^{\le\prime}$-scheme is admissible relative to $f$. This is the case, for example, if $f$ is a universal homeomorphism.
\item[(d)] If $X^{\prime}$ is admissible (relative to $f\colon S^{\e\prime}\to S$) and $Y^{\le\prime}\to X^{\prime}$ is an affine morphism of $S^{\e\prime}$-schemes, then $Y^{\le\prime}$ is admissible as well. This is immediate from Definition \ref{adm} using \cite[Proposition 9.1.10]{ega1}.
\item[(e)] If $X^{\prime}$ is an $S^{\e\prime}$-scheme which is admissible relative to $f$ and $g\colon T\to S$ is an {\it affine} morphism of schemes, then the $(S^{\e\prime}\times_{S}T\le)$-scheme $X^{\prime}\times_{S^{\e\prime}}(S^{\e\prime}\be\times_{\be S}
T\le)=X^{\prime}\times_{S}T$ is admissible relative to $f\times_{S}T\colon S^{\e\prime}\times_{S}T\to T$. Indeed, let $t\in T$, set $s=g(t)$ and let $\s C_{t}$ be a collection of $\gamma(\le f\!\times_{S}\!T;t\le)$ points in $(X^{\prime}\!\times_{S}\! T\le)\times_{T}\spec k(t)=(X^{\prime}\be\times_{S}\be\spec k(s))\times_{\spec k(s)}\spec k(t)$. Since $\gamma(f\times_{S}T;t\le)=\gamma(f;s\le)$ by Remark \ref{gpts}(c), $\s C_{t}$ defines a collection $\s C_{s}$ of at most $\gamma(f;s\le)$ points in $X^{\prime}\times_{S}\spec k(s)$. Since $\s C_{s}$ is contained in an affine open subscheme of $X^{\le\prime}$ and $X^{\prime}\times_{S}T\to X^{\prime}$ is affine, the set $\s C_{t}$ is clearly contained in an affine open subscheme of $X^{\prime}\be\times_{S}\be T$.
\item[(f)] If $X^{\prime}$ is an $S^{\le\prime}$-scheme which is admissible relative to $f\colon S^{\e\prime}\to S$ and $g\colon T^{\le\prime}\to S^{\e\prime}$ is a universal homeomorphism such that $h=f\circ g\colon T^{\e\prime}\to S$ is finite and locally free, then the $T^{\e\prime}$-scheme $X^{\prime}\times_{S^{\le\prime}}T^{\e\prime}$ is admissible relative to $h$. In effect, for every $s\in S$, since $X^{\prime}\times_{S^{\le\prime}}T^{\e\prime}\to X^{\prime}$ is a universal homeomorphism, every collection of $\gamma(h;s\le)$ points in  $(X^{\prime}\times_{S^{\prime}}T^{\e\prime})\times_{S}\spec k(s)$ defines a collection of $\gamma(f;s\le)=\gamma(h;s\le)$ points in $X^{\prime}\times_{S}\spec k(s)$ (see Remark \ref{gpts}(d)). Since the latter collection of points is contained in an open affine subscheme of $X^{\prime}$ and $X^{\prime}\times_{S^{\prime}}T^{\e\prime}\to X^{\prime}$ is affine (since a universal homeomorphism is an affine morphism), the claim follows.
\item[(g)] If $B/A$ is a finite and free extension of rings of rank $d$, then 
$\Re_{B/A}(\A^{\be n}_{\lbe B})=\A^{\be nd}_{\be A}$ for every integer $n\geq 0$. See \cite[\S7.6, proof of Theorem 4, pp.~194-195]{blr}. Choosing $n=0$ above, we obtain $\Re_{B/A}(\spec B)=\spec A$. 
\end{enumerate}
\end{remarks}

We can now strengthen \cite[\S7.6, Theorem 4, p.~194]{blr}:

\begin{theorem}\label{wr-rep} Let $f\colon S^{\e\prime}\to S$ be a finite and locally free morphism of schemes and let $X^{\prime}$ be an $S^{\e\prime}$-scheme which is admissible relative to $f$. Then $\Re_{S^{\e\prime}\be/S}(X^{\prime}\le)$ exists.
\end{theorem}
\begin{proof} See \cite[\S7.6, Theorem 4, p.~194]{blr} and note that in the last paragraph of that proof the set of points $\{z_{j}\}$ in $S^{\e\prime}\times_{S}T$ lying over a given point $z\in T$, where $g\colon T\to S$ is an arbitrary $S$-scheme, has cardinality at most $\gamma(f;s\le)$ by Remark \ref{gpts}(a), where $s=g(z)$. Thus the corresponding set of points $\{x_{j}\}\subseteq X^{\prime}$ considered in \cite[p.~195, line -14]{blr} has cardinality at most $\gamma(f;s\le)$, whence it is contained in an open affine subscheme of $X^{\prime}$ by Definition \ref{adm}. This is the condition needed in [loc.cit.] to complete that proof. 
\end{proof}

\begin{corollary}\label{wr-uh} Let $f\colon S^{\e\prime}\to S$ be a finite and locally free morphism of schemes which is a universal homeomorphism and let $X^{\prime}$ be any $S^{\le\prime}$-scheme. Then $\Re_{S^{\e\prime}\be/S}(X^{\prime}\e)$ exists.
\end{corollary}
\begin{proof} This is immediate from the theorem and Remark \ref{rems-adm}(c). 
\end{proof}

\begin{proposition}\label{w-lim}
Let $k^{\e\prime}\be/k$ be a finite field extension and let $(X_{\lambda})_{\lambda\le\in\le\Lambda}$ be a projective system of $k^{\le\prime}$-schemes, where $\Lambda$ is a directed set containing an element $\lambda_{\e 0}$ such that the transition morphisms $X_{\mu}\to X_{\lambda}$ are affine if $\mu\geq\lambda\geq \lambda_{\e 0}$. Assume that $X_{\lambda_{0}}$ is admissible relative to $k^{\le\prime}\be/k$ (see Definition {\rm \ref{adm}}). Then $\Re_{\e k^{\le\prime}\be/k}\lbe
\big(\varprojlim \be X_{\lambda}\big)$ and $\varprojlim \Re_{\e k^{\e\prime}/k}(X_{\lambda})$ exist and
\[
\Re_{\e k^{\le\prime}\be/k}\lbe
\big(\varprojlim \be X_{\lambda}\big)
=\displaystyle\varprojlim \Re_{\e k^{\le\prime}\be/k}(X_{\lambda}).
\]
\end{proposition}
\begin{proof} By \cite[IX, \S 3, dual of Theorem 1]{mac}, we may replace $\Lambda$ with the cofinal subset $\{\lambda\in \Lambda\,|\, \lambda \geq \lambda_{\e 0}\}$ (whence $\lambda_{0}$ is an initial element of $\Lambda$). The stated formula will follow from \eqref{plim} and \eqref{wr} once the existence assertion is established. Set $X=\varprojlim X_{\lambda}$. Since the canonical morphism $X\to X_{\lambda_{\e 0}}$ is affine \cite[Proposition 3.2(iv) and Remark 5.16]{bga}, $X$ is admissible relative to $k^{\e\prime}/k$ by Remark \ref{rems-adm}(d). Thus, by Theorem \ref{wr-rep}, $\Re_{\e k^{\e\prime}/k}\lbe\big(X\big)=\Re_{\e k^{\e\prime}/k}\lbe\big(\varprojlim \be X_{\lambda}\big)$ exists. Similarly, for every $\lambda\geq\lambda_{\e 0}$, $X_{\lambda}$ is admissible relative to $k^{\le\prime}\be/k$ and $\Re_{\e k^{\le\prime}\be/k}(X_{\lambda})$ exists. It remains only to check that the transition morphisms $\Re_{\e k^{\e\prime}\be/k}(X_{\mu})\to\Re_{\e k^{\e\prime}\be/k}(X_{\lambda})$ are affine if $\mu\geq\lambda\geq \lambda_{0}$. Let $U$ be an affine open subscheme of $X_{\lambda}$. Then $X_{\mu}\times_{X_{\lambda}}U$ is affine and therefore so also is
\[
\Re_{\e k^{\e\prime}\be/k}(X_{\mu}\!\be\times_{\lbe X_{\lbe\lambda}}\!\be U\e)=\Re_{\e k^{\e\prime}\be/k}(X_{\mu}) \times_{\Re_{\e k^{\prime}\be/k}(X_{\lambda})}\Re_{\e k^{\e\prime}\be/k}(U\e)
\]
by \cite[Proposition A.5.2, (2) and (3)]{cgp}. Since $\Re_{\e k^{\e\prime}\be/k}(X_{\lambda})$ is covered by affine open subschemes of the form $\Re_{\e k^{\e\prime}\be/k}(U)$ \cite[p.~195]{blr}, the proposition follows.
\end{proof}

 We conclude this Subsection by recalling the definition (to be relevant in Section \ref{forgr}) of the Weil restriction of an adic formal scheme over a discrete valuation ring.

\begin{definition}\label{wr-for} Let $R\to R^{\e\prime}$ be a finite extension of complete discrete valuation rings and let $\m S^{\e\prime}\be\to \m S$ be the corresponding morphism of adic formal schemes. Let $\m X^{\e\prime}$ be an adic formal $\m S^{\prime}$-scheme. We will say that {\it the Weil restriction of $\m X^{\e\prime}$ along $\m S^{\e\prime}\be\to \m S$ exists} if 
the contravariant functor
\[
(\text{Ad-For}/\m S)\to({\rm Sets}),\m T\to \Hom_{\, \m S^{\prime}}(\m T\times_{\m S}\m S^{\e\prime},\m X^{\e\prime}\e),
\]
is represented by an adic formal $\m S$-scheme  $\Re_{\e\m S^{\e\prime}\be/\m S}\big(\m X^{\e\prime}\le\big)$ (which will then be called {\it the Weil restriction of $\m X^{\e\prime}$ along $\m S^{\e\prime}\be\to \m S$}). 
\end{definition}

\subsection{The fpqc topology}\label{ftop}
Recall from \cite[\S2.3.2, pp.~27-28]{v} that a morphism of schemes $f\colon X\to Y$ is said to be an {\it fpqc morphism} if it is faithfully flat and has the following property: if $x$ is a point of $X$, then there exists an open neighborhood $U$ of $x$ in $X$ such that the image $f(U)$ is open in $Y$ and the induced morphism $U\to f(U)$ is quasi-compact. It is immediate that a faithfully flat and quasi-compact morphism is an $\fpqc$ morphism. By \cite[Proposition 2.35(v), p.~28]{v}, the class of fpqc morphisms is stable under base change. An {\it fppf morphism} of schemes is a faithfully flat morphism locally of finite presentation. Every fppf morphism is an fpqc morphism by \cite[Proposition 2.35(iv), p.~28]{v}.
Let $S$ be a scheme and let $\mathcal C$ be a full subcategory of $(\mathrm{Sch}/S\e)$ which contains the final object $1_{\lbe S}$. The fpqc (respectively, fppf) topology on $\mathcal C$ is the topology where the coverings are collections of flat morphisms $\{ X_{\alpha}\to X\}$ in $\mathcal C$ such that the induced morphism $\coprod X_{\alpha}\to X$ is an fpqc (respectively, fppf) morphism.  Clearly, the fpqc topology is finer than the fppf topology.
If $\tau=\fpqc$ or $\fppf$, we will write $\mathcal C_{\tau}$ for the category $\mathcal C$ endowed with the $\tau$ topology. The category of sheaves of sets on $\mathcal C_{\tau}$ will be denoted by $\mathcal C^{\,\sh}_{\tau}$. Both sites mentioned above are subcanonical, i.e., every representable presheaf is a sheaf \cite[Theorem 2.55, p.~34]{v} and the induced functor  
\begin{equation}\label{fufa}
h_{\lbe
S}\colon \mathcal C\to \mathcal C^{\,\sh}_{\tau},Y\mapsto \Hom_{S}(-,Y),
\end{equation}
is fully faithful, whence it identifies $\mathcal C$ with a full subcategory of $\mathcal C^{\,\sh}_{\tau}$. 
A sequence $1\to F\to G\to H\to 1$ of group schemes in $\mathcal C$ will be called {\it exact for the $\tau$ topology on $\mathcal C$} if the sequence of sheaves of groups $1\to h_{\lbe S}(F\le)\to h_{\lbe S}(G)\to h_{\lbe S}(H)\to 1$ is exact. See \cite[\S 2]{bga} for more details.

\begin{lemma}\label{flat1} Let $k$ be a field and let $q\colon G\to H$ be a dominant and quasi-compact morphism of $k$-group schemes, where $H$ is {\rm reduced}. Then $q$ is faithfully flat.
\end{lemma}
\begin{proof} See \cite[Proposition 1.3, p.~19]{per} or \cite[${\rm{VI}}_{\rm A}$, Corollary 6.2]{sga3}.
\end{proof}

\begin{lemma}\label{flat2} Let $k$ be a field and let $q\colon G\to H$ be a morphism of $k$-group schemes locally of finite type.
\begin{enumerate}
\item[(i)] If $q$ is flat, then $q^{\le 0}\colon G^{\e 0}\to H^{\le 0}$ is surjective.
\item[(ii)] If $q^{\le 0}$ is surjective and $H$ is reduced, then $q$ is flat.
\end{enumerate}
\end{lemma}
\begin{proof} See \cite[${\rm{VI}}_{\rm B}$, Proposition 3.11 and its proof]{sga3}.
\end{proof}

\begin{lemma}\label{con-sm} Let $k$ be a field and let $q\colon G\to H$ be a faithfully flat morphism of $k$-group schemes locally of finite type. If $G$ is connected (respectively, smooth), then $H$ is connected (respectively, smooth).
\end{lemma}
\begin{proof} By Lemma \ref{flat2}(i), $q^{\le 0}\colon G^{\e 0}\to H^{\le 0}$ is surjective. Consequently, if $G$ is connected, i.e., $G=G^{\e 0}$, then $H=q(G)=q^{\le 0}(G^{\e 0})=H^{\e 0}$, i.e., $H$ is connected as well. Now, if $G$ is smooth, then $H$ is smooth by \cite[$\text{IV}_{4}$, Proposition 17.7.7]{ega}.
\end{proof}

\begin{lemma} \label{conn} Let $k$ be a field and let $q\colon G\to H$ be a surjective and quasi-compact morphism of $k$-group schemes locally of finite type, where $H$ is smooth. Then the sequence 
\[
1\to \krn q\to G\overset{\!q}\to
H\to 1
\]
is exact for both the fppf and fppf topologies on $(\mr{Sch}/k)$. If $\e\krn q$ and $H$ are connected, then $G$ is connected. If $\krn q$ is smooth, then $G$ is smooth.
\end{lemma}
\begin{proof} By Lemmas \ref{flat1} and \ref{flat2}(i), $q$ is faithfully flat and $q^{\le 0}\colon G^{\le 0}\to H^{\le 0}$ is surjective. In particular $q$ is an fppf morphism (and therefore also an fpqc morphism) by \cite[Proposition 2.4(i)]{bga}. The exactness assertion of the lemma now follows from \cite[Lemma 2.3]{bga}. Assume next that $\krn q$ and $H$ are connected.  Since $q^{0}\colon G^{\e 0}\to H^{\e 0}=H$ is surjective and $\krn q=(\krn q)^{\le 0}\subseteq G^{\e 0}$, we have $G^{\e 0}=G$, i.e., $G$ is connected. Now, if $\krn q$ is smooth, then $q$ is smooth by
\cite[$\text{IV}_{4}$, Proposition 17.5.1]{ega} and \cite[${\rm{VI}}_{\rm B}$, Proposition 1.3]{sga3}. Thus, since the structure morphism of $G$ factors as $G\overset{\!q}{\to}H\to\spec k$, $G$ is smooth over $k$, as claimed.
\end{proof}

\begin{lemma}\label{kerseq} Let $k$ be a field and let $F\overset{f}\to G\overset{g}\to H$ be morphisms of $k$-group schemes locally of finite type.
\begin{enumerate}
\item[(i)] The given pair of morphisms induces a sequence of $k$-group schemes locally of finite type
\[
1\to\krn f\to\krn(\e g\be\circ\! f\e)\to\krn g
\]
which is exact for both the fppf and fpqc topologies on $(\mr{Sch}/k)$.
\item[(ii)] If $f$ is faithfully flat, then the sequence
of $k$-group schemes locally of finite type
\[
1\to\krn f\to\krn(\e g\be\circ\! f\e)\to\krn g\to 1
\]
is exact for both the fppf and fpqc topologies on $(\mr{Sch}/k)$.
\end{enumerate}
\end{lemma}
\begin{proof} Since $F$ and $G$ are locally of finite type, $\krn f,\krn g$ and $\krn(\e g\circ f\e)$ are locally of finite type. We regard $F$ as an $H$-scheme via $g\be\circ\be f$, so that $f\colon F\to G$ is an $H$-morphism. Then the induced morphism $f\times_{H}\spec k\colon F\times_{H}\spec k\to G\times_{H}\spec k$ is a $k$-morphism $\krn(\e g\be\circ \be f\e)\to\krn g$ whose kernel is canonically isomorphic to 
$\krn f$. Assertion (i) is now clear. If $f$ is faithfully flat, then so also is $f\times_{H}\spec k$. Thus, by \cite[Proposition 2.4(i)]{bga}, $f\times_{H}\spec k$ is an fppf morphism and assertion (ii) follows from \cite[Lemma 2.3]{bga}.
\end{proof}

Let $k$ be a field. The category of commutative and quasi-compact $k$-group schemes will be denoted by $\s C_{\qc}$. The full subcategory of $\s C_{\qc}$ whose objects are the $k$-group schemes of finite type will be denoted by $\s C_{\alg}$. By \cite[${\rm{VI}}_{\rm A}$, Theorem 5.4.2 and Corollary 6.8]{sga3}, $\s C_{\alg}$ and $\s C_{\qc}$ are abelian categories. Further, by \cite[${\rm{VI}}_{\rm A}$, 0.3]{sga3} and \cite[Propositions 6.1.5(v), p.~291, and 6.3.8(v), p.~305]{ega1}, every morphism in $\s C_{\qc}$  (respectively, $\s C_{\alg}$) is quasi-compact (respectively, of finite presentation).

\begin{lemma}\label{qc-alg} Let $0\to F\to G\to H\to 0$ be an exact sequence in the abelian category $\s C_{\qc}$ (respectively, $\s C_{\alg}$). Then the given sequence is exact as a sequence of sheaves for the fpqc (respectively, fppf) topology on $({\rm Sch}/k)$.
\end{lemma}
\begin{proof} The morphism $f\colon G\to H$ can be identified with the canonical projection morphism $G\to G/F$, which is faithfully flat by \cite[${\rm{VI}}_{\rm A}$, Proposition 5.4.1 and Corollary 6.7(i)]{sga3}. Consequently, $f$ is an fpqc (respectively, fppf) morphism and the lemma follows from \cite[Lemma 2.3]{bga}. 

\end{proof}

The above lemma shows that an exact sequence of arbitrary finite length in $\s C_{\qc}$ (respectively, $\s C_{\alg}$) is also exact for the fpqc (respectively, fppf) topology on $\s C_{\qc}$ (respectively, $\s C_{\alg}$). Observe now the following partial converse to the previous lemma:

\begin{proposition}\label{ex-ex} Let $0\to F\to G\to H\to 0$ be a sequence in $\s C_{\qc}$ which is exact for the fpqc topology on $({\rm Sch}/k)$. If $H$ is reduced, then the given sequence is exact in the abelian category $\s C_{\qc}$. A similar result holds if above $\s C_{\qc}$ is replaced by $\s C_{\alg}$ and the fpqc topology is replaced by the fppf topology.
\end{proposition}
\begin{proof} If $f\colon G\to H$, then $F\simeq \krn f$ and $f$ is surjective by \cite[Lemma 2.2]{bga}. On the other hand, by \cite[${\rm{VI}}_{\rm A}$, Corollary 6.7(i)]{sga3}, $f$ factors as
\[
G\overset{\!h}{\to} G/\e\krn f\simeq\img f\overset{\!i}{\to}H,
\]
where $h$ is faithfully flat and $i$ is a closed immersion. It follows that $i$ is a surjective closed immersion, i.e., a nilimmersion \cite[(4.5.16), p.~273]{ega1}. Since $H$ is reduced, $i$ is an isomorphism.
\end{proof}

\begin{corollary} \label{cok=0} Let $f\colon G\to H$ be a morphism in  $\s C_{\alg}$, where $H$ is reduced. If $f\big(\e\kbar\e\big)\colon G\lbe\big(\e\kbar\e\big)\to
H\lbe\big(\e\kbar\e\big)$ is surjective, then $\cok f=0$. 
\end{corollary}
\begin{proof} Since $H\big(\e\kbar\e\big)=f\big(\e\kbar\e\big)\big(G\big(\e\kbar\e\big)\big)\subseteq f(\vert G\vert)$ and $H\big(\e\kbar\e\big)$ is dense in $H$ by \cite[Corollary 3.8, p.~71]{per}, $f$ is dominant and therefore faithfully flat by Lemma \ref{flat1}. Now \cite[Corollary 2.5]{bga} shows that the sequence $0\to\krn f\to G\to H\to 0$ is exact for the  fppf  topology on $({\rm Sch}/k)$, whence $\cok f=0$ by Proposition \ref{ex-ex}.
\end{proof}

\begin{lemma}\label{ker-cok} Let $k$ be a field and let $F\overset{\!\be f}\to G\overset{g}\to H$ be morphisms in $\s C_{\qc}$ (respectively, $\s C_{\alg}$). Then there exists an induced sequence 
\[
0\to\krn f\to\krn(\e g\be\circ\! f\e)\to\krn g\to\cok f\to\cok\le(\e g\be\circ\! f\e)\to\cok g\to 0
\]
which is exact in $\s C_{\qc}$ (respectively, $\s C_{\alg}\e$).
\end{lemma}
\begin{proof} This proposition is valid in any abelian category. See \cite[Hilfssatz 5.5.2, p.~45]{bp}.
\end{proof}

If $k$ is a field, $G$ is a commutative $k$-group scheme and $n$ is an integer, let $n_{\le G}\colon G\to G$ denote the morphism which maps $x\in G(T)$ to $x^{n}\in G(T)$ for every $k$-scheme $T$. Since $G$ is commutative, $n_{\le G}$ is a morphism of $k$-group schemes, i.e., a homomorphism.

\begin{lemma}\label{n-div} Let $k$ be a field and let $G$ be a commutative and connected $k$-group scheme of finite type. If $n$ is an integer which is not divisible by $\mr{char}\e k$, then $n_{\le G}\be\big(\e\kbar\e\big)\colon G\big(\e\kbar\e\big)\to G\big(\e\kbar\e\big)$ is surjective.
\end{lemma}
\begin{proof} By \cite[${\rm VII_{A}}$, \S 8.4, Proposition]{sga3}\,\footnote{The reader should be warned that the cited proposition is correct only in the commutative case, as noted by Brian Conrad. Indeed, the proof of \cite[${\rm VII_{A}}$, \S 8.4, Proposition]{sga3} requires that $n_{\le G}$ be a {\it homomorphism} in order to apply \cite[${\rm VI_{B}}$, Proposition 1.3]{sga3}.}, $n_{\le G}$ is \'etale and therefore flat. Thus, by Lemma \ref{flat2}(i), $n_{\le G}=n_{\le G}^{\le 0}$ is surjective. The lemma now follows from \cite[I, \S3, Corollary 6.10, p.~96]{dg}.
\end{proof}

\begin{lemma}\label{flat3} Let $k$ be a field and let $q\colon G\to H$ be a morphism of smooth and commutative $k$-group schemes. Assume that
\begin{enumerate}
\item[(i)] $q\big(\e\kbar\e\big)\colon G\big(\e\kbar\e\big)\to H\big(\e\kbar\e\big)$ is surjective, and
\item[(ii)] $\pi_{0}(G\le)\big(\e\kbar\e\big)$ is a finitely generated abelian group.
\end{enumerate}
Then $q$ is flat.
\end{lemma}
\begin{proof} By Lemma \ref{flat2}(ii), it suffices to check that $q^{\le 0}\colon G^{\e 0}\to H^{\le 0}$ is surjective. Since $G^{\e 0}$ and $H^{\le 0}$ are both of finite type by \cite[${\rm VI_{A}}$, Proposition 2.4(ii)]{sga3}, $q^{\le 0}$ is a morphism in $\s C_{\alg}\e$ that factors as
\[
G^{\e 0}\overset{h}{\to}\img q^{\le 0}\overset{i}{\to} H^{\e 0},
\]
where $h$ is faithfully flat, $\img q^{\le 0}$ is smooth by Lemma \ref{con-sm} and $i$ is a closed immersion (see the proof of Proposition \ref{ex-ex}). Thus it suffices to check that $i$ is an isomorphism, i.e., that $C=\cok i=0$. By Lemma \ref{cok=0} we only need to show, in fact, that $\cok i\big(\e\kbar\e\big)=0$. Since the canonical projection morphism $H^{\e 0}\to C$ is surjective by \cite[${\rm VI_{A}}$, Theorem 3.3.2(ii)]{sga3}, $H^{\le 0}\big(\e\kbar\e\big)\to C\big(\e\kbar\e\big)$ is surjective by \cite[I, \S3, Corollary 6.10, p.~96]{dg}, whence $\cok i\big(\e\kbar\e\big)=C\big(\e\kbar\e\big)$. Further, $C$ is connected by Lemma \ref{con-sm}. Thus Lemma \ref{n-div} implies that $C\big(\e\kbar\e\big)$ is $n$-divisible for every integer $n$ prime to $\mr{char}\e k$. On the other hand, since $\img q^{\le 0}\be\big(\e\kbar\e\big)\subseteq \big(\img q^{\le 0}\big)\be\big(\e\kbar\e\big)$, $C\big(\e\kbar\e\big)=\cok \e i\big(\e\kbar\e\big)$ is a quotient of $\cok q^{\le 0}\big(\e\kbar\e\big)$. Now an application of the snake lemma to the exact and commutative diagram
\[
\xymatrix{
0\ar[r]&G^{\e 0}\be\big(\e\kbar\e\big)\ar[r]\ar[d]^(.45){q^{\le 0}\lbe(\e\kbar\e)}&  G\lbe\big(\e\kbar\e\big)\ar[r]\ar@{->>}[d]^(.45){q(\e\kbar\e)}&\pi_{0}(G\le)\big(\e\kbar\e\big)\ar[r]\ar[d]^(.45){\pi_{0}(q)(\e\kbar\e)}&0\\
0\ar[r]&H^{\e 0}\be\big(\e\kbar\e\big)\ar[r]&  H\be\big(\e\kbar\e\big)\ar[r] &\pi_{0}(H\le)\big(\e\kbar\e\big)\ar[r] &0
}
\](whose middle vertical arrow is surjective by (i)) shows that $\cok\e q^{\e 0}\be\big(\e\kbar\e\big)$ is a quotieisnt of $\krn\e \pi_{0}(q)\big(\e\kbar\e\big)$, which is finitely generated by hypothesis (ii). We conclude that $C\big(\e\kbar\e\big)$  $n$-divisible (for every $n$ as above) and finitely generated, whence $C\big(\e\kbar\e\big)=0$.
\end{proof}

\begin{remark} 
The lemma and its proof show that both $q$ and $q^{\le 0}$ are faithfully flat. Thus, by \cite[Proposition 2.4(i)]{bga}, $q$ and $q^{\le 0}$ are fppf morphisms.
\end{remark}

\begin{lemma}\label{red=1} Let $k$ be a perfect field and let $G$ be a $k$-group scheme locally of finite type. If $G\big(\e\kbar\e\big)=\{1\}$, then $G_{\red}=1$.
\end{lemma}
\begin{proof} Since the projection $G_{\red}\times_{\spec k}\spec \kbar\to G_{\red}$ is faithfully flat and $G_{\red}\times_{\spec k}\spec \kbar=(G\times_{\spec k}\spec \kbar\e)_{\red}$ by \cite[Corollary 4.5.12, p.~271]{ega1} and \cite[I, \S2, no.4, Corollary 4.13, p.~55]{dg}, we may assume that $k=\kbar$. By \cite[${\rm VI_{A}}$, 0.2 and Lemma 0.5.2]{sga3}, $G_{\red}$ is a reduced and closed $k$-subgroup scheme of $G$. Further, the hypothesis implies that $G_{\red}(k)=\{1\}$. Now \cite[II, \S5, no.4, Proposition 4.3, p.~245]{dg} shows that $G_{\red}=1$.
\end{proof}

\begin{remarks}\label{inf}\indent
\begin{enumerate}
\item[(a)]  By definition, an infinitesimal $k$-group scheme is a finite and local $k$-group scheme. By \cite[II, \S 4, lines below 7.1, p.~230]{dg}, such an object is a connected and artinian one-point scheme. Consequently, a $k$-group scheme of finite type is infinitesimal if, and only if, it is a one-point scheme \cite[Exercise 3, p.~92]{am}. We now observe that a quotient $U/V$ of infinitesimal and unipotent $k$-group schemes is unipotent and infinitesimal. Indeed, unipotency is clear and $U/\le V$ is a one-point scheme since the projection $U\to U/\le V$ is faithfully flat. 
\item[(b)] By \cite[${\rm VI_{A}}$, Proposition 5.6.1 and its proof\,]{sga3}, the group $G$ of the lemma is a one-point scheme which is equal to the spectrum of a local $k$-algebra of finite rank with residue field $k$. Thus $G$ is infinitesimal and therefore $\dim G=0$.
\end{enumerate}	
\end{remarks}

\section{Greenberg algebras}\label{ga}
Let $k$ be a perfect field of positive characteristic and let $m\geq 1$ be an integer. In \cite[Appendix A]{lip}, Lipman translated into scheme-theoretic language Greenberg's construction of the $\mathbb W_{\be m}$-module variety associated to a finitely generated $W_{\be m}\lbe(k)$-module \cite[Proposition 3, p.~628]{gre1}. In this Section we extend Lipman's translation to other constructions/statements from \cite{gre1,gre2}, covering also (in Subsection \ref{sec-k})  a case not discussed by Lipman. We note that the references \cite{gre1,gre2} are concerned with {\it varieties}, i.e.,  discuss the  restricted scheme-theoretic setting of reduced and irreducible $k$-schemes of finite type (see, e.g., \cite[comments preceding Lemma 1, p.~257]{gre2}). Consequently, it is a nontrivial problem to translate statements from [loc.cit.] into a general scheme-theoretic setting. For example, it is shown in \cite[proof of Proposition 3(5), p.~629]{gre1} that, if $\mathfrak B\subseteq \mathfrak C$ is an inclusion of finitely generated $W_{\be m}\lbe(k)$-modules  (where $k$ is as above), then the associated morphism of $\mathbb W_{\be m}$-module varieties $\s B\to \s C$ is a closed immersion. This result depends strongly on the fact that the author allows the replacement of $\s B$ with a $\mathbb W_{\be m}$-module scheme $\widetilde{ \s B}$ such that $\widetilde{\s B}_{\red}=\s B$. In the scheme-theoretic setting, the corresponding statement is false since the kernel of the induced morphism of $\mathbb W_{\be m}$-module schemes $\s B\to \s C$ can be a non-reduced scheme. See Example \ref{gr-err} and Remarks \ref{rl-gre} and \ref{gr-max}.

\subsection{Finitely generated modules over arbitrary fields}\label{sec-k} 
In this Subsection, $k$ is an arbitrary field.  Let $\mathfrak M$ be a finitely generated $k$-module of rank $r\geq 1$ and fix a basis $\{m_{1},\dots,m_{r}\}$ of $\mathfrak M$, i.e., a $k$-isomorphism $\mathfrak M\simeq  k^{\le r}, \sum_{\e i} x_{i}m_{i}\mapsto (x_{i})$. The $k$-module structure on $\mathfrak M$ induces an 
$\mathbb O_{k}$-module structure on $\A^{\!r}_{k}$. The {\it Greenberg module} associated to $\mathfrak M$, denoted by $\s M\be$, is the $k$-scheme $\A^{\!r}_{k}$ equipped with the above $\mathbb O_{k}$-module scheme structure. By definition, for every $k$-algebra $A$, there exists an isomorphism of $A$-modules
\begin{equation}\label{wem}
\s M\be(A)\overset{{\rm def.}}{=}\Hom_{\e k}(\spec A,\s M\e)\simeq \mathfrak M\otimes_{\le k}\be A,
\end{equation}
which is explicitly given by $A^{r}\overset{\!\sim}{\to}\oplus_{\le i=1}^{\le r} A\le m_{i},\, (a_{i})\mapsto (a_{i}\le m_{i})$.

Let $\mathfrak R$ be a finite $k$-algebra with associated ring scheme $\mathbb O_{\le \mathfrak R}$. Since $\mathfrak R$ is a finitely generated $k$-module, its associated Greenberg module $\s R$ can be defined as above. Now $\s R\lbe(A)=\mathfrak R\otimes_{k} A$ is naturally endowed with an $\mathfrak R$-algebra structure and the $k$-ring scheme $\s R$ is called the {\it Greenberg algebra associated to $\mathfrak R$}. By \eqref{wr} and \eqref{wem}, we have  
\[
\s R =\Re_{\e\mathfrak R \lbe/ k}(\mathbb O_{\le \mathfrak R}\lbe),
\] 
where $\Re_{\e\mathfrak R \lbe/k}$ is the Weil restriction functor associated to the finite and locally free morphism $ \spec \mathfrak R\to\spec k$. In particular,
\begin{equation}\label{uu}
\hskip .6cm\s R =\mathbb O_{k} \quad\text{if $\mathfrak R=k\e$}.
\end{equation}
Further, for every $k$-algebra $A$, there exists an isomorphism of $\mathfrak R$-$A$-bialgebras
\begin{equation} \label{gr-weil}
\s R(A)=\mathfrak R\otimes_{\e k}\! A.
\end{equation}
Consequently, there exists a (non-canonical) isomorphism of $k$-group schemes 
\begin{equation}\label{uul}
\s R \simeq \G_{a,\le k}^{\ell}\e,
\end{equation}
where $\ell=\dimn\le\mathfrak R \geq 1$. Note that $\s R\lbe(k)=\mathfrak R$. Clearly, if $A$ is finitely generated as a $k$-algebra (respectively, $k$-module), then $\s R\lbe(A)$ is finitely generated as an $\mathfrak R$-algebra (respectively, $\mathfrak R$-module). Further, if $f\in A$, then $\s R\lbe(\lbe A)_{f}=\s R\lbe(\lbe A)\otimes_{A}A_{f}$ by \cite[Proposition 3.5, p.~39]{am}, whence
\begin{equation}\label{eqrf}
\s R\lbe(A)_{f}=\s R\lbe(\lbe A_{f}\lbe).
\end{equation}

Now let $\mathfrak R\to \mathfrak R^{\e\prime}$ be a homomorphism of finite $k$-algebras with kernel $\mathfrak K$ and let $\s R,\s R^{\e\prime}$ and $\s K$ be the Greenberg modules associated to $\mathfrak R, \mathfrak R^{\e\prime}$ and $\mathfrak K$, respectively. By \eqref{wem} and \eqref{gr-weil}, the canonical exact sequence of $k$-modules
\[
0\to\mathfrak K\to \mathfrak R\to \mathfrak R^{\e\prime}
\]
induces, for every $k$-algebra $A$, an exact sequence of $\mathfrak R$-$A$-bimodules
\[
0\to\s K\be(A)\to\s R(A)\to\s R^{\e\prime}(A),
\]
where
\begin{equation}\label{kar}
\s K\be(A)=\mathfrak K\otimes_{k}\! A=\mathfrak K\,\s R(A).
\end{equation}
We conclude that
\begin{equation*} 
\s K=\krn\!\left[\e \s R\to \s R^{\e\prime}\,\right],
\end{equation*}
where the indicated morphism of $k$-group schemes is induced by the given homomorphism $\mathfrak R\to \mathfrak R^{\e\prime}$.  In particular, let $\mathfrak I$ be an ideal of $\mathfrak R$, write $\mathfrak R^{(\mathfrak I\le)}=\mathfrak R/\le\mathfrak I$ and let 
$\s R^{\e(\lbe\s I)}$ denote the Greenberg algebra associated to $\mathfrak R^{(\mathfrak I\le)}$. Then
\begin{equation}\label{eqc}
\s I=\krn\!\lbe\left[\e \s R\to \s R^{\e(\lbe\s I)}\e\right]
\end{equation}

\begin{remarks}\label{resp0} By \eqref{wem} and the exactness of the bifunctor $(-)\lbe\otimes_{\e k}\be(-)$ on the category of $k$-modules, the following holds.
\begin{enumerate}
\item[(a)] If $\mathfrak M$ is a finitely generated $k$-module and 
$A\to B$  is an injective (respectively, surjective) homomorphism of $k$-algebras, then the induced homomorphism of $k$-modules $\s M\lbe(A)\to \s M\lbe(B)$ is injective (respectively, surjective)
\item[(b)]  If $\mathfrak M\to\mathfrak M^{\e\prime}$ is a surjective homomorphism of finitely generated $k$-modules and $A$ is any $k$-algebra, then the induced map
$\s M\lbe(A)\to \s M^{\e\prime}\lbe(A)$ is a surjective homomorphism of $A$-modules. 
\end{enumerate}
\end{remarks}

\medskip

\subsection{Modules over rings of Witt vectors}\label{sec-w}
In this Subsection, $k$ is a perfect field of characteristic $p>0$. Let $\mathfrak M$ be a finitely generated $W_{\be m}(k)$-module, where $m>1$ is an integer.

\begin{remark}\label{excu} Above we have assumed that $m>1$ since $W_{\be 1}(k)$-modules, i.e., $k$-modules, have been discussed in the previous Subsection for arbitrary fields $k$. See also Remark \ref{uch}(a) below.
\end{remark}

Let $\mathbf M$ denote the fpqc sheaf on the category of affine $k$-schemes associated to the presheaf $\spec A\mapsto \mathfrak M\otimes_{\e W\be(k)}\be W\be(A)$, where $A$ is a $k$-algebra.  
By \cite[Proposition A.1]{lip}, there exists  an affine $\mathbb W_{\be m}$-module scheme $\s M$, called the {\it Greenberg module associated} to $\mathfrak M$, which represents $\mathbf M$,  i.e., $\mathbf M(\spec A)=\s M\be(A)$, where
\begin{equation*} 
\s M\be(A)\overset{{\rm def.}}{=} \Hom_{\e k}(\spec A,\s M\e).
\end{equation*}
Therefore $\s M$ is unique up to a unique isomorphism. Further, by \cite[Corollary A.2]{lip}, the canonical map $\mathfrak M\le\otimes_{\e W_{\be m}\lbe(k)}\! W_{\be m}\lbe(A)\to \s M\lbe(A)$ of [\le loc.cit.] is surjective  for every $k$-algebra $A$. By construction, a choice of an isomorphism of $W_{\be m}(k)$-modules $\mathfrak M\simeq\prod_{\e i=0}^{\e r}W_{\! n_{i}}\be(k)$, where $n_{i}\leq m$ for every $i$, induces an isomorphism of $\mathbb W_{\be m}$-module schemes $\s M\simeq \prod_{\le i=0}^{\le r}\mathbb W_{\! n_{i}}$. In particular, the dimension of $\s M$ equals the length of the $W_{\be m}(k)$-module $\mathfrak M$. Further, a homomorphism of finitely generated $W_{\be m}(k)$-modules $\mathfrak M\to \mathfrak M^{\e\prime}$ induces a morphism of associated $\mathbb W_{\be m}$-module schemes $\s M\to \s M^{\e\prime}$ \cite[Proposition A.1, p.~74]{lip}.

\begin{remarks}\label{resp}\indent
\begin{enumerate}
\item[(a)] If $\mathfrak M$ is a finitely generated $W_{\be m}(k)$-module and 
$A\to B$  is an injective (respectively, surjective) homomorphism of $k$-algebras, then the induced homomorphism of $W_{\be m}(k)$-modules $\s M\lbe(A)\to \s M\lbe(B)$ is injective (respectively, surjective). This follows from the fact that there exist isomorphisms of $k$-schemes $\s M\simeq \prod_{\le i=0}^{\le r}\mathbb W_{\! n_{i}}\simeq \A_{k}^{\!N}$, where $N=\sum_{\e i\le =1}^{\e r} n_{i}$.
\item[(b)]  If $\mathfrak M\to\mathfrak M^{\e\prime}$ is a surjective homomorphism of finitely generated $W_{\be m}(k)$-modules and $A$ is a $k$-algebra, then the commutativity of the diagram of $W_{\be m}(A)$-modules
\[
\xymatrix{\mathfrak M\otimes_{\e W_{\be m}(k)}W_{\be m}(A)\ar[r]\ar@{->>}[d]&\s M(\be A)\ar[d]\ar[r]& 0\\ 
\mathfrak M^{\e\prime}\otimes_{\e W_{\be m}(k)}W_{\be m}(A)\ar[r]&\s M^{\e\prime}(\be A)\ar[r]& 0
} 
\]
(whose left-hand vertical map is surjective by the right-exactness of the tensor product functor) shows that the right-hand vertical map above is a surjective homomorphism of $W_{\be m}(A)$-modules. 
\end{enumerate}
\end{remarks}

\smallskip

Let $\mathfrak R$ be a finite $W_{\be m}\lbe(k)$-algebra. The fpqc sheaf on the category of affine $k$-schemes associated to the presheaf $\spec A\mapsto \mathfrak R\e\otimes_{\e W\be(k)}\be W\be\le(\lbe A)$ is represented by a $\mathbb W_{\le m}$-algebra scheme $\s R $ called the {\it Greenberg algebra associated to $\mathfrak R$}. The scheme $\s R$ is defined as follows. There exists an isomorphism of $W_{\le m}(k)$-modules $\mathfrak R\simeq \prod_{\le i=0}^{\le r}W_{\! n_{i}}\be(k)$, and $\s R$ is the $\mathbb W_{\le m}$-module scheme $\prod_{\le i=0}^{\le r}\!\mathbb W_{\! n_{i}}$ endowed with the $k$-ring scheme structure induced by the ring structure on $\mathfrak R$ \cite[Proposition A.1 and Corollary A.2]{lip}. By construction, there exists a (non-canonical) isomorphism of $k$-schemes
\begin{equation}\label{uul2}
\hskip 4cm\s R\simeq \A^{\!\ell}_{k} \quad\text{(\e$\ell\e=\mr{length}_{\e W_{\be m}\lbe(k)}\e\mathfrak R$\e)}
\end{equation}
and we have $\s R\lbe(k)=\mathfrak R$. Further, if $\mathfrak R=W_{\! m}\lbe(k)$, then $\s R=\mathbb W_{\! m}$.

\begin{remarks}\label{uch}\indent
\begin{enumerate}
\item[(a)] The preceding considerations work equally well if $m=1$ and the resulting Greenberg module (respectively, algebra) associated to the finitely generated $W_{\be 1}(k)=k$-module $\mathfrak M$ (respectively, finite $k$-algebra $\mathfrak R$) coincides with that defined in the previous Subsection.
\item[(b)] If $\mathfrak R $ is an artinian local ring with residue field $k$ and  $m>1$ is defined by the equality ${\rm char}\,\mathfrak R =p^{\le m}$, then $\mathfrak R$ has a canonical structure of finite $W_{\!m}(k)$-algebra \cite[Case 2, p.~627]{gre1}. Now, if $s\geq 0$ is an integer, then the canonical projection $W_{\!m+s}(k)\to W_{\!m}\lbe(k)$ induces a $W_{\!m+s}(k)$-algebra structure on $\mathfrak R$ which produces the same isomorphism $\mathfrak R\simeq  \prod_{\le i=0}^{\le r}W_{\! n_{i}}\be(k)$ of $W(k)$-modules as that obtained in the case $s=0$. Consequently, the $k$-ring scheme $\s R$ depends only on the (canonical) $W(k)$-algebra structure of $\mathfrak R$. 
\item[(c)] If $\mathfrak  K$ is an ideal of $\mathfrak R$, then the image of the canonical homomorphism $\s K\be(A)\to \s R(A)$ equals $\mathfrak K\,\s R(A)$, as follows at once from the commutative diagram in Remark \ref{resp}(b) (setting $\mathfrak M=\mathfrak K$ and $\mathfrak M^{\e\prime}=\mathfrak R$ in that diagram).
\end{enumerate}
\end{remarks}

Every finitely generated $\mathfrak R$-module $\mathfrak B$ defines an $\s R$-module scheme $\s B$ and every homomorphism $\mathfrak B\to \mathfrak C$ of finitely generated $\mathfrak R$-modules induces a $k$-morphism $\s B\to \s C$ of associated $\s R$-module schemes.   If $\mathfrak I$ is an ideal of $\mathfrak R$, then the canonical projection $\mathfrak R\to \mathfrak R^{(\mathfrak I\le)}=\mathfrak R/\le\mathfrak I$ induces a $k$-morphism of associated $\s R$-module schemes
\begin{equation}\label{rri}
\s R\to \s R^{\e(\lbe\s I)}.
\end{equation}

\begin{proposition}\label{uprop}  Let $\mathfrak R$ be a finite $W_{\be m}(k)$-algebra $\mathfrak M$ a finitely generated $\mathfrak R$-module and $\s R$ (respectively, $\s M$) the Greenberg algebra (respectively, module) associated to $\mathfrak R$ (respectively, $\mathfrak M$). Then, for every $k$-algebra $A$, there exist a canonical surjective homomorphism of $\e\mathfrak R$-$\e W_{\!m}(A)$-bialgebras 
\[
\mathfrak R\otimes_{\e W_{\be m}(k)}\be W_{\!m}(A)\twoheadrightarrow \s R(A)
\]
and a canonical surjective homomorphism of $\e \mathfrak R$-$\e W_{\!m}(A)$-bimodules
\[
\mathfrak M \otimes_{\e W_{\be m}(k)}\be W_{\!m}(A)\twoheadrightarrow \s M (A).
\] 
If $A=A^{\le p}$, both maps are isomorphisms.
\end{proposition}
\begin{proof} In \cite[Corollary A.2, p.~75]{lip} set $R=W_{\be m}(k)$ and  $M=\mathfrak R$ (respectively, $M=\mathfrak M$) to obtain the first (respectively, second) homomorphism of the statement.
\end{proof}

\begin{remarks}\label{up} 
In the setting of the proposition, if $A=A^{\le p}$, then the  isomorphism \eqref{wepi1} induces an isomorphism $\mathfrak R\otimes_{\e W\lbe(k)}\be W\be(\be A)\simeq \mathfrak R\otimes_{\e W_{\be m}(k)}\be W_{\!m}\lbe(A)$. Composing the preceding map with the first isomorphism of the proposition, we obtain a canonical isomorphism $\mathfrak R\le\otimes_{\e W\lbe(k)}\! W\lbe(\be A)\simeq\s R\le(A)$  of $\e\mathfrak R$-$\e W\be(\be A)$-bialgebras. Similarly, there exists a canonical isomorphism $\mathfrak M \otimes_{\e W(k)}\be W\be(\be A)\simeq\s M\be(A)$ of $\e\mathfrak R$-$\e W(A)$-bimodules. 
\end{remarks}

Together with \ref{eqrf}, the following proposition is the key to establishing the representability of the Greenberg functor \eqref{fun} in a general scheme-theoretic setting.

\begin{proposition}\label{very} Let $\mathfrak R$ be a finite $W_{\be m}(k)$-algebra with associated Greenberg algebra $\s R$ and let $A$ be any $k$-algebra. For every $f\in A$, there exists a canonical isomorphism of $\s R\le(A)$-algebras
\[
\s R\le(A)_{\le[\e f\e]}\overset{\!\sim}{\to} \s R\le(A_{\lbe f})
\]
where $[\e f\e]=(\e f,0,\dots,0)\in W_{\!m}(A)$.
\end{proposition}
\begin{proof} First we observe that, since $\s R\le(A)$ is a $W_{\be m}(A)$-module,  $\s R\le(A)_{[\e f\e]}$ exists for every $f\in A$. Let $\omega\colon \mathfrak R\otimes_{\e W_{\be m}(k)}\be W_{\!m}(A)_{[\e f\e]}\overset{\!\sim}{\to} \mathfrak R\otimes_{\e W_{\be m}(k)}\be W_{\!m}(A_{f})$ be the isomorphism of $(\mathfrak R\otimes_{\e W_{\be m}(k)}\be W_{\!m}(A))$-algebras induced by \eqref{wloc} and let $\psi\colon \mathfrak R\otimes_{\e W_{\be m}(k)}\be W_{\!m}(A)\twoheadrightarrow \s R\le(A)$ (respectively, $\psi_{f}\colon \mathfrak R\otimes_{\e W_{\be m}(k)}\be W_{\!m}(A_{\lbe f})\twoheadrightarrow \s R\le(A_{\lbe f})$) denote the first homomorphism of Proposition \ref{uprop} associated to $A$ (respectively, $A_{\lbe f}$). We will make the identification
\[
(\mathfrak R\otimes_{\e W_{\be m}(k)}\be W_{\!m}(A))\otimes_{\e W_{\be m}(A)}\be W_{\!m}(A)_{[\e f\e]}=\mathfrak R\otimes_{\e W_{\be m}(k)}\be W_{\!m}(A)_{[\e f\e]}.
\]
Now let $\psi_{\e[\e f\e]}\colon \mathfrak R\otimes_{\e W_{\be m}(k)}\be W_{\!m}(A)_{[\e f\e]}\twoheadrightarrow \s R\le(A)_{[\e f\e]}$ be the composition of $\psi\otimes_{\e W_{\be m}(A)}\be W_{\!m}(A)_{[\e f\e]}$ and the canonical isomorphism
$\s R\le(A)\otimes_{\e W_{\be m}(A)}\be W_{\!m}(A)_{[\e f\e]}\simeq \s R\le(A)_{[\e f\e]}$
in \cite[Proposition 3.5, p.~39]{am}. Then the following diagrams (with canonical vertical maps) commute:
\[
\xymatrix{ \mathfrak R\otimes_{\e W_{\be m}(k)}\be W_{\!m}(A)\ar@{->>}[r]^(.6){\psi}\ar[d]& \s R\le(A)\ar[d]\\
\mathfrak R\otimes_{\e W_{\be m}(k)}\be W_{\!m}(A)_{[\e f\e]}\ar[r]^(.63){\psi_{\le[\e f\e]}}& \s R\le(A)_{[\le f\le]}
}
\]
and
\[
\xymatrix{ \mathfrak R\otimes_{\e W_{\be m}(k)}\be W_{\!m}(A)\ar@{->>}[r]^(.6){\psi}\ar[d]& \s R\le(A)\ar[d]\\
\mathfrak R\otimes_{\e W_{\be m}(k)}\be W_{\!m}(A_{f})\ar[r]^(.63){\psi_{\le f}}& \s R\le(A_{f}).
}
\]
We assume first that $A=A^{\le p}$. Then $A_{\lbe f}=(A_{\lbe f})^{\le p}$ by Lemma \ref{sp}(i) and each of $\psi,\psi_{f}$ and $\psi_{\e[\e f\e]}$ above is a ring isomorphism by Proposition \ref{uprop}. Let $\varphi=\varphi_{\be A}\colon \s R\le(A)_{[\e f\e]}\overset{\!\sim}{\to} \s R\le(A_{f})$ be the following composition of ring isomorphisms:
\[
\s R\le(A)_{[\e f\e]}\overset{\psi_{\e[\e f\e]}^{-1}}{\lra} \mathfrak R\otimes_{\e W_{\be m}(k)}\be W_{\!m}(A)_{[\e f\e]}\overset{\omega}{\lra}\mathfrak R\otimes_{\e W_{\be m}(k)}\be W_{\!m}(A_{f})\overset{\psi_{f}}{\lra}\s R\le(A_{f}).
\]
Then the following diagram of rings commutes 
\[
\xymatrix{\mathfrak R\otimes_{\e W_{\be m}(k)}\be W_{\!m}(A)_{[\e f\e]}\ar[rr]^(.55){\psi_{\le[\le f\le]}}_(.55){\simeq}\ar[d]_{\omega}^{\simeq} && \s R\le(A)_{[\e f\e]}\ar[d]_{\varphi}^{\simeq}\\
\mathfrak R\otimes_{\e W_{\be m}(k)}\be W_{\!m}(A_{\lbe f})\ar[rr]^(.55){\psi_{\lbe f}}_(.55){\simeq}&& \s R\le(A_{\lbe f}).
}
\]
Now let $A$ be any $k$-algebra. By Lemma \ref{sp}(ii), there exist injective homomorphisms of $k$-algebras $A\hookrightarrow B$ and $A_{f}\hookrightarrow B_{f}$, where $B=B^{\le p}$ and $B_{f}=(B_{f})^{\le p}$. These maps induce four ring homomorphisms $\alpha\colon \mathfrak R\otimes_{\e W_{\be m}(k)}\be W_{\!m}(A)_{[\e f\e]}\to \mathfrak R\otimes_{\e W_{\be m}(k)}\be W_{\!m}(B)_{[\e f\e]}$, $\beta\colon \mathfrak R\otimes_{\e W_{\be m}(k)}\be W_{\!m}(A_{f})\to \mathfrak R\otimes_{\e W_{\be m}(k)}\be W_{\!m}(B_{f})$, $\gamma\colon \s R\le(A)_{[\e f\e]}\hookrightarrow \s R\le(B)_{[\e f\e]}$ and $\delta\colon \s R\le(A_{f})\hookrightarrow \s R\le(B_{f})$, where the latter two are injective by Remark \ref{resp}(a) and the flatness of $W_{\!m}(A)_{f}$ over $W_{\!m}(A)$. The preceding maps fit into the following diagram of rings
\[
\xymatrix{& \mathfrak R\otimes_{\e W_{\be m}(k)}\be W_{\!m}(A)_{[\e f\e]}\ar@{->>}[rr]^(.55){\psi_{\le[\le f\le]}^{A}}\ar[ddl]_(.5){\alpha}\ar[d]^{\simeq}_{\omega^{\lle A}} &&  \s R\le(A)_{[\e f\e]}\ar@{^{(}->}[ddl]^(.64){\gamma}\ar@{-->}[d]^(.5){\varphi_{\be A}} \\
& \mathfrak R\otimes_{\e W_{\be m}(k)}\be W_{\!m}(A_{\lbe f})\ar@{->>}[rr]^{\psi_{\be f}^{A}}\ar[ddl]_(.43){\beta}   &&   \s R\le(A_{\lbe f})\ar@{^{(}->}[ddl]^(.5){\delta} \\ 
\mathfrak R\otimes_{\e W_{\be m}(k)}\be W_{\!m}(B)_{[\e f\e]}\ar[rr]^(.6){\psi_{\le[\le f\le]}^{B}}_(.6){\simeq}  \ar[d]_{\omega^{\lle B}}^{\simeq} &&  \s R\le(B)_{[\e f\e]}\ar[d]_{\varphi_{\lbe B}}^{\simeq} & \\
\mathfrak R\otimes_{\e W_{\be m}(k)}\be W_{\!m}(B_{f}) \ar[rr]^(.55){\psi_{\be f}^{B}}_(.53){\simeq}   &&  \s R\le(B_{f})  &,} 
\]
where the left-hand vertical, top and bottom rectangles commute. The diagram shows that, if $x$ and $y$ are elements of $\mathfrak R\otimes_{\e W_{\be m}(k)}\be W_{\!m}(A)_{[\e f\e]}$ such that $\psi_{\le[\le f\le]}^{\le A}(x)=\psi_{\le[\le f\le]}^{\le A}(y)$, then $\psi_{\lbe f}^{\le A}(\omega^{\le A}(x))=\psi_{\lbe f}^{\le A}(\omega^{\le A}(y))$. Consequently, there exists a unique isomorphism of rings $\varphi=\varphi_{\be A}\colon \s R\le(A)_{[\e f\e]}\overset{\!\sim}{\to} \s R\le(A_{f})$ (i.e., the broken arrow in the above diagram) so that the full diagram commutes. It remains only to check that $\varphi$ is an isomorphism of $\s R\le(A)$-algebras. To this end, we consider the diagram
\[
\xymatrix{ \mathfrak R\otimes_{\e W_{\be m}(k)}\be W_{\!m}(A) \ar@{->>}[rrrr]^(.6){\psi}\ar[dr] \ar[ddr]& &&&   \s R\le(A) \ar[dl]_{c}\ar[ddl]^{d}\\ 
&\mathfrak R\otimes_{\e W_{\be m}(k)}\be W_{\!m}(A)_{[\e f\e]}\ar@{->>}[rr]^(.55){\psi_{[\e f\e]}}\ar[d]_{\simeq}^{\omega} &&  \s R\le(A)_{[\e f\e]}\ar[d]_{\varphi}^{\simeq}&\\
&\mathfrak R\otimes_{\e W_{\be m}(k)}\be W_{\!m}(A_{\lbe f})\ar@{->>}[rr]^(.55){\psi_{f}}&&   \s R\le(A_{\lbe f}) &,
} 
\]
where all sub-diagrams, except perhaps the right-hand triangle, commute. The diagram shows that $d\circ\psi=\varphi\e\circ\e c\e\circ\e\psi$. Since the top horizontal map $\psi$ is surjective, we conclude that $d=\varphi\e\circ\e c$, i.e., the right-hand triangle commutes as well. This completes the proof.
\end{proof}

As noted at the beginning of this Section, it is shown in \cite[proof of Proposition 3(5), p.~629]{gre1} that, if $\mathfrak B\subseteq \mathfrak C$ is an inclusion of finitely generated $W_{\be m}(k)$-modules, then the  induced morphism of associated $\mathbb W_{\be m}$-module {\it varieties} is a closed immersion. In a general scheme-theoretic setting (in particular, when non-reduced schemes are allowed), the corresponding statement fails, as the following example shows.

\begin{example}\label{gr-err} Let $n\geq 1$ be an integer and let $\mathfrak B\subseteq\mathfrak C$ be an inclusion of finitely generated $W_{\be n+1}(k)$-modules with associated $\mathbb W_{\be n+1}$-module schemes $\s B$ and $\s C$, respectively. Let $\mathfrak B=p\e W_{\!n+1}(k)$ and $\mathfrak C=W_{\!n+1}(k)$. The isomorphism of $W_{\!n+1}(k)$-modules from Remark \ref{f-gr}
\[
W_{\!n}(k)\overset{\!\sim}{\to} p\e W_{\!n+1}(k), (a_{0},\dots, a_{n-1})\mapsto (0,a_{0}^{\e p},\dots, a_{n-1}^{\e p}),
\]
extends to an isomorphism of $\mathbb W_{\be n+1}$-module schemes $\mathbb W_{\!n}\simeq \s B$. Now  $\s B\to \s C$ corresponds to the morphism $\mathbb W_{\!n}\to \mathbb  W_{\!n+1}$ given by
\[
W_{\!n}(A)\to W_{\!n+1}(A),(a_{0},\dots, a_{n-1})\mapsto(0,a_{0}^{\e p},\dots, a_{n-1}^{\e p}),
\]
for every $k$-algebra $A$. Consequently, if $a$ is a nonzero element of $A$ such that $a^{\le p}=0$, then $(a,0,\dots,0)\in W_{\!n}(A)$ is a nontrivial element in the kernel of the preceding map. Thus $\s B(A)\to \s C(A)$ is not injective.  
\end{example}

The behavior pointed out in the above example has the following undesirable consequence.

\begin{remark}\label{rl-gre} Let $m>1$ be an integer and let $\mathfrak R\to\mathfrak R^{\le\prime}$ a homomorphism of finite $W_{\!m}(k)$-algebras with kernel $\mathfrak K$. Let $\s R\to\s R^{\e\prime}$ be the induced morphism of associated $\mathbb W_{\!m}$-module schemes and let $\s K$ be the $\s R$-module scheme which corresponds to $\mathfrak K$. Since the composite map $\mathfrak K\to\mathfrak R\to \mathfrak R^{\le\prime}$ is the zero homomorphism, the composite of induced morphisms $\s K\to \s R\to\s R^{\e\prime}$ is the zero morphism. However, in contrast to \eqref{eqc}, $\s K\be(Y\le)$ may fail to be equal to the kernel of $\s R\le(Y\le)\to\s R^{\e\prime}(Y\le)$ (for certain $k$-schemes $Y$). For example, if $R=W\lbe(k)$, $n$ and $A$ are as in Example \ref{gr-err}, $\mathfrak R\to \mathfrak R^{\le\prime}$ is the canonical homomorphism $W_{\! n+1}(k)\to W_{\be 1}(k)$ (so that $\mathfrak K=p\e W_{\!n+1}(k)$) and $Y=\spec A$, then $\s K\be(Y\le)$ is not equal to the kernel of $\s R(Y)\to\s R^{\e\prime}(Y)$ since $\s K\be(Y)\to \s R(Y)$ is not injective. 
\end{remark}

It follows from the above remark that the obvious scheme-theoretic analog of the following statement from \cite[Lemma 1, p.~257]{gre2} fails:
 \begin{quotation} 
{\it Suppose that $\mathfrak I$ is the kernel of a surjective homomorphism {\rm[of finite and local $W_{\le m}(k)$-algebras]} $\varphi\colon 
\mathfrak R\to\mathfrak R^{\le\prime}$ and $\mathfrak I\e \mathfrak M=0$ {\rm[where $\mathfrak M$ is the maximal ideal of $\mathfrak R$]}. Then, for every pre-scheme $Y$ over $k$, the homomorphism $\varphi(Y)\colon \s R(Y)\to\s R^{\e\prime}(Y)$ is surjective with kernel $\s I(Y)$ and $\s M\be(Y)\s I\be(Y)=0$.}
\end{quotation}
 
In fact, the preceding statement if false {\it even} in the context of \cite{gre2}, as explained in the following remark.

\begin{remark}\label{gr-max}
The quoted statement is false if $\s I$ and $\s M$ are the (maximal) Greenberg module varieties associated to $\mathfrak I$ and $\mathfrak M$. We believe that Greenberg was well aware of this fact, which led him to changing the way in which a module variety is attached to a $W_{\!m}(k)$-module depending on the particular situation being considered. To justify our assertion, we begin by recalling that Greenberg introduced the modules that bear his name in \cite[\S 1]{gre1} using a pre-Grothendieck terminology. At the beginning of the indicated section, the author declares that he intends to use the language of ``algebraic spaces'' [sic], as introduced in Cartier's seminar \cite[Expos\'e 1]{chev} (Cartier actually defines algebraic {\it sets}, not {\it spaces}). In modern terms, Greenberg works with  {\it varieties}, i.e., reduced schemes of finite type over $k$ \cite[lines above Lemma 1, p.~257]{gre2}. Let $\Omega$ be an algebraically closed field extension of $k$, $M$ a finitely generated $W_{\! m}(k)$-module and $M_{\e\Omega}=M\be\otimes_{\e W_{\! m}(k)}\be W_{\!m}(\Omega)$. In \cite[Proposition 3, p.~628]{gre1}, the author shows that there exists a unique structure of module-variety on $M_{\e\Omega}$ over (the variety) $W_{\! m}(\Omega)$ such that $M_{\e\Omega}(k)=M$ and such that the $W_{\! m}(\Omega)$-action induces separable maps. He calls such a structure {\it maximal} and shows that other structures of $W_{\! m}(\Omega)$-module variety on $M_{\e\Omega}$ are obtained as purely inseparable regular images of the maximal one. The maximal structure of module variety on $M_{\Omega}$ is the object that truly corresponds to the (scheme-theoretic) Greenberg module $\s M$ introduced in \cite[Appendix]{lip}, as can be seen by comparing the constructions in \cite[proof of Proposition 3, p.~628, first few lines]{gre1} and \cite[p.~75, lines 1--5]{lip}. As part of the same proposition \cite[Proposition 3, p.~628]{gre1}, Greenberg gives a very succinct proof of the following statement: ``every submodule of $M_{\e\Omega}$ generated by elements of $M$ is a $k$-closed subvariety''. One {\it might} interpret the above statement as saying that the module variety associated to a submodule is a submodule variety when both varieties are equipped with their maximal structures, but this is not the case, as Example \ref{gr-err} shows. Greenberg is evidently aware of this fact when he writes, in the lines following the proof, that ``... the induced structure of module-variety on the submodule need not be its maximal structure''. Further, in \cite[Lemma 1 and lines above it, p.~257]{gre2}, the author is {\it not} working with the maximal structures of module varieties of the ideals $I$ and $M$ (see \cite[p.~257, lines 5--8]{gre2}). Translated into modern terms, the above means that, when stating and proving \cite[Lemma 1, p.~257]{gre2}, the author is {\it not} considering the Greenberg module schemes associated to $\mathfrak I$ and $\mathfrak M$.
\end{remark}

In order to obtain a correct scheme-theoretic version of Greenberg's statement quoted above, we proceed as follows.

Let $\varphi\colon\mathfrak R\to\mathfrak R^{\le\prime}$ be a homomorphism of finite $W_{\!m}(k)$-algebras with kernel $\mathfrak K$ and let $\s R\to\s R^{\e\prime}$ be the induced morphism of associated $\mathbb W_{\!m}$-module schemes. The {\it ideal subscheme of $\s R$ associated to $\varphi$} is, by definition, the $\s R$-module scheme
\begin{equation}\label{kbar}
\overbar{\s K}=\krn\!\lbe\left[\e\s R\to\s R^{\,\prime}\e\right].
\end{equation}
If $\s K$ is the $\s R$-module scheme which corresponds to $\mathfrak K=\krn\le\varphi$ then, as noted in Remark \eqref{rl-gre}, the canonical exact sequence of $W_{\!m}(k)$-modules $0\to\mathfrak K\to \mathfrak R\to\mathfrak R^{\le\prime}$ induces a complex of $\mathbb W_{\!m}$-module schemes $\s K\to \s R\to\s R^{\e\prime}$. Consequently, there exists a canonical morphism of $\s R$-module schemes
\begin{equation}\label{nu}
\Theta_{\varphi}\colon \s K\to  \overbar{\s K}.
\end{equation}
By Remark \ref{uch}(c), we have  
\begin{equation}\label{nua}
\mr{Im}[\e\Theta_{\varphi}\lbe(A)\colon \s K\be(A)\to  \overbar{\s K}\lbe(A)]=\mathfrak K\,\s R(A)
\end{equation}
for every $k$-algebra $A$.
Now, if $\mathfrak R$ is a finite $W_{\!m}(k)$-algebra and $\mathfrak I$ is any ideal of $\mathfrak R$, then the {\it ideal subscheme of $\s R$ associated to $\mathfrak I$}, denoted $\overbar{\s I}$, is the ideal subscheme of $\s R$ associated to the canonical projection $\varphi\colon\mathfrak R\to\mathfrak R^{(\mathfrak I\le)}=\mathfrak R/\le\mathfrak I$, i.e.,  
\begin{equation}\label{ibar}
\overbar{\s I}=\krn\!\lbe\left[\,\s R\to \s R^{\e(\lbe\s I\lle)}\le\right],
\end{equation}
where the indicated map is the morphism \eqref{rri}.  In this case, the map \eqref{nu} will be denoted by
\begin{equation}\label{nu2}
\Theta_{\le\mathfrak I}\colon \s I\to  \overbar{\s I}.
\end{equation}
Clearly, $\overbar{\s I}=0$ if $\mathfrak I=0$. Note that, as indicated in Remark \ref{rl-gre}, \eqref{nu2} is not an isomorphism in general.

\begin{proposition}\label{rnm-bar} Let $\mathfrak R$ be a finite $W_{\!m}(k)$-algebra, where $m>1$, and let $\mathfrak I$ be an ideal of $\mathfrak R$. If $A$ is a $k$-algebra such that $A=A^{p}$, then the homomorphism of $\s R(\lbe A)$-modules 
\[
\Theta_{\le\mathfrak I}(A)\colon \s I\lbe(A) \to \overbar{\s I}\lbe(A)
\]
is surjective. Further, if $A$ is perfect, then the preceding map is an isomorphism.
\end{proposition}
\begin{proof}  Recall $\mathfrak R^{(\mathfrak I\le)}=\mathfrak R/\le\mathfrak I$. There exists a canonical commutative diagram of $W_{\! m}(A)$-modules
\[
\xymatrix{0\ar@{-->}[r] & \mathfrak I\otimes_{\e W(k)} W(A) \ar[r]\ar[d]^\simeq   & \mathfrak R \otimes_{\e W(k)} W(A) \ar[r]\ar[d]^\simeq  & \mathfrak R^{(\mathfrak I\le)}\! \otimes_{\e W\lbe(k)}\be W\lbe(A) \ar[r]\ar[d]^\simeq   & 0\\
0\ar@{-->}[r]&  \s I\be(A)\ar[r]\ar[d]^{\Theta_{\mathfrak I}(A)} &  \s R\le(A)\ar[r]\ar@{=}[d] &  \s R^{\e(\lbe\s I)}(A) \ar@{=}[d] \ar[r] &   0\\
0\ar[r] &    \overbar{\s I}\be(A) \ar[r]&  \s R\le(A)\ar[r]& \s R^{\e(\lbe\s I)}(A)\ar[r]& 0.}
\]
The vertical arrows in the top rectangle are isomorphisms by Remark \ref{up}. Further, the top row of the diagram (excluding the broken arrow) is exact by the right-exactness of the tensor product functor. Thus the middle row (excluding the broken arrow) is exact as well. Since the bottom row of the diagram is exact by  \eqref{ibar} and Remark \ref{resp}(b), the surjectivity of $\Theta_{\mathfrak I}(A)$ follows.

Now assume that $A$ is perfect. Then the broken arrows in the above diagram can be filled in since $W\lbe(A)$ is flat  over $W\lbe(k)$ by Lemma \ref{flatw}. The bijectivity of $\Theta_{\mathfrak I}(A)$ is then immediate.
\end{proof}

\begin{corollary} 
Let $\mathfrak R$ be a finite $W_{\!m}(k)$-algebra, where  $m>1$, and let $\mathfrak I$ be an ideal of $\mathfrak R$. Then the perfection of the map \eqref{nu2}, i.e., $\Theta_{\le\mathfrak I}^{\le\pf}\colon \s I^{\lle\pf}\to\overbar{\s I}^{\e\pf}\!$, is an isomorphism of perfect $k$-schemes.
\end{corollary}
\begin{proof} This follows from the last assertion of the proposition using \cite[Remark 5.18(a)]{bga}.
\end{proof}

\begin{lemma}\label{barp} Let $m>1$ and let $\mathfrak R\to\mathfrak R^{\e\prime}$ and $\mathfrak R\to \mathfrak R^{\le\prime\prime}$ be surjective homomorphisms of finite $W_{\!m}(k)$-algebras with kernels $\mathfrak I$ and $\mathfrak J$ which satisfy $\mathfrak J\e\mathfrak I=0$. Then, for every $k$-scheme $Y\be$, the ring homomorphism $\s R\le(Y\le)\to \s R^{\e\prime}\lbe(Y\le)$ induced by $\mathfrak R\to\mathfrak R^{\le\prime}$ is surjective with
kernel $\overbar{\s I}\be(Y\le)$ and $\overbarr{\s J}\be(Y\le)\le\overbar{\s I}\be(Y\le)=0$.
\end{lemma}
\begin{proof}   The induced isomorphism $\mathfrak R^{(\mathfrak I\le)}=\mathfrak R/\mathfrak I\overset{\!\sim}{\to}\mathfrak R^{\e\prime}$ defines an isomorphism of associated Greenberg algebras $\s R^{\e(\lbe\s I)}\simeq \s R^{\e\prime}$. Consequently, the maps $\s R\le(Y\le)\to \s R^{\e\prime}\lbe(Y\le)$ and $\s R\le(Y\le)\to \s R^{\e(\lbe\s I\le)}\lbe(Y\le)$ have the same kernel, namely $\overbar{\s I}(Y\le)$ \eqref{ibar}. Now, since $\s R^{\e\prime}$ is affine, the morphism $\s R\to \s   R^{\e\prime}$ has a section by Remark \ref{resp}(b), which yields the surjectivity of $\s R\le(Y\le)\to \s R^{\e\prime}\lbe(Y\le)$. In order to check that $\overbarr{\s J}\be(Y\le)\le\overbar{\s I}\be(Y\le)=0$, we may assume that $Y=\spec A$, where $A$ is a $k$-algebra. By Lemma \ref{sp}(ii), there exists an injective homomorphism of $k$-algebras $A\to B$,  where $B^{\le p}=B$. Thus, since $\s  R\le(A)$ injects into $\s R\le(B)$ by Remark \ref{resp}(a), we may assume that $A=A^{p}$. In this case there exist canonical exact and commutative diagrams of $W_{\!m}(A)$-modules
\[
\xymatrix{&\mathfrak I\otimes_{\e W_{\!m}(k)}\be W_{\!m}(A)\ar[r]\ar@{->>}[d]_{\pi_{\lbe\mathfrak I}} & \mathfrak R\otimes_{\e W_{\!m}(k)} \be W_{\!m}(A) \ar[d]^\simeq \ar[r]&  \mathfrak R^{\e\prime}\otimes_{\e W_{\!m}(k)}\be W_{\!m}(A) \ar[d]^\simeq \ar[r]& 0 \\  
0\ar[r]& \overbar{\s I}\be(A) \ar[r]^{\alpha}& \s R\le(A)\ar[r]&  \s R^{\e\prime}(A)\ar[r]& 0 
}
\]
and
\[
\xymatrix{&\mathfrak J\otimes_{\e W_{\!m}(k)}\be W_{\!m}(A)\ar[r]\ar@{->>}[d]_{\pi_{\lbe\mathfrak I}} & \mathfrak R\otimes_{\e W_{\!m}(k)}\be W_{\!m}(A) \ar[d]^\simeq \ar[r]&  \mathfrak R^{\le\prime\prime}\otimes_{\e W_{\!m}(k)}\be W_{\!m}(A) \ar[d]^\simeq \ar[r]& 0 \\  
0\ar[r]& \overbarr{\s J}\be(A) \ar[r]^{\beta}& \s R\le(A)\ar[r]&  \s R^{\e\prime\prime}(A)\ar[r]& 0,
}
\]
where the rows are exact by Remark \ref{resp}(b) together with the right-exactness of the tensor product functor and the middle and right-hand vertical maps (in both diagrams) are the isomorphisms of Proposition \ref{uprop}. In order to show that $\overbarr{\s J}\lbe(A)\le\overbar{\s I}\lbe(A)$  is the zero ideal of $\s R\le(A)$, it suffices to check that
\[
(\e\beta \circ\pi_{\lbe\mathfrak J})(\e\textstyle\sum_{\e j} y_{j}\otimes w_{j}) \cdot (\alpha \circ\pi_{\lbe\mathfrak I})(\e\textstyle\sum_{\e i} x_{i}\otimes z_{i})=0
\]
for all $y_{j}\in \mathfrak J, x_{i}\in \mathfrak I$ and $w_{j},z_{i}\in W_{\!m}(A)$.
By the commutativity of the left-hand squares in the preceding diagrams, the latter is equivalent to the vanishing of the image of $(\sum_{\e j} y_{j}\otimes w_{j})(\e\sum_{\e i} x_{i}\otimes z_{i})$ in $\mathfrak R\le\otimes_{\e W_{\!m}(k)}\be W_{\!m}(A)$. Since the preceding product equals $\sum_{\e i,j}y_{j}x_{i}\otimes w_{j}z_{i}$ and $y_{j}x_{i}\in \mathfrak J\e\mathfrak I=0$ for all $i,j$, the lemma follows.
\end{proof}

\begin{remark} \label{fgt} Proposition \ref{rnm-bar} also holds, rather trivially, in the setting of Subsection \ref{sec-k}. In this case $\overbar{\s I}=\s I$ by \eqref{eqc} and \eqref{ibar}, whence \eqref{nu2} is the identity morphism. Thus, for every $k$-algebra $A$, the map $\Theta_{\mathfrak I}(A)$ in the indicated proposition is the identity map.
Further, Lemma \ref{barp} also holds in the setting of Subsection \ref{sec-k}. The proof is similar to (and, in fact, simpler than) the above proof, using \eqref{wem}, \eqref{eqc} and Remarks \ref{resp0} in place of Remarks \ref{resp}.  
\end{remark}

Now let $\mathfrak R$ be either a finite $W_{\be m}(k)$-algebra, where $k$ is a perfect field of positive characteristic and $m>1$ is an integer, or a finite $k$-algebra over an arbitrary field $k$. In order to discuss both cases simultaneously, we adopt the following convention:

{\it $\mathfrak R$ will denote a finite $W_{\be m}(k)$-algebra, where $m\geq 1$ and $k$ is assumed to be perfect and of positive characteristic if $m>1$.}

Let $\mathfrak I$ be an ideal of $\mathfrak R$, $i\geq 1$ an integer and $A$ a $k$-algebra. We will write ${\s I}^{\lbe i}$ for the $\mathbb W_{\be m}$-module scheme associated to the ideal $\mathfrak I^{\e i}$ (we warn the reader that ${\s I}^{\lbe i}$ should not be confused with the $i$-th power of $\s I$. The latter, in fact, cannot be defined since, in general, $\s I$ is not an ideal subscheme of $\s R\e$).

By Lemma \ref{barp} and Remark \ref{fgt}, the exact sequence of $\mathfrak R$-modules
$0\to \mathfrak I^{\e i}\to \mathfrak R\to \mathfrak R/\mathfrak I^{\e i}\to 0$  associated to the pair $(\mathfrak R, \mathfrak I\le)$ induces an exact exact sequence of $\s R(A)$-modules
\begin{equation}\label{rnis}
0\to \overbar{{\s I}^{\lbe\bbe i}}(\be A)\to \s R(\be A)\to \s R^{\e(\lbe\s I^{\lbe i})}\bbe(\be A)\to 0.
\end{equation}
We note that, although $\overbar{{\s I}}$ is an ideal subscheme of $\s R$ (by definition) and $\overbar{{\s I}}^{\e i}$ is thus defined (compare with the comment above), the latter ideal subscheme is not, in general, equal to $\overbar{{\s I}^{\lbe i}}$. See \eqref{incl} below.

Now let $s\geq i\geq 1$ be integers and consider $(\mathfrak R/\mathfrak I^{\le s},\mathfrak I/\mathfrak I^{\le s})$ in place of $(\mathfrak R, \mathfrak I\le)$ above.  We will make the identifications
\[
\frac{\mathfrak R/\lle\mathfrak I^{\le s}}{(\mathfrak I/\lle\mathfrak I^{\le s})^{\le i}}=\frac{\mathfrak R/\lle\mathfrak I^{\le s}}{\mathfrak I^{\lle i}/\lle\mathfrak I^{\le s}}=\mathfrak R/\lle\mathfrak I^{\le i}.
\]
Thus there exists a canonical exact sequence of $W_{\be m}(k)$-modules
\begin{equation}\label{rnis2}
0\to \mathfrak I^{\le i}\lbe/\lle\mathfrak I^{\le s}\to \mathfrak R/\lle\mathfrak I^{\le s}\to \mathfrak R/\lle\mathfrak I^{\le i}\to 0.
\end{equation}
We will write $\overbar{{\s I}}_{\be\!\!i/\lbe s}$ for the ideal subscheme of $\s R^{\e(\s I^{\lbe s}\lbe)}$  associated to $\mathfrak I^{\e i}\lbe/\lle\mathfrak I^{\le s}$, i.e., the kernel of the morphism of $\mathbb W_{\be m}$-module schemes $\s R^{\e(\s I^{\lbe s})}\to \s R^{\e(\s I^{\lbe i})}$ induced by the map $\mathfrak R/\lle\mathfrak I^{\le s}\to \mathfrak R/\lle\mathfrak I^{\le i}$ in \eqref{rnis2}. Now consider the exact and commutative diagram of $\s R (A)$-modules
\begin{equation*} 
\xymatrix{
0\ar[r]&\overbar{{\s I}^{\lbe\bbe i} }(A)\ar[d] \ar[r]& \s R (A) \ar[r]\ar@{->>}[d]& \s R^{\e(\s I^{\lbe i})} (A)\ar[r]\ar@{=}[d]&0\\
0\ar[r]&\overbar{{\s I}}_{\be\!\!i/\lbe s}(A)\ar[r]& \s R^{\e(\s I^{\be s})}(A)  \ar[r]& \s R^{\e(\s I^{\lbe i})}(A) \ar[r]&0, 
}
\end{equation*}
whose top row is \eqref{rnis} and bottom row is similarly induced by \eqref{rnis2} via Lemma \ref{barp} and Remark \ref{fgt}. The middle vertical map above is part of sequence \eqref{rnis} with ${\s I}^{\lbe i}$ replaced by ${\s I}^{\be s}$. The diagram thus yields an exact sequence of $\s R (A)$-modules
\begin{equation}\label{rnis3} 
0\to\overbar{{\s I}^{\be s}}(\lbe A)\to \overbar{{\s I}^{\lbe i} }(A)\to \overbar{{\s I}}_{\be\!\!i/\lbe s}(A)\to 0.
\end{equation}
Now, if $1\leq j\leq s$ is an integer such that  $i+j\geq s$, then $(\mathfrak I^{\le i}\lbe/\lle\mathfrak I^{\le s})(\mathfrak I^{\e j}\lbe/\lle\mathfrak I^{\le s})=0$ and therefore Lemma \ref{barp} shows that
 \begin{equation}\label{need}
\overbar{{\s I}}_{\be\!\!i/\lbe s}(\be A)\le\overbar{{\s I}}_{\be\!\!j/\lbe s}(\be A)=0 \qquad\text{if $i+j\geq s$}.
\end{equation}
It now follows from \eqref{rnis3}  with $s=i+j$  that
\begin{equation*} 
\overbar{{\s I}^{\lbe\bbe i} }(\be A)\le\overbar{{\s I}^{\be j} }(\be A)\subseteq \overbar{{\s I}^{\lbe\bbe i+j} }(\be A)
\end{equation*}
for every pair of integers $i,j$. In particular,  for every integer $r\geq 1$, 
\begin{equation}\label{incl}
\overbar{{\s I}}\be(\be A)^{r}   \,\subseteq\, \overbar{{\s I}^{\bbe r}}(\be A).
\end{equation}
Consequently, 
\begin{equation}\label{nlp}
\hskip 1.9cm\overbar{{\s I}}\be(\be A)^{n}=0 \qquad\text{if $\mathfrak I^{\le n}=0$}, 
\end{equation}
since ${\s I}^{n}=0$ when $\mathfrak I^{\le n}=0$.

Now, for every $k$-scheme $Y$, we will write $\s R(\lbe\cO_{Y}\be)$ for the Zariski sheaf on $Y$ defined by
\begin{equation}\label{zdef}
\hskip 3.5cm\varGamma(U,\s R(\cO_{Y}))=\Hom_{\e k}(U,\s R\le)\qquad\qquad{\text{($U\subset Y$ open)}}
\end{equation}
If $U=\spec A$ is an affine subscheme of $Y$, then
\begin{equation}\label{use}
\varGamma(U,\s R(\cO_{Y}))=\s R(A).
\end{equation}
We define $\overbar{\s I}\lbe(\be\cO_{Y}\be)$ similarly. Note that, if $V$ is an open subscheme of $Y$, then $\s R(\lbe\cO_{Y}\be)\!\be\mid_{\lle V}=\s R(\lbe\cO_{\le V}\be)$ and $\overbar{\s I}\lbe(\be\cO_{Y}\be)\!\be\mid_{\lle V}=\overbar{\s I}\lbe(\be\cO_{\le V}\be)$.

\begin{lemma}\label{rnm2} Let $\mathfrak R$ be a $W_{\be m}(k)$-algebra, where $m\geq 1$. For every ideal $\mathfrak I$ of $\mathfrak R$ and every $k$-scheme $Y$, there exists a canonical exact sequence of Zariski sheaves on $Y$
\[
0\to\overbar{\s I}\lbe(\be\cO_{Y}\be)\to \s R\le(\be\s
O_{Y}\be)\to \s R^{\e(\lbe\s I)}\bbe(\be\s O_{Y}\be)\to 0.
\]
\end{lemma}
\begin{proof}  This follows directly from Lemma \ref{barp} and Remark \ref{fgt}.
\end{proof}

We will also need the following lemma. By Remarks \ref{resp0}(b) and \ref{resp}(b), if $A$ is a $k$-algebra and $I$ is a proper ideal of $A$, then the canonical surjective homomorphism of $k$-algebras $A\twoheadrightarrow A/I$ induces a surjective homomorphism of $\mathfrak R$-algebras $\s R\le\le(A)\twoheadrightarrow \s R\le(A/I\e)$. We define
\[
\s R\le(I\e)=\krn[\e\s R\le(A\e)\to \s R\le(A/I\e)\e],
\]
so that
\begin{equation}\label{rrri}
0\to \s R\le(I\e)\to \s R\le(A\e)\to \s R\le(A/I\e)\to 0
\end{equation}
is an exact sequence of $\mathfrak R$-modules.

\begin{lemma}\label{r-nilp} Let $A$ be a $k$-algebra and let $I$ and $J$ be
ideals of $A$. Then
\[
\s R\le(I\e)\e\s R\le(J\e)\subseteq \s R\le(IJ\e).
\]
\end{lemma}
\begin{proof} This follows from the fact that the functor $\s R\le(-)$ is representable.
\end{proof}

\section{The Greenberg algebra of a truncated discrete valuation ring} \label{truc}

In this Section we discuss the Greenberg algebras associated to truncated discrete valuation rings, which are the motivating examples of the theory.

Let $R$ be a discrete valuation ring with valuation $v$, field of fractions $K$, maximal ideal $\mm$ and residue field $k=R/\mm$. We will write $\widehat{R}$ for the $\mm$-adic completion of $R$ and $\widehat{K}$ for the field of fractions of $R$. Let $\kbar$ be a fixed algebraic closure of $k$. In the unequal characteristics case, i.e., when ${\rm{char}}\, R=0$ and ${\rm{char}}\e k=p>0$, we assume that $k$ is {\it perfect}. For every  $n\in\N$, set $R_{\le n}=R/\mm^{n}$. Then, by \cite[III, \S4, Proposition 8, p.~205]{bou}, $R_{\le n}=\widehat{R}/\mm^{n}=\widehat{R}_{\le n}$ for every $n\in\N$. Consequently, {\it in all constructions that depend only on the truncations $R_{\le n}$, such as those in this Section, we will assume, without loss of generality, that $R$ is complete}. Now, for each $n\in\N$, set  $M_{\lbe n}=\mm/\mm^{n}$. Clearly, $R_{\le n}$ is an artinian local ring with maximal ideal $M_{n}$. We will write $S=\spec R$, $S_{n}=\spec R_{\le n}$ and $q_{\le n}$ for the canonical map $R\to R_{\le n}$. For every pair of integers $r\geq 1$ and $i\geq 0$, we let $\theta_{\lbe i}^{\e r}\colon S_{\le r}\to S_{\le r+i}$ be the morphism induced by the canonical map $R_{\e r+i}\to R_{\e r}$. 
Note that $\theta_{\lbe i}^{\e r}$ is a nilpotent immersion and thus a universal homeomorphism. Further, every $S_{\le r}$-scheme has a canonical $S_{\le r+i}\e$-scheme structure via $\theta_{\lbe i}^{\e r}$.

Now let $n,s$ be integers such that $n\geq s\geq 1$. Then multiplication by $\pi^{\le s}$ on $R$ induces a surjective homomorphism of $R_{\le n}$-modules $R_{\le n} \to M_{\lbe n}^{s}$ whose kernel is  $M_{n}^{\le n-s}$. Thus we obtain an isomorphism of $R_{\le n}$-modules
\begin{equation}\label{vid}
R_{\e n-s}\overset{\!\sim}{\to} M_{\lbe n}^{\lle s},\,\,r+\mm^{n-s}\mapsto \pi^{\le s}r+\mm^{ n} \quad (r\in R\e).
\end{equation}
Note that, since $M_{\lbe n}^{\le s}\le M_{n}^{\le n-s}=M_{\lbe n}^{\le n}=0$, the preceding map is also an isomorphism of $R_{\e n-s}\le$-modules.

If $R$ is an equal characteristics ring then, by \cite[II, \S4, Theorem 2 and comment that follows, p.~33]{self}, there exists an isomorphism $\xi\colon k[[\e t\e ]]\overset{\!\be\sim}\to R$, where $t$ is an indeterminate. Consequently $\pi=\xi(t\e)$ is a uniformizing element of $R$, i.e, $\mm=(\pi)$. Note that, if we set
$\pi_{n}=q_{n}(\pi)\in R_{\le n}$, then $R_{\le n}$ is a free $k$-module of rank $n$ with basis  $1,\pi_{n},\dots,\pi_{n}^{\le n-1}$ for every $n\geq 1$.  In particular, the ring $R_{\le n}$ is of the type discussed in Subsection \ref{sec-k}.
We now fix the preceding isomorphism and write $\tau_{_{\! R}}\colon k\to R$ for the canonical inclusion.
We will regard $S$ and each $S_{n}$ as a $k$-scheme via $\spec\be(\tau_{_{\! R}})$ and $\spec\be(q_{n}\tau_{_{\! R}})$, respectively. By Subsection \ref{sec-k} above, the Greenberg algebra associated to $R_{\le n}$ is the $k$-ring scheme
\begin{equation}\label{rneq}
\s R_{\lle n}=\Re_{R_{\le n}\lbe/\le k}(\mathbb O_{R_{\le
n}}\lbe),
\end{equation}
where $\Re_{\le R_{\le n}/\le k}$ is the Weil restriction functor associated to the finite and locally free morphism $\spec\be(q_{n}\tau_{_{\! R}})\colon S_{n}\to\spec k$.  See \cite[Example 2.6(2)]{ns2} for an explicit description of the ring structure on $\s R_{\le n}$. Note that $\s R_{1}=\mathbb O_{\le k}$ and $\s R_{\le n}(k)=R_{\le n}$ for every $n\geq 1$. Further, the $k$-scheme (respectively, $k$-group scheme) underlying $\s R_{\le n}$ is $\Re_{R_{\le n}\lbe/ k}(\A^{\be 1}_{R_{\le n}}\lbe)=\A_{\le k}^{\be n}$ (respectively, $\Re_{R_{\le n}\lbe/ k}(\G_{a,\le R_{\le n}}\lbe)=\G_{a,\e k}^{\le n}$). Now, by \eqref{wem} and \eqref{gr-weil}, for every $k$-algebra $A$ we have
\begin{equation}\label{eqrn}
\s R_{\le n}\lbe(A)=R_{\le n}\otimes_{\e k}A  
\end{equation}
and
\begin{equation*} 
\s M_{\le n}\lbe(A)=M_{\le n}\otimes_{\e k}A =\pi_{n}\e \s R_{n}(A)\subseteq \s R_{n}(A).
\end{equation*}

\begin{remark} 
By Lemma \ref{uh0}, $\spec\be(q_{n}\tau_{_{\! R}})\colon S_{\le n}\to \spec k$ is a universal homeomorphism. Thus, by Corollary \ref{wr-uh}, $\Re_{R_{\le n}\lbe/\le k}\big(Z)$ exists for every $R_{\le n}$-scheme $Z$.
\end{remark}

\smallskip

In the unequal characteristics case, we follow the exposition in
\cite[pp.~1591-94]{ns} and \cite[\S2.2]{ns2}. See also \cite[II, \S 5]{self}. Recall that $k$ is assumed to be perfect in this case, of characteristic $p>0$. The integer $\ari=v(p)\geq 1$, which agrees with the ramification index
of $K/\mathbb Q_{\le p}$, is called the {\it absolute ramification index of $R$}. 
When $\ari=1$, $R$ is called {\it absolutely unramified}. There exists a unique $k$-monomorphism of
rings $W(k)\hookrightarrow R$ such that $R/\le W\lbe(k)$ is a totally ramified (possibly trivial) extension of degree $\ari$. In particular, $R$ is absolutely unramified if, and only if, $R=W(k)$. Assume now that $\ari>1$, so that $R$ is totally ramified over
$W(k)$. Then there exists an isomorphism
\begin{equation}\label{eis}
\xi\colon W(k)[\e T\e]/(\le f\e)\overset{\!\be\sim}\to R
\end{equation}
where $f$ is an Eisenstein polynomial of degree $\ari$ over $W(k)$, i.e., $f(T\e)=T^{\e\ari   }+a_{1}T^{\e\ari   -1}+\cdots+a_{\ari}$, where $a_{i}\in W(k)$, $p\!\!\mid\!\! a_{i}$ for all $i$ and $p^{2}\!\!\nmid\!\! a_{\ari}$. 
We now write $W(k)[\e T\e]/(\le f\le)=W(k)[\le t\le]$, where $t$ satisfies the equation 
$t^{\e\ari}+a_{1}t^{\e\ari-1}+\cdots+a_{\ari}=0$, fix the isomorphism $W(k)[\le t\le]\simeq R$ \eqref{eis} and write $\pi=\xi(t\e)$, which is a uniformizing element of $R$. For every integer $n\geq 1$, let $\pi_{n}=q_{n}(\pi)\in R_{\le n}$. 
The artinian local ring $R_{\le n}$ has characteristic $p^{\le m}$, where 
\begin{equation}\label{m}
m=\lceil n/\ari\e \rceil
\end{equation}
is  the smallest integer that is larger than or equal to $n/\ari$.
Consequently, $R_{\le n}$ is canonically an algebra over $W_{\lbe m}(k)\simeq W(k)/(\e p^{\e m})$. Note that, since  $m-1<n/\ari   \leq n$, we have $1\leq m\leq n$ and therefore 
$R_{\le n}$ is also an algebra over $W_{\be n}(k)$. As a $W_{\be m}\lbe(k)$-module, $R_{\le n}$ can be written as an internal direct sum  $W_{\be m}\lbe(k)\oplus W_{\be m}\lbe(k)\cdot\pi_{n}\oplus\cdots\oplus W_{\be m}\lbe(k)\cdot \pi_{n}^{\le r}$, where
\begin{equation}\label{r}
r=\mr{min}\{\e\ari-1,n-1\}.
\end{equation}

\begin{lemma}\label{isom} For each integer $i$ such that $0\leq i\leq r$, where $r$ is given by \eqref{r}, there exists an isomorphism of  $W_{\be m}\lbe(k)$-modules
\[
W_{\be m}\lbe(k)\be\cdot\be\pi_{n}^{\le i}\simeq W_{\be n_{i}}\lbe(k),
\]
where
\begin{equation}\label{ni}
n_{\le i}=\lceil (n-i\le)/\ari\e\rceil
\end{equation}
and $\ari$ is the absolute ramification index of $R$.
\end{lemma}
\begin{proof}  Note that, by \eqref{m} and \eqref{ni}, $n_{i}\leq m$ for every $i$ as above and therefore $W_{\be n_{i}}\lbe(k)$ has a canonical $W_{\lbe m}\lbe(k)$-module structure. Note also that $n_{i}$ is the {\it least} integer $d$ such that  $i+\ari    \e d\geq n$. Now, since  $i+\ari    \e n_{i}\geq n$ and $\pi_{n}^{\le n}=0$, we have $p^{\le n_{i}}\pi^{\le i}_{n}=0$, whence
\[
W_{\lbe m}\lbe(k)\be\cdot\be\pi^{\le i}_{n}=(W_{\lbe m}\lbe(k)/(\e p^{\le n_{i}}))\be\cdot\be\pi^{\le i}_{n}\simeq W_{\lbe n_{i}}\lbe(k)\be\cdot\be\pi^{\le i}_{n}.
\]
It remains only to check that the canonical map $W_{\lbe n_{i}}\lbe(k)\to W_{\lbe n_{i}}\lbe(k)\be\cdot\be\pi^{\le i}_{n}\e,\, a\mapsto a\be\cdot\be\pi^{\le i}_{n},$ is an isomorphism of $W_{\be m}\lbe(k)$-modules. The above map is clearly surjective. To show that it is injective, we argue by contradiction and assume that its kernel contains a nontrivial element. Then, since $p$ is a uniformizing element of $W(k)$, there exists a positive integer $r<n_{i}$ such that $p^{\e r}\!\be\cdot\be \pi^{\le i}_{n}=0$, i.e., $i+r\le \ari   \geq n$, which contradicts the minimality of $n_{i}$.
\end{proof}

The lemma shows that there exists an isomorphism of  $W_{\be m}\lbe(k)$-modules
\begin{equation}\label{dcp}
R_{\le n}\simeq \prod_{i=0}^{r}W_{\be n_{i}}\lbe(k).
\end{equation}
Note that, since  $W_{\be m}\lbe(k)$, $W\lbe(k)$ and $R$ are local rings with the same residue field,
\begin{equation}\label{lgth}
{\rm length}_{\e  W_{\be m}\lbe(k)}(R_{\le n})={\rm length}_{\e  W\lbe(k)}(R_{\le n})={\rm length}_{R}(R_{\le n})= n.
\end{equation}
See \cite[Ch. 7, Lemma 1.36(a), p.~262]{liu}. On the other hand, ${\rm length}_{\e  W_{\be m}\lbe(k)}(W_{\be n_{i}}\lbe(k))=n_{i}$ for every $i$ and \eqref{dcp} shows that  $n=n_{\le 0}+\cdots+n_{r}$,  where $n_{\le 0}=m$ by  \eqref{m} and \eqref{ni}. Further, the underlying set of $R_{\le n}$ can be identified with  $k^{\e n_{\le 0}}\times\cdots\times k^{n_{r}}=k^{\e n}$ in such a way that the ring structure on $R_{\le n}$, which is defined by the rules $f(\pi_{n})=\pi_{n}^{\le n}=0$, corresponds to a ring structure on  $k^{\e n}$ given by polynomial maps.  The resulting $k$-ring scheme $\s R_{\le n}$ agrees with the Greenberg algebra associated to $R_{\le n}$ in Subsection \ref{sec-w}. 
As a  $\mathbb W_{\be m}$-module scheme, $\s R_{\le n}$ is isomorphic to $\prod_{\e i=0}^{\e r}\!\mathbb  W_{\be n_{i}}$ and the $k$-scheme underlying $\s R_{\le n}$ is  $\A^{\be n}_{k}$. Further, $\s R_{\le n}(k)=R_{\le n}$ and  $\s R_{\le 1}\simeq \mathbb O_{k}$. In addition, if $R$ is absolutely unramified, i.e.,  $\ari   =1$ (or, equivalently, $R=W\lbe(k)$), then $r=0$ \eqref{r},  $n_{0}=n$ and $\s R_{\le n}$ is isomorphic to  $\mathbb W_{\be n}$ as a $k$-ring scheme.
\bigskip

\begin{remarks}\label{nim0}\indent
\begin{enumerate}
\item[(a)]  Write $n=q\ari   +\zeta$, where $0\leq \zeta<\ari   $ and $q\geq 0$. Note that $\zeta=0$ if, and only if, $\ari   $ divides $n$. Now the integer $n_{i}$ in \eqref{ni} equals $q+1$ if $i<\zeta$ and $q$ if $i\geq \zeta$. In particular, $m=n_{\le 0}$ equals $q+1$ if $\zeta\neq 0$ and $q$ if $\zeta=0$. Consequently, if $\zeta\neq 0$, i.e., $\ari   $ does not divide $n$,  then $n_{\le i}=m$ for $i<\zeta$ and $n_{\le i}=m-1$ for $i\geq \zeta$. On the other hand, if $\zeta=0$, i.e., $\ari   $ divides $n$, then $n_{i}=m$ for all $i$. 
\item[(b)] If $n\leq \ari$, then $m=1$ \eqref{m} and $R_{\le n}$ is a finitely generated $W_{\lbe 1}(k)=k$-algebra, i.e., a type of ring discussed in Subsection \ref{sec-k}. On the other hand, if $n>\ari   $, then $m>1$ and $R_{\le n}$ is a type of ring discussed in Subsection \ref{sec-w}. Further, in the latter case ${\rm char}\, R_{\le n}=p^{m}\neq {\rm char}\, k$.
\end{enumerate}
\end{remarks}

\smallskip

Let $R$ again be an arbitrary discrete valuation ring and let $n,s$ be integers such that  $n\geq s\geq 1$. Then $R_{\le n}$ and $R_{\le s}$ are finite $W_{\be m}(k)$-algebras, where $m$ is given by \eqref{m} if $R$ is an unequal characteristics ring and is equal to $1$ otherwise. Thus we may apply here the discussion that starts after Remark \ref{fgt} with $(\mathfrak R\le, \mathfrak I\e)=(R_{\le n}, M_{n})$ and $(\mathfrak R/\mathfrak I^{\le s},\mathfrak I/\mathfrak I^{\le s})=(R_{\le n}\le/M_{n}^{s},M_{\le n}\le/ M_{n}^{s}\e)$. Since $R_{\le n}\le/ M_{n}^{s}\simeq R_{\le s}$ and $M_{n}^{\e i}\lle/ M_{n}^{\le s}\simeq M_{s}^{i}$ for every $i\geq 1$, we may make the identifications $\s R^{\e(\s I^{\be\lle s}\lbe)}=\s R_{\le n}^{\e(\s M_{n}^{\lbe s})}=\s R_{\lbe s}$ and $\overbar{{\s I}}_{\be\!\!i/\lbe s}=\overbar{{\s M}}_{\! i/\lbe s}=\overbar{{\s M}_{\be s}^{i}}$. Thus, for every $k$-algebra $A$, \eqref{need} yields 
\begin{equation*} 
\kern 4cm\overbarr{{\s M}_{\be s}^{i}}(A)\le\overbarr{{\s M}_{\lbe s}^{\lle\jmath}}(A)=0\qquad\text{if $i+j\geq s$}.
\end{equation*}
In particular, 
\begin{equation}\label{ned2} 
\kern 4cm\overbarr{{\s M}_{\be n}^{i}}(A)\le\overbarr{{\s M}_{\lbe n}^{\lle\jmath}}(A)=0\qquad\text{if $i+j\geq n$}.
\end{equation}
Further, if $r\geq 1$ is an integer, then \eqref{incl} yields
\begin{equation}\label{mincl}
\overbarr{{\s M}_{\lbe n}}\lbe(A)^{r}\e\subseteq\e \overbar{{\s M}^{r}_{\lbe n}}\lbe(A)
\end{equation}
Thus, since $M_{n}^{n}=0$, we have
\begin{equation*} 
\overbarr{{\s M}_{\lbe n}}(\be A)^{n}=0.
\end{equation*}
Further, \eqref{rnis} with $i=s$ is identified with an exact sequence
of $\s R_{n}\lbe(A)$-modules
\begin{equation}\label{rnms}
0\to\overbarr{{\s M}_{\be n}^{\be\lle s}}(\be A)\to \s R_{n}(\be A)\to \s R_{\lbe s}(\be A)\to 0.
\end{equation}
In other words, there exists a canonical isomorphism of $\s R_{\le n}$-module schemes
\begin{equation*} 
\overbarr{{\s M}_{\lbe\bbe n}^{\be s}}=\krn\!\left[\s R _{\lbe n}\to  \s R_{\lbe s}\right].
\end{equation*}
In particular, $\overbarr{{\s M}_{\lbe n}} =\krn\!\lbe\left[\,\s R_{\lbe n}\to\mathbb O_{k}\e\right]$.

Now observe that, by the exactness of \eqref{rnms}, $\pi_{\lbe n} ^{\le s}\le\s R_{\le n}(A)\subseteq\overbarr{\s M_{\be n}^{\lbe\bbe s}}(A)$. Thus, by \eqref{ned2} with  $i=s$, we have 
\begin{equation}\label{rnms1}
\pi_{\be n} ^{\e s}\lbe\overbarr{\s M_{n}^{\le\jmath}}(A)=0 \qquad\text{if $j\geq n-s$}.
\end{equation}
In other words, $\overbarr{\s M_{n}^{\le\jmath}}(A)$ is a $\pi_{\lbe n}^{\le s}$-torsion $\s R_{n}\lbe(A)$-module for every $j\geq n-s$.

We will write
\begin{equation}\label{ttr}
\Theta_{n,\le s}\colon \s M_{\be n}^{s}\to\overbarr{{\s M}_{\be n}^{\be s}}
\end{equation}
for the canonical map \eqref{nu2}. Recall that, by Remark \ref{fgt}, \eqref{ttr} is the identity morphism in the equal characteristic case.

\begin{remarks}\label{power}\indent
\begin{itemize}
\item[(a)] If $R=W(k)$ in the unequal characteristics case, then $\s R_{n}=\mathbb{W}_{\be n}$ for every $n\in\N$. Further, if $n>s\geq 1$, then \eqref{rnms} can be identified with the sequence
\eqref{vr}. Thus $\overbarr{{\s M}_{\lbe n}^{\lbe\bbe s}}(\be A)=V^{\be s}\le W_{\be n-s}(A)\subset W_{\be n}(A)$ \eqref{sta}.
\item[(b)]  In general, the inclusion $\overbarr{{\s M}_{\lbe n}}\lbe(A)^{r}\e\subseteq\e \overbar{{\s M}^{r}_{\lbe n}}\lbe(A)$ \eqref{mincl} is strict. For example, choose $R=W(k)$ and set $n=3$ and $s=1$ in (a). By \eqref{ot} and \eqref{vv}, we have
\[
\begin{array}{rcl}
V\lbe(a_{0},a_{1})V(b_{0},b_{1})&=&V^{\lle 2}(a_{0}^{\le p}\e b_{0}^{\e p}\le)=(0,0,a_{0}^{\le p}\e b_{0}^{\e p}\e)=FV^{2}(a_{0}\le b_{0})\\
&=& p \le V(a_{0}\le b_{0})\in p\le W_{\!3}(\lbe A), 
\end{array}
\]
whence $\overbarr{{\s M}_{\lbe 3}}(A)^{2}=(\le V\le W_{\!2}(\lbe A))^{2}\subseteq p\e W_{\be 3}(A)$ by (a). On the other hand, if $A\neq  A^{\le p}$ and $c\in A\setminus A^{p}$, then $(0,0,c)$ is an element of $\overbarr{{\s M}^{\e 2}_{\lbe 3}}(A)=V^{\lle 2}\le W_{\!1}(A)$ which is not contained in $p\le W_{\!3}(A)=FV\le W_{\!2}(A)$. Thus $(0,0,c)\notin\overbarr{{\s M}_{\lbe 3}}(A)^{2}$.
\item[(c)]  The containment \eqref{mincl} is an equality in the unequal characteristics case if $A$ is perfect and $n>\ari$, where $\ari$ is the absolute ramification index of $R$ (so that $m>1$ \eqref{m}). Indeed, by Proposition \ref{rnm-bar}, the map $\Theta_{M_{\lbe n}^{\lbe s}}\be(A)\colon \s M_{\lbe n}^{\lbe s}(A)\to \overbarr{\s M_{\be n}^{\lbe\bbe s}}(A)$ is an isomorphism for every $n$ and $s\geq 1$. On the other hand, $\s M_{\lbe n}^{\lbe s}\lbe(A)\simeq  \pi_{\lbe n}^{\le s}\s R_{\le n}(A)\simeq \s M_{\lbe n}\be(A)^{\lbe s}$, as follows from Remark \ref{up} and Lemma \ref{flatw}. 
 
\item[(d)] If $R_{\le n}$ is a $k$-algebra (which holds if $R$ is an equal characteristic ring or if $R$ is an unequal characteristics ring and $n\leq \ari$, as $m=1$ \eqref{m} in the latter case), then  \eqref{mincl} is also an equality. Indeed, in this case $\s M_{\lbe n}^{\lbe s}(A)=\overbarr{\s M_{\be n}^{\lbe\bbe s}}(A)$ for every $A$ by Remark \ref{fgt} and $\s M_{\lbe n}^{\lbe s}\be(A)\simeq  \pi_{\lbe n}^{\le s}\le\s R_{\le n}\be(A)\simeq\s M_{\lbe n}\lbe(A)^{\lbe s}$ by \eqref{wem}.
\end{itemize}
\end{remarks}

Let $n,s$ be integers such that  $n> s\geq 1$. 
The isomorphism of  $R_{\le n}$-modules \eqref{vid} induces an isomorphism of  $\s R_{\le n}\le$-module schemes $\s R_{\le n-s}\!\overset{\!\sim}{\to}\s M_{\lbe n}^{\lbe s}$.
 We will write 
\begin{equation}\label{pis}
\varphi_{n,\le s}\colon \s R_{n-s}\to \overbarr{\s M_{\be n}^{\lbe\bbe s}}
\end{equation}
for the composition 
\[
\s R_{\le n-s}\overset{\!\sim}{\to}\s M_{\lbe n}^{s}\to\overbarr{\s M_{\lbe n}^{\lbe\bbe s}},
\]
where the second map is the morphism of $\s R_{\le n}\le$-module schemes $\Theta_{n,\le s}$ \eqref{ttr}. Note that, by Remark \ref{fgt}, $\Theta_{n,\le s}$ is the identity morphism  in the equal characteristic case, whence \eqref{pis} is an isomorphism. 
Now the canonical homomorphism of $R_{\le n}$-modules $R_{\le n}\to M_{n}^{s}, r\mapsto \pi_{n}^{s}r,$ induces a morphism of $\s R_{\le n}$-module schemes $\s R_{\le n}\to \s M_{n}^{\bbe s}$. We will write
\begin{equation}\label{pis2}
\vartheta_{n,\le s}\colon \s R_{\le n}\to \overbarr{\s M_{\lbe n}^{\lbe\bbe s}}
\end{equation}
for the composition 
\[
\s R_{\le n}\to\s M_{\lbe n}^{s}\to\overbarr{\s M_{\lbe n}^{\lbe\bbe s}}.
\]

\begin{proposition}\label{rnm-2}  Let $n,s$ be integers such that $n>s\geq 1$. Then 
the following diagram of $\s R_{\le n}$-module schemes commutes
\[
\xymatrix{\s R_{\le n}\ar@{->>}[dr]\ar[rr]^{\vartheta_{n,\le s}}&&\overbarr{\s M_{\be n}^{\lbe s}}\,,\\
&\s R_{\le n-s}\ar[ur]_{\varphi_{\le n,\le s}}
}
\]	
where the unlabeled map is the canonical morphism \eqref{rnms} and the maps $\varphi_{n,\le s}$ and $\vartheta_{n,\le s}$ are given by \eqref{pis} and \eqref{pis2}, respectively. If $R$ is an equal characteristic ring, then $\varphi_{n,\le s}$ is an isomorphism. 
If $R$ is a ring of unequal characteristics and $A$ is a $k$-algebra, then $\varphi_{n,\le s}(A)$ is a surjection if $A=A^{\le p}$ and an isomorphism  if either $A$ is perfect or $n\leq \ari$.
\end{proposition}
\begin{proof}  The commutativity of the indicated triangle is immediate from the fact that the composition of canonical homomorphisms $R_{\le n}\to R_{\le n-s}\to M_{n}^{s}$ (where the second map is the isomorphism \eqref{vid}) is the map $R_{\le n}\to M_{n}^{s}, r\mapsto \pi_{n}^{s}r$.  The fact that $\varphi_{n,\le s}$ is an isomorphism in the equal characteristic case was noted above. For the unequal characteristics case, see Proposition \ref{rnm-bar} and note that, by Remark \ref{power}(d), $\varphi_{n,\le s}(A)$ is an isomorphism for every $A$ if $n\leq \ari$. 
\end{proof}

\begin{remarks}\label{twist}\indent  Let $k$ be a perfect field of characteristic $p>0$. 
\begin{enumerate}
\item[(a)] If $R=W(k)$ and  $n>s\geq 1$, then $\overbarr{{\s M}_{\be n}^{\lbe\bbe s}}(\be A)=V^{\lbe s}W_{\be n-s}(A)\subseteq W_{n}(A)$ for every $k$-algebra $A$ by Remark \ref{power}(a). The homomomorphism of $W_{\be n }(A)$-modules \eqref{pis}
\begin{equation}\label{pis3}
\varphi_{n,\le s}(A)\colon W_{\be n-s}(A)\to V^{\lbe s}\le W_{\be n-s}(A)
\end{equation}
is the multiplication by $p^{\le s}=V^{s}F^{s}=F^{s}V^{s}=F^{s}\!\circ\be V_{n-s,\le n}$ map (see \eqref{vr} and \eqref{ot}), i.e.,
\[
\qquad\qquad\varphi_{n,\le s}(A)(a_{0},\dots,a_{\le n-s-1})=(0,\dots,0,a_{0}^{\le p^{\le s}},\dots, a_{\le n-s-1}^{\le p^{\le s}})\qquad\text{($s$ zeroes)}.
\]
In particular, if $n\geq 2$, then $\varphi_{\le n,\e n-1}(A)$ is the map  $A\to V^{n-1}\le W_{\be 1}(A)\subseteq W_{\be n}(A), a\mapsto (0,\dots,0, a^{\le p^{n-1}})$ ($n-1$ zeroes), which is an isomorphism if, and only if, $A$ is perfect.
\item[(b)] Let $R=W(k)$ be as in (a) and let $A$ be a $k$-algebra. By  Remark \ref{power}(a) and (a) above, for every integer $n\geq 1$, $\overbarr{{\s M}_{\lbe n+1}^{\le n}}\lbe(\be A)=V^{\lbe n}W_{\be 1}(A)$ has a canonical structure of $A$-module given by $a\!\cdot\! V^{n}(b)=(a,0,\dots,0)V^{n}(b)=V^{n}\lbe (a^{\le p^{n}}b)$ \eqref{vv}.  Now recall the $A$-algebra ${}^{p^n}\!\!\lbe A$. By definition, ${}^{p^n}\!\!\lbe A$ is the ring $A$ endowed with the $A$-module structure given by $a\cdot b=a^{p^{n}}b$ for $a,b\in A$. Then the map ${}^{p^n}\!\!\lbe A \to V^{\lbe n}W_{\be 1}(\lbe A), b\mapsto V^{n}(b),$ is bijective and $A$-linear (hence, in particular, additive \cite[(1.1.5), p.~505]{ill}). If we identify ${}^{p^n}\!\!\lbe A$ and  $V^{n}W_{1}(A)$ as $A$-modules via the preceding map, then the homomorphism of $A$-modules  $\varphi_{n+1,\le n}(A)\colon W_{1}(A)\to V^{n}W_{\lbe 1}(A)$ \eqref{pis3} is identified with the $A$-linear map $A\to {}^{p^n}\!\!\lbe A, a\mapsto a^{\le p^{n}}$.
\end{enumerate}
\end{remarks}

Let ${}^{p^{n}}\lbe\be\mathbb O_{k}$ denote the $\mathbb O_{k}$-module scheme defined by ${}^{p^{n}}\lbe\be\mathbb O_{k}(\be A)={}^{p^{n}}\!\!\lbe A$ for every $k$-algebra $A$. The following statement should be compared with that in \cite[p.~257, line 10]{gre2}.

\begin{proposition}\label{gwist} Let $k$ be a perfect field of characteristic $p>0$, let $\mathfrak R$ be an artinian local ring with residue field $k$ such that ${\rm char}\, \mathfrak R\neq {\rm char}\e k$ and let $\mathfrak I$ be a minimal ideal of $\mathfrak R$. Then there exists an isomorphism of $\mathbb O_{k}$-module schemes $\overbar{{\s I}}\simeq {}^{p^{\le t}}\be\mathbb O_{k}$,  where $t\geq 0$ is a uniquely defined integer.
\end{proposition}
\begin{proof} If $m>1$ is defined by the equality ${\rm char}\, \mathfrak R=p^{\le m}$, then $\mathfrak R$ has a canonical $W_{\!m}\lbe(k)$-algebra structure by Remark \ref{uch}(b). Thus $\mathfrak R$ is a finitely generated module over the principal ideal domain $W\be(k)$, whence there exist integers $\{n_{1},\dots, n_{r}\}$ with $1\leq n_{1}\leq \dots \leq n_{r}\leq m$ and an isomorphism of $W(k)$-modules $\lambda\colon \mathfrak R\overset{\sim}{\to} \prod_{\e i=1}^{\e r}\!W_{\! n_{i}}\be(k)$. We will construct a $W\be(k)$-automorphism $\delta$ of $\prod_{\e i=1}^{\e r}\!W_{\! n_{i}}\be(k)$ such that the composition $\delta\be\circ\be \lambda \colon \mathfrak R\overset{\sim}{\to} \prod_{\e i=1}^{\e r}\!W_{\! n_{i}}\be(k)$ induces an isomorphism $\mathfrak I\overset{\!\sim}{\to}p^{\le n_{q}-1}W_{\! n_{q}}\be(k)\subset \prod_{\e i=1}^{\e r}\!W_{\! n_{i}}\be(k)$ for some $q\in\{1,\dots,r\}$. Thus, setting
\begin{equation}\label{tnq}
t=n_{q}-1,
\end{equation}
we obtain the existence of an isomorphism of $k$-modules $\mathfrak I\simeq p^{\e t}\e W_{\! s+1}\be(k)=V^{\le t}\e W_{\!1}(k)$ \eqref{opr} which induces, for every $k$-algebra $A$, an isomorphism of $A$-modules $\overbar{\s I}\be(\lbe A)\simeq V^{\le t}W_{\!1}\lbe(\lbe A)$. The proposition then follows from Remark \ref{twist}(b).

Let $\mathfrak M$ be the maximal ideal of $\mathfrak R$. The minimality hypothesis implies that $\mathfrak I$ is principal and $\mathfrak M\e\mathfrak I=0$. Let  $g$ be a fixed generator of $\mathfrak I$ and write $\lambda(\e g\le)=(w_{i})\in \prod_{\e i=1}^{\e r}\!W_{\! n_{i}}\be(k)$. Note that $(w_{i})\neq (0,\dots,0)$. Since $p\le g\in \mathfrak M\le\mathfrak I=0$, we have $w_{i}=p^{\e n_{i}-1}\widetilde{w}_{i}$ for every $i$, for some $\widetilde{w}_{i}\in W_{\! n_{i}}\be(k)$. Clearly $w_{i}\neq 0$ if, and only if, $\widetilde{w}_{i}\in W_{\! n_{i}}\be(k)^{\times}$. Let $n_{q}={\rm min}\{n_{i}\colon w_{i}\neq 0\}$. Then $\lambda(\e g\le)=p^{\le n_{q}-1}(w_{1}^{\e\prime},\dots,w_{r}^{\e\prime})$, where $w_{\lbe q}^{\e\prime}=\widetilde{w}_{q}\in W_{\! n_{q}}\be(k)^{\times}$ and, if $i\neq q$, either $w_{i}^{\e\prime}=0$ when $w_{i}=0$ or $w_{i}^{\e\prime}=p^{\le n_{i}-n_{q}}\widetilde{w}_{i}\in W_{\! n_{i}}\be(k)$ when $w_{i}\neq 0$. Clearly $p^{\le n_{q}}w_{i}^{\e\prime}\in p^{\le n_{i}}W_{\! n_{i}}\be(k)$ for every $i$. Now let $\zeta\colon W(k)^{r}\to \prod_{\e i=1}^{\e r}\!W_{\! n_{i}}\be(k)$ be the canonical projection and write $J$ for its kernel. Thus $J=\oplus_{\e i\lle =\lle 1}^{\e r}\,W(k)\e p^{\le n_{i}}e_{i}\subset W(k)^{r}$, where $\{e_{1},\dots, e_{\le r}\}$ is the canonical basis of $W(k)^{r}$. Since
$w_{\lbe q}^{\e\prime}\in W_{\! n_{q}}\be(k)^{\times}$, we may choose $v=\sum_{\e i\le=1}^{\e r}\alpha_{i}e_{i}\in W(k)^{r}$ with $\alpha_{q}\in W\lbe(k)^{\times}$ such that $\zeta\lle(v)=(w_{1}^{\le\prime},\dots,w_{r}^{\le\prime})$. Note that, since $p^{\le n_{q}}w_{i}^{\e\prime}\in p^{\le n_{i}}W_{\! n_{i}}\be(k)$ for every $i$, we have $p^{\le n_{q}}\alpha_{i}\in p^{\le n_{i}}W\be(k)$ for every $i$, whence $p^{\le n_{q}}v=\sum_{\e i\le=1}^{\e r}p^{\le n_{q}}\alpha_{i}e_{i}\in J$. Now let $T$ be the automorphism of $W(k)^{r}$ defined by $T(e_{i})=e_{i}$ for $i\neq q$ and $T(e_{q})=(1+\alpha_{q}^{-1})e_{q}-\alpha_{q}^{-1}v$. Since $p^{\le n_{q}}v\in J$, we have $T(\e p^{\le n_{i}}e_{i})\in J$ for every $i$, whence $T(\lbe J\le)\subseteq J$. On the other hand, since $T(v)=e_{q}$, we have $p^{\le n_{q}}e_{q}=T(\e p^{\le n_{q}}v)\in T(\lbe J\le)$. Further, $p^{\le n_{i}}e_{i}=T(\e p^{\le n_{i}}e_{i})\in T(\lbe J\le)$ for every $i\neq q$. We conclude that $T(\lbe J\le)=J$. Let $\overline{T}$ be the automorphism of $W(k)^{r}\be/\be J$ induced by $T$ and let $\delta$ be the corresponding automorphism of $\prod_{\e i} W_{\! n_i}(k)$, i.e., $\delta=\overline{\zeta}\circ \overline{T}\circ \overline{\zeta}^{\e -1}$, where
$\overline{\zeta}\colon W(k)^{r}\be/\be J\overset{\!\sim}{\to}\prod_{\e i} W_{\! n_{i}}(k)$ is the  isomorphism induced by $\zeta$. Writing $\overline{v}$ (respectively, $\overline{e_{q}}\e$) for the class of $v$ (respectively, $e_{q}$) in $W(k)^{r}\be/\be J$, we have
\[
\begin{array}{rcl}
\delta(\lambda(\e \mathfrak I\le))&=&\delta(\lambda(\e g\le))\be\displaystyle\prod_{\e i=1}^{r} W_{\! n_i}(k)=p^{\le n_{q}-1}\le\overline{\zeta}\e\big(\e\overline{T}\le(\e\overline{v}\e)\big)\be\displaystyle\prod_{\e i=1}^{r} W_{\! n_i}(k)\\
&=& p^{\le n_{q}-1}\le\overline{\zeta}\e(\overline{e_{q}}\e)\be\displaystyle\prod_{\e i=1}^{r} W_{\! n_{i}}(k)=
p^{\le n_{q}-1}W_{\! n_{q}}(k),
\end{array}
\]
as desired.
\end{proof}

\begin{remark}\label{twist3} Let $R$ be a discrete valuation ring of unequal characteristics and let $n> \ari=v(\e p\lle)\geq 1$ be an integer.
Then $R_{\le n}$ has characteristic $p^{m}$, where $m>1$ is as defined in \eqref{m}, and the pair $(\mathfrak R, \mathfrak I\e)=(R_{\le n}, M_{\lbe n}^{\lle n-1})$ satisfies the conditions of the proposition (see Remark \ref{nim0}(b)). By Remark \ref{nim0}(a) and Lemma \ref{isom}\e, the factor $W_{\be n_{i}}\be(k)$ in \eqref{dcp} corresponds to the $W_{\be m}\be(k)$-submodule  $W_{\be m}\be(k)\cdot \pi_{n}^{i}$ of $R_{\le n}$ if $i<\zeta $ and to $W_{\be m-1}(k)\cdot \pi_{n}^{i}=W_{\be m}\be(k)\cdot \pi_{n}^{i}$ if $i\geq \zeta$. Now observe that, since  $p^{\le m-1}$ divides $\pi_{n}^{n-1}$, we have $W_{\be m-1}\be(k)\cdot \pi_{n}^{n-1}=0$ in $R_{\le n}$. Consequently, multiplication by $\pi_{n}^{n-1}$ annihilates every factor $W_{\be n_{i}}(k)$ if $i\geq \zeta$. We conclude that $n_{q}=m$ and therefore $t=m-1$ in \eqref{tnq}. Thus there exists an isomorphism  of $\mathbb O_{k}$-module schemes $\overbarr{{\s M}_{\lbe n }^{\le n-1}}\simeq {}^{p^{m-1}}\be\mathbb O_{k}$ which generalizes the isomorphism $V^{\le n-1}W_{1}\simeq {}^{p^{\le n-1}}\be\mathbb O_{k}$ of Remark \ref{pis3}(b) (where $\ari=1$ and therefore $m=n$). 
\end{remark}

\section{Greenberg algebras and ramification}\label{ram}
We keep the notation and hypotheses of the previous Section. An {\it extension of discrete valuation rings} is a local and flat homomorphism of discrete valuation rings $R\to R^{\e\prime}$. We will write $\mm^{\le\prime}$,  $k^{\e\prime}$ and $K^{\le\prime}$ for the maximal ideal, residue field and fraction field of $R^{\e\prime}$, respectively. Note that $R\to R^{\e\prime}$ is faithfully flat and therefore injective, whence it induces field extensions $k\hookrightarrow  k^{\e\prime}$ and $K\hookrightarrow K^{\prime}$. 
We will say that the (possibly infinite) extension $R^{\,\prime}\be/R$ is of {\it ramification index $1$}
if $\mm R^{\e\prime}=\mm^{\le\prime}$. If $R^{\,\prime}\be/R$ is finite, i.e., $R^{\,\prime}$ is a finitely generated $R$-module, we will write $e$ for the ramification index of $R^{\,\prime}/R$, i.e., $\mm R^{\,\prime}=(\mm^{\le\prime}\le)^{e}$. Note that, if $R^{\,\prime}\be/R$ is finite, then the associated morphism $S^{\e\prime}\to S$ is finite and locally free.

We now consider extensions $R^{\,\prime}$ of $R$ as above and their corresponding Greenberg algebras, under the assumption that $R$ is {\it complete}. We will discuss first (possibly infinite) extensions of $R$ of ramification index 1. 

\medskip

If $R$ is an equal characteristic ring and $k^{\e\prime}\be/k$ is any subextension of $\e\kbar\lbe/k$, the {\it extension $R^{\,\prime}$ of $R$ of ramification index $1$ which corresponds to $k^{\e\prime}/k\e$} is
\begin{equation*} 
R^{\,\prime}=R\otimes_{k}k^{\e\prime}.
\end{equation*}
Note that $R^{\e\prime}$ is a discrete valuation ring with maximal ideal $\mm^{\e\prime}=\mm R^{\e\prime}$ and residue field $k^{\e\prime}$. Further, $R^{\,\prime}=\bigcup\, (R\otimes_{\e k}\e k^{\e\prime\prime}\le)$, where the union extends over the family of finite subextensions $k^{\e\prime\prime}/k$ of $k^{\e\prime}/k$ and each ring $R\otimes_{k}k^{\e\prime\prime}$ is a complete discrete valuation ring. For every $n\in\N$, we have
\[
R^{\,\prime}_{\le n}=R_{\le n}\otimes_{R}R^{\,\prime}=R_{\le n}\otimes_{\e k}k^{\e\prime}.
\]
Thus, by \eqref{eqrn},
\begin{equation}\label{eq}
R^{\,\prime}_{\le n}=\s R_{\le n}(k^{\e\prime}\le).
\end{equation}

Now assume that $R$ is a ring of unequal characteristics (in particular, $k$ is perfect). For every finite subextension $k^{\e\prime}/k$ of $\e\kbar\lbe/k$, there exists a
uniquely determined finite unramified extension $K^{\le\prime}\be/K$ whose residual extension is (isomorphic to) $k^{\e\prime}/k$ \cite[III, \S5, Theorem 2, p.~54]{self}. 
We will write $R^{\e\prime}$ for the integral closure of $R$ in $K^{\le\prime}$. Then $R^{\,\prime}$ is a complete discrete valuation ring with maximal ideal $\mm\le R^{\,\prime}$ and residue field $k^{\e\prime}$ (see, for example, \cite[Appendix, 1.2, p.~649]{dg}). Further, by \cite[III, \S5, Remark 1, p.~55 ]{self}, $R^{\,\prime}=R\otimes_{\e W(k)}W(k^{\e\prime}\e)$. Consequently, for every  $n\in\N$, we have
\begin{equation}\label{uneqc}
R_{\le n}\otimes_{R}R^{\,\prime}=R_{\le
n}\otimes_{\e W_{\be n}(k)}\be W_{\be n}(k^{\e\prime}\le)
\end{equation}
by \eqref{wepi1}. Now the maximal unramified extension of $K$ is $K^{\nr}=\varinjlim K^{\le\prime}$, where the inductive limit extends over the finite unramified extensions $K^{\le\prime}/K$ as above. If $k^{\e\prime}/k$ is any (i.e., possibly infinite) subextension of $\e\kbar\lbe/k$, then there exists a (possibly infinite) subextension $K^{\le\prime}\be/K$ of $K^{\nr}/K$ which corresponds to $k^{\e\prime}/k$. We define the {\it extension of $R$ of ramification index $1$ which corresponds to $k^{\e\prime}/k\e$} as the integral closure $R^{\,\prime}$ of $R$ in $K^{\le\prime}$. Then $R^{\,\prime}$ is a discrete valuation ring with maximal ideal $\mm\le R^{\,\prime}$ and residue field $k^{\e\prime}$. If $k^{\e\prime}/k$ is finite, then $R^{\,\prime}$ is complete. We have
\begin{equation}\label{rl}
R^{\,\prime}=\varinjlim R_{L},
\end{equation}
where the inductive limit extends over the finite subextensions $L/K$ of $K^{\le\prime}\be/K$  which correspond to the finite subextensions of $k^{\e\prime}\be/k$. 
Now, since the functor  $W_{\lbe n}(-)$ commutes with filtered inductive limits for every $n\in\N$, \eqref{uneqc} and \eqref{rl} show that
\begin{equation*} 
R^{\,\prime}_{\le n}=R_{\le n}\otimes_{R}R^{\e\prime}=R_{\le
n}\otimes_{\e W_{\be n}(k)}\be W_{\be n}(k^{\e\prime}\le).
\end{equation*}
Thus, by \eqref{wepi1},
\begin{equation}\label{uneq}
R^{\,\prime}_{\le n}=R_{\le
n}\otimes_{\e W_{\be n}(k)}W_{\be n}(k^{\e\prime}\le)=R_{\le n}\otimes_{\e W(k)}W(k^{\e\prime}\le)
\end{equation}
for every  $n\in\N$.

\medskip

In both the equal and unequal characteristics cases, if $k^{\e\prime}/k$ is any subextension of $\e\kbar\lbe/k$, we will write $S_{n}^{\e\prime}$ for $\spec R^{\,\prime}_{\le n}$ and
${\s R}^{\e\prime}_{\le n}$ for the $k^{\e\prime}$-ring scheme associated
to $R^{\,\prime}_{\le n}$. If $k^{\e\prime}=\kbar$, we will write $R^{\e\prime}=R^{\e\rm{nr}}$, 
$R_{\le n}^{\e\prime}=R_{\le n}^{\e\rm{nr}}$, ${\s
R}^{\e\prime}_{\lbe n}={\s R}^{\e\rm{nr}}_{\lbe n}$ and
$\Snnr=\spec R^{\e\rm{nr}}_{n}$.

\begin{lemma}\label{rnk} Let $k^{\e\prime}\be/k$ be a subextension of $\e\kbar\lbe/k$ and let $R^{\e\prime}$ be the extension of $R$ of ramification index $1$ that corresponds to $k^{\e\prime}\be/k$. Then, for every  $n\in\N$, there exists a canonical isomorphism of $R_{n}$-algebras
\[
R^{\,\prime}_{\le n}=\s R_{n}(k^{\e\prime}\e).
\]
\end{lemma}
\begin{proof} The equal characteristics case is \eqref{eq}. In the unequal characteristics case, the lemma follows by combining \eqref{uneq} and Proposition \ref{uprop}.
\end{proof}

By the lemma, we have
\begin{equation}\label{rnr}
R_{\le n}^{\e\rm{nr}}=\s R_{ n}\be\big(\e\kbar\e\big).
\end{equation}
Further, in the lemma the field $k^{\e\prime}$ is being regarded as a $k$-algebra. In general, every $k^{\e\prime}$-algebra $A$, where $k^{\e\prime}\be/k$ is a subextension of  $\e\kbar\lbe/k$, can be regarded as a $k$-algebra. By Lemma \ref{rnk}, the $R_{\le n}=\s R_{n}(k)\le$-algebra $\s R_{n}(A)$ is canonically endowed with an $R^{\,\prime}_{\le n}=\s R_{n}(k^{\e\prime}\e)$-algebra structure.

\begin{lemma}\label{unr1} Let $k^{\e\prime}\be/k$ be a subextension of $\e\kbar\lbe/k$ and let $R^{\e\prime}$ be the extension of $R$ of ramification index $1$ that corresponds to $k^{\e\prime}\be/k$. Then, for every $n\in\N$, there exists a canonical isomorphism of $k^{\e\prime}$-ring schemes
\[
{\s R}^{\e\prime}_{\lbe n}=\s R_{n}\times_{\spec k}\spec k^{\e\prime}.
\]
\end{lemma}
\begin{proof} In the equal characteristic case, the result follows
from \eqref{wrbc}, \eqref{rneq} and \eqref{eq}.
In the unequal characteristics case, it suffices to check that the fpqc sheaves of sets on the category of $k^{\e\prime}$-algebras which are represented by the $k^{\e\prime}$-schemes ${\s R}^{\e\prime}_{n}$ and $\s R_{ n}\times_{\spec k}\spec k^{\e\prime}$ are isomorphic. Since, by \eqref{bc} and \cite[Appendix A]{lip}, the indicated sheaves are the sheaves associated to the functors on $k^{\e\prime}$-algebras $A\mapsto R^{\,\prime}_{\le n}\otimes_{\e W\lbe(k^{\le\prime})}\be W\lbe(\be A)$ and $A\mapsto R_{\le n}\otimes_{\e W\lbe(k)}\be W\lbe(\be A)$ (respectively), it suffices to check that the canonical map
\[
R_{\le n}\otimes_{\e W\lbe(k)}\be W\lbe(\be A)\longrightarrow R^{\,\prime}_{\le n}\otimes_{\e W\lbe(k^{\le\prime})}\be W\lbe(\be A)
\]
is a bijection  for every $k^{\e\prime}$-algebra $A$. This follows from \eqref{uneq}.
\end{proof}

\begin{lemma}\label{unr2}
Let $k^{\e\prime}\be/k$ be a finite subextension of $\e\kbar\lbe/k$ and let $R^{\e\prime}$ be the extension of $R$ of ramification index $1$ which corresponds to $k^{\e\prime}\be/k$. Then, for every  $n\in\N$ and every $k$-algebra $A$, there exists a canonical isomorphism of $R_{\le n}^{\e\prime}$-algebras
\[
{\s R}_{n}^{\,\prime}\lbe\big(\lbe A\be\otimes_{\le k}\be k^{\e\prime}\e\big)={\s R}_{\le
n}(A)\otimes_{R_{\le n}}\! R_{\le n}^{\,\prime}.
\]
\end{lemma}
\begin{proof} In the equal characteristic case, \eqref{eqrn} yields
\[
{\s R}_{n}^{\,\prime}\lbe\big(\lbe A\be\otimes_{\e k}\be k^{\e\prime}\e\big)=\big(A\be\otimes_{\e k}\be k^{\e\prime}\le\big)\otimes_{\e k^{\e\prime}}\! {R}_{\le
n}^{\,\prime}=(A\!\otimes_{\e k}\! R_{\le n})\otimes_{R_{\le n}}\!{R}_{\le
n}^{\,\prime}={\s R}_{\le
n}(A)\otimes_{R_{\le n}}\! {R}_{\le n}^{\e\prime}.
\]
Now assume that $R$ is a ring of unequal characteristics (in particular, $k$ is perfect).
Since $k^{\e\prime}$ is an \'etale
$k$-algebra, by \cite[Theorem C.5(i), p.~84]{lip} there exists a canonical isomorphism of $\s R_{\le n}(A)$-algebras
\[
\s R_{n}\be\big(A\be\otimes_{\le k}\be k^{\e\prime}\e\big)=
\s R_{\le n}(A)\otimes_{\e W_{\be n}\lbe(k\lbe)}\be W_{\be n}\lbe(k^{\e\prime}\le).
\]
Now, by \eqref{bc} and Lemma
\ref{unr1}, there exists a canonical isomorphism of ${R}_{\le n}^{\e\prime}$-algebras
$\s R_{n}\big(A\otimes_{\e k}\e k^{\e\prime}\le\big)={\s R}_{n}^{\e\prime}\big(A\otimes_{k} k^{\e\prime}\le\big)$. On the other hand, by \eqref{uneq}, there exists a ring isomorphism  $\s R_{n}(A)\otimes_{\, W_{\be n}\be(k)}W_{\be n}\lbe(k^{\e\prime})=\s R_{n}(A)\otimes_{R_{\le n}}\! {R}_{\le n}^{\,\prime}$.  Thus there exists a ring isomorphism
\[
\alpha_{\lbe A}\colon \s R_{n}^{\,\prime}\be\big(A\otimes_{k} k^{\e\prime}\le\big)\to\s R_{n}(A)\otimes_{R_{\le n}}\! R_{\le n}^{\,\prime}
\]
which is functorial in $A$. Consequently the diagram
\[
\xymatrix{ R_{\le n}^{\,\prime}= \s R_{n}^{\e\prime}(k\otimes_k k^{\e\prime}) \ar[r]\ar[d]^{\alpha_{k}}& \s R_{n}^{\e\prime}(A \otimes_{k}\!k^{\e\prime}) \ar[d]^{\alpha_{\be A}}\\
R_{\le n}^{\,\prime}= \s R_{n}(k)\otimes_{R_{\le n}} R_{\le n}^{\,\prime}\ar[r] &\s  R_{n}(A)\otimes_{R_{\le n}} R_{\le n}^{\,\prime}
}
\]
commutes, whence $\alpha_{\lbe A}$ is an isomorphism of $R_{\le n}^{\e\prime}$-algebras. This completes the proof.
\end{proof}

\begin{remark}\label{filt} The above lemma remains valid if $k^{\e\prime}\be/k$ is infinite. The proof reduces to the above proof via a limit argument using the fact that the functors $\s R_{\le n}(-)$ and $W_{\lbe n}(-)$ commute with filtered inductive limits. 
\end{remark}

The following lemma applies to possibly ramified {\it finite} extensions of $R$.

\begin{lemma}\label{rne1}
Let $R^{\e\prime}$ be a finite extension of $R$ of ramification index $e$ with associated residue field extension $k^{\e\prime}/k\subseteq \kbar/k$. Then, for every
integer $n\geq 1$ and every $k$-algebra $A$, there exists a
canonical isomorphism of $R_{\le ne}^{\,\prime}$-algebras
\[
\s R_{\le n}(\be A)\otimes_{R_{\le n}}\! R_{\le ne}^{\,\prime}=\s R_{ne}^{\e\prime}\be\big(A\otimes_{k}\lbe k^{\e\prime}\e\big).
\]
\end{lemma}
\begin{proof} In the equal characteristic case, the proof is similar to the
proof of the corresponding case of Lemma \ref{unr2}. If $R$ is an unequal
characteristics ring and $R^{\,\prime}\be/R$ is totally ramified (respectively, of ramification index 1), then the lemma follows from \cite[Lemma 2.7, p.~1593]{ns} (respectively, Lemma \ref{unr2}). The general case follows by combining these two cases in a well-known manner.
\end{proof}

\begin{remark}\label{sp0} In the unequal characteristics case, the ring $R$ is a {\it totally ramified} (possibly trivial) extension of $W\lbe(k)$ of degree $\ari=v(p)\geq 1$. Thus, for every $i\in\N$ and every $k$-algebra $A$, the lemma yields a canonical isomorphism of $R_{\e i\le\ari}\,$-\e algebras
\[
\s R_{\le i\le\ari}(A)=R_{\e i\le\ari}\otimes_{\e W_{\be i}(k)}W_{\be i}(A).
\]
Note that, if $A=A^{\lle p}$, then $\s R_{\le i\e\ari}(A)=R_{\e i\le\ari}\otimes_{\e W(k)}\be W\lbe(\be A)$  by \eqref{wepi1}. Note also that, if $0< j\leq \ari   $ and $A$ is any $k$-algebra, then $\s R_{\le j}(A)=R_{\le j}\otimes_{k}A$ by \cite[Lemma 2.7, p.~1593]{ns}. See also Remark \ref{nim0}(b).  
\end{remark}

\medskip

Let $R^{\e\prime}$ be a finite extension of $R$ with maximal ideal $\mm^{\le\prime}$, residue field $k^{\e\prime}$ and ramification index $e$. Recall $S=\spec R$ and let $S^{\e\prime}=\spec R^{\e\prime}$. For every integer $n\geq 1$, set $S_{n}^{\e\prime}=\spec{R}^{\e\prime}_{\le n}=\spec(R^{\e\prime}/(\mm^{\e\prime}\le)^{\le n}R^{\e\prime}\e)$. Since $\mm=(\mm^{\le\prime}\le)^{e}$, there exists a canonical isomorphism
\begin{equation}\label{ram-ten}
S^{\e\prime}_{\le ne}=S^{\e\prime}\times_{S}S_{\le n}.
\end{equation}
In particular,
\begin{equation}\label{ram-0}
S^{\e\prime}_{\lbe e}=S^{\e\prime}\times_{S}S_{\lle 1}=S^{\e\prime}_{\le ne}\times_{S_{n}}S_{\lle 1}.
\end{equation}
Consequently, if $Z$ is an  $S^{\e\prime}_{\lle ne}$-scheme, then $Z\times_{S^{\e\prime}_{\lbe ne}}\!S^{\e\prime}_{e}=Z\times_{S_{n}}\!S_{\lle 1}$.

Now observe that, since $R$ is noetherian and $R\to R^{\e\prime}$ is finite and flat, $S^{\e\prime}\to S$ is finite and locally free. Therefore  the induced morphism  $f_{n}\colon S^{\e\prime}_{ne}=S^{\e\prime}\times_{S}S_{n}\to S_{ n}$ is finite and locally free as well. Further, since  $S^{\e\prime}_{\lle 1}\to S^{\e\prime}_{e}$  is a universal homeomorphism we have, by \eqref{ram-0} and Remark \ref{gpts}(e),
\[
\begin{array}{rcl}
\gamma(\lle f_{n})&=&\#\big(S^{\e\prime}_{ne}\times_{S_{n}}\spec \kbar\,\big)=\#\big(S^{\e\prime}_{e}\times_{S_{\lle 1}}\spec \kbar\,\big)\\
&=&\#\big(S^{\e\prime}_{\lle 1}\times_{S_{\lle 1}}\spec \kbar\,\big)=[k^{\e\prime}\colon k\e]_{\sep}.
\end{array}
\]
Thus $Z$ is admissible relative to $S^{\e\prime}_{\le ne}\to S_{\le n}$ (see Definition \ref{adm}) if, and only if, every set of $[k^{\e\prime}\colon k\e]_{\sep}$ points in  $Z\times_{S^{\e\prime}_{ne}}\!S^{\e\prime}_{e}(=Z\times_{S_{n}}\!S_{\lle 1})$ is contained in an open affine subscheme of $Z$.

\begin{remark}\label{tram} If
$R^{\,\prime}\be/R$ is {\it totally ramified}, then $k^{\e\prime}=k$ and therefore {\it every} $S_{ne}^{\e\prime}$-scheme is admissible relative to $S^{\e\prime}_{ne}\to S_{\le n}$ (see Remark \ref{gpts}(b)). Consequently, by Theorem \ref{wr-rep}, the Weil restriction $\Re_{S_{ne}^{\e\prime}\lle/S_{n}}(Z\lle)$ exists for every $S_{ne}^{\e\prime}$-scheme $Z$. Note that, in this case, $S^{\e\prime}_{ne}\to S_{\le n}$ is, in fact, a universal homeomorphism. This follows from \cite[Proposition 3.8.2(iv), p.~249]{ega1} and the commutative diagram
\[
\xymatrix{S_{\lle 1}^{\e\prime}\ar[d]\ar@{=}[r]& S_{\lle 1}\ar[d]\\
S^{\e\prime}_{ne}\ar[r]& S_{\le n},
}
\]
whose vertical morphisms are universal homeomorphisms.
\end{remark}
\begin{lemma}\label{adm1}  Let $n\geq 1$ be an integer and let $Z$ be an $S^{\e\prime}_{ne}$-scheme which is admissible relative to $S^{\e\prime}_{ne}\to S_{\le n}$. Then the $k^{\e\prime}$-scheme $Z\times_{ S^{\e\prime}_{ne}}S^{\e\prime}_{\lle 1}$ is admissible relative to $k^{\e\prime}\be/k$.
\end{lemma}
\begin{proof}  Since $S_{\lle 1}\to S_{n}$ is affine and $S^{\e\prime}_{ ne}\times_{S_{n}}S_{\lle 1}$ equals $S^{\e\prime}_{e}$ by \eqref{ram-0}, the $S^{\e\prime}_{e}$-scheme $Z\be\times_{S^{\e\prime}_{ne}}\be S^{\e\prime}_{e}$ is admissible relative to $f_{n}\be \times_{S_{n}}\be S_{1}\colon S^{\e\prime}_{e}\to S_{\lle 1}$ by Remark \ref{rems-adm}(e), where $f_{n}\colon S^{\e\prime}_{ne}\to S_{\le n}$. Now, since $S^{\e\prime}_{\lle 1}\to S^{\e\prime}_{e}$ is a universal homeomorphism, Remark \ref{rems-adm}(f) shows that $(Z\be \times_{S^{\e\prime}_{\le ne}}\be S^{\e\prime}_{e}\le)\be \times_{S^{\e\prime}_{e}}\be S^{\e\prime}_{\lle 1}=Z\be \times_{S^{\e\prime}_{\le ne}}\be S^{\e\prime}_{\lle 1}$ is, indeed, admissible relative to $S^{\e\prime}_{1}\to S_{\lle 1}$.
\end{proof}
\begin{lemma}\label{adm2}  Let $X^{\prime}$ be an $S^{\e\prime}$-scheme which is admissible relative to $S^{\e\prime}\to S$ and let $n\geq 1$ be an integer. Then the $S^{\le\prime}_{ne}$-scheme $X^{\prime}\times_{S^{\le\prime}}S^{\e\prime}_{ne}$ is admissible relative to
$S^{\e\prime}_{ne}\to S_{\le n}$.
\end{lemma}
\begin{proof} This is immediate from Remark \ref{rems-adm}(e) using \eqref{ram-ten}.
\end{proof}

\section{The Greenberg algebra of a discrete valuation ring}\label{gad}
Let $R$ be a discrete valuation ring. The {\it Greenberg algebra associated to $R$} is the affine $k$-scheme
\begin{equation}\label{hatr}
\widetilde{\s R}=\varprojlim_{n\e\in\e\N}\s R_{\le n},
\end{equation}
where the transition morphisms are induced by the canonical maps $R_{\le n+1}\to R_{\le n}$.
Since, for every $n\in\N$, the underlying scheme of $\s R_{n}$ is isomorphic to  $\A_{k}^{\be n}$ (see Section \ref{truc}), the underlying scheme of the ring scheme  $\widetilde{\s R}$ is isomorphic to  $\A^{\!(\N)}_{\le k}=\spec k\le [\le x_{n}; n\in\N\le]$. In particular, $\widetilde{\s R}$ is not locally of finite type. Now, if $A$ is a $k$-algebra, set
\begin{equation}\label{ra}
\widetilde{\s R}(\be A)=\Hom_{\e k}\be\big(\spec A,\widetilde{\s R}\,\big)=\varprojlim\e(\s R_{n}(A)),
\end{equation}
where the second equality follows from \eqref{hatr} via \eqref{plim} and \eqref{use}.

We will also need to consider, for every $k$-scheme $Y$, the Zariski sheaf on $Y$ defined by
\begin{equation}\label{zsf}
\widetilde{\s R}(\cO_{Y})=\varprojlim \s R_{\le n}(\be\s O_{Y}\be),
\end{equation}
where, for every $n\in\N$,  $\s R_{\le n}(\be\s O_{Y}\be)$ is the Zariski sheaf on $Y$ given by \eqref{zdef}. Then, for every open subset $U\subset Y$, we have
\[
\varGamma(U,\widetilde{\s R}(\cO_{Y}))=\varprojlim \varGamma(U,\s R_{\le n}(\be\s O_{Y}\be)).
\]
In particular, if $U=\spec A$ is an affine subscheme of $Y$, then \eqref{use} and \eqref{ra} yield
\begin{equation}\label{zsf2}
\varGamma(U,\widetilde{\s R}(\cO_{Y}))=\widetilde{\s R}(\be A).
\end{equation}

Note that the underlying set of the ring $\widetilde{\s R}(A)$ is isomorphic to  $A^{(\N)}$.
The functor $\widetilde{\s R}(-)$ thus defined is a covariant and representable functor from the category of $k$-algebras to the category of $R$-algebras (a representing object is the coordinate ring of the affine scheme $\widetilde{\s R}\,$). 
If $\phi\colon B\to A$ is a monomorphism (respectively, epimorphism) of $k$-algebras, then $\widetilde{\s R}(\phi)\colon\widetilde{\s R}(B)\to \widetilde{\s R}(A)$ is a monomorphism (respectively, epimorphism) of $R$-algebras. This follows from the fact that, as a map of sets, $\widetilde{\s R}(\phi)$ is simply the map $\phi^{(\N)}\colon B^{\le(\N)}\to A^{(\N)}$. Note also that, 
since the $k$-algebra that represents $\widetilde{\s R}\e(-)$ is not of finite presentation, the functor $\widetilde{\s R}(-)$ does {\it not} commute with filtered inductive limits. See \cite[$\text{IV}_{3}$, Corollary 8.14.2.2]{ega} and compare with Remark \ref{filt}. In addition, if $k^{\e\prime}\be/k$ is a subextension of $\e\kbar\lbe/k$ and $R^{\,\prime}$ is the extension of $R$ of ramification index 1 which corresponds to $k^{\e\prime}\be/k$, then, by Lemma \ref{rnk},
\begin{equation}\label{rk}
\widetilde{\s R}(k^{\e\prime}\e)=\widehat{R^{\e\prime}}
\end{equation}

\begin{remark}\label{rm-ns} \indent
\begin{enumerate}
\item[(a)] If $R\simeq k\le[[\le t\le]]$ is an equal characteristic ring and $A$ is a $k$-algebra then, by \eqref{eqrn},
\[
\widetilde{\s R}(A)=\varprojlim \left(R_{\le n}\otimes_{k} A\right)\simeq \varprojlim A[t]/(t^{\le n})\simeq  A[[\le t\le]]\simeq R\,\e\widehat{\otimes}_{\le k}\e A,
\]
where the last term is the completion of $R\otimes_{\le k}\lbe A$ relative to the $(t)$-adic topology. Consequently, definition \eqref{ra} coincides with that in \cite[p.~256]{ns2}.
In particular, if $A=L$ is a field extension of $k$, then
\[
\widetilde{\s R}(L)\simeq L\le[[\le t\le]].
\]
\item[(b)] Let $R$ be an unequal characteristics ring and let $A$ be a $k$-algebra such that $A=A^{p}$. By \eqref{eis}, $R\simeq W\be(k)[\le T\le]/(f)$, where $f$ is an Eisenstein polynomial. Thus, by Remark \ref{up}, for every $n\geq 1$ we have
\[
\s R_{n}(A)=R_{\le n}\otimes_{\e W(k)}\be W\lbe(\be A)\simeq W\lbe(\be A)[\le T\le]/(f, T^{\le n}).
\]
Consequently, by Lemma \ref{ww},
\begin{equation}\label{raw}
\widetilde{\s R}(A)\simeq \varprojlim  W\lbe(\be A)[T]/(f, T^{\le n})\simeq W(A)[\le T\le]/(f)\simeq R\otimes_{\le W(k)} \be W\lbe(\be A).
\end{equation}
Note that, since $R$ is a finitely generated $W(k)$-module, $R\otimes_{\e W(k)}W\lbe(\be A)\simeq R\,\widehat{\otimes}_{\e W(k)} \, W\lbe(\be A)$, whence $\widetilde{\s R}(A)\simeq R\,\widehat{\otimes}_{\e W(k)} \,W\lbe(\be A)$. Thus definition \eqref{ra} above generalizes the definition given in \cite[p.~256]{ns2} when $A=A^{\le p}$. In particular, if $A=L$ is a perfect field extension of $k$, then
\begin{equation*} 
\widetilde{\s R}(L)\simeq R\otimes_{\e W(k)} \be W(L).
\end{equation*}
Further, since $R/\le W(k)$ is an extension of complete discrete valuation rings of degree $\ari$, $\widetilde{\s R}(L)/\e W(L)$ is also an extension of complete discrete valuation rings of degree $\ari$. Moreover, the extension of complete discrete valuation rings $\widetilde{\s R}(L)/R$ has ramification index 1.
\item[(c)] We conclude that, if $L/k$ is a field extension (where $L$ is assumed to be perfect in the unequal characteristics case), then $\widetilde{\s R}(L)$ is a reduced noetherian ring.
\end{enumerate}
\end{remark}

\section{The Greenberg functor}\label{gr-art}
The Greenberg realization of a scheme of finite type over an artinian local ring was introduced in \cite{gre1}. In this Section we revisit Greenberg's construction using a scheme-theoretic approach.

Let $\mathfrak R$ be an artinian local ring with maximal ideal $\mathfrak M$ and residue field $k$ which is either a finite $W_{\be m}(k)$-algebra, where $k$ is perfect of positive characteristic and $m>1$, or a finite $k$-algebra, where $k$ is arbitrary. As before, we discuss both cases simultaneously by letting $m\geq 1$ and assuming that $k$ is perfect of positive characteristic if $m>1$.
Note that, if $\mathfrak R=k$, then $m=1$ and we are in the setting of Subsection \ref{sec-k} (see Remark \ref{excu}). Further, $\s R=\mathbb O_{k}$ by \eqref{uu}.

Let $Y$ be a $k$-scheme. By Lemma \ref{rnm2} applied to the pair $(\mathfrak R,\mathfrak M \le)$ and the identification $\s R ^{\le(\lbe\s M  )}=\mathbb O_{k}$ induced by the canonical isomorphism $\mathfrak R/\mathfrak M=k$, there exists a canonical exact sequence of Zariski sheaves on $Y$
\begin{equation}\label{rr}
0\to\overbarr{\s M}\lbe(\be\s O_{Y}\be)\to\s R\lbe(\be\s O_{Y}\be)\to \s O_{Y} \to 0.
\end{equation}
Note that, since $\mathfrak M $ is a nilpotent ideal, $\overbarr{\s M}\be(\lbe A)$ is a nilpotent ideal as well for every $k$-algebra $A$ by \eqref{nlp}. We now consider the locally ringed space over $\mathfrak R$
\begin{equation*} 
\hra(Y)=(\le|Y|, \s R\lbe(\be\cO_{Y}\be)).
\end{equation*}
By \eqref{use}, if $U=\spec A$ is an open affine subset of $|\hra(Y)|=|Y|$, then
\begin{equation}\label{use2}
\varGamma\!\left(U,\s O_{\be\hra(Y)}\right)=\varGamma(U,\s R(\be\cO_{Y}\be))=\s R(A).
\end{equation}
Further, there exists a nilpotent immersion of the special fiber $\hra(Y)_{\mr s}$ of $\hra(Y)$ into $\hra(Y)$ whose ideal sheaf is $\mathfrak M\,\s R\be(\be\cO_{Y}\be)$, where $(\mathfrak M\,\s R\lbe(\be\cO_{Y}\be))(U)=\mathfrak M\,\s R(\be A) \subseteq \overbarr{\s M}\be(\be A\be)$. Now, by
the exactness of \eqref{rr}, there exists a canonical nilpotent immersion of $k$-schemes
\begin{equation}\label{io}
\iota_{\lle Y}\colon Y\to \hra(Y)_{\mr s}
\end{equation}
which is induced by the composition  $\s R\lbe(\be\cO_{Y}\be)\be/\mathfrak M\,\s R\be(\be\cO_{Y}\be)\twoheadrightarrow\s R\lbe(\be\cO_{Y}\be)/\be\overbarr{\s M}\be(\be\s O_{Y}\be)\simeq\s O_{Y}$. 
Note that, in the setting of Subsection \ref{sec-k}, $\iota_{\lle Y}$ is an isomorphism for arbitrary $Y$ (see\eqref{kar} and Remark \ref{fgt}). In the setting of Subsection \ref{sec-w}, $\iota_{\lle Y}$ is an isomorphism for every $Y$ such that the absolute Frobenius endomorphism of $Y$ is a closed immersion (see \eqref{nua} and Proposition \ref{rnm-bar}).

\begin{proposition}\label{exist-r} Let $Y$ be a $k$-scheme. Then $\hra(Y)$ is an $\mathfrak R \le$-scheme which is affine if $\e Y$ is affine. If  $Y^{\prime}$ is a closed (respectively, open) subscheme of $Y$, then $\hra(Y^{\prime}\le)$ is a closed (respectively, open) subscheme of $\hra(Y)$.
\end{proposition}
\begin{proof}
Assume first that $Y=\spec A$ is affine. 
Then $\varGamma(\le|\hra(Y)|,\s O_{\be \hra(Y)})=\s R(A)$ by \eqref{use2}. Let $\sigma^{\mathfrak R}\colon \hra(Y)\to\spec \s R\lbe(A)$ be the morphism of locally ringed spaces which corresponds to the identity map of $\s R(A)$ under the bijection
\[
\Hom_{\e\text{loc}}\!\be\left(\hra(Y),\spec\s R\lbe(\be A)\right)\overset{\!\sim}{\to} \Hom(\s R(A),\s R(A))
\]
of \cite[Proposition 1.6.3, p.~210]{ega1}. If $\mathfrak R=k$, then $\s R=\mathbb O_{k}$ by \eqref{uu}. Further, $h^{\lbe k}\lbe(Y)=Y$ and $\sigma^{k}\colon h^{\lbe k}\lbe(Y)\to \spec A$ is the identity morphism of $Y$. Now, if $\mathfrak R$ is arbitrary, then the identity map of $|Y|$ and the projection in \eqref{rr} define a morphism of locally ringed spaces $\delta \colon Y\to \hra\lbe(Y)$.
On the other hand, by \eqref{use} and Remarks \ref{resp0}(b) and \ref{resp}(b), the sequence \eqref{rr} induces a surjective homomorphism of $W_{\be m}\lbe(k)$-algebras $\s R(A)\to A$  with (nilpotent) kernel $\overbarr{\s M}\lbe(A)$. Thus the morphism $\varsigma \colon\spec A\to \spec\s R \lbe(A)$ induced by $\s R(A)\to A$ is a nilpotent immersion. By the functoriality of the bijection in \cite[Proposition 1.6.3, p.~210]{ega1}, the following diagram commutes:
\begin{equation}\label{dhs}
\xymatrix{h^{\lbe k}\lbe(Y)\ar[d]_(.43){\delta} \ar[rr]^(.4){\sigma^{k}}_(.4){\sim}&& \spec A\ar[d]^{\varsigma}\\
\hra\lbe(Y)\ar[rr]^(.4){\sigma^{\mathfrak R}}&& \spec\s R(A).
}\end{equation} 
Since $\delta$ and $\varsigma$ are homeomorphisms, the diagram shows that $\sigma^{\mathfrak R}$ is a homeomorphism as well. On the other hand, \eqref{dhs} with $Y=D(f\le)=\spec A_{f}$, where $f\in A$, and Proposition \ref{very} together show that $\sigma^{\mathfrak R}$ maps the open locally ringed subspace $\hra\lbe(\lbe D(f))$ of $\hra\lbe(Y)$ onto the open subscheme $\spec\s R \lbe(A)_{[f]}$ of $\spec\s R \lbe(A)$. Further,
\[
\Gamma(|D(f)|, \cO_{\hra\lbe (Y)})=\s R(A_{f})\simeq \s R \lbe(A)_{[f]}= \Gamma\left (\sigma^{\mathfrak R}\lbe(\le|D(f)|), \cO_{\spec\s R(A)}\right).
\]
We conclude that $\sigma^{\mathfrak R}$ is an isomorphism of locally ringed spaces and, consequently, $\hra\lbe(Y)$ is a scheme.

If $Y$ is arbitrary, let $\{Y_{i}\}$ be a covering of $Y$ by open affine subschemes. By definition, the restriction of $\s R(\s O_{Y})$ to $|Y_{i}|$ is $\s R(\s O_{Y_{i}})$. Thus  $\hra(Y)$ is obtained by gluing the affine $\mathfrak R$-schemes $\hra(Y_{i})$, whence $\hra(Y)$ is an $\mathfrak R$-scheme, as claimed. Further, if $Y^{\prime}$ is an open subscheme of $Y$, then $\hra\le(Y^{\prime})$ is an open subscheme of $\hra(Y)$.
Finally, let $Y^{\prime}$ be a closed subscheme of $Y$. In order to show that $\hra(Y^{\prime})$ is a closed subscheme of $\hra\lbe(Y)$, we may assume that $Y$ is affine. In this case the desired conclusion follows from \eqref{rrri}. 
\end{proof}

It follows from the above proof that if $A$ is a $k$-algebra, then
\begin{equation}\label{hrn-aff}
\hra(\spec A)=\spec \s R(A).
\end{equation}
In particular, 
\begin{equation}\label{hrn-aff2}
\hra(\spec k)=\spec \mathfrak R.
\end{equation}
Thus there exists a covariant functor
\begin{equation}\label{hra-f}
\hra\colon (\mathrm{Sch}/k)\to (\mathrm{Sch}/\mathfrak R), \quad Y\mapsto
\hra(Y),
\end{equation}
which respects open, closed and arbitrary immersions. Further, \eqref{hra-f} is {\it local for the Zariski topology}, i.e., if $Y$ is a
$k$-scheme and $\{\iota_{\alpha}\colon U_{\alpha}\to Y\}_{\alpha}$ is a
Zariski covering of $Y$, then $\{\hra(\iota_{\alpha})\colon
\hra(U_{\alpha})\to\hra(Y)\}_{\alpha}$ is a Zariski covering of $\hra(Y)$.

Now, for every $\mathfrak R $-scheme $Z$, consider the contravariant functor
\begin{equation}\label{fun}
(\mathrm{Sch}/k)\to (\mathrm{Sets}), \quad Y\mapsto\Hom_{\e\mathfrak R } \big(h^{\lbe \mathfrak R}(Y),Z\big).
\end{equation}

\begin{proposition-definition}\label{pr-def1}  For every $\mathfrak R$-scheme $Z$, the functor \eqref{fun} is represented by a $k$-scheme which is denoted by ${\mathrm{Gr}} ^{\lbe \mathfrak R}(Z)$ and called the {\rm Greenberg realization of $Z$}. The assignment
\begin{equation}\label{grf}
\gra\colon (\e\mathrm{Sch}/\mathfrak R)\to
(\e\mathrm{Sch}/k), \quad Z\mapsto
\gra(Z),
\end{equation}
is a covariant functor called the {\rm Greenberg functor associated to $\mathfrak R$}, and the bijection
\begin{equation}\label{adj}
\Hom_{\e k}\big(Y, \gra(Z)\big)\simeq\Hom_{\e\mathfrak R}\big(\hra(Y),Z\big)
\end{equation}
is functorial in the variables $Y\in(\e\mathrm{Sch}/k)$ and $Z\in (\e\mathrm{Sch}/\mathfrak R)$. If $Z$ is of finite type (respectively, locally of finite type), then $\gra(Z)$ is of finite type (respectively, locally of finite type).
\end{proposition-definition}
\begin{proof} An argument completely analogous to the proof of \cite[Theorem, p.~643]{gre1}\,\footnote{ Note that in \cite{gre1,gre2} $\hra$ and $\gra$ are denoted by $G$ and $F$, respectively.} shows that, if $Z$ is of finite type over $\mathfrak R$, then $\gra(Z)$ exists, is of finite type over $k$ and the bijection \eqref{adj} is bifunctorial. In \cite{gre1}, $\gra(Z)$ is constructed in a number of steps from the particular case
\begin{equation}\label{gr-a}
\gra\lbe\big(\lbe\A_{\le\mathfrak R}^{\! d}\le\big)=\s R^{\e d}\be,
\end{equation}
where $d\geq 0$ (see \cite[Proposition 3, p.~638]{gre1} for this particular case). The {\it same} construction can be used to define $\gra(Z)$ for any $Z$ starting from the following definition:

Let $\{x_{i}\}_{i\e\in\e I}$ be a (possibly infinite) family of independent indeterminates and set $\A_{\e \mathfrak R}^{\be (I\e)}=\spec \mathfrak R[\{x_{i}\}_{i\e\in\e I}]$. For every finite subset $J$ of $I$ of cardinality $|J\e|$, let $\A_{\e \mathfrak R}^{\be (J\e)}=\spec \mathfrak R[\{x_{i}\}_{i\in J}]\simeq\A_{\e \mathfrak R}^{\be |J\e|}$. Then $\mathfrak R[\{x_{i}\}_{i\in I}]=\varinjlim_{J\subseteq I}\mathfrak R[\{x_{i}\}_{i\in J}]$, where the inductive limit extends over all finite subsets $J$ of $I$ (ordered by inclusion). Thus $\A_{\e \mathfrak R}^{\be (I\e)}=
\varprojlim_{J\subseteq I}\A_{\e \mathfrak R}^{\be( J\e)}$ in the category of $\mathfrak R$-schemes. Now set
\[
\gra\big(\lbe\A_{\e \mathfrak R}^{\be (I\e)}\le\big)=\varprojlim_{J\subseteq I}\,
\gra\big(\lbe\A_{\e \mathfrak R}^{\be (J\e)}\le\big)\simeq\varprojlim_{J\subseteq I}\s R^{\, |J\e|}.
\]
Since $\big(\gra \big(\lbe\A_{\e \mathfrak R}^{\be (J\e)}\le\big)\big)$ is a projective system of affine $k$-schemes, $\gra\big(\lbe\A_{\e \mathfrak R}^{\be (I\e)}\le\big)$ is an affine $k$-scheme
by \cite[$\text{IV}_{3}$, Proposition 8.2.3]{ega}. It remains to check that \eqref{adj} holds for $Z=\A_{\e \mathfrak R}^{\be (I\e)}$. Since \eqref{adj} holds for each $Z=\A_{\e \mathfrak R}^{\be (J\e)}$, we have, for every $k$-scheme $Y$,
\begin{align*}
\Hom_{\e\mathfrak R}\be\big(\hra(Y),\A_{\e \mathfrak R}^{\be (I\e)}\big)&=\varprojlim_{J\subseteq I}\Hom_{\e\mathfrak R}\be\big(\hra(Y),
\A_{\e \mathfrak R}^{\be (J\e)}\big)= \varprojlim_{J\subseteq I}\Hom_{\e k}\big(Y,\gra\big(\A_{\e \mathfrak R}^{\be(J\e)}\big)\big)\\
& =  \Hom_{\le k}\big(Y,\varprojlim_{J\subseteq I} \gra\big(\A_{\lbe \mathfrak R}^{\be(J\e)}\big)\big)=\Hom_{\le k}\big(Y, \gra\big(\A_{\e \mathfrak R}^{\be(I\e)}\big)\big),
\end{align*}
by \eqref{plim}. Finally, the fact that $\gra(Z)$ is locally of finite type over $k$ if $Z$ is locally of finite type over $\mathfrak R$ follows as in \cite[proof of Proposition 7, p.~642]{gre1}, using the fact that \eqref{grf} transforms affine $\mathfrak R$-schemes of finite type into affine $k$-schemes of finite type (cf. \cite[Corollary 1, p.~639]{gre1}).
\end{proof}

\begin{remark} 
It follows from the above proof that the $k$-scheme $\gra(Z)$ agrees with the realization constructed in \cite[Proposition 7, p.~641]{gre1} when $Z$ is of finite type over $\mathfrak R$.
\end{remark}

By \cite[Proposition 3, p.~638]{gre1}, the functor \eqref{grf} satisfies
\begin{equation}\label{grnsn}
\gra(\spec \mathfrak R)=\spec k.
\end{equation}
Further, for every $\mathfrak R$-scheme $Z$ and $k$-algebra $A$, \eqref{hrn-aff} and \eqref{adj} yield the equality
\begin{equation}\label{grzz}
\gra(Z)(A)=Z(\s R(A)).
\end{equation}

\begin{remarks}\label{rems1}\indent
\begin{enumerate}
\item[(a)] Both $h^{k}$ and ${\mr{Gr}}^{k}$ are the identity functors on $(\mr{Sch}/k)$.

\item[(b)] The functor \eqref{grf} transforms affine $\mathfrak R$-schemes into affine $k$-schemes and respects open, closed and arbitrary immersions. Further, if $\{Z_{i}\}$ is
an open covering of an $\mathfrak R$-scheme $Z$, then the open subschemes $\gra\lbe(Z_{i})$ cover $\gra\lbe(Z)$. The proofs of the preceding statements are similar to the proofs of \cite[Corollary 1, p.~639, Corollaries 1 and 3, p.~640, Proposition 8, p.~642, and Corollary 1, p.~642]{gre1}, using the fact that every affine $\mathfrak R$-scheme is isomorphic to a closed subscheme of $\A_{\le \mathfrak R}^{\be (I\e)}$ for some set $I$.

\item[(c)] Assume that $\mathfrak R$ is a (finite) $k$-algebra and let $Z$ be an $\mathfrak R$-scheme. Since
$|Y|=|Y\!\times_{\spec k}\spec \mathfrak R\e|$ for every $k$-scheme $Y$, \eqref{gr-weil} yields
\begin{equation}\label{hrny}
\hra(Y)=Y\!\times_{\spec k}\!\spec \mathfrak R.
\end{equation}
Thus, in this case, \eqref{fun} coincides with the Weil restriction functor of $Z$ relative to the universal homeomorphism $\spec\mathfrak R\to \spec k$. Consequently
\begin{equation}\label{eqq}
\gra=\Re_{\,\mathfrak R\lbe/k}.
\end{equation}
\item[(d)] The functor \eqref{grf} respects fiber products (the proof of this fact is similar to that in \cite[Theorem, p.~643]{gre1}). Consequently, $\gra$ defines a covariant functor from the category of $\mathfrak R$-group schemes to the category of $k$-group schemes.
In particular, there exists a canonical isomorphism of $k$-group schemes $\gra\be\big(\lbe\G_{a,\e \mathfrak R}\le\big)=\s R$.
\item[(e)]  If $G$ is a smooth $\mathfrak R$-group scheme and $d=\dim G_{\lbe\rm s}$ then, by Lemma \ref{pc0}, there exists an isomorphism of $\mathfrak R$-group schemes $\mathbb V(\omega_{G/\le \mathfrak R}^{1}\le)\simeq \G_{\lbe a,\e \mathfrak R}^{d}$. It now follows from \eqref{gr-a} and \eqref{uul2} that, if $\mathfrak R$ is a finite $W_{\be m}\lbe(k)$-algebra, then there exists an isomorphism of $k$-schemes $\mr{Gr}^{\le\mathfrak R}\lbe(\le\mathbb V\lbe(\omega_{\le G/\le\mathfrak R}^{1}))\simeq\s R^{\e d}\simeq \A_{  k}^{\be\ell\le d}$, where $\ell={\rm length}_{\e W_{\be m}\lbe(k)}(\mathfrak R)$. On the other hand, if $\mathfrak R$ is a finite $k$-algebra, then (d) and \eqref{uul} show that there exists an isomorphism of $k$-group schemes $\mr{Gr}^{\le\mathfrak R}\lbe(\le\mathbb V\lbe(\omega_{\le G/\le\mathfrak R}^{1}))\simeq\s R^{\e d}\simeq\G_{a,\e k}^{\ell\le d}$, where $\ell=\dimn\e \mathfrak R$. For example, if $\mathfrak R=W_{2}(k)$, then $\gra(\mathbb G_{a,\e \mathfrak R})=\mathbb W_{2}$, which is isomorphic to $\A_{k}^{2}$ as a $k$-scheme but is not isomorphic to $\mathbb G_{a,\e k}^{\le 2}$ as a $k$-group scheme.
\end{enumerate}
\end{remarks}

\medskip

For every $k$-scheme $Y$ and $\mathfrak R$-scheme $Z$, let  
\begin{equation}\label{vphi}
\varphi_{\le Y,\le Z}^{\mathfrak R} \colon\Hom_{\e k}\lbe\big(Y,
\gra(Z)\big)\overset{\!\sim}{\to} \Hom_{\e\mathfrak R}\lbe\big(\hra(Y),Z\big)
\end{equation}
be the bijection \eqref{adj} and let 
\begin{equation}\label{vpsi}
\psi_{\le Y,\le Z}^{\mathfrak R}\colon \Hom_{\e\mathfrak R} \lbe\big(\hra(Y),Z\big)\overset{\!\sim}{\to}
\Hom_{\e k}\lbe\big(Y, \gra(Z)\big)
\end{equation}
be its inverse. If $Y=\spec A$ and $Z=\spec B$ are affine, we will write 
$\varphi_{\le Y,\le Z}^{\mathfrak R} =\varphi_{\be A,\le B}^{\mathfrak R} $ and similarly for $\psi_{\le Y,\le Z}^{\mathfrak R}$.
By \eqref{hrn-aff2} and \eqref{grnsn}, the morphisms $1_{\spec k}$ and  $1_{\lbe \spec \mathfrak R}$ are elements of $\Hom_{\e k}\big(\spec k, \gra(\spec \mathfrak R)\big)$ and $\Hom_{\e \mathfrak R} \big(\hra(\spec k),\spec \mathfrak R\big)$, respectively, and we have
\begin{equation}\label{1-1}
\varphi_{\lbe k,\e \mathfrak R}^{\mathfrak R} (1_{\spec k})=1_{\lbe \spec \mathfrak R}
\end{equation}
and
\begin{equation}\label{1-2}
\psi_{\lbe k,\e \mathfrak R}^{\mathfrak R} (1_{\lbe \spec \mathfrak R})=1_{\spec k}.
\end{equation}
Further, since \eqref{adj} is bifunctorial, the following identities hold:
\begin{align}\label{fun1}
\varphi_{\le Y\be,\e Z}^{\mathfrak R} (g\circ u)&=\varphi_{Y^{\lbe\prime}\!,\e Z}^{\mathfrak R} (g)\circ\hra(u)\\\label{fun2}
\psi_{\le Y^{\lbe\prime}\!,\e Z}^{\mathfrak R} (v)\circ u&=\psi_{\le Y,\e Z}^{\mathfrak R} (v\circ \hra(u))\\\label{fun3}
f\circ \varphi_{\le Y^{\lbe\prime}\!,\e Z}^{\mathfrak R} (g)&=\varphi_{\le Y^{\lbe\prime}\!,\e  Z^{\prime}}^{\mathfrak R}(\gra(f)\circ g)\\\label{fun4}
\psi_{\le Y^{\lbe\prime}\!,\e Z^{\prime}}^{\mathfrak R} (f\circ v)&=\gra(f)\circ\psi_{\le Y^{\prime}\be,\e  Z}^{\mathfrak R} (v),
\end{align}
where $f\colon Z\to Z^{\e\prime}$ and $v\colon \hra(Y^{\prime})\to Z$ are $\spec \mathfrak R$-morphisms and $u\colon Y\to Y^{\prime}$ and $g\colon Y^{\prime}\to \gra(Z)$ are $k$-morphisms. In particular, \eqref{fun1} shows that
\begin{equation}\label{vphi1}
\varphi_{\le Y^{\prime}\!,\le Z}^{\mathfrak R}(\e g)=\lambda_{Z}^{\mathfrak R}  \circ \hra\lbe(\e g),
\end{equation}
where
\begin{equation}\label{lzn}
\lambda_{Z}^{\lbe\mathfrak R}=\varphi_{\lbe \gra(Z),\e Z}^{\le\mathfrak R} \lbe\big(1_{\le\gra(Z)}\big)\colon \hra(\gra(Z))\to Z.
\end{equation}
Note that, by \eqref{1-1},
\begin{equation}\label{lzi}
\lambda_{\spec \mathfrak R}^{\mathfrak R}=1_{\spec \mathfrak R}.
\end{equation}
Further, \eqref{fun3} yields the identity
\begin{equation}\label{phigr}
\varphi_{\le\gra(Z),\e Z^{\prime}}^{\mathfrak R} (\gra\lbe(\le f\le))=f\circ \lambda_{Z}^{\mathfrak R} .
\end{equation}

The following lemma extends the adjunction formula \eqref{adj}.

\begin{lemma}\label{g-adj} Let $Z^{\prime}$ be an $\mathfrak R$-scheme, $Z$ a $Z^{\prime}$-scheme and $Y$ a $\gra\lbe(Z^{\prime}\le)$-scheme. Then
\[
\Hom_{\e \gra(\lbe Z^{\prime})}\big(Y,
\gra(Z)\big)=\Hom_{\e Z^{\prime}}\big(\hra(Y),Z\big).
\]
\end{lemma}
\begin{proof} Let $f\colon Z \to Z^{\prime}$ and
$u^{\le\prime}\colon Y\to \gra(Z^{\e\prime})$ be the given structural
morphisms. Then the morphism $\varphi_{\le Y,\le Z^{\prime}}^{\mathfrak R} \be\big(u^{\le\prime}\e\big)\colon \hra(Y)\to Z^{\e\prime}$ endows $\hra(Y)$ with a $Z^{\e\prime}$-scheme structure.
Let $u\in\Hom_{\e
\gra(\lbe Z^{\prime}\lbe)}\big(Y,
\gra(Z)\big)$, i.e, $\gra(f)\circ u=u^{\le\prime}$.  Then, by \eqref{fun3},
\[
f\circ\varphi_{\le Y,\le Z}^{\mathfrak R}\lbe (u)=\varphi_{\le Y,\le Z^{\prime}}^{\mathfrak R}\lbe(\gra(f)\circ u)=\varphi_{\le Y,\le Z^{\prime}}^{\mathfrak R} \be\big(u^{\le\prime}\e\big),
\]
i.e., $\varphi_{\le Y,\le Z}^{\mathfrak R} (u)\in \Hom_{Z^{\prime}}\big(\hra(Y),Z\big)$. On the other hand, if $v\in \Hom_{Z^{\prime}}(\hra(Y),Z)$, i.e.,
$f\circ v=\varphi_{\le Y,\le Z^{\prime}}^{\mathfrak R} \be\big(u^{\le\prime}\e\big)$, then, by \eqref{fun4},
\[
\gra(f)\circ\psi_{Y,Z}^{\mathfrak R} (v)=\psi_{Y,Z^{\prime}}^{\mathfrak R}\lbe(f\circ v)=\psi_{Y,Z^{\prime}}^{\mathfrak R}\lbe (\varphi_{\le Y,\le Z^{\prime}}^{\mathfrak R} \be\big(u^{\le\prime}\e\big))=u^{\le\prime},
\]
i.e., $\psi_{Y,Z}^{\mathfrak R} (v)\in \Hom_{\e \gra(\lbe Z^{\prime}\lbe)}(Y,
\gra(Z))$.
\end{proof}

\section{The Greenberg functor of a truncated discrete valuation ring}\label{gberg}

The definitions and constructions of the preceding Section apply, in particular, to truncated  discrete valuation rings. We recall the notation introduced in Section \ref{truc}. Thus $R$ is a discrete valuation ring with maximal ideal $\mm$ and residue field $k$ (assumed to be perfect when $R$ has unequal characteristics). In this Section we may assume, without loss of generality, that $R$ is {\it complete}. Let $n\geq 1$ be an integer and let $\s R_{\le n}$ be the Greenberg algebra associated to $R_{\le n}=R/\mm^{n}$. Recall that the covariant functor \eqref{hra-f}
\begin{equation}\label{hrn-f}
\hrn=h^{R_{\lle n}}\colon (\mathrm{Sch}/k)\to (\mathrm{Sch}/R_{\le n}), \quad Y\mapsto (|Y|, \s R_{\le n}(\cO_{Y})),
\end{equation} 
respects open, closed and arbitrary immersions. Further, by \eqref{hrn-aff},
\begin{equation*} 
\hrn(\spec A)=\spec \s R_{n}\be(A),
\end{equation*}
$\hrn$ is local for the Zariski topology and, for every $R_{\le n}$-scheme $Z$, the contravariant functor 
\begin{equation*} 
(\mathrm{Sch}/k)\to (\mathrm{Sets}), \quad Y\mapsto\Hom_{\le R_{
n}}\!\big(\hrn(Y),Z\big)
\end{equation*}
is represented by a $k$-scheme $\grn(Z)=\mr{Gr}^{\le R_{\le n}}\be(Z)$ called the {\rm Greenberg realization} of $Z$. See Proposition \ref{pr-def1}. The assignment
\begin{equation}\label{grf2} 
\grn\colon (\e\mathrm{Sch}/R_{\le n})\to
(\e\mathrm{Sch}/k), \quad Z\mapsto
\grn(Z),
\end{equation}
is a covariant functor called the {\it Greenberg functor of level $n$} (associated to $R\e$), and the bijection
\begin{equation}\label{adj-bis}
\Hom_{\e k}\big(Y, \grn(Z)\big)\simeq\Hom_{\e R_{n}}\!\big(\hrn(Y),Z\big)
\end{equation}
is functorial in the variables $Y\in(\mathrm{Sch}/k)$ and $Z\in (\e\mathrm{Sch}/R_{\le n})$. If $Z$ is of finite type (respectively, locally of finite type) over $R_{\le n}$, then $\grn(Z)$ is of finite type (respectively, locally of finite type) over $k$.   
By \eqref{grnsn}, the functor \eqref{grf2} satisfies
\begin{equation*} 
\grn(S_{n})=\spec k.
\end{equation*}

\begin{lemma}\label{rat-pts} Let $n\geq 0$ be an integer and let $Z$ be an $R_{\le n}$-scheme.
\begin{enumerate}
\item[(i)] If $A$ is a $k$-algebra, then $\grn(Z)(A)=Z(\s R_{\le n}\lbe(A))$.
\item[(ii)] If $k^{\e\prime}\be/k$ is a subextension of $\e\kbar/k$ and $R^{\e\prime}$ is the extension of $R$ of ramification index $1$ which corresponds to $k^{\e\prime}\be/k$, then $\grn(Z)(k^{\e\prime}\e)=Z(R^{\,\prime}_{\le n})$.
\end{enumerate}
\end{lemma}
\begin{proof} Assertion (i) follows from \eqref{grzz}. Assertion (ii) follows from (i) using Lemma \ref{rnk}.
\end{proof}

\begin{remark}\label{rems1-bis}\indent
Assume  that $R$ is a ring of unequal characteristics, let $n\geq 1$ be an integer and recall the integer  $m=\lceil n/\e\ari\,\rceil$ \eqref{m}, where $\ari$ is the absolute ramification index of $R$. Let $Y$ be any $k$-scheme such that the absolute Frobenius morphism of $Y$ is a closed immersion. By Proposition \ref{uprop} and the fact that \eqref{hrn-f} is local for the Zariski topology, we have
\begin{equation}\label{for}
\hrn\be\big(Y \big)=W_{\be m}\lbe\big(Y \big)\be\times_{W_{\lbe m}(k)} S_{n},
\end{equation} 
where $W_{\be m}\be\big(Y \big)$ is the scheme defined in \cite[\S 1.5]{ill}.
In particular, if $m=1$, i.e., $n\leq \ari$, then $\hrn$ coincides with the base change functor $-\times_{\spec k}S_{n}$ on the category of $k$-schemes $Y$ which satisfy the indicated condition. Consequently, by \eqref{wr} and Proposition-Definition \ref{pr-def1}, we have
\begin{equation*} 
\Hom_{\e k}(Y,\grn(Z\le))=\Hom_{\e k}(Y,\Re_{R_{\le n} /k}(Z\le)\le)\qquad(\text{if $1\leq n\leq \ari\,  $})
\end{equation*}
for every $R_{\le n}$-scheme $Z$ and {\it perfect} $k$-scheme $Y$.
We call attention to the fact that \eqref{for} does {\it not} hold for arbitrary $k$-schemes $Y$. In particular, the formula in \cite[p.~276,
line~-18]{blr} is incorrect, as previously noted in \cite[p.~1592]{ns}. Note, however, that \eqref{for} is indeed valid for every $Y$ provided $n=m\le\ari$, as follows from Remark \ref{sp0} with $i=m$.
Note also that, if $R$ is {\it absolutely unramified}, then $\s R_{\le n}=\mathbb W_{\be n}$ and $\hrn\lbe(Y)=W_{n}(Y)$ for every $k$-scheme $Y$ and integer $n\geq 1$.
\end{remark}

\begin{example}\label{ex.alpha} Let $R$ be a complete discrete valuation ring of equal characteristic $p>0$. Fix an isomorphism $R\simeq k[[t]]$, so that $R_{\le n}\simeq k[t]/(t^{\le n })$ for every $n\in\N $ (see Section \ref{ga}). By Remarks \ref{rems1}(c) and \ref{rems-adm}(g), $\grn(\A^{\! 1}_{R_{\le n}}\be)=\Re_{R_{n}/k}(\A^{\!1}_{R_{\le n}}\be)=\A_{\le k}^{\be n }$. On the other hand, $\hrn\lbe(\A_{\le k}^{\be n })=\A^{\be n }_{R_{\le n}}$ by \eqref{hrny}. Thus the canonical morphism \eqref{lzn}
\[
\lambda_{\A^{\be 1}_{\lbe R_{n}}} \colon \hrn(\grn(\A^{\! 1}_{R_{\le n}}))\to \A^{\! 1}_{ R_{\le n}}  
\]
is induced by a ring homomorphism $q^{\le(n)}\colon R_{\le n}[\le x\le]\to  R_{\le n}[\le x_{\le 0},\dots, x_{n-1}\le]$. It follows from \cite[\S 7.6, proof of Theorem 4, p.~195]{blr} that $q^{\le(n)}$ is given by the formula $q^{\le(n)}(x)=\sum_{\le i=0}^{\le n-1}x_{i}\e t^{\le i}$. Since $t^{\le j}=0$ in $R_{\le n}$ for $j\geq n$, we have $q^{\le(n)}(x^{\le p})= \sum_{\le i=0}^{\le \lfloor{(n-1)/p}\rfloor}x_{i}^{\le p}t^{\le ip}$. We conclude that
\begin{equation}\label{alp}
\grn(\spec\be(R_{\le n}[x]/(x^{\le p})))\simeq\spec (k[\le x_{\le 0},\dots, x_{n-1}\le]/(x_{i}^{\e p}, i\leq (n-1)/p)).
\end{equation}
Compare with \cite[\S 7.6, proof of Proposition 2(ii), pp.~193-194]{blr}.
In particular, \eqref{alp} is not a finite $k$-scheme for every $n> 1$.

\end{example}

\section{The change of rings morphism}\label{s-cr}

We return to the setting of Section \ref{gr-art}. Thus $\mathfrak R$ is an artinian local ring with maximal ideal $\mathfrak M$ and residue field $k$ which is either a finite $W_{\be m}(k)$-algebra, where $k$ is perfect of positive characteristic and $m>1$, or a finite $k$-algebra, where $k$ is arbitrary. As before, we discuss both cases simultaneously by letting $m\geq 1$ and assuming that $k$ is perfect of positive characteristic if $m>1$. Let $\mathfrak I$ be an ideal of $\mathfrak R$, write $\mathfrak R^{\le\prime}$ for the artinian local ring $\mathfrak R^{(\mathfrak I\le)}=\mathfrak R/\le\mathfrak I$ and let $\s R^{\e\prime}$ be the corresponding Greenberg algebra $\s R^{\le(\s I\le)}$. Note that, if $\mathfrak I=\mathfrak M$, then $\mathfrak R^{\le\prime}=k$ and $\s R^{\e\prime}=\mathbb O_{k}$ by \eqref{uu}. If $X$ is an $\mathfrak R$-scheme, we will write $X^{\lbe\prime}$ for $X_{\mathfrak R^{\le\prime}}$. Note that the canonical morphism $X^{\prime}\to X$ is a nilpotent immersion and hence a universal homeomorphism. If $f\colon Z\to X$ is a morphism of $\mathfrak R$-schemes, $f_{\mathfrak R^{\prime}}$ will be denoted by $f^{\le\prime}$.

Let $Y$ be a $k$-scheme and recall the schemes $h^{\lbe \mathfrak R}\lbe (Y)$ and $h^{\lbe \mathfrak R^{\prime}}\!(Y)$ introduced in Section \ref{gr-art}. By construction,
the surjective morphism of Zariski sheaves on $Y$ with nilpotent kernel $\s R(\be\s
O_{Y}\be)\to \s R^{\e\prime}\bbe(\be\s O_{Y}\be)$ (see Lemma \ref{rnm2} and \eqref{nlp}) associates to the canonical projection $\mathfrak R\to\mathfrak R^{\le\prime}$ a nilpotent immersion 
\begin{equation}\label{dlt}
\delta_{Y}^{\e\mathfrak R,\mathfrak R^{\prime}}\colon h^{\lbe \mathfrak R^{\le\prime}}\!(Y\le)\to  h^{\lbe \mathfrak R}(Y\le)
\end{equation}
which is functorial in $Y$, i.e., if $u\colon Y\to W$ is a morphism of $k$-schemes, then the 
diagram
\begin{equation}\label{dlt2}
\xymatrix{h^{\mathfrak R^{\prime}}\!(Y)\ar[d]_(.5){\delta_{\lbe Y}^{\le\mathfrak R,\mathfrak R^{\prime}}} \ar[rr]^(.47){h^{\mathfrak R^{\prime}}\!\lbe(\le u\le)}&& h^{\lbe \mathfrak R^{\prime}}\be(W)\ar[d]^(.5){\delta_{W}^{\le\mathfrak R,\mathfrak R^{\prime}}}\\
h^{\lbe \mathfrak R}\lbe(Y)\ar[rr]^(.47){h^{\mathfrak R}\lbe\lbe(\le u\le)}&& h^{\mathfrak R}\lbe(W)
}
\end{equation} 
commutes. If $Y=\spec A$ is affine, we will write $\delta_{\! A}^{\le\mathfrak R,\mathfrak R^{\prime}}$ for $\delta_{Y}^{\le\mathfrak R,\mathfrak R^{\prime}}$. Via \eqref{hrn-aff2},
\begin{equation}\label{dlt3}
\delta_{k}^{\e\mathfrak R,\mathfrak R^{\prime}}\colon  \spec \mathfrak R^{\e\prime}\to \spec \mathfrak R,
\end{equation} 
is the nilpotent immersion defined by the projection $\mathfrak R\to \mathfrak R^{\prime}$.

Now let $X$ be an $\mathfrak R$-scheme and let $u\colon Y\to \gra (X)$ be a $k$-morphism. The image of $u$ under the bijection \eqref{vphi}
\begin{equation*} 
\varphi_{\le Y,\le X}^{\mathfrak R}  \colon\Hom_{\le k}\big(Y,
\gra(X)\big)\,\overset{\!\sim}{\to}\, \Hom_{\mathfrak R}\big(\hra(Y),X\big)
\end{equation*}
is a morphism of  $\mathfrak R$-schemes $\varphi_{\le Y,\le X}^{\le\mathfrak R}\be(u)\colon\hra\le(Y)\to X$.
By \eqref{bc}, there exists a unique morphism of $\mathfrak R^{\le\prime}$-schemes $\widetilde{v}\colon h^{\lbe \mathfrak R^{\prime}}\!(Y)\to X^{\prime}$ such that the following diagram commutes:
\begin{equation}\label{step1}
\xymatrix{h^{\lbe \mathfrak R^{\prime}}\!(Y)\ar[drr]^{v}\ar[d]_(.43){\delta_{Y}^{\le\mathfrak R,\mathfrak R^{\prime}}} \ar@{-->}[rr]^(.5){\widetilde{v}}&& X^{\prime}\ar[d]^(.43){\mr{pr}_{\lbe X}}\\
h^{\lbe \mathfrak R}\be(Y)\ar[rr]_(.5){\varphi_{\le Y,\le X}^{\mathfrak R}(u)}&& X,
}
\end{equation} 
where $\delta_{Y}^{\le\mathfrak R,\mathfrak R^{\prime}}$ is the map \ref{dlt} and we have written $v=\varphi_{\le Y,\le X}^{\mathfrak R}(u)\circ \delta_{Y}^{\le\mathfrak R,\mathfrak R^{\prime}}$. Similarly, there exists a unique morphism of $\mathfrak R^{\le\prime}$-schemes $w_{\lbe X}\colon h^{\lbe \mathfrak R^{\prime}}\!(\mr{Gr}^{\lbe \mathfrak R}(X))\to X^{\prime}$ such that the following diagram commutes:
\begin{equation}\label{step2}
\xymatrix{h^{\lbe \mathfrak R^{\prime}}\!(\mr{Gr}^{\lbe \mathfrak R}\be(X))\ar[d]_(.43){\delta_{\be\mr{Gr}^{\lbe \mathfrak R}\lbe(X)}^{\le\mathfrak R,\mathfrak R^{\prime}}} \ar@{-->}[rr]^(.55){w_{\lbe X}}&& X^{\prime}\ar[d]^(.43){\mr{pr}_{\lbe X}}\\
h^{\lbe \mathfrak R}(\mr{Gr}^{\lbe \mathfrak R}\be(X))\ar[rr]_(.55){\lambda_{\be X}^{\be\mathfrak R} }&& X,
}
\end{equation} 
where $\lambda_{\be X}^{\be\mathfrak R}$ is the map \eqref{lzn}. When $X=\spec \mathfrak R$, we have $\lambda_{\be X}^{\be\mathfrak R}=1_{\le\spec\mathfrak R}$ by \eqref{lzi} and both vertical maps above can be identified with $\delta_{k}^{\mathfrak R,\mathfrak R^{\prime}}\!$ \eqref{dlt3} via \eqref{hrn-aff2} and \eqref{grnsn}, whence
\begin{equation}\label{vr=1k0}
w_{\e\spec \mathfrak R}=1_{\spec\mathfrak R^{\prime}}.
\end{equation}
Now observe that, since  $\varphi_{\le Y,\le X}^{\le\mathfrak R}(u)$ factors as $\lambda_{\lbe X}^{\be\mathfrak R} \circ h^{\lbe \mathfrak R}\lbe(\le u\le)$ \eqref{vphi1}, diagram \eqref{step1} decomposes as
\begin{equation*} 
\xymatrix@C=35pt{
h^{\lbe \mathfrak R^{\prime}}\be(Y)\ar[r]^(.45){h^{\lbe \mathfrak R^{\prime}}\!\!(u)}\ar@/^2.0pc/[rr]|-{~\widetilde{v}~} \ar@{..>}[drr]|-(0.35){~v~}\ar[d]_(.4){\delta_{Y}^{\le\mathfrak R,\mathfrak R^{\prime}}}& h^{\lbe \mathfrak R^{\prime}}\!(\mr{Gr}^{\lbe \mathfrak R}(X))\ar[r]^(0.6){w_{\lbe X}}\ar[d]^(0.4){\delta_{\be\mr{Gr}^{\mathfrak R}\lbe(X)}^{\le\mathfrak R,\mathfrak R^{\prime}}}& X^{\prime}\ar[d]^{{\mr{pr}_{\lbe X}}}\\
h^{\lbe \mathfrak R}(Y)\ar[r]_(.4){h^{\lbe \mathfrak R}\lbe(u)}\ar@/_2.0pc/[rr]|-(0.5){~\varphi^{\mathfrak R}_{\le Y\le,\le X}\be(u)~}& h^{\lbe \mathfrak R}(\mr{Gr}^{\lbe \mathfrak R}(X))\ar[r]_(0.58){\lambda_{\be X}^{\lbe\mathfrak R} } & X,}
\end{equation*} 
where the left-hand commutative square is an instance of \eqref{dlt2} and the right-hand commutative square is \eqref{step2}. We conclude, by uniqueness, that
\begin{equation}\label{til}
\widetilde{v}=w_{\lbe X}\circ h^{\lbe \mathfrak R^{\prime}}\!\lbe(u).
\end{equation}
Thus, we have defined a map
\begin{equation*} 
\Hom_{\le k}\big(Y, \mr{Gr}^{\lbe \mathfrak R}(X)\big)\to \Hom_{\le \mathfrak R^{\le\prime}}\be\big(h^{\lbe \mathfrak R^{\prime}}\!(Y),X^{\prime}\big), \, u\mapsto w_{\lbe X}\circ h^{\lbe \mathfrak R^{\prime}}\!\lbe(u).
\end{equation*}
Composing the above map with the bijection \eqref{vpsi}
\[
\psi_{\le Y,\le X^{\prime}}^{\mathfrak R^{\prime}}\colon \Hom_{\le\mathfrak R^{\prime}} \be\big(h^{\lbe \mathfrak R^{\prime}}\!(Y),X^{\prime}\e\big)\overset{\!\sim}{\to}
\Hom_{\le k}\be\big(Y, \gra\big(\be X^{\prime}\big)\big)
\]
and using the formula \eqref{fun2}
\[
\psi_{\le Y,\le X^{\lbe\prime}}^{\mathfrak R^{\prime}}\lbe (w_{\lbe X}\!\circ\! h^{\lbe \mathfrak R^{\prime}}\!\lbe(u))=\psi^{\mathfrak R^{\prime}}_{\mr{Gr}^{\mathfrak R}\lbe(X)\le,\e X^{\prime}}\be(w_{\be X}\be)\circ u
\]
we obtain a map 
\begin{equation}\label{cp2}
\Hom_{\le k}\big(Y, \mr{Gr}^{ \mathfrak R}(X)\big)\to \Hom_{\le k}\big(Y,\mr{Gr}^{ \mathfrak R^{\prime}}(X^{\prime})\big), \, u\mapsto \psi^{\mathfrak R^{\prime}}_{\mr{Gr}^{\mathfrak R}\lbe(X)\le,\e X^{\prime}}\be(w_{\be X}\be)\!\circ\be u.
\end{equation}
The morphism of $k$-schemes
\begin{equation}\label{tr0} 
\varrho_{\lbe X}^{\le\mathfrak R,\mathfrak R^{\prime}}=\psi^{\mathfrak R^{\prime}}_{\mr{Gr}^{\mathfrak R}(X)\le,\e X^{\lbe\prime}}\be(w_{\be X}\be)\colon \mr{Gr}^{\le\mathfrak R}\lbe(X)\to\mr{Gr}^{\le\mathfrak R^{\prime}}\!(X^{\lbe\prime}\lle),
\end{equation}
is called the {\it change of rings morphism associated to the $\mathfrak R$-scheme $X$}. Then \eqref{cp2} is the map
\begin{equation}\label{cp4}
\Hom_{\le k}\big(Y, \mr{Gr}^{\mathfrak R}(X)\big)\to \Hom_{\le k}\big(Y,\mr{Gr}^{ \mathfrak R^{\prime}}(X^{\prime})\big), \, u\mapsto \varrho_{\lbe X}^{\le\mathfrak R,\mathfrak R^{\prime}}\be\!\circ\be u.
\end{equation}
Observe that, by \eqref{til} and \eqref{fun1},
\[
\widetilde{v}=w_{\lbe X}\circ h^{\lbe \mathfrak R^{\prime}}\!\lbe(u)=\varphi^{\mathfrak R^{\prime}}_{\mr{Gr}^{\mathfrak R}(X)\le,\e X^{\lbe\prime}}\be\big(\varrho_{\lbe X}^{\le\mathfrak R,\mathfrak R^{\prime}}\le\big)\circ h^{\lbe \mathfrak R^{\prime}}\!\lbe(u)=\varphi_{\le Y,\le X^{\lbe\prime}}^{\le\mathfrak R^{\le\prime}}\be\big(\varrho_{\lbe X}^{\le\mathfrak R,\mathfrak R^{\prime}}\!\!\circ\be u\le\big).
\]
Thus, by the definition of $\widetilde{v}$ \eqref{step1}, the following holds.

\begin{proposition}\label{rdad} Let $Y$ be a $k$-scheme, $X$ an $\mathfrak R$-scheme and $u\colon Y\to \gra\le(X)$ a morphism of $k$-schemes. If $\varrho_{\lbe X}^{\le\mathfrak R,\mathfrak R^{\prime}}$ is the change of rings morphism \eqref{tr0}, then $\varrho_{\lbe X}^{\le\mathfrak R,\e\mathfrak R^{\prime}}\!\circ\le u$ is the unique morphism of $k$-schemes\, $a\colon Y\to \mr{Gr}^{\le\mathfrak R^{\prime}}\!(X^{\lbe\prime}\lle)$ such that the diagram
\begin{equation}\label{rdad2}
\xymatrix{h^{\mathfrak R^{\prime}}\be(Y)\ar[d]_(.43){\delta_{\lbe Y}^{\le\mathfrak R,\mathfrak R^{\prime}}} \ar[rrr]^(.53){\varphi_{\lbe Y\!,\le X^{\lbe\prime}}^{\le\mathfrak R^{\le\prime}}\lbe(\le a\le)}&&& X^{\prime}\ar[d]^(.43){\mr{pr}_{\lbe X}}\\
h^{\mathfrak R}\be(Y)\ar[rrr]^(.5){\varphi_{\lbe Y,\le X}^{\mathfrak R}\be(u)}&&& X
}
\end{equation}
commutes. 
\end{proposition}

We will now discuss the functoriality of the assignment $X\mapsto \varrho_{\lbe X}^{\le\mathfrak R,\mathfrak R^{\prime}}$ \eqref{tr0}. Let $f\colon Z\to X$ be a morphism of $\mathfrak R$-schemes with associated morphism of $k$-schemes $\gr(\lle f\lle)\colon \mr{Gr}^{\mathfrak R}\lbe(Z)\to\mr{Gr}^{\mathfrak R}(X)$. Further, let $\varrho^{\e\mathfrak R\le,\e\mathfrak R^{\prime}}\!(\lle f\lle)\colon \mr{Gr}^{\mathfrak R}\lbe(Z)\to \mr{Gr}^{\mathfrak R^{\prime}}\!(X^{\prime})$ be the image of $\gr(\lle f\lle)$ 
under the map \eqref{cp4} for $Y=\mr{Gr}^{\mathfrak R}\be(Z)$, i.e.,
\begin{equation}\label{vrho}
\varrho^{\mathfrak R, \mathfrak R^{\prime}}\!(\lle f\lle)=\varrho_{\lbe X}^{\le\mathfrak R,\mathfrak R^{\prime}}\be\!\circ\be \gr(\lle f\lle).
\end{equation}
By \eqref{phigr} with $Z^{\le\prime}=X$, the commutativity of \eqref{step2} with $X$ replaced by $Z$ and the formula
\begin{equation}\label{fp}
f\circ\mr{pr}_{\lbe  Z}= \mr{pr}_{\lbe  X}\circ f^{\le\prime},
\end{equation} 
the diagram
\[
\xymatrix{h^{\lbe \mathfrak R^{\prime}}\!(\mr{Gr}^{\mathfrak R}\lbe(Z))\ar[d]_(.43){\delta_{\lbe \mr{Gr}^{\mathfrak R}\lbe(Z)}^{\le\mathfrak R,\mathfrak R^{\prime}}} \ar[rrr]^(.53){f^{\le\prime}\circ w_{Z}}&&& X^{\prime}\ar[d]^(.43){\mr{pr}_{\lbe X}}\\
h^{\lbe \mathfrak R}\be(\mr{Gr}^{\mathfrak R}\lbe(Z))\ar[rrr]^(.58){\varphi_{\lbe \mr{Gr}^{\mathfrak R}\lbe(Z),\le X}^{\mathfrak R}\be(\gr(\lle f\lle))}&&& X
}
\]
commutes. Thus, by \eqref{vrho}, \eqref{fun4}, \eqref{tr0} and Proposition \ref{rdad} for $Y=\mr{Gr}^{\mathfrak R}\be(Z)$ and $u=\gr(\lle f\lle)$, we have
\[
\begin{array}{rcl}
\varrho^{\e\mathfrak R, \mathfrak R^{\prime}}\!(\lle f\lle)&=&\varrho_{\lbe X}^{\le\mathfrak R,\mathfrak R^{\prime}}\be\!\circ\be \gr(\lle f\lle)=\psi_{\lbe \mr{Gr}^{\mathfrak R}\lbe(Z),\e X^{\lbe\prime}}^{\le\mathfrak R^{\le\prime}}(\e f^{\e\prime}\circ w_{\lbe Z})\\
&=&\mr{Gr}^{\mathfrak R^{\prime}}\!(\lle f^{\le\prime}\lle)\circ \psi_{\lbe \mr{Gr}^{\mathfrak R}\lbe(Z),\e Z^{\lbe\prime}}^{\le\mathfrak R^{\le\prime}}(w_{\lbe Z})= \mr{Gr}^{\mathfrak R^{\prime}}\!(\lle f^{\le\prime}\lle)\circ \varrho_{\lbe Z}^{\le\mathfrak R,\le\mathfrak R^{\prime}}.
\end{array}
\] 
Thus the following diagram commutes
\begin{equation}\label{funct}
\xymatrix{\mr{Gr}^{ \mathfrak R}\be(Z)\ar[d]_{\mr{Gr}^{\lbe \mathfrak R}(\le f\le)}\ar[drr]^{\e\varrho^{\le\mathfrak R,\mathfrak R^{\prime}}\!\be(\le f\le)}\ar[rr]^(.45){\varrho_{\lbe Z}^{\mathfrak R,\mathfrak R^{\prime}}}&& \mr{Gr}^{ \mathfrak R^{\prime}}\!(\be Z^{\le\prime})\ar[d]^{\mr{Gr}^{\mathfrak R^{\prime}}\!(\le f^{\prime}\lle)}\\
\mr{Gr}^{\lbe \mathfrak R}(X\e)\ar[rr]^(.45){\varrho_{X}^{\mathfrak R,\mathfrak R^{\prime}}}&&
\mr{Gr}^{\lbe \mathfrak R^{\prime}}\!(X^{\le\prime})}.
\end{equation}
In particular, if $Z$ is an $\mathfrak R\le$-group scheme, then the change of rings morphism \eqref{tr0} is a morphism of $k$-group schemes (i.e., a homomorphism).

\begin{remarks}\label{vrk}\indent
\begin{itemize}
\item[(a)] Note that, by \eqref{vr=1k0} and \eqref{1-2}, 
\begin{equation}\label{vr=1k}
\varrho_{\e \spec \mathfrak R}^{\le\mathfrak R,\mathfrak R^{\prime}}=1_{\spec k}.
\end{equation}
Further, $\varrho^{\le\mathfrak R\le,\e \mathfrak R^{\prime}}\!(\lle 1_{\lbe X}\lle)=\varrho_{\lbe X}^{\le\mathfrak R,\mathfrak R^{\prime}}$ \eqref{vrho}. In addition,
if $\mathfrak R^{\le\prime}=k$ (so that $X^{\prime}=X_{\mr s}$ is the special fiber of $X$), then \eqref{tr0} is a morphism of $k$-schemes $\varrho_{X}^{\mathfrak R,\le k}\colon\gra\lbe(X)\to X_{\mr s}$ (see Remark \ref{rems1}(a)).
\item[(b)] By Proposition \ref{rdad}, if $A$ is a $k$-algebra, then \eqref{grzz} identifies $\varrho_{\lbe X}^{\le\mathfrak R,\mathfrak R^{\prime}}\!(A)$ with the map $X(\s R\lbe(A))\to X(\s R^{\e\prime}\lbe(A))$ induced by the canonical homomorphism $\s R\lbe(A)\to \s R^{\e\prime}\be(A)$. In particular, $\varrho_{\lbe X}^{\le\mathfrak R,\le\mathfrak R^{\prime}}\!(k)\colon X(\mathfrak R)\to X(\mathfrak R^{\e\prime})$ is induced by the projection $\mathfrak R\to \mathfrak R^{\e\prime}$.
\item[(c)] If $\mathfrak J$ is an ideal of $\mathfrak R$ which contains $\mathfrak I$ and $\mathfrak R^{\prime\prime}=\mathfrak R/\mathfrak J$, then (b) shows that
$\varrho_{\lbe X}^{\le\mathfrak R,\le\mathfrak R^{\prime\prime}}=
\varrho_{\lbe X^{\prime}}^{\le\mathfrak R^{\prime}\lbe,\le\mathfrak R^{\prime\prime}}\!\circ \varrho_{\lbe X}^{\le\mathfrak R,\le\mathfrak R^{\prime}}$, where $X$ is an $\mathfrak R$-scheme and $X^{\prime}=X\times_{\spec \mathfrak R} \spec \mathfrak R^{\prime}$.

\item[(d)] For every $k$-scheme $Y$, let $\iota_{\e\mr s}\colon h^{\lbe \mathfrak R}(Y)_{\mr s}\to h^{\lbe \mathfrak R}(Y)$ denote the canonical immersion of the special fiber of $h^{\lbe \mathfrak R}(Y)$ into $h^{\lbe \mathfrak R}(Y)$. Then the following diagram of nilpotent immersions
\begin{equation}\label{vrk-d}
\xymatrix{
Y\ar[rr]^{\iota_{Y}}\ar[drr]_(0.4){\delta^{\le\mathfrak R,\le k}_{Y}} &&\ar[d]^{\iota_{\le\mr s}} h^{\lbe \mathfrak R}(Y) _{\mr s}
\\
&&h^{\lbe \mathfrak R}(Y)
}
\end{equation}
commutes, where $\iota_{Y}$ is the morphism \eqref{io} and $\delta^{\le\mathfrak R,\le k}_{Y}$ is the morphism \eqref{dlt} for $\mathfrak R^{\e\prime}=k$. Indeed, $\delta^{\le\mathfrak R,\le k}_{Y}$ (respectively, $\iota_{Y}$) is induced by the morphism of Zariski sheaves  $\s R(\be\s
O_{Y}\be)\to \s R(\be\s O_{Y}\be)/\be\overbarr{\s M}\be(\be\s O_{Y}\be)\simeq\s O_{Y}$ (respectively, $\s R(\be\cO_{Y}\be)/\mathfrak M\,\s R\be(\be\s O_{Y}\be)\to \s R(\be\cO_{Y}\be)/\be\overbarr{\s M}\lbe(\be\s O_{Y}\be)$), whence the indicated commutativity follows.
\end{itemize}
\end{remarks}

\begin{proposition}\label{f-et} Let $f\colon Z\to X$ be a formally \'etale morphism of $\mathfrak R$-schemes. Then the diagram \eqref{funct} is cartesian.  Consequently, there exists a canonical isomorphism of $k$-schemes 
\[
\mr{Gr}^{\le\mathfrak R}\lbe(Z)=\mr{Gr}^{\le\mathfrak R}\lbe(X\e)\!\times_{\e \varrho_{\lbe X}^{\le\mathfrak R,\mathfrak R^{\prime}}\!,\,\mr{Gr}^{\le\mathfrak R^{\le\prime}}\!\lbe(X^{\lbe\prime}\lle),
\, \mr{Gr}^{\le\mathfrak R^{\prime}}\!(\le f^{\e\prime}\e) }\!\mr{Gr}^{\le\mathfrak R^{\le\prime}}\!(Z^{\le\prime}\e).
\]
\end{proposition}
\begin{proof} We need to show that, if $T$ is a scheme and $t_{1}\colon T\to\mr{Gr}^{ \mathfrak R^{\le\prime}}\!(Z^{\le\prime}\e)$ and $t_{2}\colon T\to 
\mr{Gr}^{\lbe \mathfrak R}(X\le)$ are morphisms of schemes such that 
\begin{equation}\label{pin}
\mr{Gr}^{\lbe \mathfrak R^{\prime}}\be(\lle f^{\le\prime}\le)\!\circ\be t_{1}=\varrho_{X}^{\le\mathfrak R,\le\mathfrak R^{\prime}}\!\be\circ\be t_{2},
\end{equation}
then there exists a unique morphism of schemes
$g\colon T\to \mr{Gr}^{\lbe \mathfrak R}(Z)$ such that $t_{1}=\varrho_{Z}^{\le\mathfrak R,\le\mathfrak R^{\prime}}\!\be\circ\be g$ and $t_{2}=\mr{Gr}^{\lbe \mathfrak R}(\lle f\lle)\be\circ\be g\,$. See the following diagram:
\[
\xymatrix{T\ar@{-->}[dr]_(.4){g}\ar@/^.9pc/[drrr]^(.4){t_{1}}\ar@/_1.2pc/[ddr]_(.37){t_{2}}&&\\
&\mr{Gr}^{\mathfrak R}(Z)\ar[d]^{\mr{Gr}^{\mathfrak R}(\le f\lle)}\ar[rr]^(.47){\varrho_{\lbe Z}^{\le\mathfrak R,\le\mathfrak R^{\prime}}}&& \mr{Gr}^{ \mathfrak R^{\le\prime}}\!(Z^{\le\prime}\e)\ar[d]^(.45){\mr{Gr}^{\mathfrak R^{\prime}}\!(\le f^{\le\prime}\le)}\\
&\mr{Gr}^{\mathfrak R}(X\e)\ar[rr]^(.48){\varrho_{\lbe X}^{\le\mathfrak R,\le\mathfrak R^{\prime}}}&&
\mr{Gr}^{\mathfrak R^{\le\prime}}\!(X^{\le\prime}\e).}
\]
By \eqref{pin} and \eqref{fun3}, we have 
\[
\varphi^{\mathfrak R^{\prime}}_{T,\e X^{\prime}}\lbe(\varrho_{X}^{\le\mathfrak R,\le\mathfrak R^{\prime}}\!\be\circ\be t_{2})=\varphi^{\mathfrak R^{\prime}}_{T,\e X^{\prime}}\lbe(\mr{Gr}^{\lbe \mathfrak R^{\prime}}\be(\lle f^{\le\prime}\le)\!\circ\be t_{1})=f^{\le\prime}\be\circ\be \varphi^{\mathfrak R^{\prime}}_{T,\e Z^{\prime}}(t_{1}).
\]
Thus, by \eqref{fp} and the commutativity of \eqref{rdad2} for $Y=T$ and $u=t_{2}$, the following diagram commutes:
\[
\xymatrix{h^{\lbe \mathfrak R^{\prime}}\!(T)\ar[d]_(.43){\delta_{ T}^{\le\mathfrak R,\mathfrak R^{\prime}}} \ar[rrr]^(.53){\varphi_{\lbe T\!,\le Z^{\lbe\prime}}^{\le\mathfrak R^{\le\prime}}\lbe(\le t_{1}\le)}&&& Z^{\prime}\ar[d]^(.43){f\e\circ\e\mr{pr}_{\lbe Z}}\\
h^{\mathfrak R}\lbe(T)\ar[rrr]^(.5){\varphi_{\lbe T,\le X}^{\le\mathfrak R}\be(t_{2})}&&& X.
}
\] 
Consequently, if we regard $h^{\lbe \mathfrak R}\lbe(T)$ and $h^{\mathfrak R^{\prime}}\!(T)$ as $X$-schemes via the maps $\varphi_{\lbe T,\le X}^{\le\mathfrak R}\be(t_{2})$ and $f\!\circ\!\mr{pr}_{Z}\!\circ\!\varphi_{\lbe T\!,\e Z^{\prime}}^{\le\mathfrak R^{\le\prime}}\lbe(\le t_{1}\le)$, respectively, then we obtain a well-defined map
\[
\Hom_{X}(h^{\mathfrak R}\lbe(T), Z\le)\to\Hom_{X}(h^{\mathfrak R^{\prime}}\!(T),Z\le), v\mapsto v\circ \delta^{\le\mathfrak R,\le\mathfrak R^{\prime}}_{T}.
\]
Now, since $\delta^{\le\mathfrak R,\le\mathfrak R^{\prime}}_{T}$ is a nilpotent immersion and $f\colon Z\to X$ is formally \'etale, the preceding map is a bijection by \cite[$\text{IV}_{4}$, Remark 17.1.2(iv)]{ega}. Consequently, there exists a unique morphism of $X$-schemes $v\colon h^{\lbe \mathfrak R}(T)\to Z$, i.e., $f\circ v=\varphi^{\mathfrak R}_{T,\e X}(t_{2})$, such that $v\circ \delta^{\le\mathfrak R,\mathfrak R^{\prime}}_{T}=\mr{pr}_{\lbe Z}\e\circ\e\varphi^{\mathfrak R^{\prime}}_{T,\e Z^{\prime}}(t_{1})\colon h^{\mathfrak R^{\prime}}\!(T)\to Z$. Let $g=\psi^{\mathfrak R}_{T,\e Z}(v)\colon T\to \mr{Gr}^{\lbe \mathfrak R}(Z)$. Then $\mr{pr}_{\lbe Z}\be\circ\be\varphi^{\le\mathfrak R^{\prime}}_{T,\e Z^{\prime}}(t_{1})=v\circ \delta^{\le\mathfrak R,\mathfrak R^{\prime}}_{T}=\varphi^{\mathfrak R}_{T,\e Z}(\e g)\circ \delta^{\le\mathfrak R,\mathfrak R^{\prime}}_{T}$. Thus, by Proposition \ref{rdad} applied to 
$Y=T$, $X=Z$ and $u=g$, we have $t_{1}=\varrho_{\lbe Z}^{\mathfrak R,\mathfrak R^{\prime}}\!\be\circ\be g$. Finally, by \eqref{fun4},
\[
t_{2}=\psi^{\mathfrak R}_{T,\e X}(\varphi^{\mathfrak R}_{T,\e X}(t_{2}))=\psi^{\mathfrak R}_{T,\le X}(\e f\be\circ\be v)=\mr{Gr}^{\lbe \mathfrak R}(f)\circ g.
\] 
\end{proof}

\begin{corollary}\label{ffet}  Let $f\colon Z\to X$ be a formally \'etale morphism of $\mathfrak R$-schemes. Then there exists a canonical isomorphism of $k$-schemes 
\[
\mr{Gr}^{\le\mathfrak R}\lbe(Z)= Z_{\e\rm s}\!\times_{f_{\mr s},\e\lle X_{\lle\rm s},\e \varrho_{\lbe X}^{\le\mathfrak R,\lle k}}\!\mr{Gr}^{\le\mathfrak R}\lbe(X\e).
\]
Consequently, $\gra(\lle f\lle)\colon\gra(Z)\to\gra(X\le)$ can be identified with $f_{\mr s}\!\times_{\be\lle X_{\lle\rm s}}\!\mr{Gr}^{\le\mathfrak R}\lbe(X\e)$. 
\end{corollary}
\begin{proof}  This is immediate from the proposition by setting $\mathfrak R^{\e\prime}=k$ there. See also Remarks \ref{rems1}(a) and \ref{vrk}(a). 
\end{proof}

\begin{corollary}\label{fetR}
Let $Z$ be a formally \'etale $\mathfrak R$-scheme. Then the change of rings morphism $\varrho_{\lbe Z}^{\mathfrak R,\le\mathfrak R^{\prime}}\colon\mr{Gr}^{\mathfrak R}\lbe(Z\le)\to\mr{Gr}^{ \mathfrak R^{\prime}}\!(Z^{\e\prime}\le)$ \eqref{tr0} is an isomorphism.
\end{corollary}
\begin{proof} Let $f\colon Z\to \spec \mathfrak R$ be the structure morphism of $Z$. By \eqref{vr=1k} and the proposition, diagram \eqref{funct} yields a cartesian diagram
\[
\xymatrix{\mr{Gr}^{\lbe \mathfrak R}(Z\le)\ar[d]_{\mr{Gr}^{\lbe \mathfrak R}(\lle f\lle)}\ar[rr]^(.5){\varrho_{\lbe Z}^{\mathfrak R,\le\mathfrak R^{\prime}}}&& \mr{Gr}^{\lbe \mathfrak R^{\prime}}\!(\lbe Z^{\le\prime}\e)\ar[d]^{\mr{Gr}^{\lbe \mathfrak R^{\prime}}\!(\le f^{\lle\prime}\lle)}\\
\spec k\ar[rr]^(.5){1_{\spec k}}&&
\spec k.}
\]
Consequently, $\varrho_{\lbe Z}^{\mathfrak R,\le\mathfrak R^{\prime}}$ is an isomorphism (see Subsection \ref{not}).
\end{proof}

\begin{proposition}\label{aff}
Let $Z$ be an $\mathfrak R$-scheme. Then $\varrho_{\lbe Z}^{\mathfrak R\le,\e\mathfrak R^{\le\prime}} \colon \mr{Gr}^{\lbe \mathfrak R}\be(Z)\to\mr{Gr}^{ \mathfrak R^{\prime}}\!(Z^{\e\prime}\e)$ is affine.
\end{proposition}
\begin{proof} Since $Z^{\e\prime}\to Z$ is a universal homeomorphism, we may choose an open affine covering $\{U_{\lbe j}\}$ of $Z$ such that $\{U_{\be j}^{\e\prime}\}$ is an open affine covering of $Z^{\e\prime}$. The proof of \cite[Proposition 7, p.~641]{gre1}\,\footnote{This proof depends only on \eqref{adj} and is valid independently of the finiteness assumption in [loc.cit.].} shows that $\mr{Gr}^{\mathfrak R^{\prime}}\be(Z^{\le\prime}\le)$ is covered by the open affine subschemes $\mr{Gr}^{\mathfrak R^{\lle\prime}}\be(U_{\be j}^{\le\prime}\e)$. Now, since the canonical injection morphism $U_{j}\to Z$ is formally \'etale for each $j$ by \cite[$\text{IV}_{4}$, Proposition 17.1.3]{ega}, Proposition \ref{f-et} yields a canonical isomorphism
\[
\big(\varrho_{\lbe Z}^{\mathfrak R,\e\mathfrak R^{\prime}}\big)^{\!-1}\be(\mr{Gr}^{\mathfrak R^{\lle\prime}}\be(U_{\be j}^{\le\prime}\e))=\mr{Gr}^{\mathfrak R}(Z\e)\times_{\mr{Gr}^{ \mathfrak R^{\prime}}\be(Z^{\prime}  )}\mr{Gr}^{\mathfrak R^{\prime}}\be(U_{\be j}^{\le\prime}\e)=\mr{Gr}^{\mathfrak R}\lbe (U_{j})
\]
for every $j$. Since $\mr{Gr}^{\mathfrak R}\lbe (U_{j})$ is affine for every $j$ by Remark \ref{rems1}(b), the proof is complete.
\end{proof}

The following is an immediate corollary of the proposition (see \cite[Proposition 9.1.3]{ega1}):

\begin{corollary}\label{rqc}
Let $Z$ be an $\mathfrak R$-scheme. Then $\varrho_{\lbe Z}^{\mathfrak R\le,\e\mathfrak R^{\le\prime}} \colon \mr{Gr}^{\lbe \mathfrak R}\be(Z)\to\mr{Gr}^{ \mathfrak R^{\prime}}\!(Z^{\e\prime}\e)$ is quasi-compact and separated. \qed
\end{corollary}

\begin{proposition}\label{sm-surj} Let $Z$ be a formally smooth $\mathfrak R$-scheme. Then the change of rings morphism $\varrho_{\lbe Z}^{\mathfrak R\le,\e\mathfrak R^{\le\prime}} \colon \mr{Gr}^{\lbe \mathfrak R}\be(Z)\to\mr{Gr}^{ \mathfrak R^{\prime}}\!(Z^{\e\prime}\e)$ is surjective.
\end{proposition}
\begin{proof} By \cite[Proposition 3.6.2, p.~244]{ega1}, we need only check that the canonical map
\[
\Hom_{\e k}(\spec K, \mr{Gr}^{\mathfrak R}\lbe(Z)\le)\to\Hom_{\e k}(\spec K, \mr{Gr}^{ \mathfrak R^{\prime}}\!(Z^{\e\prime}\e)\le),\,g\mapsto \varrho_{\lbe Z}^{\mathfrak R\le,\e\mathfrak R^{\e\prime}}\!\! \circ\be g,
\]
is surjective for every field extension $K\be/\le k$. Let $t\colon\spec K\to \mr{Gr}^{ \mathfrak R^{\prime}}\be(Z^{\e\prime}\e)$ be a $k$-morphism. Since $Z$ is formally smooth over $\mathfrak R$ and $\delta_{\spec K}^{\le\mathfrak R\le,\e\mathfrak R^{\prime}} \colon h^{ \mathfrak R^{\le\prime}}\!(\spec K\lle)\to h^{\mathfrak R}\lbe(\spec K\lle)$ \eqref{dlt} is a nilpotent immersion of affine $\mathfrak R$-schemes \eqref{hrn-aff}, the canonical map 
\[
\Hom_{ \mathfrak R}(h^{\mathfrak R}\lbe(\spec K\lle),Z\le)\to\Hom_{\le\mathfrak R}(h^{ \mathfrak R^{\prime}}\!(\spec K\lle), Z\le),\, v\mapsto v\be\circ\be \delta^{\e\mathfrak R,\mathfrak R^{\prime}}_{\spec K},
\]
is surjective. Thus, since $\mr{pr}_{\lbe Z}\circ\varphi^{\mathfrak R^{\le\prime}}_{\le K,\le Z^{\le\prime}}\be(t)\in \Hom_{\e\mathfrak R}(h^{ \mathfrak R^{\prime}}\!(\spec K\lle), Z\le)$, there exists an $\mathfrak R$-morphism $v\in \Hom_{ \e\mathfrak R}(h^{\mathfrak R}\lbe(\spec K\lle),Z\le)$ such that
$v\be\circ\be \delta^{\e\mathfrak R\le,\e\mathfrak R^{\prime}}_{\spec K}=\mr{pr}_{\lbe Z}\circ\varphi^{\mathfrak R^{\le\prime}}_{\le K,\le Z^{\le\prime}}\be(t)$. Let $g=\psi^{\mathfrak R}_{\le K,\le Z}(v)\in \Hom_{\le k}(\spec K, \mr{Gr}^{\mathfrak R}\be(Z)\le)$. Then the same argument used in the latter part of the proof of Proposition \ref{f-et} shows that $t=\varrho^{\le\mathfrak R\le,\e\mathfrak R^{\le\prime}}_{\lbe Z}\!\!\circ\be g$, which completes the proof.
\end{proof}

\section{The change of level morphism}\label{s-clm}

We keep the notation of the previous Section. In particular, $\mathfrak M$ denotes the maximal ideal of $\mathfrak R$. Set
\begin{equation}\label{nig}
N=\text{min}\{n\in\N\colon \mathfrak M^{\le n}=0\}.
\end{equation}
For every integer $n\geq 1$, set $\mathfrak R_{\le n}=\mathfrak R/\mathfrak M^{\le n}\,$ and $\mathfrak M_{\le n}=\mathfrak M/\mathfrak M^{\le n}\,$. Note that $\mathfrak R=\mathfrak R_{\le n}$ and $\mathfrak M_{n}=\mathfrak M$ for every $n\geq N$, where $N$ is given by \eqref{nig}. Thus the set of rings $\{\mathfrak R_{\le n}\colon n\in\N\e \}$ equals the finite set $\{\mathfrak R_{\le n}\colon n\leq N\}$.
For example, if $\mathfrak R=R_{\le s}=R/\mm^{s}$ is the truncation of order $s$ $(\geq 1)$ of a discrete valution ring $R$, as in Section \ref{truc}, then $\mathfrak M=\mm/\mm^{s}$ satisfies $\mathfrak M^{s}=0$ whence $\mathfrak R_{\le n}\simeq R_{\le n}$ for every $n\leq s$ and the sets $\{\mathfrak R_{\le n}\colon n\in\N\e \}$ and $\{R_{\le n}\colon n\leq s\e \}$ can be identified.  
As in Section \ref{gberg}, we will write $h_{n}^{\mathfrak R}$ and $\mr{Gr}_{n}^{\mathfrak R}$ for $h^{\le\mathfrak R_{\lle n}}$ and $\mr{Gr}^{\le\mathfrak R_{\lle n}}$, respectively. Further, for every pair of integers $n\geq 1$ and $j\geq 0$, we will write
\begin{equation}\label{tht}
\theta_{\lbe j}^{\e  n }\colon \spec \mathfrak R_{\le n}\to \spec \mathfrak R_{\e n+j}
\end{equation}
for the morphism induced by the canonical surjective map $\mathfrak R_{\e n+j}\to \mathfrak R_{\le n}$. Note that $\theta_{\lbe j}^{\e  n}$ is the identity morphism of $\spec \mathfrak R$ if $n\geq N$. Further, $\mathfrak R_{\e n+j}\to \mathfrak R_{\e n}$ is a map of the form $\mathfrak R\to \mathfrak R^{\le\prime}=\mathfrak R/\mathfrak I$, with $\mathfrak R=\mathfrak R_{\le n+j}$ and $\mathfrak I=\mathfrak M^{\e n}\!/\le\mathfrak M^{\e n+j}$, which were discussed in the previous Section. Thus, for every $\mathfrak R_{\le n+j}$-scheme $Z$, the change of rings map $\varrho_{\le Z}^{\,\mathfrak R_{\le n+j},\e \mathfrak R_{\le n}}\!\!$ \eqref{tr0} is defined. The latter map will be denoted by $\varrho_{\le n,\le Z}^{\, j}$ and called the {\it change of level morphism associated to the $\mathfrak R_{\le n+j}$-scheme $Z$}. If $\mr{Gr}_{n}^{\le\mathfrak R}(Z\le)$ denotes $\mr{Gr}_{n}^{\le\mathfrak R}(Z\!\be\times_{\mathfrak R_{\le n+j}}\!\spec \mathfrak R_{\le n})$, then $\varrho_{\le n,Z}^{\, j}$ is a map
\begin{equation}\label{clm0}
\varrho_{\le n,Z}^{\, j} \colon \mr{Gr}_{n+j}^{\le\mathfrak R}(Z\le)\to\mr{Gr}_{n}^{\le\mathfrak R}(Z\le).
\end{equation}

\begin{example} 
Let $V=\spec\! \big(W(k)[x]/(\e p\lle x)\big)$ be the vector bundle associated to the $W(k)$-module $W(k)/(\e p)=k$, and let $A$ be any $k$-algebra $A$. Then $V_{\lbe\rm s}\lle(A)=\mr{Hom}_{\e k-{\rm alg}}(k[x],A )\simeq A$ and
\[
\mr{Gr}_{2}^{\lbe R}(V)(\lbe A)=\mr{Hom}_{\e W_{2}(k)-{\rm alg}}(W_{2}(k)[x]/(\e p\lle x), W_{2}(A))\simeq \{(a_{0},a_{1})|\, a_{0}^{\le p}=0\}\subseteq W_{2}(A).
\]
By Remark \ref{vrk}(b), the change of level morphism $\varrho_{\e 1,\e V}^{\e 1}(\lbe A)\colon \mr{Gr}_{2}^{\lbe R}(V)(\lbe A)\to V_{\lbe\rm s}\lle(A)$ \eqref{clm0} maps $(a_{0},a_{1})$ to $a_{0}$, which is a $p$-nilpotent element of $V_{\lbe\rm s}\lle(A)=A$. Setting $A=k$ above, we conclude that $\varrho_{\e 1,\le V}^{\e 1}(k)$ is the zero map. Note, however, that $\varrho_{\le 1,\le V}^{\e 1}\neq 0$. 
\end{example}

Now, for every $k$-scheme $Y$, let
\begin{equation}\label{dnj}
\delta_{Y}^{\e n,\le j}= \delta_{Y}^{\e\mathfrak R_{\le n+j},\e\mathfrak R_{\le n}}\colon  h_{n}^{\mathfrak R}\lbe(Y\le)\to  h_{n+j}^{\mathfrak R}\lbe(Y\le)
\end{equation}
be the nilpotent immersion \eqref{dlt} which corresponds to the projection $\mathfrak R_{\le n+j}\to \mathfrak R_{\le n}$. We will write $\s R_{n}$ for the Greenberg algebra $\s R^{\e(\s M^{\lbe n})}$ associated to $\mathfrak R_{\le n}$, as described in Subsections \ref{sec-k} and \ref{sec-w}. Further, note that $\s R_{n}=\s R$ if $n\geq N$.

\begin{remark}
For every morphism of $k$-schemes $u\colon Y\to W$ and every pair of integers $n\geq 1$ and $j\geq 0$, \eqref{dlt2} provides a commutative diagram
\begin{equation}\label{nim}
\xymatrix{h_{n}^{\mathfrak R}\lbe(Y)\ar[d]_(.5){\delta_{Y}^{\e n,\le j}} \ar[rr]^(.47){h_{n}^{\mathfrak R}\lbe(\le u\le)}&& h_{n}^{\mathfrak R}\lbe(W)\ar[d]^(.5){\delta_{W}^{\e n,\le j}}\\
h_{n+j}^{\mathfrak R}\lbe(Y)\ar[rr]^(.47){h^{\mathfrak R}_{n+j}(\le u\le)}&& h_{n+j}^{\mathfrak R}\lbe(W)
}
\end{equation}
where $\delta_{Y}^{\e n,\le j}$ is the map \eqref{dnj}
In particular, if $W=\spec k$ and $u\colon Y\to\spec k$ is the structural morphism then, by \eqref{hrn-aff2} and \eqref{dlt3}, \eqref{nim} is a diagram
\[
\xymatrix{h_{n}^{\mathfrak R}\lbe(Y)\ar[d]_(.5){\delta_{Y}^{\e n,\le j}} \ar[rr]^(.47){h_{n}^{\mathfrak R}\lbe(\le u\le)}&& \spec \mathfrak R_{\le n}\ar[d]^(.5){\theta^{\e n}_{\lbe j}}\\
h_{n+j}^{\mathfrak R}\lbe(Y)\ar[rr]^(.47){h^{\mathfrak R}_{n+j}(\le u\le)}&& \spec \mathfrak R_{\le n+j},
}
\]
where the right-hand vertical map is \eqref{tht}. We conclude that $\delta_{Y}^{\e n,\le j}$ defines a morphism of $\mathfrak R_{\le n+j}$-schemes $h_{n}^{\mathfrak R}\lbe(Y)\to h_{n+j}^{\mathfrak R}\lbe(Y)$ when $h_{n}^{\mathfrak R}\lbe(Y)$ is regarded as an $\mathfrak R_{\le n+j}$-scheme via the composition $\theta^{\e n}_{\lbe j}\circ h_{n}^{\mathfrak R}\lbe(\le u\le)$. 
\end{remark}

\section{Basic properties of the Greenberg functor}\label{bas}
We keep the notation introduced in Section \ref{gr-art}.

\begin{proposition}\label{q-proj}
Let $Z$ be a quasi-projective $\mathfrak R$-scheme. Then $\gra(Z)$ is a quasi-projective $k$-scheme.
\end{proposition}
\begin{proof}  Since $Z\to\spec \mathfrak R$ is of finite type, $\gra\lbe(Z)\to\spec k=\gra\lbe(\spec \mathfrak R\le)$ is a morphism of finite type which factors as $\gra\lbe(Z)\to Z_{\le\rm s}\to\spec k$, where the first map is the change of rings morphism $\varrho_{\lbe Z}^{\e\mathfrak R,\e k}$ \eqref{tr0} and the second morphism is quasi-projective. Now, by Proposition \ref{aff} and \cite[Proposition 6.3.4(v), p.~305]{ega1}, $\varrho_{\lbe Z}^{\e\mathfrak R,\le k}$ is an affine morphism of finite type. Thus $\varrho_{\lbe Z}^{\e\mathfrak R,\le k}$ is also quasi-projective whence the proposition follows (see \cite[II, Proposition 5.3.4, (i) and (ii)]{ega}).
\end{proof}

\begin{remark} When $R$ is an equal characteristic discrete valuation ring (in which case the Greenberg functor agrees with the Weil restriction functor by Remark \ref{rems1}(c)), and $\mathfrak R$ is a truncation of $R$, then the preceding proposition also follows from \cite[Proposition A.5.8]{cgp}.
\end{remark}

\begin{proposition}\label{gr-prop}  Consider, for a morphism of schemes, the property of being:
\begin{enumerate}
\item[(i)] quasi-compact;
\item[(ii)] quasi-separated;
\item[(iii)] separated;
\item[(iv)] locally of finite type;
\item[(v)]  of finite type;
\item[(vi)] affine.
\end{enumerate}
If $\bm{P}$ denotes one of the above properties and the $\mathfrak R $-morphism $f\colon X\to Y$ has property $\bm{P}$, then the $k$-morphism $\gra(\lle f\lle)\colon \gra(X)\to \gra(Y)$ has property $\bm{P}$ as well.
\end{proposition}
\begin{proof} Recall diagram \eqref{funct} associated to the canonical projection $\mathfrak R\to k\e$:
\[
\xymatrix{\gra(X)\ar[rr]^(0.5){\varrho_{X}^{\mathfrak R,\le k}}\ar[d]_{\gra(\lle f\lle)}&&  X_{\le\rm s}\ar[d]^{ f_{\le\rm s}}
\\
\gra(Y) \ar[rr]^(0.5){\varrho_{Y}^{\mathfrak R,\le k}} && Y_{\lbe\rm s}. }
\]
By Proposition \ref{aff}, the horizontal morphisms in the above diagram are affine and hence separated and quasi-compact. Thus, if $f$ is quasi-compact (whence $f_{\le\rm s}$ is quasi-compact as well), then the quasi-compactness of $\gra(\lle f\lle)$ follows from the diagram using \cite[Propositions 6.1.4 and 6.1.5(v), p.~291]{ega1}.
To prove the proposition for properties (ii) and (iii), assume that $f$ is quasi-separated (respectively, separated), i.e., the diagonal morphism $\Delta_{f}\colon X\to X\!\times_{Y}\!X$ is quasi-compact (respectively, a closed immersion). Then, by Remarks \ref{rems1}, (b) and (d), and the first part of the proof,
\[
\gra\lbe(\Delta_{f})=\Delta_{\e\gra\lbe(\lle f\lle)}\colon \gra(X)\to \gra(X)\!\times_{\be\gra(Y)}\!\gra(X)
\]
is quasi-compact (respectively, a closed immersion), i.e., $\gra(f)$ is quasi-separated (respectively, separated). To prove the proposition for property (iv), we 
may assume that $X=\spec A$ and $Y=\spec B$, where $B$ is an $\mathfrak R$-algebra and $A$ is a quotient of the polynomial $B$-algebra $B[x_{1},\dots,x_{d}]$ for some $d\geq 0$. 
Since $\spec A\to\spec B$ factors as $\spec A\hookrightarrow\A_{B}^{\be d}\to\spec B$, where the first morphism is a closed immersion (and therefore of finite type), and $\mr{Gr}^{\mathfrak R}$ respects closed immersions by Remark \ref{rems1}(b), we may, in fact, assume that $X=\A_{B}^{\be d}$. In this case $f$ is the map $\A_{\mathfrak R}^{\be d}\times_{\mathfrak R}\spec B\to\spec B$, whence (by Remark \ref{rems1}(d)) $\gra(\lle f\lle)$ is the base extension along $\gra(\spec B)\to \spec k$ of the canonical morphism $\s R^{\e d}\to\spec k$, which is clearly a morphism of finite type. This completes the proof for property (iv). The proposition holds in the case of property (v) since it holds for properties (i) and (iv). Finally, by \cite[proof of Proposition 7, p.~642]{gre1}, $\gra(Y)$ is covered by affine open subschemes of the form $\gra(U)$, where $U$ is an affine open  subscheme of $Y$. Since $\gra(X)\!\times_{\be\gra(Y)}\!\gra(U)=\gra(X\!\times_{\lbe Y}\!U)$ is affine, the proof is complete.
\end{proof}

\begin{corollary}\label{aff-cor} Let  $X$ be an affine scheme of finite type over $\mathfrak R$. Then $\gra\lbe(X)$ is an affine scheme of finite type over $k$. 
\end{corollary}
\begin{proof}  Since $X$ and $\spec \mathfrak R$ are affine schemes, the structural morphism 
$X\to\spec \mathfrak R$ is an affine morphism of finite type. Thus, by \eqref{grnsn} and parts (v) and (vi) of the proposition, $\gra\lbe(X)\to \spec k$ is an affine morphism of finite type. The corollary is now clear. 
\end{proof}

\begin{proposition} 
Let $f\colon Z\to Z^{\prime}$ be a formally smooth (respectively, formally unramified, formally \'etale) $\mathfrak R$-morphism of schemes. Then the induced $k$-morphism $\gra\lbe(\lle f\lle)\colon $ $\gra(Z)\to \gra(Z^{\prime})$ is formally smooth (respectively, formally unramified, formally \'etale).
\end{proposition}
\begin{proof} We need to show that, if $Y$ is an affine scheme, $\iota\colon Y_{*}\to Y$ is a nilpotent immersion and $Y\to \gra(Z^{\prime})$ is an arbitrary morphism of schemes, then the map induced by $\iota$
\[
\Hom_{\gra\lbe(Z^{\prime})}(Y,  \gra(Z)))\to \Hom_{\gra\lbe(Z^{\prime})}(Y_{*},\gra(Z))
\]
is surjective (respectively, injective, bijective). By \eqref{adj},  the morphism $Y\to\gra(Z^{\prime})$ corresponds to an $\mathfrak R$-morphism $\hra(Y)\to Z^{\prime}$. Further, since $\hra$ respects closed immersions by Proposition \ref{exist-r}, formula   \eqref{hrn-aff} and Lemma \ref{r-nilp} show that $\hra(Y)$ is an affine $\mathfrak R$-scheme and $\hra(\iota)\colon \hra(Y_{*})\to\hra(Y)$ is a nilpotent immersion. Thus, since $f$ is formally smooth (respectively, formally unramified, formally \'etale), the bottom horizontal map in the following diagram (whose vertical maps are the canonical isomorphisms of Lemma \ref{g-adj})
\[
\xymatrix{\Hom_{\e \gra(Z^{\prime})}(\e Y,\gra(Z))\ar@{=}[d]\ar[r]& \Hom_{\e \gra(Z^{\prime})}(\e Y_{*},
\gra(Z))\ar@{=}[d]\\
\Hom_{Z^{\prime}}\be\big(\hra(Y),Z\big)\ar[r]&
\Hom_{Z^{\prime}}\be\big(\hra(Y_{*}\e),Z\big),}
\]
is surjective (respectively, injective, bijective). Thus the top horizontal map has the same property.
\end{proof}

\begin{corollary}\label{gr-sm} Let $f\colon Z\to Z^{\prime}$ be a smooth (respectively, unramified, \'etale) $\mathfrak R$-morphism.
Then $\gra(f)\colon \gra\lbe(Z)\to \gra\lbe(Z^{\le\prime}\le)$ is a
smooth (respectively, unramified, \'etale) $k$-morphism.
\end{corollary}
\begin{proof} This follows by combining the proposition and Proposition \ref{gr-prop}(iv).
\end{proof}

\begin{corollary}\label{gr-sm2} If $Z$ is a smooth (respectively, unramified, \'etale) $\mathfrak R$-scheme, then $\gra\lbe(Z)$ is a smooth (respectively, unramified, \'etale) $k$-scheme.
\end{corollary}
\begin{proof} This is immediate from the previous corollary using \eqref{grnsn}.
\end{proof}

\begin{corollary}\label{fflet}  Let $Z$ be a smooth $\mathfrak R$-scheme and let $\mathfrak R\to \mathfrak R^{\e\prime}$ be a surjective homomorphism of artinian local rings. Then the change of rings morphism \eqref{tr0}
\[
\varrho_{\lbe Z}^{\mathfrak R,\e\mathfrak R^{\prime}}\colon\mr{Gr}^{\mathfrak R}\lbe(Z\le)\to\mr{Gr}^{ \mathfrak R^{\prime}}\!(Z^{\e\prime}\le)
\]
is faithfully flat.
\end{corollary}
\begin{proof}\footnote{The proof in the previous version  was incomplete. We thank I. Vanni for calling our attention to this inaccuracy.} By Corollary \ref{rqc}, Proposition \ref{sm-surj} and Corollary \ref{gr-sm2}, $\varrho_{\lbe Z}^{\mathfrak R,\e\mathfrak R^{\prime}}$
is a quasi-compact and surjective morphism of smooth $k$-schemes. For proving flatness, we may work locally on $Z$ and assume that $Z$ is \'etale over an affine space $\A^{\be n}_{\mathfrak R}$. By Proposition \ref{f-et}, it suffices to consider the case $Z=\A^{\be n}_{\mathfrak R}$. The latter scheme can be endowed with the usual additive group scheme structure, whence the change of rings morphism is a morphism of $k$-group schemes (see the lines below diagram \eqref{funct}). We can now apply Lemma \ref{flat1} to complete the proof. 
\end{proof}

\begin{proposition}\label{gr-projlim} Let $\Lambda$ be a directed set and let $(Z_{\lambda})_{\lambda\e\in\e\Lambda}$ be a projective system of $\mathfrak R$-schemes with affine transition morphisms. Then $(\gra(Z_{\lambda}))_{\lambda\in\Lambda}$ is a projective system of $k$-schemes with affine transition morphisms and there exists a canonical isomorphism of $k$-schemes
\[
\gra\be\big(\varprojlim Z_{\lambda}\big)=\varprojlim\gra\lbe(Z_{\lambda}).
\]
\end{proposition}
\begin{proof} By Subsection \ref{not}, $\varprojlim Z_{\lambda}$ exists in $(\e\mathrm{Sch}/\mathfrak R)$ and therefore $\gra\be\big(\varprojlim Z_{\lambda}\big)$ is defined. The first assertion is clear from Proposition \ref{gr-prop}(vi), and the second follows from \eqref{adj}, \eqref{plim} and Yoneda's lemma, as in the proof of Proposition-Definition \ref{pr-def1}. 
\end{proof}

\begin{examples}\label{non-prop2} The functor $\gra$ fails to preserve some properties of morphisms such as those described below.  
\begin{enumerate}
\item[(a)] If $f$ is a proper morphism of $\mathfrak R$-schemes, then $\gra(f\le)$ may fail to be a proper morphism of $k$-schemes (if $n>1$). See \cite[Example A.5.6]{cgp} for an equal characteristic example with $\mathfrak R=k[\lle t\lle]/(t^{n})$ (in which case $\mr{Gr}^{\mathfrak R}$ is the Weil restriction funtor by Remark \ref{rems1}(c)).
\item[(b)] Flat morphisms are not, in general, preserved by $\gra$. Indeed, let $k$ be a field of positive characteristic $p\neq 2$ and let $\mathfrak R=k[\lle t\lle]/(t^{\le 2})$. Then $\mr{Gr}_{2}^{R}=\Re_{\e\mathfrak R/k}$ by  \eqref{eqq}. Now consider the free $\mathfrak R[x]$-module $\mathfrak R[x,y]/(y^{2}-t\le x)$ and let $\varphi\colon \mathfrak R[x]\to \mathfrak R[x,y]/(y^{2}-t\le x) $ be the canonical inclusion. Then $f=\spec(\varphi)$ is a flat morphism of affine $\mathfrak R$-schemes. On the other hand, the morphism of $k$-schemes $\mr{Gr}_{2}^{R}(\lle f\le)=\Re_{\e\mathfrak R/k}(\lle f\le)$  is the morphism associated to the homomorphism of $k$-algebras
\[
A=k[x_{0},x_{1}]\to B=k[x_{0},x_{1},y_{0},y_{1}]/(y_{0}^{2}, x_{0}-2y_{0}y_{1})
\]
which is induced by the inclusion $k[x_{0},x_{1}]\to k[x_{0},x_{1},y_{0},y_{1}]$, whence $\Re_{\e\mathfrak R/k}(\lle f\le)$ is not flat. In effect, since $x_{0}y_{0}=2y_{0}^{2}y_{1}=0$ in $B$, the $A$-regular element $x_{0}\in A$ is a zero divisor in $B$ and therefore is not $B$-regular.  See \cite[(1.0), p.~12, and (3.F), p.~21]{mat}.
\item[(c)] If $f$ is a finite morphism of $\mathfrak R$-schemes, then $\gra(f\le)$ may fail to be a finite morphism of $k$-schemes. See  \eqref{alp}.
\end{enumerate}
\end{examples}

\section{Greenberg's structure theorem}\label{gstr}
In this Section we  establish a generalized version of the main theorem of \cite{gre2}, showing in particular that the original result in \cite{gre2} is unaffected by the various changes allowed by the author to the structural sheaves of the schemes that intervene in his arguments (see Remark \ref{gr-max}). We keep the notation of  Section \ref{s-cr}. In particular, $\mathfrak R$ is an artinian local ring with maximal ideal $\mathfrak M$ and residue field $k$ which is either a finite $k$-algebra, where $k$ is arbitrary, or a finite $W_{\be m}(k)$-algebra, where $k$ is perfect of  characteristic $p>0$ and $m>1$.

\smallskip

We will consider the following cases:
\begin{itemize}
\item[(i)] $\mathfrak R$ is a  $k$-algebra and $\mathfrak I$ is an ideal of $\mathfrak R$ such that $\mathfrak I\le\mathfrak M=0$, or
\item[(ii)] $\mathfrak R$ is a finite $W(k)$-algebra of characteristic $p^{\le m}$, where $m>1$, and $\mathfrak I$ is a minimal ideal of $\mathfrak R$.
Note that $\mathfrak R$ is also a finite $W_{\be m}(k)$-algebra.
\end{itemize}

Note that $\mathfrak I\e\mathfrak M=0$ in either case. In particular, since $\mathfrak I\subset\mathfrak M$, we have $\mathfrak I^{\le 2}=0$. As before, we will write $\mathfrak R^{\le\prime}=\mathfrak R\le/\le\mathfrak I$. In case (i), let $t$ denote ${\rm dim}_{\le k}\le\mathfrak I$. In case (ii), let $t$ be the unique non-negative integer such that $\overbar{{\s I}}\simeq {}^{p^{\le t}}\be\mathbb O_{k}$ as $\mathbb O_{k}$-module schemes (see Proposition {\rm \ref{gwist}}). For every $\mathfrak R$-scheme $X$, consider the quasi-coherent $\s O_{\! X_{\rm s}}$-module
\begin{equation}\label{ex}
\s E_{\lbe X_{\lbe\rm s}\lbe/\lle k}=
\begin{cases}
\displaystyle\bigoplus_{i=1}^{t}\Omega_{X_{\mr s}/k}^{1} \qquad\text{\hskip .04cm in case {\rm (i)}}\\
\left(\! F^{\e p^{\lle t}}_{\!\be X_{\mr s}}\lbe\right)^{\!\be *}\! \Omega_{X_{\mr s}/k}^{1}\quad\text{in case {\rm (ii)},}
\end{cases}
\end{equation}
where, in case (ii), $F_{\! X_{\mr s}}$ denotes the absolute Frobenius endomorphism of $X_{\mr s}$.  Note that, if $X_{\rm s}=\spec A$ is an affine $k$-scheme (where $A$ is a $k$-algebra), then the sheaf $\left(\! F^{\e p^{\lle t}}_{\!\be X_{\mr s}}\lbe\right)^{\!\be *}\! \Omega_{X_{\mr s}/k}^{1}$ in case (ii) corresponds to the $A$-module $(\Omega_{A/k}^{1})^{\le( ^{\le nt}{\!\be A})}$ discussed in \cite[\S 4]{bga}.

\begin{remark}\label{val}  Let $R$ be a discrete valuation ring with residue field $k$ and let $n>0$ be an integer. Consider the following cases: (a) $R$ is an equal characteristic ring, (b) $R$ has unequal characteristics $(0,p)$ and $n\leq \ari=v(\e p\lle)$ and (c) $R$ has unequal characteristics and $n> \ari$. Then $(\mathfrak R, \mathfrak I\e)=(R_{\le n},M_{\lbe n}^{n-1})$ is a valid choice in case (i) if either (a) or (b) holds, and in case (ii) if (c) holds. Note that, in cases (a) and (b) (respectively, case (c)) we have $t=1$ by \eqref{vid} (respectively, $t=m-1$, where $m$ is given by \eqref{m},
as noted in Remark \ref{twist3}). Therefore \eqref{ex} is given by
\begin{equation*} 
\s E_{\lbe X_{\lbe\rm s}\lbe/\lle k}=
\begin{cases}
\displaystyle \hskip .2cm\Omega_{X_{\mr s}/k}^{1} \qquad\text{\hskip 1cm if  ${\rm char} \e R={\rm char} \e k$}\\
\hskip .2cm\Omega_{X_{\mr s}/k}^{1} \qquad\text{\hskip 1cm if  ${\rm char} \e R\neq p= {\rm char} \e k$ and $n\leq \ari=v(p)$}\\
\left(\! F^{\e p^{\le m-1}}_{\!\be X_{\mr s}}\lbe\right)^{\!\be *}\! \Omega_{X_{\mr s}/k}^{1}\quad\text{if ${\rm char}\e R\neq p= {\rm char} \e k$ and $n> \ari=v(p)$.}
\end{cases}
\end{equation*}
\end{remark}

\medskip

Now let the following data be given: a $k$-scheme $Y$, an $\mathfrak R$-scheme $X$ and a $k$-morphism $u^{\e\prime}\lbe\colon\lbe Y\!\lbe\to\be \mr{Gr}^{\mathfrak R^{\le\prime}}\!(X^{\lle\prime}\le)$, where $X^{\prime}=X\be\times_{\mathfrak R}\,\spec \mathfrak R^{\le\prime}$. Note that $Y$ is an $X_{\mr s}$-scheme via the $k$-morphism $\est\colon Y\to X_{\mr s}$ which is defined by the commutativity of the diagram
\begin{equation}\label{dfe}
\xymatrix{
Y\ar[rr]^(0.43){u^{\le\prime}}\ar[drr]_(.45){\est}  && \mr{Gr}^{\le \mathfrak R^{\le\prime}}\!(X^{\prime}\le) \ar[d]^(.4){\varrho_{\lbe X^{\lbe\prime}}^{\le\mathfrak R^{\prime}\!\be,\le k}}\\
&& X_{\mr s},
 }
\end{equation} 
where $\varrho_{\lbe X^{\lbe\prime}}^{\le\mathfrak R^{\prime}\!\lbe,\e k}$ is the change of rings morphism \eqref{tr0}. Next, consider the Zariski sheaf of abelian groups on $Y$
\begin{equation}\label{hh}
\s H_{\est}={\s H\!om}_{\cO_{\lbe Y}}\!\big(\est^{\lbe *}\lle\Omega_{X_{\mr s}\lbe/k}^{1}\le, \overbar{\s I}\be(\be\cO_{Y}\be)\big).
\end{equation} 
Then, by \cite[4.4.7.1 p. 102]{ega1}, for every open subset $U$ of $Y$ we have
\begin{equation*} 
\s H_{\est}(U)=\mr{Hom}_{\e\cO_{\lbe U}}\be\lbe\big((\est\!\be\mid_{\le U})^{\lbe *}\lle\Omega_{X_{\mr s}\lbe/k}^{1}\le, \overbar{\s I}\be(\be\cO_{\lbe U}\be)\big)=\mr{Hom}_{\e\s O_{\be X_{\lbe\rm s}}}\!\big(\Omega_{X_{\mr s}\lbe/k}^{1}\le, (\est\!\be\mid_{\le U})_{*}\lle\overbar{\s I}\be(\be\cO_{\lbe U}\be)\big).
\end{equation*} 

\begin{proposition}\label{sth} Let $\mathfrak R$ be as in {\rm (i)} or {\rm (ii)} above, let $X$ be an $\mathfrak R$-scheme and let $Y$ be a $\mr{Gr}^{\le\mathfrak R^{\le\prime}}\!(\lbe X^{\prime}\le)$-scheme. Then there exists an isomorphism of Zariski sheaves of abelian groups on $Y$
\[
\s H_{\est}\simeq  \mathbb V\!\be\left(\lbe\s E_{\lbe X_{\lbe\rm s}\lbe/\lle k}\right)\times_{\sigma,\, X_{\mr s}\e,\,\varrho_{\lbe X^{\lbe\prime}}^{\le\mathfrak R^{\prime}\!\be,\le k} }\mr{Gr}^{\le\mathfrak R^{\le\prime}}\!(\lbe X^{\prime}\le),
\]
where  $\s H_{\est}$ and $\s E_{\lbe X_{\lbe\rm s}\lbe/\lle k}$ are given by \eqref{hh} and \eqref{ex}, respectively,  and $\sigma\colon \mathbb V\!\be\left(\lbe\s E_{\lbe X_{\lbe\rm s}\lbe/\lle k}\right)\to X_{\mr s}$ is the canonical structural morphism.
\end{proposition}
\begin{proof}  Note that every open subscheme $U$ of $Y$ is a $\mr{Gr}^{\le\mathfrak R^{\le\prime}}\!(\lbe X^{\prime}\le)$-scheme and $\mathbb V\!\be\left(\lbe\s E_{\lbe X_{\lbe\rm s}\lbe/\lle k}\right)\times_{ X_{\mr s} }\mr{Gr}^{\le\mathfrak R^{\le\prime}}\!(\lbe X^{\prime}\le)$ is the Zariski sheaf on $Y$ whose sections on $U$ are the $\mr{Gr}^{\le\mathfrak R^{\le\prime}}\!(\lbe X^{\prime}\le)$-morphisms $U\to \mathbb V\!\be\left(\lbe\s E_{\lbe X_{\lbe\rm s}\lbe/\lle k}\right)\times_{  X_{\mr s} }\mr{Gr}^{\le\mathfrak R^{\le\prime}}\!(\lbe X^{\prime}\le)$.  In case (i), i.e., $\mathfrak R$ is a $k$-algebra, $\mathfrak I$ is an ideal of $\mathfrak R$ such that $\mathfrak I\le\mathfrak M=0$ and $t={\rm dim}_{\le k}\le\mathfrak I$, the choice of an isomorphism of $k$-modules $\mathfrak I\simeq k^{\e t}$ determines an isomorphism of $\mathbb O_{k}$-module schemes $\s I=\overbar{\s I}\simeq \mathbb O_{k}^{\e t}$ so that $\overbar{\s I}\be(\be\cO_{\lbe Y}\be)\simeq \oplus_{\e i=1}^{\e t}\cO_{\lbe Y}$. Thus, by \eqref{hh}and \eqref{vb2}, for every open subset $U$ of $Y$ we have
\[
\begin{array}{rcl}
\s H_{\est}(U)&\simeq&\mr{Hom}_{\e\s O_{\be X_{\lbe\rm s}}}\!\lbe\big(\Omega_{X_{\mr s}\lbe/k}^{1}\le, (\est \!\be\mid_{\le U}\lbe)_{*}\lle\cO_{\lbe U}\lbe\big)^{\be t}\simeq
\mr{Hom}_{\e X_{\lbe\rm s}} \lbe\big(U, \mathbb{V}( \textstyle\bigoplus_{i=1}^{t}\!\Omega_{X_{\mr s}\lbe/k}^{1})\big) \\
&=&\mr{Hom}_{\e X_{\lbe\rm s}} \lbe\big(U, \mathbb V\lbe(\s E_{\lbe X_{\lbe\rm s}\lbe/\lle k}\le) )\simeq \mr{Hom}_{\e\mr{Gr}^{\le\mathfrak R^{\le\prime}}\!(\lbe X^{\prime}\le)} \lbe\big(U, \mathbb V\lbe(\s E_{\lbe X_{\lbe\rm s}\lbe/\lle k}\le)\times_{X_{\mr s}}\mr{Gr}^{\le\mathfrak R^{\le\prime}}\!(\lbe X^{\prime}\le)).
\end{array}
\]
In case (ii), the isomorphism of $\mathbb O_{k}$-modules $\overbar{\s I}\simeq {}^{p^{\le t}}\lbe\lbe\mathbb O_k$ of Proposition \ref{gwist} yields an isomorphism of Zariski sheaves $\overbar{\s I}\be(\be\cO_{\lbe U}\be)\simeq {}^{p^{\lle t}}\!\be\cO_{\lbe U}$ for every open subset $U$ of $Y$. Thus, by \cite[(4.12) and Caveat 4.14]{bga}, we have  
\begin{eqnarray*}
\s H_{\est}(U)&\simeq&\mr{Hom}_{\le\le\cO_{\lbe U}}\be\lbe\big((\est\!\be\mid_{\le U})^{\lbe *}\Omega_{X_{\mr s}\lbe/k}^{1}\le, {}^{p^{\lle t}}\!\be\cO_{\lbe U}\big)\simeq\mr{Hom}_{\le\cO_{\lbe U}}\!\!\left(\left(\! F^{\e p^{\lle t}}_{\be U}\lbe\right)^{\!\be *}\!\be (\est \!\be\mid_{\le U}\lbe)^{\lbe *}\le\Omega_{X_{\mr s}\lbe/k}^{1}\le, \cO_{\lbe U}\be\right)\\
&\simeq&\mr{Hom}_{\le\le\cO_{\lbe U}}\!\!\left((\est \!\be\mid_{\le U}\lbe)^{\lbe *}\!\!\left(\! F^{\e p^{\lle t}}_{\!\be X_{\mr s}}\lbe\right)^{\!\be *}\!\Omega_{X_{\mr s}\lbe/k}^{1}\le, \cO_{\lbe U}\be\right)=
\mr{Hom}_{\e\cO_{\be X_{\mr s}}}\!\!\left(\s E_{\lbe X_{\lbe\rm s}\lbe/\lle k}, (\est\!\be\mid_{\le U}\lbe)_{*}\cO_{\lbe U}\be\right)\\
&\simeq &\mr{Hom}_{\e X_{\lbe\rm s}} \lbe\big(U, \mathbb V\lbe(\s E_{\lbe X_{\lbe\rm s}\lbe/\lle k}\le) )  \simeq \mr{Hom}_{\mr{Gr}^{\le\mathfrak R^{\le\prime}}\!(\lbe X^{\prime}\le)} \lbe\big(U, \mathbb V\lbe(\s E_{\lbe X_{\lbe\rm s}\lbe/\lle k}\le)\times_{X_{\mr s}}\mr{Gr}^{\le\mathfrak R^{\le\prime}}\!(\lbe X^{\prime}\le)). 
\end{eqnarray*}
\end{proof}

\begin{corollary}\label{sth1}  Let $R$ be a discrete valuation ring with residue field $k$ and let $n\geq 1$ be an integer. If $X$ is an $R_{\le n }$-scheme and $Y$ is a $\mr{Gr}_{\be\le n-1}^{R} (X)$-scheme, let $\s H_{a}={\s H\!om}_{\cO_{\lbe Y}}\be\!\left(\lbe\est^{\lbe *}\lle\Omega_{X_{\mr s}\lbe/k}^{1}\le, \overbar{\s M_{\lbe n}^{\lle n-1}}(\be\cO_{Y}\be)\right)$, where $a$ is defined by the commutativity of diagram \eqref{dfe}. Then there exists a canonical isomorphism of Zariski sheaves of abelian groups on $Y$
\[
\s H_{a}\simeq 
\mathbb V\lbe \left( \s E_{\lbe X_{\lbe\rm s}\lbe/\lle k}\right)\times_{X_{\mr s} }\mr{Gr}_{\be\le n-1}^{R} (X) ,
\]
where $\s E_{\lbe X_{\lbe\rm s}\lbe/\lle k}=\Omega_{X_{\mr s}/k}^{1}$ if ${\rm char} \e R={\rm char} \e k$ or ${\rm char}\e R=0\neq p={\rm char}\e k$ and $n\leq \ari=v(p)$, and $\s E_{\lbe X_{\lbe\rm s}\lbe/\lle k}=\left(F^{\e p^{\lle m-1 }}_{\!\be X_{\mr s}}\right)^{\!\be *}  \Omega_{X_{\mr s}/k}^{1}$ if ${\rm char}\e R=0\neq p={\rm char}\e k$ and $n>\ari$ with $m=\lceil n/\ari\e\rceil$.
\end{corollary}
\begin{proof} This is immediate from the proposition and Remark \ref{val}.
\end{proof}

We now consider the following extension problem: find a morphism $u\colon Y\to \gra\lbe(X)$ such that the following diagram commutes
\begin{equation}\label{4d}
\xymatrix{
Y\ar@{-->}[rr]^(0.45){u}\ar[drr]_(.4){u^{\le\prime}}  && \mr{Gr}^{\le \mathfrak R}\lbe(X) \ar[d]^(.4){\varrho_{X}^{\le\mathfrak R,\mathfrak R^{\prime}}} \\
&& \mr{Gr}^{\le\mathfrak R^{\le\prime}}\!(X^{\prime}),
 }
\end{equation}
where $\varrho_{X}^{\le\mathfrak R,\le\mathfrak R^{\prime}}$ is the change of rings morphism \eqref{tr0}. By Proposition \ref{rdad}, the preceding problem is equivalent to that of finding an $\mathfrak R$-morphism $f\colon h^{\lbe \mathfrak R}\lbe(Y)\to X$ such that the following diagram commutes
\begin{equation}\label{dlft}
\xymatrix{h^{\lbe \mathfrak R}\lbe(Y)\ar[rrr]^(.5){f} &&& X\\
h^{\mathfrak R^{\prime}}\!(Y)\ar[u]^(.48){\delta_{\lbe Y}^{\le\mathfrak R,\mathfrak R^{\prime}}}\ar[rrr]^(.53){\varphi_{Y,\le X^{\lbe\prime}}^{\le\mathfrak R^{\prime}}\lbe(u^{\le\prime}\le)}&&& X^{\prime}\ar[u]_(.45){\mr{pr}_{\lbe X}}.
}
\end{equation}
Indeed, if such an $f$ exists, then $u=\psi_{Y,X}^{\e\mathfrak R}\lbe(\le f\le)$ solves the original problem, where $\psi_{Y,X}^{\e\mathfrak R}$ is the bijection \eqref{vpsi}. Note that both vertical maps in \eqref{dlft} are nilpotent immersions. Further, $\delta^{\le\mathfrak R,\mathfrak R^{\prime}}_{Y}$ has a square-zero ideal of definition.

\begin{remark} \label{warn} If $u$ in \eqref{4d}, and therefore $f$ in \eqref{dlft}, exist, then the following holds.
\begin{enumerate}
\item[(a)] By Remark \ref{vrk}(c) and the commutativity of diagrams \eqref{dfe} and \eqref{4d}, the diagram
\[
\xymatrix{
Y\ar[rr]^(0.43){u}\ar[drr]_(.45){\est}  && \mr{Gr}^{\le \mathfrak R}\be(X\le) \ar[d]^(.4){\varrho_{X}^{\le\mathfrak R,\le k}}\\
&& X_{\mr s}\,
 }
\]
commutes. In other words, $u$ is a lifting of $\est$ to $\mr{Gr}^{\le \mathfrak R}\be(X\le)$. 
\item[(b)]  We claim that, if $\iota_{\le Y}$ is the map \eqref{io} and $\est$ is defined by the commutativity of diagram \eqref{dfe}, then the following diagram commutes
\[
\xymatrix{Y\ar[rr]^(0.43){\iota_{\le Y}}\ar[drr]_(.45){\est}  && \hra(Y)_{\mr s} \ar[d]^(.4){f_{\mr s}}\\
&& X_{\mr s}.
}
\] 
Indeed, for every $\mathfrak R$-scheme $Z$, let $\iota_{\e\mr s,\e Z}$ denote the nilpotent immersion $Z_{\mr s}\to Z$. Then, by (a), the commutativity of diagram \eqref{dlft} for $\mathfrak R^{\prime}=k$ and the commutativity of \eqref{vrk-d}, we have the following equalities of morphisms $Y\to X$
\[
\iota_{\e\mr s,\e X}\circ \est= f\circ \delta^{\e\mathfrak R,k}_{Y}=f\circ \iota_{\e\mr s,\e \hra(Y)}\circ \iota_{Y}=\iota_{\e\mr  s,\e X}\circ f_{\e\mr s}\circ   \iota_{Y}.
\] 
Since $\iota_{\mr s,X}$ is a nilpotent immersion, we conclude from the above that $\est=f_{\e\mr s}\circ\iota_{\e Y}$, as claimed.
\end{enumerate}
\end{remark}

Now let $\s P(\lbe u^{\le\prime}\le)$ be the following Zariski sheaf of sets on $Y$: for every open subset $U\subseteq Y$, let $\s P(u^{\le\prime})(U)$ be the set of $k$-morphisms
$v\colon U\to\gra\lbe(X)$ (if any exist) such that the diagram
\[
\xymatrix{
U\ar@{-->}[rr]^(0.45){v}\ar[drr]_(.4){u^{\le\prime}\mid_{\le U}}  && \mr{Gr}^{\le \mathfrak R}\lbe(X) \ar[d]^(.4){\varrho_{X}^{\le\mathfrak R,\mathfrak R^{\prime}}} \\
&& \mr{Gr}^{\le\mathfrak R^{\le\prime}}\!(X^{\prime})
}
\]
commutes. Then, by \eqref{adj} and Proposition \ref{rdad}, $\s P(u^{\le\prime})(U)$ is in bijection with the set of $\mathfrak R$-morphisms $f_{U}\colon \hra\lle(U)\to X$ (if any exist) such that the following diagram commutes
\[
\xymatrix{h^{\lbe \mathfrak R}\lbe(U)\ar[rrr]^(.5){f_{U}} &&& X\\
h^{\mathfrak R^{\prime}}\!(U)\ar[u]^(.48){\delta_{\lbe U}^{\le\mathfrak R,\mathfrak R^{\prime}}} \ar[rrr]^(.53){ \varphi_{Y,\le X^{\lbe\prime}}^{\le\mathfrak R^{\prime}}\lbe(u^{\le\prime}\le)\mid_{\e h^{\lbe 
\text{\scalebox{0.7}{$\mathfrak{R}^{\prime}$}} }\!\be(U)} }&&& X^{\prime}\ar[u]_(.45){\mr{pr}_{\lbe X}}.
}
\]
Clearly, the existence of diagrams \eqref{4d} and \eqref{dlft} is equivalent to the non-emptyness of $\s P\lbe(\lbe u^{\le\prime}\le)(Y\le)\,$.

\begin{lemma} 
For every (respectively, every smooth) $\mathfrak R$-scheme $X$, the Zariski sheaf $\s P\lbe(\lbe u^{\le\prime}\le)$ defined above is a formally principal homogeneous (respectively, principal homogeneous) sheaf on $|Y|$ under the abelian sheaf $\s H_{\est}={\s H\!om}_{\cO_{Y}}\!\big(\est^{\lbe *}\lle\Omega_{X_{\mr s}\lbe/k}^{1}\le, \overbar{\s I}\be(\be\cO_{Y}\be)\big)$, where $a$ is defined by the commutativity of diagram \eqref{dfe}.
\end{lemma}
\begin{proof}  From \cite[III, Proposition 5.1]{sga1} with $S=\spec \mathfrak R$ and $g_{0}=\mr{pr}_{X}\circ \varphi_{Y,\le X^{\lbe\prime}}^{\le\mathfrak R^{\prime}}\lbe(u^{\le\prime}\le)$, i.e., the following diagram commutes  
\[
\xymatrix{h^{\lbe \mathfrak R}\lbe(Y)\ar[rrr]^(.5){f} &&& X\\
h^{\mathfrak R^{\prime}}\!(Y)\ar[urrr]^{g_{\lle 0}}\ar[u]^(.48){\delta_{\lbe Y}^{\le\mathfrak R,\mathfrak R^{\prime}}}\ar[rrr]_(.53){\varphi_{Y,\le X^{\lbe\prime}}^{\le\mathfrak R^{\prime}}\lbe(u^{\le\prime}\le)}&&& X^{\prime}\ar[u]_(.45){\mr{pr}_{\lbe X}}
}
\]
\eqref{dlft}, we conclude that $\s P\lbe(\lbe u^{\le\prime}\le)$ is a formally principal homogeneous sheaf under the sheaf
${\s H\!om}_{\e\s R^{\prime}(\lbe\cO_{\lbe Y}\lbe)}\be \big(\e g_{0}^{\lbe *}\le\Omega_{X\lbe/\mathfrak R}^{1}\le, \overbar{\s I}\be(\be\cO_{\lbe Y}\be)\big)$.
Thus it suffices to check that the latter sheaf equals $\s H_{\est}={\s H\!om}_{\cO_{\lbe Y}}\!\big(\est^{\lbe *}\lle\Omega_{X_{\mr s}\lbe/k}^{1}\le, \overbar{\s I}\be(\be\cO_{Y}\be)\big)$. By definition of $g_{0}$, we have $g_{0}^{\lbe *}\e\Omega_{X\lbe/\mathfrak R}^{1}=(\varphi_{Y,\le X^{\lbe\prime}}^{\le\mathfrak R^{\prime}}\lbe(u^{\le\prime}\le) )^{*}\e\Omega_{X^{\prime}\be/\mathfrak R^{\prime}}^{1}$. Now, since $\mathfrak M\e\mathfrak I=0$, the $\s R^{\e\prime}(\cO_{\lbe Y})$-module structure on $\s I(\cO_{\lbe Y})$ induces an $\s R^{\e\prime}(\cO_{\lbe Y})/\overbarr{\s M^{\lle\prime}}\be(\be\s O_{Y}\be) =\cO_{\lbe Y}$-module structure on $\s I(\cO_{\lbe Y})$, where $\mathfrak M^{\prime}=\mathfrak M/\mathfrak I$ is the maximal ideal of $\mathfrak R^{\prime}$. Therefore
\[
{\s H\!om}_{\e\s R^{\le\prime}\lbe(\cO_{\lbe Y})}\be\big(\e g_{0}^{\lbe *}\e\Omega_{X\lbe/\mathfrak R}^{1}\le, \overbar{\s I}\be(\be\cO_{\lbe Y}\be)\big)=
{\s H\!om}_{\le\cO_{\lbe Y}}\!\big( 
(\varphi_{Y,\le X^{\lbe\prime}}^{\le\mathfrak R^{\prime}}\lbe(u^{\le\prime}\le) )^{*}\e\Omega_{X^{\prime}\be/\mathfrak R^{\prime}}^{1}\otimes_{\e\s R^{\le\prime}(\cO_{\lbe Y})}\cO_{\lbe Y} , \overbar{\s I}\be(\be\cO_{Y}\be)\big).
\] 
It remains to check that $(\varphi_{Y,\le X^{\lbe\prime}}^{\le\mathfrak R^{\prime}}\lbe(u^{\le\prime}\le) )^{*}\e\Omega_{X^{\prime}\be/\mathfrak R^{\prime}}^{1}\otimes_{\e\s R^{\le\prime}(\cO_{\lbe Y})}\cO_{\lbe Y}= \est^{*}\e\Omega_{X_{\mr s}\lbe/k}^{1}$.
This is immediate from Remark \ref{warn}(b) setting $\mathfrak R=\mathfrak R^{\prime}$ and $f=\varphi_{Y,\le X^{\lbe\prime}}^{\le\mathfrak R^{\prime}}\lbe(u^{\le\prime}\le)$ there, using the fact that $\s L\otimes_{\s R^{\prime}\lbe(\cO_{\lbe Y})}\cO_{\lbe Y}=(\iota_{\le\mr s,\le h^{\mathfrak R^{\prime}}(Y)}\circ \iota_{Y})^{*}\lbe\s L$ for every sheaf of $\s R^{\e\prime}(\cO_{\lbe Y})$-modules $\s L$ together with the equality $f^{\e\prime}\circ  \iota_{\le\mr s,\le h^{\mathfrak R^{\prime}}\be(Y)}=\iota_{\le\mr s,\le X}\circ f_{\mr s}^{\e\prime}$ for arbitrary morphisms $f^{\e\prime}\colon h^{\mathfrak R^{\prime}}(Y)\to X^{\prime}$.

Finally, assume that $ X$ is smooth over $\mathfrak R$. If $Y$ is affine, then $\s P\lbe(\lbe u^{\le\prime}\le)$ has global sections by the lifting property \cite[\S 2.2, Proposition 6, p. 37]{blr}. In general, $\s P\lbe(\lbe u^{\le\prime}\le)$ has non-empty fibers and is therefore a principal homogeneous sheaf under $\s H_{\est}$.
\end{proof}

\begin{remarks} 
\indent
\begin{enumerate}
\item[(a)]  In the terminology of \cite[III,  1.1.5 p.~107 and 1.4.1, p.~117]{gi}, the previous lemma states that $\s P\lbe(\lbe u^{\le\prime}\le)$ is a pseudo-torsor (respectively, torsor) under $\s H_{\est}$ on the Zariski topos, i.e., the category of  sheaves of sets on the small Zariski site of $Y$. Note also that the smoothness of $X$ guarantees the nonemptiness of the fibers of $\s P\lbe(\lbe u^{\le\prime}\le)$, whence condition \cite[III, 1.4.1(a)]{gi} does hold by \cite[III, 1.4.1.1]{gi}.
\item[(b)] By \eqref{bc}, the global sections of $\s P\lbe(\lbe u^{\le\prime}\le)$ over $Y$ correspond to the set-theoretic sections of the projection  $\mr{pr}_{Y}\colon Y\times_{u^{\prime}\lbe,\e \mr{Gr}^{\mathfrak R^{\le\prime}}\!(X^{\prime}),\e \varrho_{X}^{\le\mathfrak R,\le\mathfrak R^{\prime}} }\gra(X)\to Y$.
\item[(c)]  By definition of the term {\it formally principal homogeneous sheaf} (= pseudo-torsor for the Zariski topos), there exists an isomorphism of Zariski sheaves on $Y$
\[
\s H_{\est} \times \s P\lbe(\lbe u^{\le\prime}\le) \overset{\!\sim }{\to} \s P\lbe(\lbe u^{\le\prime}\le) \times  \s P\lbe(\lbe u^{\le\prime}\le).
\]
Now global sections of the sheaf $\s P\lbe(\lbe u^{\le\prime}\le) \times  \s P\lbe(\lbe u^{\le\prime}\le)$ are pairs $(u_{1},u_{2})$ of morphisms $Y\to \gra(X)$ which lift $u^{\le\prime}$. On the other hand, by Proposition \ref{sth}, the global sections of the sheaf $\s H_{\est} \times \s P\lbe(\lbe u^{\le\prime}\le)$ correspond to pairs $(x,u)$ where $u\colon Y\to \gra(X)$ is a lifting of $u^{\le\prime}$ (and thus of $a$) and  $x \colon Y\to\mathbb V\lbe(\s E_{\lbe X_{\lbe\rm s}\lbe/\lle k})$ is a morphism whose composition with the canonical morphism  $\mathbb V\lbe(\s E_{\lbe X_{\lbe\rm s}\lbe/\lle k})\to X_{\lbe\rm s}$ is $a$. Thus we obtain a bijection of fiber products of sets
\begin{equation*}
\mathbb V\lbe(\s E_{\lbe X_{\lbe\rm s}\lbe/\lle k})(Y)\times_{\{ \est\}}\! \gra(X)(Y) \overset{\!\sim }{\to}  \gra(X)(Y)\times_{\{u^{\le\prime}\}}\!\gra(X)(Y).
\end{equation*}
When $Y$ and $u^{\le\prime}$ vary, the latter bijections induce an isomorphism of $k$-schemes
\begin{equation}\label{qtrs}
\mathbb V\lbe(\s E_{\lbe X_{\lbe\rm s}\lbe/\lle k}) \times_{X_{\rm s}}\! \gra(X) \overset{\!\sim }{\to} \gra(X) \times_{ \mr{Gr}^{\mathfrak R^{\le\prime}}\be\!(X^{\prime}) }\!\gra(X).
\end{equation}
Consequently, if $y$ is a $k$-rational point of $\gra(X)$, then the fiber of $\varrho_{X}^{\le\mathfrak R,\le\mathfrak R^{\prime}}$ at $\varrho_{X}^{\le\mathfrak R,\le\mathfrak R^{\prime}}\!(\le y)$ is isomorphic to the fiber of $\sigma\colon \mathbb V\lbe(\s E_{\lbe X_{\lbe\rm s}\lbe/\lle k})\to X_{\rm s}$ at $\varrho_{X}^{\le\mathfrak R,\le k}(y)$. In effect, the base change of \eqref{qtrs} along $y\colon \spec k\to \gra(X)$ is an isomorphism from $\mathbb V\lbe(\s E_{\lbe X_{\lbe\rm s}\lbe/\lle k})\times_{X_{\rm s}}\spec k=\mathbb V\lbe(\s E_{\lbe X_{\lbe\rm s}\lbe/\lle k})_{\varrho_{X}^{\le\mathfrak R,\le k}\!(\le y)}$ to
$\gra(X)\times_{ \mr{Gr}^{\mathfrak R^{\le\prime}}\!(X^{\prime}) }\!\spec k=\gra(X)_{\varrho_{X}^{\le\mathfrak R,\le\mathfrak R^{\prime}}\!(\le y)}$.
\end{enumerate} 
\end{remarks}

\begin{theorem} 
Let $X$ be an arbitrary (respectively,  smooth) $\mathfrak R$-scheme. Then the $\mr{Gr}^{\le\mathfrak R^{\le\prime}}\!(X^{\prime})$-scheme $\mr{Gr}^{\le \mathfrak R}\lbe(X)$ with structural morphism $\varrho_{X}^{\le\mathfrak R,\le\mathfrak R^{\prime}}$ is a pseudo-torsor (respectively, torsor) under 
$\mathbb V\lbe(\s E_{\lbe X_{\lbe\rm s}\lbe/\lle k})\times_{\be X_{\mr s}}\! \mr{Gr}^{\mathfrak R^{\le\prime}}\!(X^{\prime})$ in the category of fppf sheaves of sets on $(\mr{Sch}/\mr{Gr}^{\mathfrak R^{\le\prime}}\!\lbe(X^{\prime}))$. 
\end{theorem}
\begin{proof}
(Compare \cite[Proposition 2, p.~262]{gre2}).  By \eqref{qtrs} and the identification
\[
\mathbb V\lbe(\s E_{\lbe X_{\lbe\rm s}\lbe/\lle k}) \times_{\be X_{\mr s} }\!  \gra(X)= (\mathbb V\lbe(\s E_{\lbe X_{\lbe\rm s}\lbe/\lle k})\times_{\be X_{\mr s} }\!  \mr{Gr}^{\mathfrak R^{\le\prime}}\!(X^{\prime}))\times_{ \mr{Gr}^{\mathfrak R^{\le\prime}}\!(X^{\prime}) } \gra(X),
\]  
$\gra(X)$ is, indeed, a pseudo-torsor \cite[III, Definition 1.1.5, p.~107]{gi}.
Assume now that $X$ is a smooth $\mathfrak R$-scheme. By definition of the term {\it fppf torsor} (cf. \cite[\S 6.4, p. 153]{blr}), it remains only to check that $\varrho_{X}^{\le\mathfrak R,\le\mathfrak R^{\prime}}\colon 
\gra(X) \to \mr{Gr}^{\mathfrak R^{\le\prime}}\!(X^{\prime})$ is an fppf morphism. This follows from Corollaries \eqref{gr-sm2} and \ref{fflet} together with \cite[Proposition 6.2.3(v), p.~298]{ega1}.
\end{proof}

\begin{example}\label{k2} Let $R$ be a discrete valuation ring and let $X$ be a smooth $R_{\le n}$-scheme. If $R$ is a ring of unequal characteristics $(0,\le p\lle)$ and  $n>\ari=v(\e p\lle)$, then the $\mr{Gr}_{\lbe n-1}^{R}(X)$-scheme $\grn(X)$ is an fppf torsor under $\mathbb V\!\lbe\lbe\left(\!\left(\! F^{\e p^{\lle m-1\lbe}}_{\!\be X_{\mr s}}\be\right)^{\!\be *}\! \Omega_{X_{\mr s}/k}^{1}\be\right)\times_{\be X_{\mr s}}\! \grn(X)$, where $m=\lceil n/\ari\e \rceil$. If $R$ has unequal characteristics and $n\leq \ari$, or if $R$ is an equal characteristic ring, then the $\mr{Gr}_{\lbe n-1}^{R}(X)$-scheme $\grn(X)$ is an fppf torsor under $\mathbb V\!\lbe\lbe\left( \Omega_{X_{\mr s}/k}^{1}\be\right)\times_{\be X_{\mr s}}\! \grn(X)$. See Corollary \ref{sth1}. 
\end{example}

\section{Weil restriction and the Greenberg functor}\label{wrbe}

In this Section we determine the behavior under Weil restriction of the Greenberg functor of truncated discrete valuation rings discussed in Section \ref{gberg}. See Theorem \ref{wr-gr} below.

Let $R$ be a complete discrete valuation ring with maximal ideal $\mm$ and residue field $k$ (assumed to be perfect when $R$ has unequal characteristics). We fix an integer $n\geq 1$ and recall $R_{\le n}=R/\mm^{n}$ and $S_{n}=\spec R_{\le n}$.

\begin{lemma}\label{hne1}
Let $R^{\e\prime}$ be a finite extension of $R$ of ramification index $e$ with associated residue field extension $k^{\e\prime}\be/k\subseteq \kbar/k$. Then, for every $k$-scheme $Y$, we have
\[
\hrn(Y)\times_{S_{n}}\!\be S^{\e\prime}_{\lbe ne}=h_{n e}^{R^{\le\prime}}\be\big(Y\!\be\times_{k}\be\spec
k^{\e\prime}\e\big)
\]

\end{lemma}
\begin{proof} Since $\hrn$ is local for the Zariski topology, we may assume that $Y=\spec A$ is affine, where $A$ is a $k$-algebra. By Lemma \ref{rne1},
\[
\begin{array}{rcl}
\hrn(Y)\times_{S_{n}}\!\be S^{\e\prime}_{\lbe ne}&=&\spec \s R_{n}(A)\times_{R_{\le n}}\! \spec  R_{\e ne}^{\e\prime}\\
&=&\spec\! \big(\s R_{\le
n}(A)\otimes_{R_{\le n}}\! R_{\le ne}^{\e\prime}\big) =\spec\!\big(\s R_{
ne}^{\e\prime}\be\big(A\otimes_{k} k^{\e\prime}\big)\big)\\
&=&
h_{n e}^{R^{\le\prime}}\be\big(Y\!\be\times_{k}\be\spec
k^{\e\prime}\e\big).
\end{array}
\]
\end{proof}

For the meaning of the term ``admissible" in the next two statements, see Definition \ref{adm}.

\begin{lemma} Let $R^{\e\prime}$ be a finite extension of $R$ of ramification index $e$ with associated residue field extension $k^{\e\prime}\be/k\subseteq \kbar/k$. If $Z$ is an $S_{\lbe ne}^{\e\prime}$-scheme which is admissible relative to $S_{\lbe ne}^{\e\prime}\to S_{n}$, then the $k^{\e\prime}$-scheme
$\mr{Gr}_{\lbe ne}^{R^{\le\prime}}\be(Z\le)$  is admissible 
relative to $k^{\e\prime}\be/k$.
\end{lemma}
\begin{proof} By Lemma \ref{adm1}, $Z\!\times_{S_{n\lbe e}^{\e\prime}}\! S_{1}^{\e\prime}$ is admissible relative to $k^{\e\prime}\be /k$. Thus, since
\[
\mr{Gr}_{\lbe ne}^{R^{\le\prime}}\be(Z\le)\to \mr{Gr}_{1}^{R^{\le\prime}}\!(Z\!\times_{S_{n\lbe e}^{\e\prime}}\! S_{1}^{\e\prime}\e)=Z\!\times_{S_{n\lbe e}^{\e\prime}}\! S_{1}^{\e\prime}
\]
is an affine morphism of $k^{\e\prime}$-schemes by Proposition \ref{aff}, $\mr{Gr}_{ne}^{R^{\le\prime}}(Z)$ is admissible relative to $k^{\e\prime}\be/k$ by Remark \ref{rems-adm}(d).
\end{proof}

We can now prove the main result of this section.

\begin{theorem}\label{wr-gr}
Let $R^{\e\prime}$ be a finite extension of $R$ of ramification index $e$ with associated residue field extension $k^{\e\prime}\be/k\subseteq \kbar/k$. If $Z$ is an $S_{\lbe ne}^{\e\prime}$-scheme which is admissible relative to $S_{\lbe ne}^{\e\prime}\to S_{n}$, then $\Re_{\e k^{\prime}\be/k}\big(\mr{Gr}_{ne}^{R^{\le\prime}}(Z\le)\big)$
and $\Re_{\le S_{\lbe ne}^{\le\prime}/S_{\le n}}\be(Z\le)$
exist and
\[
\Re_{\e k^{\prime}\be/k}\big(\mr{Gr}_{ne}^{R^{\le\prime}}(Z\le)\big)=\mr{Gr}_{n}^{R}\be\big(
\Re_{\le S_{\lbe ne}^{\le\prime}/S_{\le n}}\be(Z\le)\big).
\]
\end{theorem}
\begin{proof} The existence assertions follow from Theorem \ref{wr-rep} using the previous lemma. The formula of the theorem now follows from Lemma \ref{hne1} using the adjunction formula \eqref{adj}, the definition of the Weil restriction functor \eqref{wr} and Yoneda's lemma.\end{proof}

\begin{remark}\label{wgeq} The theorem is new only in the unequal characteristics case. For in the equal characteristic case the formula of the theorem is the well-known identity
\[
\Re_{\e k^{\prime}\be/k}\big(\Re_{R^{\le\prime}_{ne}/k^{\prime}}(Z\le)\big)=\Re_{R_{\le n}/k}\lbe\big(\Re_{R_{\le ne}^{\le\prime}/R_{\le n}}\be(Z\le)\big)
\]
which follows at once from \eqref{wrcomp}.
\end{remark}

\begin{corollary}\label{wr-gr1}
Let $Z$ be a quasi-projective  $R_{ne}^{\e\prime}$-scheme. Then $\Re_{\e k^{\prime}\be/k}\be\big(\mr{Gr}_{ne}^{R^{\le\prime}}\be(Z\le)\big)$
and $\Re_{\le R_{\lbe ne}^{\le\prime}/\lbe R_{\le n}}\be(Z\le)$ exist and
\[
\Re_{\e k^{\prime}\be/k}\be\big(\mr{Gr}_{ne}^{R^{\le\prime}}\be(Z\le)\big)=\mr{Gr}_{n}^{R}\be\big(\Re_{\le R_{\lbe ne}^{\le\prime}/\lbe R_{\le n}}\be(Z\le)\big).
\] 
\end{corollary}
\begin{proof} This follows from Proposition \ref{q-proj}, Remark \ref{rems-adm}(a) and the theorem.
\end{proof}

\begin{remark}
An application of the above corollary can be found in \cite[\S14]{cr}.
\end{remark}

\begin{proposition}\label{tot-gr}
Let $R^{\e\prime}$ be a finite and totally ramified extension of $R$ of degree $e$ and let $Z$ be an arbitrary $S_{\lbe ne}^{\e\prime}$-scheme. Then $\Re_{\le S_{ne}^{\le\prime}/S_{n}}\be(Z\le)$ exists and
\[
\mr{Gr}_{n}^{R}\be\big(
\Re_{\le S_{ne}^{\le\prime}/S_{n}}\be(Z\le)\big)=\mr{Gr}_{ne}^{R^{\le\prime}}\be(Z\le).
\]  
\end{proposition}
\begin{proof} The existence assertion is Remark \ref{tram}. The formula now follows from Lemma \ref{hne1} using \eqref{adj}, \eqref{wr} and Yoneda's lemma, as in the previous proof.
\end{proof}

The behavior of the Greenberg realization functor \eqref{grf2} under finite extensions of $R$ was discussed in \cite[Theorem 3.1]{ns} for $R_{n}$-schemes of finite type. Below we extend the indicated theorem to arbitrary $R_{n}$-schemes. We begin with the case of ramification index 1, where infinite extensions of $R$ are allowed.

\begin{proposition}\label{unr3}
Let $k^{\e\prime}\be/k$ be a subextension of $\e\kbar/k$ and let $R^{\e\prime}$ be the extension of $R$ of ramification index $1$ which corresponds to $k^{\e\prime}\be/k$. Then, for every $S_{n}$-scheme $Z$, there exists a canonical isomorphism of $k^{\e\prime}$-schemes
\[
\grn(Z\le)\be\times_{k}\lbe\spec
k^{\e\prime}=\mathrm{Gr}_{\lbe n}^{R^{\le\prime}}\!\big(Z\!\times_{S_{n}}\!
{S}^{\e\prime}_{n}\le\big).
\]
\end{proposition}
\begin{proof} Let $g_{n}\colon {S}^{\e\prime}_{\le n}\to S_{n}$ be the morphism induced by the canonical map $R_{\le n}\to R^{\,\prime}_{\le n}$. Note that $g_{0}$ is the morphism $\spec k^{\e\prime}\to \spec k$. For every $k^{\e\prime}$-scheme $T$, \eqref{bc}  and Lemma \ref{unr1} yield a canonical isomorphism of $S_{n}$-schemes
$h_{n}^{R^{\e\prime}}\!(T\e) =\hrn(T)$. The proposition
now follows from \eqref{bc} and \eqref{adj}.
\end{proof}

\begin{proposition}\label{b-c}
Let $R^{\e\prime}$ be a finite extension of $R$ of ramification index $e$ with associated residue field extension $k^{\e\prime}\be/k\subseteq \kbar/k$. Then, for every $S_{n}$-scheme $Z$, there exists a canonical closed immersion of $k^{\e\prime}$-schemes
\[
\mathrm{Gr}_{\lbe n}^{R}\be(Z\le)\times_{k}\spec
k^{\e\prime}\hookrightarrow \mathrm{Gr}_{\lbe ne}^{R^{\le\prime}}\be\big(Z\!\times_{S_{n}}\!S^{\e\prime}_{\le ne}\big)
\]
which is an isomorphism if $e=1$.
\end{proposition}
\begin{proof} The indicated map is an isomorphism if $e=1$ by Proposition \ref{unr3}. 
If $Z$ is of finite type over $S_{n}$, the proposition was established in \cite[Theorem 3.1]{ns}. The method used in [loc.cit.] easily extends to arbitrary $S_{n}$-schemes $Z$ provided the finite-dimensional affine space $\A^{\be N}_{R_{\le n}}$ considered in \cite[proof of Lemma 3.5, p.~1598]{ns} is replaced by the affine space $\A_{R_{\le n}}^{(I\e)}$ introduced in the proof of Proposition-Definition \ref{pr-def1}.
\end{proof}

\section{The kernel of the change of level morphism}\label{gp-sch}

Let $R$ be a complete discrete valuation ring with perfect residue field in the unequal characteristics case. In this Section we describe the kernel of the change of level morphism \eqref{clm0} when $Z$ is a smooth group scheme over $\mathfrak R=R_{\le n+j }$.
Recall that $\Rnr$ denotes the extension of $R$ of ramification index 1 which corresponds to $\kbar/k$ and $\Snnr=\spec \Rnnr$ for $n\geq 1$.

\begin{lemma}\label{coh=0} Let $n\geq 1$ be an integer and let $G$ be a smooth $\Snnr$-group scheme. Then $\HH^{1}_{\fppf}(\Rnnr,G\e)$ is a one-point set.
\end{lemma}
\begin{proof} Since $\Rnnr$ is a henselian local ring with residue field $\kbar$ and $G$ is smooth over $\Rnnr$, \cite[Theorem 11.7(2), p.~181, and Remark 11.8.3, p.~182]{dix} show that there exists a canonical bijection of pointed sets
\[
\HH^{1}_{\fppf}(\Rnnr,G)=\HH^{1}_{\et}\big(\e\kbar,G\be
\times_{S_{n}^{\nr}}\be\spec\kbar\,\big).
\]
Thus, since $\HH^{1}_{\et} \big(\e\kbar,G\be\times_{S_{n}^{\nr}}\be\spec\kbar\,\big)$ is clearly a one-point set, the lemma follows.
\end{proof}

\begin{proposition}\label{ex-green}
Let $n\geq 1$ be an integer and let
\begin{equation}\label{eta}
1\to F\to G\stackrel{\!q}{\to} H \to 1
\end{equation}
be a sequence of $R_{\le n}$-group schemes locally of finite type. Assume that $q$ is {\rm smooth} and that the above sequence is exact for the fpqc topology on $(\mr{Sch}/R_{\le n})$. Then the induced sequence of $k$-group schemes locally of finite type
\begin{equation}\label{greta}
1\to \grn(F)\to\grn(G)\to\grn(H)\to 1
\end{equation}
is exact for both the fppf and fpqc topologies on $(\mr{Sch}/k)$.
\end{proposition}
\begin{proof} By \cite[Lemma 2.2]{bga}, $q$ is surjective and the map $F\to G$ in \eqref{eta} identifies $F$ with the kernel of $q$. Further, since a smooth morphism is flat and locally of finite presentation, $q$ is an fppf morphism. Thus \cite[Lemma 2.3]{bga} shows that \eqref{eta} is also exact for the $\fppf$ topology on $(\mr{Sch}/R_{\le n})$. Now the sequence \eqref{greta} is left-exact since $\grn$ has a left-adjoint functor. On the other hand, by Corollary \ref{gr-sm}, $\grn(q)\colon\grn(G)\to\grn(H)$ is smooth. Thus, by \cite[Corollary 2.5]{bga} and Lemma \ref{rat-pts}(i), it suffices to check that $\grn(q)\big(\e\kbar\,\big)=q(\Rnnr)$ is surjective. Note that, since $q\times_{S_{n}}\!\Snnr$ is an fppf morphism,  the base extension of \eqref{eta} along $\Snnr\to S_{n}$ is exact for the fppf topology on $\big(\mr{Sch}/\Rnnr\big)$ \cite[Lemma 2.3]{bga}. Thus, by \cite[Proposition III.3.3.1(i), p.~162]{gi}, we have reduced the proof to checking that $\HH^{1}_{\fppf}(\Rnnr,F\!\times_{\be S_{n}}\!\be\Snnr)$ is a one-point set. Since $F\be\times_{S_{n}}\be\Snnr=G\be\times_{H}\be\Snnr$ is smooth over $\Snnr$, the latter follows from the previous lemma.
\end{proof}

The change of level morphism \eqref{clm0} for group schemes has been  discussed before (in a particular case) in \cite[pp.~37-40]{beg}. We now discuss this morphism in the more general setting of this paper.

Let $Z$ be an $S$-scheme. For every integer $n\geq 1$, set
\begin{equation*} 
\grn(Z)=\grn(Z\!\times_{S}\!S_{n}).
\end{equation*}
Further, if $G$ is a flat $S$-group scheme locally of finite type, we define
\begin{equation}\label{grpi}
\grn\lbe(\pi_{0}(G\le))=\grn(\pi_{0}(G\!\times_{S}\!S_{n})),
\end{equation}
where $\pi_{0}(G\!\times_{S}\!S_{n})$ is the \'etale $S_{\le n}$-group scheme \eqref{pio}. Note that, for every $S$-scheme $Z$, $\mathrm{Gr}_{0}^{R}(Z)=Z\!\times_{\be S}\be\spec k=Z_{\mr s}$ by Remark \ref{rems1}(a). For every pair of integers $r\geq 1, i\geq 0$, we have $(Z\!\times_{\be S}\be S_{r+i})\!\times_{\be S_{r+i}}\be S_{r}=Z\!\times_{\be S}\! S_{r}$ and \eqref{clm0} is a morphism 
\[
\varrho_{r,\le Z}^{\e i}\colon \mathrm{Gr}_{r+i}^{R}(Z)\to\mathrm{Gr}_{r}^{R}(Z).
\]
Now, if $G$ is an $R$-group scheme locally of finite type then, by Proposition \ref{gr-prop}(iv), 
\begin{equation}\label{nm-hom}
\varrho_{r,\le G}^{\e i}\colon \mathrm{Gr}_{r+i}^{R}(G)\to\mathrm{Gr}_{r}^{R}(G)
\end{equation}
is a morphism of $k$-group schemes locally of finite type. Further, by Remark \ref{vrk}(c),
\begin{equation}\label{comp}
\varrho_{r,\e G}^{\e i+1}=\varrho_{r,\e G}^{\le 1}\circ \varrho_{r+1,\e G}^{\e i}.
\end{equation}

We will now describe the kernel of \eqref{nm-hom}.
To this end, recall $\omega_{\le G/R}^{1}=\varepsilon^{\le*}\e\Omega_{G/R}^{1}$, where $\varepsilon\colon \spec R\to G$ is the unit section of $G$. We begin with the case where $i=1$ and $G$ is smooth over $R$. 

\begin{proposition}\label{uc} Let $R$ be a discrete valuation ring and let $G$ be a smooth $R$-group scheme. Then there exist canonical isomorphisms of smooth, connected and unipotent $k$-group schemes
\[
\krn\varrho_{r,\e G}^{1}=
\begin{cases}
\displaystyle \mathbb V\be \big(\omega_{G_{\mr s}/k}^{1} \big) &  \text{\hskip .7cm if  ${\rm char}\e R={\rm char} \e k$}\\
\mathbb V\be \big(\omega_{G_{\mr s}/k}^{1} \big) &  \text{\hskip .7cm if  ${\rm char} \e R\neq p= {\rm char} \e k$ and $r<\ari=v(p)$}\\
\mathbb V\be\big(\omega_{\le G_{\mr s}/k}^{1}\big)^{\!(\e p^{\lle m-1}\lle)} &  \text{\hskip .7cm if ${\rm char}\e R\neq p= {\rm char} \e k$ and $r\geq \ari=v(p)$}
\end{cases} 
\]
 where $m=\lceil (r+1)/\ari\e \rceil$ if ${\rm char}\e R\neq p= {\rm char} \e k$ and $r\geq \ari=v(p)$.
\end{proposition}
\begin{proof} Assume first that ${\rm char}\e R=0$ and $r\geq\ari=v(p)$ and let $m$ be as in the statement. By \cite[II, 1.7.11(iv)]{ega} and Example \ref{k2}, we have  
\[
\krn\varrho_{r,\e G}^{1}=\mathbb V\be\big(\big(F_{G_{\rm s}}^{\e p^{\le m-1}}\big)^{\be*}\e\Omega_{\le G_{\mr s}/k}^{1}\big)\times_{G_{\rm s}} \spec k\simeq
\mathbb V\be\big(\varepsilon_{\rm s}^{*}\big(F_{G_{\rm s}}^{\e p^{\le m-1}}\big)^{\be*}\e\Omega_{\le G_{\mr s}/k}^{1}\big),
\]
where $\varepsilon_{\rm s}\colon \spec k\to G_{\rm s}$ is the unit section of $G_{\rm s}$. Now 
\[
\varepsilon_{\rm s}^{*}\big(F_{X_{\rm s}}^{\e p^{\le m-1}}\big)^{\be*}\e\Omega_{\le G_{\mr s}/k}^{1}\simeq 
\big(F_{k}^{\e p^{\le m-1}}\big)^{\be *}\varepsilon_{\rm s}^{*}\e\Omega_{\le G_{\mr s}/k}^{1}=
\big(F_{k}^{\e p^{\le m-1}}\big)^{\be *}\e\omega_{\le G_{\mr s}/k}^{1}.
\]
Consequently 
\[
\krn\varrho_{r,\e G}^{1}\simeq  
\mathbb V\be\big( 
\big(F_{k}^{\e p^{\le m-1}}\big)^*\omega_{\le G_{\mr s}/k}^{1}\big)
\simeq \mathbb V\be\big( 
\omega_{\le G_{\mr s}/k}^{1} \big)\times_{f,\e\spec k,\e F_{k}^{\e p^{\le m-1}}} \spec(k)\overset{{\rm def.}}{=}	\mathbb V\be\big(\omega_{\le G_{\mr s}/k}^{1}\big)^{\!(\e p^{\le m-1}\lle)} ,
\]
where $f$ is the structure morphism of $\mathbb V\be\big( 
\omega_{\le G_{\mr s}/k}^{1} \big)$ and the second isomorphism again follows from \cite[II, 1.7.11(iv)]{ega}. The proof in the remaining cases is rather immediate. In effect,
\[
\krn\varrho_{r,\e G}^{1}=\mathbb V\be\big( \Omega_{\le G_{\mr s}/k}^{1}\big)\times_{X_{\rm s}} \spec k\simeq
\mathbb V\be\big(\varepsilon_{\rm s}^{\le *}\Omega_{\le G_{\mr s}/k}^{1}\big)=\mathbb V\be\big(\omega_{\le G_{\mr s}/k}^{1}\big).
\]
\end{proof}

\begin{remark} The proposition should be compared with \cite[A.6.1]{cgp}, where the case $r=\ari=1$ is discussed for $k$ algebraically closed. Note that, in [loc.cit.], $\mathbb V(\omega_{\le G_{\mr s}/k}^{1})$ has been identified with the functor $\underline{\rm{Lie}}\e(G_{\mr s}/k)$.
\end{remark}

\begin{proposition}\label{sm-vker} Let $G$ be a smooth $R$-group scheme and let $r, \e i$ be positive integers. Then $\varrho_{r,\le G}^{\e i}\colon \mr{Gr}_{r+i}^{R}(G)\to\mr{Gr}_{r}^{R}(G)$ \eqref{nm-hom} is a smooth and surjective morphism of $k$-group schemes and $\krn \varrho_{r,\le G}^{\e i}$ is smooth, connected and unipotent.
\end{proposition} 
\begin{proof} By Proposition \ref{sm-surj} and Corollaries \ref{rqc} and \ref{gr-sm2},  $\varrho_{r,\le G}^{\e i}$ is a quasi-compact and surjective morphism of smooth $k$-group schemes.  Thus, by Lemma \ref{conn}, the sequence
\begin{equation}\label{rho-sm}
1\lra\krn\varrho_{r,\le G}^{\e i}\lra\mr{Gr}_{r+i}^{R}(G)\overset{\varrho_{r,\le G}^{\e i}}{\lra}\mr{Gr}_{r}^{R}(G)\lra 1
\end{equation}
is exact for both the fppf and fpqc topologies on $(\mr{Sch}/k)$. Further, by Lemma \ref{flat1}, $\varrho_{r,\le G}^{\e i}$ is faithfully flat. 
Now Proposition  \ref{uc} shows that $\varrho_{r,\le G}^{\le 1}$ is smooth, and the smoothness of $\varrho_{r,\le G}^{\e i}$ for arbitrary $i$ follows by induction from \eqref{comp}. It remains to check (by induction) that $U_{\lbe r}^{\le i}=\krn\varrho_{\le r,\le G}^{\e i}$ is connected and unipotent. By Proposition \ref{uc}, the induction hypothesis holds if $(i,r)=(1,r)$ and $r$ is any positive integer. Now, by \eqref{comp}, the faithful flatness of $\varrho_{r,\le G}^{\e i}$ and Lemma \ref{kerseq}(ii) (and its proof), there exists a sequence of $k$-group schemes locally of finite type
\begin{equation}\label{3u}
\xymatrix{
1\ar[r]& U_{r+1}^{\le i}\ar[r]& U_{r}^{\e i+1}\ar[r]^u&   U_{r}^{\le 1}\ar[r]&1, }
\end{equation}
where $u=\varrho_{r+1,\le G}^{\e i}\times_{\mr{Gr}_{r}^{R}(G)}\spec k$, which is exact for the fppf topology on $(\mr{Sch}/k)$. The conclusion then follows from Lemma \ref{conn} and \cite[XVII, Proposition 2.2(iii)]{sga3}.
\end{proof}

\begin{remark} Note that $\krn \varrho_{r,\le G}^{\e i}$ is a $k$-group scheme of finite type since every connected $k$-group scheme locally of finite type is, in fact, of finite type. See \cite[${\rm VI_{A}}$, Proposition 2.4(ii)]{sga3}.
\end{remark}

Let $A$ be any $k$-algebra and assume that $1\leq i\leq r$. Further, set $B=\s R_{\le r+i}\le(A)$ and $J=\overbar{\s M_{r+i}^{\lle r}}(\be A)$. By \eqref{rnms}, $\s R_{\e r}(A)$ is isomorphic to $B\lbe/\lbe J$ as a $B$-algebra. Thus, by \eqref{bc} and Lemma \ref{rat-pts}(i), $\varrho_{r,\le G}^{\e i}(A)$ can be identified with the canonical map $G_{\be B}(B)\to G_{\be B}(B\lbe/\be J\e)$. Now, since $2r\geq r+i$, we have $J^{\le 2}=0$ by \eqref{ned2} and \cite[II, 3.2, p.~206, and Theorem 3.5, p.~208]{dg} show that the homomorphism $\Hom_{B\text{-mod}}\le(w_{G_{\lbe\lbe B}\lbe/\lbe B}^{1}, J\e)\to G_B(B)$ is functorial in $G$ and maps $\Hom_{B\text{-mod}}\le(w_{G_{\lbe\lbe B}\lbe/\lbe B}^{1}, J\e)$ bijectively onto the kernel of $\varrho_{r,\le G}^{\e i}(A)$. Thus there exists a  canonical isomorphism of groups
\begin{equation}\label{wj-0}
\Hom_{B\text{-mod}}\le(w_{G_{\lbe\lbe B}\lbe/\lbe B}^{1}, J\e)\overset{\!\sim}{\to}\krn\varrho_{r,\e G}^{\le i}(A),
\end{equation} 
where $w_{G_{\lbe\lbe B}\lbe/\lbe B}^{1}=\varGamma\big(\spec B,\omega_{G_{\lbe\lbe B}\lbe/\lbe B}^{1}\le\big)$.

Consider the $B$-algebra $C=\s R_{\le i}(A)$. By \eqref{v2}, we may make the identifications 
\[
\mr{Gr}_{i}^{R}(\e\mathbb
V(\omega_{G/R}^{1}))(A)=V(\omega_{G_{\lbe C}\lbe/C}^{1})(C\le)=\Hom_{\e C\text{-mod}}(w_{G_{\lbe C}\lbe/C}^{1}, C\e)=\Hom_{B\text{-mod}}(w_{G_{\be B}\lbe/\lbe B}^{1}, C\e).
\]
Now recall the homomorphism of $B$-modules $\varphi_{r+i,\le r}(A)\colon C\to J$ \eqref{pis}. Under the above identifications, $\Hom_{B\text{-mod}}(w_{G_{\be B}\lbe/\lbe B}^{1}, \varphi_{r+i,\le r}(A))$ can be identified with a map
\begin{equation}\label{wj}
\mr{Gr}_{i}^{R}(\e\mathbb V\lbe(\omega_{G/R}^{1}))(A)\to \Hom_{B\text{-mod}}\le(w_{G_{\lbe\lbe B}\lbe/\lbe B}^{1}, J\e).
\end{equation}
Composing the preceding map with the isomorphism \eqref{wj-0} and letting $A$ vary, we obtain a canonical morphism of $k$-group schemes
\begin{equation}\label{wj-1}
\Phi_{\be r,\e G}^{\le i}\colon \mr{Gr}_{i}^{R}\lbe(\le\mathbb V\lbe(\omega_{\le G/R}^{1})) \to \krn\varrho_{r,\e G}^{\le i}. 
\end{equation}
Now, if $G$ is {\it smooth} over $R$, then Remark \ref{rems1}(e) yields a (non-canonical) isomorphism of $k$-schemes
\begin{equation}\label{obv0}
\mr{Gr}_{i}^{R}\lbe(\le\mathbb V\lbe(\omega_{\le G/R}^{1}))\overset{\!\sim}{\to}
\A_{\le k}^{\!i\le d},
\end{equation}
where $d=\dim G_{\lbe \rm s}$. Further, if $R_{\le i}$ is a finite $k$-algebra (i.e, $i\leq \ari=v(p)$ when $R$ has unequal characteristics $(0,p)$), then the indicated remark also yields a (non-canonical) isomorphism of $k$-group schemes
\begin{equation}\label{obv}
\mr{Gr}_{i}^{R}\lbe(\le\mathbb V\lbe(\omega_{\le G/R}^{1}))\overset{\!\sim}{\to}\G_{\lbe a,\le k}^{i\le d} \quad \text{(if either $i\leq \ari=v(p)$ or ${\rm char}\e R={\rm char}\e k$)}.
\end{equation}

\begin{proposition}\label{vker1} Assume that $R$ is an equal characteristic ring and let $G$ be an $R$-group scheme locally of finite type. Then, for every pair of integers $r$ and $i$ such that $1\leq i\leq r$, the canonical map $\Phi_{\be r,\e G}^{\le\lle  i}\colon \mr{Gr}_{i}^{R}\lbe\lbe(\le\mathbb V\lbe(\omega_{\le G/R}^{1})) \to \krn\varrho_{r,\e G}^{\le\lle  i}$ \eqref{wj-1} is an isomorphism of $k$-group schemes. Further, if $G$ is smooth over $R$, then $\krn\varrho_{r,\e G}^{\le\lle  i}$ is (non-canonically) isomorphic to $\G_{\lbe a,\e k}^{i\le d}$, where $d=\dim G_{\lbe \rm s}$.
\end{proposition}
\begin{proof}  If $A$ is any $k$-algebra, $\varphi_{r+i,\le r}(A)$ is an isomorphism by Proposition \ref{rnm-2}. Consequently, the map \eqref{wj} is an isomorphism as well. Since $\Phi_{\be r,\e G}^{\le i}(A)$ is the composition of \eqref{wj} and the isomorphism \eqref{wj-0} and $A$ is arbitrary, the first assertion of the proposition follows. Now, if $G$ is smooth over $R$, then  the composition of the inverse of the isomorphism $\Phi_{\be r,\e G}^{\le i}$ and \eqref{obv} is an isomorphism as well.
\end{proof}

\begin{proposition}\label{vker} Let $R$ be a ring of unequal characteristics $(0,p)$ and let $G$ be a smooth $R$-group scheme. Then, for every pair of integers $r$ and $i$ such that $1\leq i\leq r$, the map $\Phi_{\be r,\e G}^{\le\lle i}\colon \mr{Gr}_{i}^{R}\lbe\lbe(\le\mathbb V\lbe(\omega_{\le G/R}^{1})) \to \krn\varrho_{r,\e G}^{\le\lle i}$ \eqref{wj-1} is an isogeny  of smooth, connected and unipotent $k$-group schemes. Its kernel is an infinitesimal $k$-group scheme which is trivial if $r+i\leq \ari=v(p)$.	
Further, if $i\leq\ari$, then $\krn\varrho_{r,\e G}^{\le\lle  i}$ is (non-canonically) isomorphic to $\G_{\lbe a,\le k}^{i\le d}$, where $d=\dim G_{\lbe \rm s}$.
\end{proposition}
\begin{proof} By \eqref{obv0} and Proposition \ref{sm-vker}, $\mr{Gr}_{i}^{R}\lbe\lbe(\le\mathbb V\lbe(\omega_{\le G/R}^{1}))$ and $\krn\varrho_{r,\e G}^{\le\lle i}$ are smooth, connected and unipotent $k$-group schemes. On the other hand, by Proposition \ref{rnm-2}, $\varphi_{r+i,\le r}(A)$ is an isomorphism of abelian groups if $r+i\leq \ari$ and $A$ is any $k$-algebra or if $r+i>\ari$ and $A$ is perfect. Thus $\Phi_{\be r,\e G}^{\le\lle i}$ is an isomorphism if $r+i\leq \ari$. When $r+i>\ari$, the maps \eqref{wj} and $\Phi_{\be r,\e G}^{\le i}(A)$ \eqref{wj-1} are isomorphisms of abelian groups for every perfect $k$-algebra $A$. Consequently $(\le\krn\e \Phi_{\be r,\e G}^{\le\lle i})\!\left(\e\lle\kbar\e\lle\right)=\krn\!\be\left(\le\Phi_{\be r,\e G}^{\le\lle i}(\e\lle\kbar\e\lle)\right)=\{1\}$ and $\Phi_{\be r,\e G}^{\le\lle i}(\e\lle\kbar\e\lle)$ is surjective. Thus $\krn\e \Phi_{\be r,\e G}^{\le\lle i}$ is an infinitesimal $k$-group scheme by Lemma \ref{red=1} and Remark \ref{inf}(b) and, furthermore, $\Phi_{\be r,\e G}^{\le\lle i}$ is faithfully flat by \cite[I, \S3, Corollary 6.10, p.~96]{dg} and Lemma \ref{flat2}(ii).
The last assertion of the proposition follows from \eqref{obv} and \cite[[IV, \S 3, Corollary 6.8 p. 523]{dg}.
\end{proof}

\begin{corollary} \label{idim}Let $G$ be a smooth $R$-group scheme and let $i$ and $r$ be integers such that $1\leq i\leq r$. Then $\,\dim \krn\varrho_{r,\e G}^{\le\lle i}=i\dim G_{\mr s}$. 
\end{corollary}
\begin{proof} This is immediate from \eqref{obv0} and Propositions \ref{vker1} and \ref{vker} using \cite[${\rm{VI}}_{\rm B}$, Proposition 1.2]{sga3} and the fact that infinitesimal $k$-group schemes have dimension 0.
\end{proof}

\begin{remark}\label{cex-beg}
As noted in the statement of Proposition \ref{vker}, the infinitesimal $k$-group scheme $\krn\e\Phi_{\be r,\e G}^{\le\lle i}$ can be nontrivial for appropriate choices of $R, r,i$ and (smooth) $R$-group scheme $G$ (in particular, \cite[Lemma 4.1.1(2), p.~37\le]{beg} is false. See also Remark \ref{begueri} below). Indeed, let $R=W\be(k)$ and $G=\G_{a,\le R}$. For every $k$-algebra $A$, $\Phi_{1,\e G}^{\le\lle 1}(A)$ may be identified with the map
\[
\Hom_{\e W_{2}(\lbe A)\text{-mod}}(W_{2}(\lbe A), A\e)\to  \Hom_{\e W_{2}(\lbe A)\text{-mod}}(W_{2}(\lbe A), V\lle W_{2}(\lbe A)\e)
\]
induced by $\varphi_{\le 2,1}(A)\colon A\to VW_2(A), a\mapsto (0,a^{\le p})$ (see Remark \ref{twist}(a)). It follows that, as a homomorphism of groups, $\Phi_{1,\e G}^{\le\lle 1}(A)$ can be identified with $\varphi_{\le 2,1}(A)$ itself. Thus $\krn\e\Phi_{1,\e G}^{\le\lle 1}(A)$ is isomorphic (functorially in $A$) to the subgroup of $A$ of $p\,$-nilpotent elements, whence $\krn\e\Phi_{1,\e G}^{\le\lle 1}$ is isomorphic to the (nontrivial) infinitesimal $k$-group scheme $\alpha_{p}$.
\end{remark}

\begin{corollary}\label{id-comp} Let $n\geq 1$ be an integer and let $G$ be a smooth $R$-group scheme. Then
\begin{enumerate}
\item[(i)] $\dim \grn(G)=n\dim G_{\lbe \rm s}$. 
\item[(ii)] $\grn(G)$ is connected if, and only if, $G_{\lbe \rm s}$
is connected.
\item[(iii)] $\grn(G^{\e 0})=\grn(G)^{\le 0}$.
\item[(iv)] $\grn(\pi_{0}(G\le))=\pi_{0}(\grn(G))$.
\end{enumerate}
\end{corollary}
\begin{proof} By the exactness of \eqref{rho-sm}, \cite[${\rm{VI}}_{\rm B}$, Proposition 1.2]{sga3} and Corollary \ref{idim} for $i=1$, we have $\dim \mr{Gr}_{n+1}^{R}(G)=\dim\mr{Gr}_{n}^{R}(G)+\dim G_{\lbe \rm s}$. Assertion (i) now follows by induction.
By Proposition \ref{sm-vker}, $\varrho_{\e 1,G}^{\le n-1}\colon \grn(G)\to G_{\mr s}$ is a surjective morphism of smooth $k$-group schemes with connected kernel for every integer $n\geq 2$. Thus, by Lemma \ref{conn}, the connectedness of $G_{\mr s}$ implies that of $\grn(G)$. Conversely, if $\grn(G)$ is connected, then $\varrho_{\e 1,\e G}^{\le n-1}$ maps $\grn(G)=\grn(G)^{\le 0}$ into $G_{\mr s}^{\le 0}$, which implies that
$G_{\mr s}^{\e 0}=G_{\mr s}$. Assertion (ii) is now proved. Since $G^{\le 0}$ is an open subgroup scheme of $G$, 
$\grn(G^{\le 0})$ is open in $\grn(G)$ by Remark \ref{rems1}(b). Further, since $G^{\e 0}_{\mr s}$ is connected, $\grn(G^{\e 0})$ is connected by (ii). Thus 
$\grn(G^{\e 0})=\grn(G)^{\e 0}$ by \cite[$\text{VI}_{\mr{B}}$, Lemma 3.10.1]{sga3}, i.e., (iii) holds. To prove (iv), we apply Proposition \ref{ex-green} to the smooth morphism $G\be\times_{\be S}\be S_{n}\to\pi_{0}(G\be\times_{\be S}\be S_{n}\lbe)$. By definition of $\grn(\pi_{0}(G\le))$ \eqref{grpi}, we have an exact and commutative diagram of sheaves of groups on $(\mr{Sch}/k)_{\fppf}$:
\[
\xymatrix{
1\ar[r]&\grn(G^{\e 0}\le)\ar[r]\ar@{=}[d]_{\mr{(iii)}}& \grn(G\le)\,\,\ar[r]\ar@{=}[d]&\,\,\grn(\pi_{0}(G))\ar[r]\ar[d]&1\\
1\ar[r]&\grn(G)^{\le 0}\ar[r]& \grn(G\le)\,\ar[r]&\,\pi_{0}(\grn(G\le))\ar[r] &1.
}
\]
Assertion (iv) is now clear.
\end{proof}

\begin{remarks}\indent
\begin{enumerate}
\item[(a)] If $n$ is a positive integer and $G$ is an $R_{\le n}$-group scheme which is not necessarily of the form $H\be\times_{\be S}\be S_{n}$ for some group scheme $H$ over $R$, then statements analogous to those of Propositions/Corollary \ref{vker1}, \ref{vker}, \ref{sm-vker} and \ref{id-comp} are valid for $G$, provided the integers $r,i$ appearing in the first two of these statements satisfy the condition $r+i\leq n$. The proofs are essentially the same.
\item[(b)]  Let $n\geq 1$ be an integer, let $G$ be a smooth and commutative $R_{\le n}$-group scheme and set  $H=\grn(G)$. Then $F^{\e i}\lbe H=\krn\varrho_{\le i,\le G}^{\e n-i}$, where $1\leq i\leq n$, defines a filtration of $H$ of length $n$:
\begin{equation}\label{fil}
H\supseteq F^{\e 1}\lbe H\supseteq\cdots \supseteq F^{\e n}\lbe H=0.
\end{equation}
Note that $H/F^{\e 1}\be H=G_{\mr s}$. Further, in the notation of the proof of Proposition \ref{sm-vker} (see remark (a) above), $F^{\e i}H=U_{i}^{\e n-i}$. Thus, by the indicated proposition, the exactness of \eqref{3u} and  Corollary \ref{idim}, $F^{\e i}\lbe H/F^{\e i+1}\lbe H\simeq U_{\le i}^{1}$ is a smooth, connected and unipotent $k$-group scheme of dimension $\dim \lbe G_{\mr s}$ for $1\leq i\leq n-1$.
\item[(c)] A particular case of the filtration \eqref{fil} appeared in \cite[\S5.1]{ed}. In effect, let $D$ be a discrete valuation ring with maximal ideal $\mm$, residue field $k$ and fraction field $K$ and let $K^{\le\prime}\be/K$ be a separable field extension of degree $n$. Let $D^{\e\prime}$ be the integral closure of $D$ in $K^{\prime}$ and assume that $D^{\e\prime}$ is a discrete valuation ring with maximal ideal $\mm^{\prime}$ such that $(\mm^{\prime})^{n}=\mm$. Assume, furthermore, that $D$ contains the $n$-th roots of unity. Now let $A^{\prime}$ be an abelian variety over $K^{\prime}$ and let $\mathcal A^{\prime}$ denote its N\'eron model over $D^{\e\prime}$. By \eqref{wrbc}, $\Re_{D^{\e\prime}\be/D}(\mathcal A^{\prime})_{\mr s}=\Re_{B/k}(\mathcal A^{\prime}_{\le B})$, where $B=D^{\e\prime}\!\otimes_{D}\lbe k$. It is shown in \cite[p.~297, line $-3$]{ed} that $B\simeq R_{\le n}$, where $R=k[[t]]$. Thus
$\Re_{D^{\e\prime}\be/D}(\mathcal A^{\prime})_{\mr s}=\grn(G)
$, where $\grn=\Re_{R_{n}/k}$ and $G=\mathcal A^{\prime}_{B}$. Consequently, the filtration considered in \cite[\S5.1]{ed} is a particular case of the equal characteristic case of the filtration \eqref{fil}.
\end{enumerate}
\end{remarks}

\section{The perfect Greenberg functor}\label{pfgf}

Let $R$ be a discrete valuation ring with perfect residue field $k$ of positive characteristic $p$. We will write $(\mathrm{Perf}/\le k)$ for the category of perfect $k$-schemes. Recall that a $k$-scheme $Y$ is said to be {\it perfect} if the absolute Frobenius endomorphism $F_{\le Y}$ of $Y$ is an isomorphism. The inclusion functor $(\mathrm{Perf}/\le k)\to (\mathrm{Sch}/\le k)$ has a right-adjoint functor
\begin{equation}\label{rapf}
(\mr{Sch}/\lle k)\to(\mr{Perf}/\lle k), Y\mapsto Y^{\lle\pf},
\end{equation}
where $Y^{\lle \pf}$ is the (inverse) perfection of the given $k$-scheme $Y$. The perfect $k$-scheme $Y^{\lle \pf}$ is equipped with a morphism of $k$-schemes $\phi_{\le\le Y}\colon Y^{\lle\pf}\to Y$ such that, for every perfect $k$-scheme $Z$, there exists a canonical bijection 
\begin{equation}\label{pf}
\Hom_{\le\mr{Sch}/\lle k}(Z,Y)\overset{\!\sim}{\to}\Hom_{\e
\mr{Perf}/\lle k}\be\big(Z,Y^{\lle\pf}\le\big), \, \psi\mapsto\psi^{\le\pf},
\end{equation}
where $\psi^{\le\pf}\circ \phi_{\le\le Y}=\psi$. See \cite[\S 5]{bga} for more details.
	
If  $n\geq 1$ is an integer, the composition of the perfection functor \eqref{rapf} and the Greenberg functor of level $n$ \eqref{grf2} is a functor
\begin{equation}\label{pf-grf}
\pgrn\colon (\e\mathrm{Sch}/R_{\le n})\to
(\mathrm{Perf}/k), \quad Z\mapsto \grn(Z)^{\pf},
\end{equation}
which is called the {\it perfect Greenberg functor of level $n$\,} (associated to $R\e$). If $Z$ is an 
$R_{\le n}$-scheme, the perfect $k$-scheme $\pgrn(Z)$ is called the {\it perfect Greenberg realization} of $Z$.

\begin{proposition} 
Let $n\geq 1 $ be an integer and let $(Z_{\lambda})_{\lambda\le\in\le\Lambda}$ be a projective system of $R_{n}$-schemes with affine transition morphisms, where $\Lambda$ is a directed set. Then $(\pgrn(Z_{\lambda}))$ is a projective system of perfect $k$-schemes with affine transition morphisms and
\[
\pgrn\big(\varprojlim Z_{\lambda}\big)=\varprojlim\pgrn(Z_{\lambda})
\]
in the category of perfect $k$-schemes.
\end{proposition}
\begin{proof} This follows from \cite[Proposition 5.21]{bga} together with Propositions \ref{gr-projlim} and \ref{aff}.
\end{proof}

Let $k^{\e\prime}\be/k$ be a finite field extension and let $X^{\prime}$ be a perfect $k^{\e\prime}$-scheme. We will say that $\Re_{\le k^{\prime}\!/k}^{\e\pf}(X^{\prime})$ exists if the contravariant functor
\begin{equation*} 
(\mr{Perf}/k)\to(\mathrm{Sets}),\quad T\mapsto\Hom_{\e\mr{Perf}/k^{\le\prime}}\big(T\times_{k}\spec k^{\le\prime},X^{\prime}\e\big),
\end{equation*}
is represented by a perfect $k$-scheme $\Re_{\le k^{\le\prime}\!/k}^{\e\pf}(X^{\prime})$. 

\begin{proposition}\label{p-wrgr} Let $R^{\e\prime}$ be a finite extension of $R$ of ramification index $e$ with associated residue field extension $k^{\e\prime}/k\subseteq \kbar/k$. Let $n\geq 1$ be an integer and let $Z$ be an $S_{ne}^{\e\prime}$-scheme.
\begin{enumerate}
\item[(i)] If $Z$ is admissible relative to $S_{ne}^{\e\prime}\to S_{n}$ (see Definition {\rm \ref{adm}}), then both $\Re_{\le k^{\le\prime}\be/k}^{\le\pf}\big(\mathbf{Gr}_{\lbe ne}^{R^{\le\prime}}(Z\le)\big)$
and $\Re_{\le S_{ne}^{\le\prime}/S_{\le n}}(Z\le)$
exist and
\[
\Re_{\le k^{\e\prime}\be/k}^{\le\pf}\big(\mathbf{Gr}_{\lbe ne}^{R^{\le\prime}}(Z\le)\big)=\mathbf{Gr}_{\lbe n}^{R}\be\big(
\Re_{S_{ne}^{\e\prime}/S_{n}}\be(Z\le)\big).
\]
\item[(ii)] If $R^{\e\prime}/R$ is totally ramified and $Z$ is arbitrary, then
$\Re_{\le S_{ne}^{\le\prime}/S_{n}}(Z\le)$ exists and
\[
\mathbf{Gr}_{\lbe ne}^{R^{\le\prime}}(Z\le)=\mathbf{Gr}_{\lbe n}^{R}\lbe\lbe\big(
\Re_{\le S_{ne}^{\lle\prime}/S_{n}}\be(Z\le)\big).
\]
\end{enumerate}
\end{proposition}
\begin{proof} Assertion (ii) is immediate from Proposition \ref{tot-gr}. In (i), $\Re_{S_{ne}^{\le\prime}/S_{\le n}}(Z)$ exists by Theorem \ref{wr-rep}. Now, since $\Re_{\le k^{\prime}\be/k}\big(\mr{Gr}_{\lbe ne}^{R^{\le\prime}}(Z)\big)$ exists by Theorem \ref{wr-gr}, the perfect Weil restriction $\Re_{\le k^{\prime}\be/k}^{\le\pf}\big(\mathbf{Gr}_{\lbe ne}^{R^{\le\prime}}(Z)\big)$ exists as well and it equals $\Re_{\le k^{\prime}\be/k}\big(\mr{Gr}_{\lbe ne}^{R^{\le\prime}}(Z)\big)^{\pf}$ by \cite[Lemma 5.24]{bga}. Thus, since $\Re_{\le k^{\prime}\be/k}\big(\mr{Gr}_{\lbe ne}^{R^{\le\prime}}(Z)\big)^{\pf}=\mathbf{Gr}_{\lbe n}^{R}\be\big(
\Re_{S_{ne}^{\e\prime}/S_{n}}\be(Z)\big)$ by Theorem \ref{wr-gr}, the formula in (i) follows.
\end{proof}
		
\begin{proposition} 
Let $k^{\e\prime}/k$ be a subextension of $\e\kbar\lbe\lle/k$ and let
$R^{\e\prime}$ be the extension of $R$ of ramification index 1 which corresponds to $k^{\e\prime}\be/k$. Then, for every integer $n\geq 1$ and every $S_{n}$-scheme $Z$,
there exists a canonical isomorphism of perfect $k^{\e\prime}$-schemes
\[
\pgrn(Z\le)\times_{\spec k}\spec
k^{\e\prime}=\mathbf{Gr}_{\lbe n}^{R^{\e\prime}}\!\big(Z\be\times_{S_{n}}\!
{S}^{\e\prime}_{n}\e\big).
\]
\end{proposition}
\begin{proof} This is immediate from Proposition \ref{unr3} using \cite[Remark 5.18(d)]{bga}.
\end{proof}
		
\begin{proposition} \label{pf-bc}
Let $R^{\e\prime}$ be a finite extension of $R$ of ramification index $e$ with associated residue field extension $k^{\e\prime}/k\subseteq\kbar\lbe/k$. For every integer
$n\geq 1$ and every $S_{n}$-scheme $Z$, there exists a canonical
closed immersion of perfect $k^{\e\prime}$-schemes
\[\mathbf{Gr}_{\lbe n}^{R}(Z)\times_{\spec k}\spec
k^{\e\prime}\hookrightarrow \mathbf{Gr}_{\lbe n\lbe
e}^{R^{\le\prime}}\be\big(Z\be\times_{\be S_{n}}\be S^{\e\prime}_{\le
ne}\big). \]
If $e=1$, the preceding map is an isomorphism.
\end{proposition}
\begin{proof} This follows by applying the perfection functor to the closed immersion of Proposition \ref{b-c} using \cite[Remark 5.18(d) and Proposition 5.17(iv)]{bga}.
\end{proof}

\begin{proposition} 
Let $n\geq 1$ be an integer and let $0\to F\overset{f}{\to} G\to H\to 0$ be a complex of commutative $R_{\le n}$-group schemes, where $G$ and $H$ are smooth. Assume that
\begin{enumerate}
\item[(i)] $f$ is quasi-compact,
\item[(ii)]  $\pi_{0}(G)\lbe(R_{\le n}^{\e\rm{nr}})$ is a finitely generated abelian group, and
\item[(iii)] the induced sequence of abelian groups
\[
0\to F \lbe(R_{\le n}^{\e\rm{nr}})\to G\lbe(R_{\le n}^{\e\rm{nr}})\to H\lbe(R_{\le n}^{\e\rm{nr}})\to 0
\]
is exact.
\end{enumerate}
Then the induced complex of perfect and commutative $k$-group schemes
\[
0\to \pgrn(F)\to \pgrn(G)\to  \pgrn(H)\to 0
\]
is exact for the the fpqc topology on $(\le\mathrm{Perf}/k)$.
\end{proposition}
\begin{proof} By (iii), Lemma \ref{rat-pts}(ii) and Corollary \ref{gr-sm2}, the sequence
\[
0\to \grn(F)\to \grn(G)\to  \grn(H)\to 0
\]
is a complex of commutative $k$-group schemes such that the sequence
\[
0\to \grn(F)\big(\e\kbar\e\big)\to \grn(G)\big(\e\kbar\e\big)\to  \grn(H)\big(\e\kbar\e\big)\to 0
\]
is exact. Thus the proposition will follow
from \cite[Proposition 6.3]{bga} once we check that the following additional conditions hold: (a) $\grn(f)\colon \grn(F)\to \grn(G)$ is quasi-compact, and (b) $\grn(G)\to\grn(H)$ is flat. Condition (a) follows at once from (i) and Proposition \ref{gr-prop}. On the other hand, by Corollary \ref{id-comp}(iv), Lemma \ref{rat-pts}(ii) and \eqref{rnr}, we have
\[
\pi_{0}(\grn(G))\big(\e\kbar\e\big)=\grn(\pi_{0}(G\le))\big(\e\kbar\e\big)=\pi_{0}(G)\lbe(R_{\le n}^{\e\rm{nr}}),
\]
which is finitely generated by (ii). Thus, since $\grn(G)\lbe(\e\kbar\e)\to\grn(H)\lbe(\e\kbar\e)$ is surjective, we conclude from Lemma \ref{flat3} that $\grn(G)\to\grn(H)$ is flat, i.e., (b) holds. The proof is now complete.
\end{proof}

\begin{remark}\label{begueri} Since the perfection of an infinitesimal $k$-group scheme is the trivial $k$-group scheme (see \cite[Lemma 5.20]{bga}), Propositions \ref{vker1} and \ref{vker} show that the perfection of the canonical morphism of $k$-group schemes $\Phi_{\be r,\e G}^{\le\lle i}\colon  \mr{Gr}_{i}^{R}\lbe\lbe(\le\mathbb V\lbe(\omega_{\le G/R}^{1}))\to \krn\varrho_{\le r,\e G}^{\le\lle i}$ \eqref{wj-1} is an isomorphism for every smooth $R$-group scheme $G$. It follows from the above that, despite the fact that infinitesimal $k$-group schemes are ignored in \cite{beg} (see Remark \ref{cex-beg}), the indicated oversight fortunately had no consequences for the validity of the main results of \cite{beg}.
\end{remark}

\section{Finite group schemes}\label{fgr}
Let $R$ be a complete discrete valuation ring with maximal ideal $\mm$ and residue field $k$ (assumed to be perfect in the unequal characteristics case). Let $K$ denote the fraction field of $R$ and recall $S=\spec R$.

In this Section $F$ is an arbitrary {\it finite} and {\it flat} $R$-group scheme. 
In particular $F$ is affine (hence separated) and of finite type over $R$. Consequently, the canonical map $F(R)\to F(K)$ is injective by \cite[(5.5.4.1), p.~288]{ega1}. Since $F(K)=F_{\lbe K}\lbe(K)$ is a finite group (see Subsection \ref{not}), $F(R\le)$ is a finite group as well.  Now, if $n\geq 1$ is an integer, then $F\!\times_{\be S}\!S_{n}$ is affine and of finite type over $R_{\le n}$. Thus Corollary \ref{aff-cor} shows that
\begin{equation}\label{def}
\grn(F\le)=\grn(F\!\times_{\be S}\!S_{n})
\end{equation}
is an affine $k$-group scheme of finite type. Note, however, that $\grn(F\le)$ may fail to be finite over $k$, as Example \ref{ex.alpha} showed.

Now recall that an arbitrary $R$-group scheme $F$ is called {\it generically smooth} (respectively, {\it generically \'etale}) if $F_{K}=F\!\times_{S}\lbe\spec K$ is smooth (respectively, \'etale) over $K$. Note that, if $R$ is a ring of unequal characteristics, so that ${\rm char}\e K=0$, then {\it every} $R$-group scheme locally of finite type is generically smooth \cite[$\text{VI}_{\text{B}}$, Corollary 1.6.1]{sga3}. Further, if $R$ is an {\it equal characteristic zero ring}, then every flat $R$-group scheme locally of finite type is, in fact, smooth  \cite[${\rm IV}_{\be 4}$, Proposition 17.8.2]{ega}. Now recall $\omega_{F\be/\lbe R}^{1}=\varepsilon^{*}\Omega_{F\be/\lbe R}^{1}$, where $\varepsilon\colon S\to F$ is the unit section of $F$. If $F$ is generically smooth and of finite type, we define the {\it defect of smoothness of $F$} by
\begin{equation}\label{df}
\delta(F\le)=\mr{length}_{\lbe R }(\omega_{\lbe F\be/\lbe R}^{1})_{\mr{tors}}.
\end{equation}

\begin{remarks}\label{rm.df}\indent
\begin{enumerate}
\item[(a)] The defect of smoothness of $F$ does not change under unramified extensions of $R$. Indeed, by the structure theorem for finitely generated modules over a principal ideal domain, there exists an isomorphism of $R$-modules $\omega_{F/R}^{1}\simeq  R^{\le f}\oplus (\bigoplus_{\le j=1}^{\le r}R/(\pi^{m_{j}}))$ for appropriate non-negative integers $f,r$ and $m_{j}$. Consequently, $\delta(F\le)=\sum_{\e j=1}^{\e r} \lbe m_{j}$ by \cite[Example 1.22 and Lemma 1.23, p.~258]{liu}.  Since $\omega_{F_{\lbe R^{\le\prime}}\be/R^{\le\prime}}^{1}\simeq \omega_{F/R}^{1} \otimes_{R} R^{\e\prime}\simeq (R^{\e\prime}\le)^{\le f}\oplus (\bigoplus_{\le j=1}^{\le r}R^{\e\prime}/(\pi^{m_{j}}))$ for every unramified extension $R^{\e\prime}\lbe/R$ of discrete valuation rings, our claim follows. We conclude that \eqref{df} coincides with the defect of smoothness of $F$ at the unit section $\varepsilon$, as defined in \cite[\S3.3, p.~65]{blr}. See also  \cite[\S3.6, p.~79]{blr}.

\item[(b)]  Recall the extension $\Rnr\be/ R$ of ramification index $1$ which corresponds to $\kbar\lbe/k$.  Let $a\colon \Snr\to F$ be an $\Rnr$-valued point of $F$ and write $a$ also for the section $\Snr\to F_{\be \Snr}$ which corresponds to $a$ via \eqref{bc}. In \cite[p.~65]{blr}, the defect of smoothness of $F$ at $a$ is defined to be $\mr{length}_{\lbe R  }(a^{*}\le\Omega_{\lbe F \be/S}^{1})_{\mr{tors}} =\mr{length}_{\lbe \Rnr }(a^{*}\le\Omega_{\lbe F_{S^{\rm nr}}\be/{\Snr}}^{1}\le)_{\mr{tors}}$. Then, by (a), the defect of smoothness of $F$ at the unit section $a=\varepsilon$ agrees with $\delta(F)$ \eqref{df}. Further, we claim that $\delta(F)$ coincides with the defect of smoothness of $F$ at {\it every} $\Rnr$-valued point of $F$. In effect, the left translation $\tau_{\le a}\colon F_{S^{\rm nr}}\to F_{S^{\rm nr}}$ is an isomorphism of $\Snr$-schemes, whence $\tau_{a}^{*}\e\Omega_{F_{S^{\rm nr}}/S^{\rm nr}}^{1}\simeq \Omega_{F_{S^{\rm nr}}/S^{\rm nr}}^{1}$. Since $a=\tau_{a}\circ \varepsilon$, our claim follows.  
\item[(c)]  If $F$ is finite, flat and generically smooth (i.e., generically \'etale) and $A$ denotes the Hopf algebra of $F$, then  the $A$-modules $\Omega_{F/R}^{1}$ and $\omega_{\lbe F/R}^{1}$ are annihilated by some power of $\pi$. Indeed, the $\cO_{F_{\lbe K}}$-module $\Omega_{F/R}^{1}\otimes_R K=\Omega_{F_{\lbe K}\lbe/K}^{1}$ has trivial fibers by \cite[\S 2.2, Proposition 2, p. 34]{blr}. Consequently, the $K$-vector space $\Omega_{F/R}^{1}\otimes_{R}K$ is trivial and therefore every element of $\Omega_{F/R}^{1}$ is annihilated by some power of $\pi$. Since $\Omega_{F/R}^{1}$ is finitely generated, our claim follows. We conclude that $\delta(F\le)=\mr{length}_{\lbe R}(\omega_{F/R}^{1}\le)$. Further, we claim that, if $F^{\e\circ}$ is the open and closed subgroup scheme of $F$ defined in Subsection \ref{c-et}, then $\omega_{\lbe F\be/\lbe R}^{1}=\omega_{F^{\lle\circ}\!\lbe/\lbe R}^{1}$  In effect, if $\iota\colon F^{\lle\circ}\to F$ is the canonical open immersion, then $\Omega^{1}_{F^{\lle\circ}/R}=\iota^{*}\e\Omega_{F/R}$. Further, the unit section of $F$ factors through $F^{\lle\circ}$ and we conclude that $\omega^{1}_{F^{\lle\circ}\be/R}=\omega_{F/R}$, as claimed. Consequently, $\delta(F\le)=\delta(F^{\lle\lle\circ})$.
\item[(d)] Assume that $\textrm{char}\, k=p>0$ and $F\neq 1$ is a finite, flat, connected and generically \'etale $R$-group scheme. By \cite[Lemma 6.1, p.~220]{mr}, the affine $R$-algebra of $F$ has the form
\[
A=R\le[X_{1},\dots X_{n}]/(\Phi_{1},\dots,\Phi_{n}\lbe),
\]
where $n$ is a positive integer, $(\Phi_{1},\dots,\Phi_{n}\lbe)$ is a regular sequence in $R\le[X_{1},\dots X_{n}]$ and
\[
\Phi_{j}\equiv X_{j}^{\e p^{\lambda_{j}}}\!\!\!\!\!\!\pmod{\mm}
\]
for some $\lambda_{j}\in\N$, where $1\leq j\leq n$. By \cite[\S 3.3, Lemma 2, p.~66]{blr} (and the fact that $\dim F_{\mr s}=0$), the ideal of $R$ generated by the constant term of  $\textrm{det}(\partial\le \Phi_{\lbe j}/\partial\le X_{i})$ equals
$\mm^{\delta(F)}$, where $\delta(F)$ is the defect of smoothness of $F$ \eqref{df}. The ideal $\mm^{\delta(F)}\subset R$ is called the {\it absolute different} of $F$.	\end{enumerate}   
\end{remarks}

\smallskip

Let $r\geq 1$ and $i\geq 0$ be integers and recall the change of level morphism $\varrho_{r}^{\e i}=\varrho_{r,\le F}^{\e i}\colon \mathrm {Gr}_{\lbe r+i}^{R}(\lbe F\le)\to\mathrm {Gr}_{\lbe r}^{R}(\lbe F\le)$ \eqref{nm-hom}. Since $\varrho_{r}^{\e i}$ is quasi-compact and separated by Corollary \ref{rqc}, the schematic image $H_{r}^{\le i}$ of $\varrho_{r}^{\e i}$ exists by \cite[Propositions 6.1.4, p.~291, and 6.10.5, p.~325]{ega1}. Now, by \cite[${\rm{VI}}_{\rm A}$, Proposition 6.4 and Corollary 6.6(i)]{sga3}, $\varrho_{r}^{\e i}$ factors as
\begin{equation}\label{hmi2}
\xymatrix{\mathrm {Gr}_{\lbe r+i}^{R}(\lbe F\le)\ar[d]\ar[r]^(.5){\varrho_{r}^{\e i}}& \mathrm {Gr}_{\lbe r}^{R}(\lbe F\le).\\
H_{r}^{\le i}\e\ar@{^{(}->}[ur]&
}
\end{equation}
In the above diagram (of $k$-group schemes of finite type), the vertical (respectively, oblique) morphism is faithfully flat (respectively, a closed immersion). Further, since $\mathrm {Gr}_{\lbe r}^{R}(F\le)$ is affine, $H_{r}^{\le i}$ is affine as well by \cite[II, Proposition 1.6.2, (i) and (ii)]{ega}.

\begin{lemma}\label{f/k} Let $r\geq 1$ and $l\geq 0$ be integers. If $H_{r}^{\le l}$ is finite over $k$, then $H_{r}^{\le i}$ is finite over $k$ for every integer $i\geq l$.
\end{lemma}
\begin{proof} For every integer $i\geq l$, there exists a canonical commutative diagram of $k$-group schemes of finite type
\[
\xymatrix{\mathrm {Gr}_{\lbe r+i}^{R}(\lbe F\le)\ar@/_1.5pc/[dd]_(.45){\varrho_{r}^{\e i}}\ar[d]\ar[rr]^(.5){\varrho_{r+l}^{\e i-l}}&& \mathrm {Gr}_{\lbe r+l}^{R}(\lbe F\le)\ar[d]\ar@/^1.5pc/[dd]^(.45){\varrho_{r}^{\e l}}\\
H_{r}^{\le i}\,\ar@{^{(}->}[d]\ar@{^{(}->}[rr]&& H_{r}^{\le l}\,\ar@{^{(}->}[d]\\
\mathrm {Gr}_{\lbe r}^{R}(\lbe F\le)\ar@{=}[rr]&&\mathrm {Gr}_{\lbe r}^{R}(\lbe F\le),
}
\]
where the middle horizontal arrow is a closed immersion which identifies 
$H_{r}^{\le i}$ with the schematic image of the restriction of $\varrho_{r}^{\e l}$ to the schematic image of $\varrho_{r+l}^{\e i-l}$ \cite[Proposition 6.10.3, p.~324]{ega1}. Consequently, if $H_{r}^{\le l}$ is finite over $k$, then $H_{r}^{\le i}$ is finite over $k$ as well by \cite[II, Proposition 6.1.5, (i) and (ii)]{ega}.
\end{proof}

Now observe that, by Lemma \ref{rat-pts}(ii) and Remark \ref{vrk}(b), diagram \eqref{hmi2} induces a commutative diagram of groups
\[
\xymatrix{F(R_{\le r+i})\ar[d]\ar[r]& F(R_{\e r}),\\
H_{r}^{\le i}(k)\ar@{^{(}->}[ur]&
}
\]
where the horizontal arrow is induced by the canonical projection $R_{\le r+i}\to R_{\e r}$. If $k$ is algebraically closed, the vertical map in the preceding diagram is surjective by \cite[I, \S3, Corollary 6.10, p.~96]{dg}, whence
\begin{equation}\label{imf}
H_{r}^{\le i}(k)=\mr{Im}\e[F(R_{\le r+i})\to F(R_{\e r})] \qquad \text{if } k=\kbar\,.
\end{equation}
Note also that the canonical projection $R\to R_{\e r}$ induces a group homomorphism
$F(R\le)\to F(R_{\e r})$.

\begin{lemma}\label{gr-appr} There exist integers $c\geq 1$, $d\geq 0$ and $M\geq 0$ such that, if $r\geq M$, then
\[
\mr{Im}[\le F(\lbe R_{\e c\le r+d})\to F(R_{\e r}\lbe)]=\img\le[\le F(R)\to F(R_{\e r}\lbe)].
\] 
\end{lemma}
\begin{proof} By \cite[Corollary 1, p.~59]{gre3}, there exist integers $N\geq 1,c\geq 1$ and $s\geq 0$ such that, for every integer $\zeta\geq N$,
\begin{equation}\label{gr-2}
\img[\le F(\lbe R_{\e \zeta }\lbe)\to F(R_{\le \lfloor{\zeta/c}\rfloor-s}\lbe)]=\img\le[\le F(R)\to F(R_{\le \lfloor{\zeta/c}\rfloor-s}\lbe)].
\end{equation}
Set $d=sc$ and $M=\mr{max}\{\lceil (N-d\,)/c\rceil,0\}$. If $r\geq M$, then $\zeta=c\le r+d\geq N$ and $\lfloor{\zeta/c}\rfloor=r+s$.  The  assertion of the lemma is now immediate from \eqref{gr-2}. 
\end{proof}

\begin{proposition}\label{fdim} Let $c\geq 1$, $d\geq 0$ and $M\geq 0$ be as in Lemma {\rm \ref{gr-appr}}. If $r\geq M$ and $i\geq (c-1)r+d$, then $H_{r}^{\le i}$ is finite over $k$.
\end{proposition}
\begin{proof} By Proposition \ref{unr3} and faithfully flat and quasi-compact descent \cite[${\rm{IV}}_{2}$, Proposition 2.7.1(xv)]{ega}, we may assume that $k$ is algebraically closed. By \eqref{imf} and Lemma \ref{gr-appr}, there exist integers $c\geq 1,d\geq 0$ and $M\geq 0$ such that, if $r\geq M$, then $H_{r}^{\le (c-1)r+d}(k)=\img[F(R)\to F(R_{\e r})]$. Thus, since $F(R)$ is finite, $H_{r}^{\le (c-1)r+d}(k)$ is finite as well. It follows that the topological space $\big|H_{r}^{\le (c-1)\le r+d}\big|$ has only finitely many closed points, whence $H_{r}^{\le (c-1)\lle r+d}$ is finite over $k$ by \cite[Corollary 6.5.3, Proposition 6.5.4 and (6.5.6)]{ega1}. The proposition is now immediate from Lemma \ref{f/k}.
\end{proof}

\begin{remarks}\label{gr-3} \indent
\begin{enumerate}
\item[(a)] Let $c\geq 1$, $d\geq 0$ and $M\geq 0$ be as in the proposition and let $t\geq 0$ be any integer. If $r\geq \max\{M/c^{\e t},d/c^{\e t}\}$ is an integer, then $rc^{\e t}\geq M$ and
\[
r\le c^{\e t+1}=(c-1)rc^{\e t}+rc^{\e t}\geq (c-1)rc^{\e t}+d.
\]
Consequently, $H_{\lbe rc^{\le t}}^{rc^{\e t+1}}$ is a finite $k$-subgroup scheme of $\mr{Gr}^{ R}_{\lbe rc^{\le t}}(\lbe F)$.
\item[(b)]  If $F$ is {\it \'etale} over $R$, then Corollary \ref{fetR} shows that $\varrho_{r}^{\e i}\colon \mathrm{Gr}_{\lbe r+i}^{R}(\lbe F\le)\to \mathrm{Gr}_{\lbe r}^{R}(\lbe F\le)$ is an isomorphism of $k$-group schemes for every $r\geq 1$ and $i\geq 0$. Consequently, for every integer $r\geq 1$, $\mathrm{Gr}_{\lbe r}^{R}(\lbe F\le)\simeq \mathrm{Gr}_{1}^{R}(\lbe F\le)=F_{\mr s}$. In particular, $H^{\le i}_{r}\simeq F_{\mr s}$ is finite over $k$ for every pair of integers $r\geq 1$ and $i\geq 0$.
\end{enumerate}
\end{remarks}

\begin{lemma}\label{gr-3c} If $F$ is generically \'etale, then $H_{\lbe\lbe r}^{\lle r}$ is  finite over $k$ for every integer $r\geq\delta(F\le)+2$, where $\delta(F\le)$ is the defect of smoothness of $F$ \eqref{df}.
\end{lemma}
\begin{proof}
Recall that $H_{\lbe\lbe r}^{\lle r}$ is the schematic image of 
$\varrho_{r}^{\e r}\colon \mathrm{Gr}_{\lbe 2\lle r}^{R}(\lbe F\le)\to \mathrm{Gr}_{\lbe r}^{R}(\lbe F\le)$. If $F$ is \'etale over $R$ (which is the case if $\textrm{char}\, k=0$), then the lemma is trivially true by Remark \ref{gr-3}(b). Assume now that $\textrm{char}\, k=p>0$ and recall the canonical sequence of $R$-group schemes \eqref{c-e-k}
\[
1\to F^{\le\circ}\to F\to F^{\e \et}\to 1.
\]
The preceding sequence induces the following commutative diagram of $k$-group schemes of finite type which is exact for both the fppf and fpqc topologies on $(\mr{Sch}/k)$:
\[
\xymatrix{1\ar[r]&\mr{Gr}^{R}_{\lbe 2\lle r}(F^{\le\circ}\lle)\ar[r]\ar[d]^{\varrho_{\le r,\le F^{\lle\circ}}^{\lle r}}&\mr{Gr}^{ R}_{\lbe 2\lle r}(F )\ar[r]\ar[d]^{\varrho_{\le r,\le F}^{\e r}}& \mr{Gr}^{ R}_{\lbe 2\lle r}(F^{\e\et})\ar[d]_{\simeq}^{\varrho_{\lbe r,\le F^{\le\et}}^{r}}\\
1\ar[r]&\mr{Gr}^{\lbe R}_{\lbe r}(F^{\e\circ})\ar[r]&\mr{Gr}^{\lbe R}_{\lbe r}(F )\ar[r]& \mr{Gr}^{ R}_{\lbe r}(F^{\e \et}),
}
\]
where the right-hand vertical morphism is an isomorphism of finite $k$-group schemes by Remark \ref{gr-3}(b). Since $\mr{Gr}^{ R}_{\lbe r}(F^{\e \et})\lbe\big(\e\kbar\e\big)$ is a finite group, we conclude from the diagram and the equality $\delta(F^{\e\circ}\lle)=\delta(F\le)$ (recall the proof of Proposition \ref{fdim} and see Remark \ref{rm.df}(c)) that it suffices to prove the lemma when $F=F^{\e \circ}$. Thus we may assume that $F=F^{\e \circ}$. We will show that, in this case, $c=1$, $d=\delta(F\le)$ and $M=\delta(F\le)+2$ are valid choices in Lemma \ref{gr-appr}.
Since $M=\delta(F\le)+2\geq d=\delta(F\le)$, it is then possible to choose $i=r\geq M=\delta(F\le)+2$ in Proposition \ref{fdim}, which will complete the proof.

Choose an isomorphism $F\simeq\spec\be(R\le[X_{1},\dots X_{n}]/(\Phi_{\lbe 1},\dots,\Phi_{\lbe n}))$ as in Remark \ref{rm.df}(d), let $J(X_{1},\dots X_{n})=(\partial\e \Phi_{\lbe j}\lbe/\partial\le X_{i})$ be the corresponding Jacobian matrix and set $J=J(0,\dots,0)\in M_{n\times n}(R\le)$. Further, let $\widetilde{J}$ denote the adjoint matrix of $J$ and recall the uniformizing element $\pi$ of $R$ fixed previously. Since $\textrm{det}\e J=u\e\pi^{\le\delta(F\le)}$ for some unit $u\in R^{\times}$, we have
\begin{equation}\label{detc}
u^{-1}\be J\le\widetilde{J}=\pi^{\delta(F)} I_{n},
\end{equation}
where $I_{n}\in M_{n\times n}(R\le)$ is the identity matrix. We will now adapt the proof of \cite[Lemma 2, p.~567]{gre3} to show that, if  $\mu\geq \delta(F)+1$  and $x=(x_{1}, \dots, x_{n})\in R^{n}$ satisfies $\Phi_{\lbe j}(x)\equiv 0$ $\text{mod }\pi^{\delta(F)+\mu+1}$ if $1\leq j\leq n$, then there exists a common zero $y\in R^{\le n}$ of the polynomials $\Phi_{\lbe j}$ such that $x\equiv y$ (mod $\pi^{\mu+1}$). Taking $\mu=r-1$, we conclude that we may, in fact, choose $c=1$, $d=\delta(F)$ and $M=\delta(F)+2$ in Lemma \ref{gr-appr}.

Now, since the coefficients of $J\in M_{n\times n}(R\le)$ are those of the linear terms of the polynomials $\Phi_{\lbe j}$, which are not affected by the substitutions $X_{i}\mapsto X_{i}-x_{i}$, we may assume that $x=(0,\dots,0)$.
Henceforth we will use  multi-index notation, i.e., $\Phi= (\Phi_{1},\dots,\Phi_{n})$ and $X=(X_{1},\dots,X_{n})$. By hypothesis $\Phi(x)=\Phi(0)=\pi^{\le \mu+\delta(F)+1}a$ for some $a\in R^{\le n}$. Now \eqref{detc} yields
$\Phi(0)=\pi^{\le \mu}J(u^{-1}\pi\le\widetilde{J}\le a)$ and $\pi^{\le 2\mu}I_{n}=\pi^{\le\mu}J(u^{-1}\pi^{\mu-\delta(F\le)}\widetilde{J})$, where $\mu-\delta(F)\geq 1$. Consequently, by Taylor's formulas for the polynomials $\Phi_{\lbe j}$, we have
\[
\Phi(\pi^{\le \mu}\be X)=\Phi(0) +\pi^{\le \mu}\be J\le X+\pi^{\le 2\mu}(\dots)=\pi^{\le \mu }\be J\!\left(\pi u^{-1}\widetilde{J}a+X+u^{-1}\pi^{\le\mu-\delta(F)}\widetilde{J}(\dots)\right).
\]
Set $\Psi(X)=\pi u^{-1}\widetilde{J}a+X+u^{-1}\pi^{\le\mu-\delta(F)}\widetilde{J}(\dots)$, so that $\Phi(\pi^{\le \mu}\be X)=\pi^{\le \mu }\be J\le\Psi(X)$. Since $\mu-\delta(F)\geq 1$, we have $\Psi(0)\equiv 0\pmod{\pi}$ and $\textrm{det} (\partial\le \Psi_{\lbe j}/\partial\le X_{i})(0)\neq 0$. Thus, by \cite[Lemma 1, p.~567]{gre3}, there exists $z\in R^{\le n}$ such that $z\equiv 0\pmod \pi$ and $\Psi(z)=0$. We conclude that $y=\pi^{\le\mu}z$ satisfies $y\equiv 0 \pmod{\pi^{\le\mu+1}}$
and $\Phi(\e y)=\Phi(\pi^{\le\mu}z)=\pi^{\le\mu}\be J\le\Psi(z)=0$. The proof is now complete.
\end{proof}

Now consider the projective limit of affine $k$-group schemes of finite type
\begin{equation}\label{pg}
\gr(F\le)=\varprojlim\grn(F\le),
\end{equation}
where $\grn(F)$ is given by  \eqref{def} and the transition morphisms in the limit are the change of level morphisms $\varrho_{n,\le F}^{\e i}\colon\grni(F)\to\grn(F)$. By \cite[$\text{IV}_{3}$, Proposition 8.2.3]{ega}, \eqref{pg} is an affine $k$-group scheme.  
Now set 
\[
\pgr(F)=\gr\lbe(F\le)^{\pf}.
\]
		
\begin{proposition}\label{grfin}  The underlying topological space of the affine $k$-group scheme $\gr(F)$ \eqref{pg} is finite and each of its residue fields is an algebraic extension of $k$. In particular, $\dim \gr(F\le)=0$.  

\end{proposition}
\begin{proof} Let $c\geq 1,d\geq 0$ and $M\geq 0$ be as in Lemma \ref{gr-appr} and let $r\geq\max\{M,d\e\}$ and $t\geq 0$ be integers. Since $r\geq \max\{M/c^{\e t},d/c^{\e t}\}$, Remark \ref{gr-3}(a) shows that $H_{\lbe rc^{\le t}}^{rc^{\e t+1}}$ is a finite $k$-subgroup scheme of $\mr{Gr}^{ R}_{\lbe rc^{\le t}}(\lbe F)$. Now consider the following particular case of diagram \eqref{hmi2}
\[
\xymatrix{\mathrm {Gr}_{\lbe rc^{\le t}\lbe(c+1)}^{R}(\lbe F\le)\ar[d]\ar[rr]^(.55){\varrho_{\lbe rc^{\le t}}^{\e rc^{\le t+1}}}&& \mathrm {Gr}_{\lbe rc^{\le t}}^{R}(\lbe F\le)\\
H_{\lbe rc^{\le t}}^{\lle rc^{\le t+1}}\ar@{^{(}->}[urr]&&
}
\]
and set $H=\varprojlim_{\e t}H_{\lbe rc^{\le t}}^{\lle rc^{\le t+1}}$. Then,  by Lemma \ref{paux}, the projective limit over $t$ of the preceding diagram yields a factorization of the  identity map of $\gr(F\le)$:
\begin{equation}\label{tr-h}
\xymatrix{\gr\lbe(F\le) \ar[d]\ar[rr]^(.45){ 1_{\lle\gr\lbe(F\lle)} }&&\gr\lbe(F\le),\\
H^{\phantom{.}} \ar@{^{(}->}[urr]&&
}
\end{equation} 
where the oblique map is a closed immersion by \cite[Proposition 3.2]{bga}. In fact, the latter map is an isomorphism by the commutativity of the diagram, for if $\gr(F)=\spec A$ and $H=\spec(A/I)$ for a suitable ideal $I$ of $A$, then the identity map of $A$ factors through $A/I$, whence $I=0$. The proposition now follows by applying \cite[Proposition 3.6]{bga} to $H$.
\end{proof}

\begin{proposition}\label{grfin2}  Assume that $k$ is perfect. Then there exists a canonical isomorphism of finite and \'etale $k$-group schemes $\gr\be(F\le)_{\red}=\pgr\lbe(F\lle)\le$. 
\end{proposition}
\begin{proof} We will show that $\gr(F)_{\red}$ is finite and \'etale, which will prove that $\gr(F)_{\red}$ is perfect and canonically isomorphic to $\pgr\lbe(F\lle)$ by \cite[(5.15) and Proposition 5.19]{bga}. By the proof of Proposition \ref{grfin}, the closed immersion $H\to \gr(F)$ in \eqref{tr-h} is an isomorphism. Therefore, in the notation of that proof, $\gr(F)=\varprojlim H_{\lbe rc^{\le t}}^{\lle rc^{\le t+1}}$. Let $A_{\e t}$ denote the Hopf $k$-algebra of the finite $k$-group scheme $H_{\lbe rc^{\le t}}^{\lle rc^{\le t+1}}$. Then $\gr(F)=\spec A$, where $A=\varinjlim A_{\e  t}$. Further, $\gr(F)_{\red}$ inherits a $k$-group structure from $\gr(F)$ since $\gr(F)_{\red}=\spec A_{\e\red}=\varprojlim \spec A_{\e t,\le\red}$ and $\spec A_{\e t,\le\red}$ is a finite $k$-group scheme by \cite[Theorem p.~52 and Exercise 9, p.~53]{wa}. 
Further, $A_{\e\red}\otimes_{k}\kbar$ is reduced by \cite[\S6.2, Theorem p. 47]{wa} and thus $\gr(F)_{\red}\times_{k}\spec \kbar= (\gr(F)\times_{ k}\spec \kbar\le)_{\red}$. By Proposition \ref{r-unr3} and faithfully flat and quasi-compact descent \cite[$\mr{IV}_2$, 2.7.1(xv)]{ega}, we may now assume that $k=\kbar$.  
The $k$-group schemes $\spec A_{\e t,\red}$ are finite and reduced and therefore finite and constant (i.e., with trivial Galois action). Consequently, $\gr(F)_{\red}$ is a profinite $k$-group scheme. Since  $|\gr(F)_{\red}|=|\gr(F)|$, we conclude from Proposition \ref{grfin} that the $k$-group scheme $\gr(F)_{\red}$ has finitely many points. Therefore  $\gr(F)_{\red}$ is a finite and constant $k$-group scheme, which completes the proof.
\end{proof}

\begin{remarks}\label{rm.grf} \indent
\begin{enumerate}
\item[(a)]  Let $p$ be a prime and let $R$ be an equal characteristic $p$ ring. It follows from  \eqref{alp} that $\gr(\alpha_{p})\simeq \spec (k[\le x_{\le 0},\dots, x_{n},\dots\le]/(x_{i}^{\le p}, i\geq 0)$
has dimension zero but is not a finite $k$-group scheme. Further, the Hopf algebra of $\gr(\alpha_{p})$ is a non-noetherian ring.

\item[(b)]  The isomorphism \eqref{alp} also shows that the closed immersions 
\[
(\grn(X_{\red}) )_{\red} \to \grn(X)_{\red}
\]
are not isomorphisms in general. In effect, if $X=\alpha_{p}$, then the preceding morphism is the map $0\to \spec (k[\le x_{\le r},\dots, x_{n-1}])$, where $r=\lfloor (n+p-1)/p\rfloor$.

\item[(c)]  By the left exactness of the inverse limit functor on the category of groups and the left exactness of the Greenberg functor of finite level $\grm$, the connected-\'etale exact sequence \eqref{c-e-k} induces a short exact (for both the fppf and fpqc topologies) sequence of $k$-group schemes 
\begin{equation*} 
0\to \gr(F^{\e\circ})\to \gr(F)\to \gr(F^{\e \et}).
\end{equation*}
Now, by  Remark \ref{gr-3}(b), the $k$-group scheme $\gr(F^{\e\et})=F_{\!\mr s}^{\e \et}$ is  finite and \'etale. Therefore the map $\gr(F)\to \gr(F^{\e\et})$ factors through $\pi_{0}(\gr(F))$, whence the closed immersion $\gr(F)^{0}\to \gr(F)$ factors through $\gr(F^{\e\circ})$. Thus there exists a canonical morphism  $\gr(F)^{0}\to \gr(F^{\e\circ})$ which, in general, is not an isomorphism. For example, let $R=W(\mathbb F_{2})$ and consider the connected finite $R$-group scheme $F^{\e\circ}=\mu_{2,R}$ of square roots of unity (cf. \S \ref{c-et}). We have $F^{\e\circ}\be(R)=F^{\e\circ}\be(K)=\{\pm 1\}$ and $\pgr(F)$ is finite and \'etale by Proposition \ref{grfin2}(ii).  We will see in Proposition \ref{gr-pts2} below that $F^{\e\circ}\be(R)=\gr(F^{\e\circ})(k)$. Further, by \cite[(5.5)]{bga}, we have  $\gr(F^{\e\circ})(k)=\pgr(F^{\e\circ})(k)$. Consequently, the finite and \'etale $k$-group scheme $\pgr(F^{\e\circ})$ is disconnected. Now $\gr(F^{\e\circ})$ is homeomorphic to $\pgr(F^{\e\circ})$ (see \cite[Remark 5.18(b)]{bga}), whence $\gr(F^{\e\circ})$ is disconnected as well.

\item[(d)]  By Remark \ref{gr-3}(b) and the exactness of $0\to \grm(F^{\e\circ})\to \grm(F)\to \grm(F^{\e \et})$, we have $\dim \grm(F^{\e \circ})=\dim \grm(F)$ for every $m\geq 1$. 

\end{enumerate}
\end{remarks}

\section{The Greenberg realization of an adic formal scheme}\label{forgr}

We continue to assume that $R$ is {\it complete} with perfect residue field in the unequal characteristics case. Recall that $S=\spec R$, $\mm$ denotes the maximal ideal of $R$, $R_{\le n}=R/\mm^{\le n}$ and $S_{n}=\spec R_{\le n}$, where $n\in\N$. For details on $\mm$-adic formal schemes, see Subsection \ref{for-sch}. Unadorned limits/inductive systems below are indexed by $\N$. Let $\m S=\widehat{S}$ be the formal completion of $S$ along $S_{\le 1}=\spec k$. Then $\m S=\spf R=\varinjlim S_{n}$ is an adic formal scheme globally of finite ideal type. 
		
\smallskip
			
Let $Y$ be a $k$-scheme and recall the Zariski sheaves $\s R_{\le n}(\be\cO_{Y}\lbe)$ on $Y$ \eqref{zdef}, the $R_{\le n}$-schemes $\hrn(Y)=(|Y|, \s R_{\le n}(\cO_{Y}))$ and the nilpotent immersions $\delta_{\le Y}^{\e i\le,\e j-i}\colon h_{\le i}^{\lbe R}(Y)\to h_{\le j}^{\lbe R}(Y)$ \eqref{dnj}, where $1\leq i\leq j$. By  \cite[I, Proposition 10.6.3 p. 412]{ega1} 
\[
\Fhr(Y)=\varinjlim \hrn(Y)
\]
is a formal $\m S$-scheme equal to $(|Y|, \widetilde{\s R}\le (\be\s O_{Y}\be))$, where $\widetilde{\s R}(\be\s O_{Y}\be)$ is the Zariski sheaf on $Y$ defined by \eqref{zsf}.

\begin{example} If $k$ is perfect of positive characteristic $p$ and $R=W(k)$ is the ring of $p\e$-typical Witt vectors on $k$, then  $\widetilde{\s R}=\mathbb W$ is the $k$-ring scheme of  Witt vectors of infinite length. Using \eqref{zsf2}, $\Fhr(Y)=W\be(Y)$ is the formal scheme considered in \cite[\S1.5, p.~511]{ill}. Note that, as illustrated in Remark \ref{power}(b), the inclusion $(V W_{\! n}(\cO_{Y}))^{m}\subseteq V^{\le m}\be\lle(W_{\! n}(\cO_{Y}))$ can be strict, whence $W(Y)$ is not, in general, an adic formal scheme. However, the following holds. \end{example}

\begin{proposition}\label{fhr-a} Let $Y$ be a $k$-scheme. Assume that
\begin{enumerate}
\item[(i)] $R$ is an equal characteristic ring, or
\item[(ii)] $R$ is a ring of unequal characteristics and $Y$ is a perfect $k$-scheme.
\end{enumerate}
Then $\Fhr(Y)$ is an adic formal $\m S$-scheme. 
\end{proposition}
\begin{proof} By Remarks \ref{fgt} and \ref{power}(c)-(d), the projective system $(\s R_{ n}(\be\s O_{Y}\be))$ satisfies the conditions of \cite[Proposition 2.1.36, p.~125]{ab} (see also Remark \ref{term}(c)). Consequently, $\Fhr(Y)$ is, in fact, an adic formal $\m S$-scheme. 
\end{proof}

\begin{corollary}\label{fhr-aff} Let $A$ be a $k$-algebra. If $R$ is a ring of unequal characteristics, assume that $A$ is perfect. Then $\Fhr(\spec A)=\spf\widetilde{\s R}(A)$.
\end{corollary}
\begin{proof} This follows by combining the proposition, \eqref{ra} and \eqref{hrn-aff}.
\end{proof}

Let $\mathrm{Ind}(\m S)$ denote the category whose objects are the inductive systems of ${\m S}_{n}$-schemes $(\m X_{\e n})$, where every transition morphism $\m X_{\e n}\to \m X_{\e n+1}$ is a nilpotent immersion of $\m S_{n+1}$-schemes. The morphisms $(\m X_{\e n}^{\e\prime}) \to (\m X_{\e n})$ in $\mathrm{Ind}(\m S)$ are given by $\m S_{n}$-morphisms $f_{n}\colon\m X_{\e n}^{\e\prime} \to \m X_{\e n}$ that make the evident squares commute. If the latter squares are cartesian, then we recover the full subcategory $\mathrm{Ad}\e$-$\e\mr{Ind}(\m S)$ of $\mathrm{Ind}(\m S)$ of adic inductive $(\m S_{n})$-systems introduced in Subsection \ref{for-sch}. Now, for every object $(\m X_{\e n})$ of $\mathrm{Ind}(\m S)$, consider the contravariant functor 
\begin{equation}\label{for-fun0}
(\mathrm{Sch}/k)\to (\mathrm{Sets}), \quad Y\mapsto \Hom_{\e\mathrm{Ind}(\m S)}\!\left( (\hrn(Y\le) ) , (\m X_{\e n})  \right).
\end{equation}

\begin{proposition-definition}\label{for-def} For every object $\m X_{\bullet}=(\m X_{\e n})$ of $\mathrm{Ind}(\m S)$, the functor \eqref{for-fun0} is represented by a $k$-scheme which is denoted by $\gr(\m X_{\bullet})$. Thus, for every $k$-scheme $Y$, there exists a canonical bijection
\begin{equation}\label{for-adj0}
\Hom_{\e k}\big(Y, \gr(\m X_{\bullet})\big)=\Hom_{\e\mathrm{Ind}(\m S)}\!\left((\hrn(Y)),\m X_{\bullet}\right).
\end{equation}
\end{proposition-definition}
\begin{proof} 
Since the transition morphisms of the inductive system $\m X_{\bullet}$ are universal homeomorphisms, the transition morphisms of the projective system $(\grn(\m X_{\e n}))$ are affine (see the proof of Proposition \ref{aff}). Thus 
\begin{equation}\label{for-def2}
\gr(\m X_{\bullet})\overset{\rm def.}{=}\varprojlim\grn(\m X_{\e n})
\end{equation}
exists in $(\mathrm{Sch}/k)$. The adjunction formula \eqref{for-adj0} now follows from \eqref{plim} and \eqref{adj-bis}. 
\end{proof}

We now recall from Subsection \ref{for-sch} the equivalence of categories  
\begin{equation}\label{eai}
(\text{Ad-For}/\m S)\to \mathrm{Ad}\e\text{-}\e\mr{Ind}(\m S), \quad \m X=\varinjlim \m X_{\e n}\mapsto (\m X_{\e n}).
\end{equation}
It follows from Proposition \ref{for-def} and its proof that, for every object $\m X=\varinjlim \m X_{\e n}$ in $(\text{Ad-For}/\m S)$, the $k$-scheme 
\begin{equation*} 
\gr(\m X)\overset{\rm def.}{=}\gr((\m X_{\e n}))=\varprojlim \grn(\m X_{\e n}) 
\end{equation*}
exists. If we set, for $n\in\N$,
\begin{equation}\label{grnx}
\grn(\m X)\overset{\rm def.}{=}\grn(\m X_{\e n}),
\end{equation}
then we can write $\gr(\m X)=\varprojlim \grn(\m X)$. Thus we have defined a covariant functor
\begin{equation}\label{for-fun2}
\gr\colon (\text{Ad-For}/\m S)\to (\mathrm{Sch}/k), \quad \m X\mapsto\gr(\m X),
\end{equation}
which, by \eqref{grnsn}, satisfies 
\begin{equation}\label{for-grs}
\gr(\m S)=\spec k. 
\end{equation}
Recall that, by Lemma \ref{fhr-a}(i), if $R$ is an {\it equal characteristic} ring and $Y$ is any $k$-scheme, then $\Fhr(Y)$ is an object of $(\text{Ad-For}/\m S)$. Thus, by the adjunction formula \eqref{for-adj0} and the equivalence of categories \eqref{eai}, there exists a canonical bijection 
\begin{equation*} 
\Hom_{\e k}\big(Y, \gr(\m X\le)\big)=\Hom_{ (\text{Ad-For}/\m S)}(\Fhr(Y),\m X\le),
\end{equation*}
i.e., $\gr\colon (\text{Ad-For}/\m S)\to (\mathrm{Sch}/k)$ is right  adjoint to $\Fhr\colon (\mathrm{Sch}/k)\to (\text{Ad-For}/\m S)$. The corresponding statement in the unequal characteristics case is false. However, the following generalization of \cite[line 10, p.~256]{ns2} is valid.

\begin{lemma}\label{for-rat-pts} Let $\m X$ be an adic formal $\m S$-scheme and let $A$ be a $k$-algebra which is assumed to be perfect if $R$ is a ring of unequal characteristics. Then $\gr(\m X)(A)=\m X(\widetilde{\s R}(A))$. 
\end{lemma}
\begin{proof} Since $\gr(\m X)(A)=\Hom_{(\text{Ad-For}/\m S)}(\Fhr(\spec A),\m X\e)$ by \eqref{for-adj0}, the lemma follows at once from Corollary \eqref{fhr-aff}.
\end{proof}

Now let $X$ be an $S$-scheme and let $\widehat{X}$ be the formal completion of $X$ along its special fiber $X\times_{S}S_{1}$. Then $\widehat{X}$ is an object of $(\text{Ad-For}/\m S)$. Further, by \eqref{hat1},
\begin{equation}\label{hat2}
\widehat{X}\,=X\times_{S}\m S=\varinjlim\,(X\times_{S}S_{n}).
\end{equation}
In particular, if $S^{\e\prime}=\spec R^{\e\prime}$,  where $R^{\e\prime}$ is a finite extension of $R$ of ramification index $e$, then, by \eqref{ram-ten},
\begin{equation}\label{hat3}
\m S^{\prime}\overset{\rm def.}{=}\widehat{S^{\le\prime}}=S^{\e\prime}\times_{S}\m S=\varinjlim \,(S^{\e\prime}\times_{S}S_{n})=\varinjlim \,S^{\e\prime}_{ne}.
\end{equation}
More generally, if $X^{\prime}$ is any $S^{\e\prime}$-scheme,
\begin{equation}\label{hat4}
\widehat{X^{\prime}}=X^{\prime}\times_{S}\m S=X^{\prime}\times_{S^{\prime}}\m S^{\prime}=
\varinjlim \,(X^{\prime}\times_{S^{\e\prime}}S^{\e\prime}_{ne}\le).
\end{equation}
Let $k^{\e\prime}/k$ be a (possibly infinite) subextension of  $\e\kbar/k$ and let $R^{\e\prime}$ be the extension of $R$ of ramification index 1 which corresponds to $k^{\e\prime}/k$. Set $S^{\e\prime}=\spec R^{\e\prime}$. Since the maximal ideal $\mm^{\e\prime}$ of $R^{\e\prime}$ equals $\mm R^{\e\prime}$, \eqref{hat2} shows that
\begin{equation}\label{hat5}
\m S^{\prime}\overset{\rm def.}{=}\widehat{S^{\le\prime}}=S^{\e\prime}\times_{S}\m S=\spf\e \widehat{R^{\e\prime}},
\end{equation}
where $\widehat{R^{\e\prime}}$ is the $\mm^{\e\prime}$-adic completion of $R^{\e\prime}$.

\begin{proposition}\label{pro-gr} Consider, for a morphism of formal schemes, the property of being:
\begin{enumerate}
\item[(i)] quasi-compact;
\item[(ii)] quasi-separated;
\item[(iii)] separated;
\item[(iv)] a closed immersion;
\item[(v)] affine;
\item[(vi)] an open immersion;
\item[(vii)] formally \'etale.
\end{enumerate}
If $\bm{P}$ denotes one of the above properties and $f\colon\m X\to \m Y$ is a morphism of adic formal $\m S$-schemes with property $\bm{P}$, then the corresponding morphism of $k$-schemes $\gr(\le f\lle)\colon \gr(\m X)\to\gr(\m Y)$ has property $\bm{P}$ as well.
\end{proposition}
\begin{proof} Let $(f_{n})\colon (\m X_{n})\to (\m Y_{n})$ be the morphism of adic inductive $(S_{n})$-systems which corresponds to $f$. If $\bm{P}$ denotes one of properties (i)-(v), then $f_{n}\colon \m X_{n}\to \m Y_{n}$ has property $\bm{P}$ for every $n\in \N$ by Lemma \ref{for-prop}. Consequently, $\grn(f_{n})\colon\grn(\m X)\to \grn(\m Y)$ has property $\bm{P}$ as well for every  $n\in \N$ by Remark \ref{rems1}(b) and Proposition \ref{gr-prop}, (i)-(iii) and (vi). Therefore $\gr(f)$ has property $\bm{P}$ by \cite[Proposition 3.2]{bga}. In the case of properties (vi) and (vii), a different argument is needed since, as noted in \cite[Example 3.5]{bga}, a projective limit of open immersions may not be an open immersion. By \cite[$\text{IV}_{4}$, Proposition 17.1.3(i)]{ega} and Lemma \ref{for-et}, if $f$ has one of properties (vi) or (vii), then each $f_{n}$ is formally \'etale. Thus, via the identification $f_{n,\le \rm s}=f_{1}$ made above, Corollary \ref{ffet} shows that $\grn(f_{n})$ factors as
\[
\grn(\m X)\stackrel{\!\sim}{\to}\m X_{1}\times_{\m Y_{1}}\grn(\m Y)\to \grn(\m Y),
\]
where the second morphism can be identified with $f_{1}\times_{\m Y_{1}}\grn(\m Y)$. Consequently, since projective limits commute with base extension \cite[$\text{IV}_{3}$, Lemma 8.2.6]{ega}, $\gr(\le f\lle)$ factors as $\gr(\m X)\stackrel{\!\sim}{\to}\m X_{1}\times_{\m Y_{1}}\gr(\m Y)\to \gr(\m Y)$, where the second morphism can be identified with $f_{1}\times_{\m Y_{1}}\gr(\m Y)$. Thus, since $f_{1}$ is an open immersion (respectively, formally \'etale), $\gr(\le f\lle)$ is an open immersion (respectively, formally \'etale).
\end{proof}

\begin{proposition} 
Let $\m X$ and $\m Y$ be two adic formal $\m S$-schemes. Then there exists a canonical isomorphism of $k$-schemes
\[
\gr(\m X\times_{\m S}\m Y)=\gr(\m X)\times_{\gr(\m S)}\gr(\m Y).
\]
\end{proposition}
\begin{proof} Let $(\m X_{n})$ and $(\m Y_{n})$ be the adic inductive systems of $S_{n}$-schemes which correspond to $\m X$ and $\m Y$, respectively. Then $\m X\times_{\m S}\m Y=\varinjlim \,(\m X_{n}\times_{S_{n}}\m Y_{n})$ by \cite[Corollary 1.3.5, p.~267]{fk}. Further, since the functor $\grn$ respects fiber products by Remark \ref{rems1}(d), $\gr(\m X\times_{\m S}\m Y)=\varprojlim\,(\grn(\m X)\times_{\spec k}\,\grn(\m Y))$ by \eqref{grnsn} and \ref{for-def}. Now consider $\N\times\N$ as an ordered set with the product order. If $(m,n)\leq (m^{\e\prime},n^{\e\prime}\e)$, the canonical morphism
\[
\mr{Gr}^{R}_{m^{\prime}}(\m X)\times_{\spec k}\mr{Gr}^{R}_{n^{\prime}}(\m Y)\to
\mr{Gr}^{R}_{m}(\m X)\times_{\spec k}\grn(\m Y)
\]
is affine \cite[II, Proposition 1.6.2(iv)]{ega}. Thus $(\mr{Gr}^{R}_{m}(\m X)\times_{\spec k}\grn(\m Y))_{(m,n)\in\N\times\N}$ is a projective system of $k$-schemes with affine transition morphisms. On the other hand, since $\{(n,n)\colon n\in\N\}$ is cofinal in $\N\times\N$, we have, by \cite[IX, \S3, dual of Theorem 1, p.~213]{mac}, \cite[$\text{IV}_{3}$, proof of Proposition 8.2.3 and Lemma 8.2.6]{ega}, \cite[III, \S7.3, Proposition 4, p.~198]{bou3} and \eqref{for-grs},
\begin{align*}
\gr(\m X\times_{\m S}\m Y)&=\varprojlim\,(\grn(\m X)\times_{\spec k}\,\grn(\m Y))\\
&=\varprojlim_{(m,n)}\, (\mr{Gr}^{R}_{m}(\m X)\times_{\spec k}\grn(\m Y))\\
&=\varprojlim_{m} \varprojlim_{n}\, (\mr{Gr}^{R}_{m}(\m X)\times_{\spec k}\grn(\m Y))\\
&=\varprojlim_{m}\, (\mr{Gr}^{R}_{m}(\m X)\times_{\spec k}\gr(\m Y))\\
&=\gr(\m X)\times_{\gr(\m S)}\gr(\m Y).
\end{align*} 
\end{proof}

The following corollary of the proposition is immediate. 
\begin{corollary} 
If $\m X$ is an adic formal $\m S$-group scheme, then $\gr(\m X)$ is a $k$-group scheme.\qed
\end{corollary}

\begin{proposition}\label{wr-gr2}
Let $R^{\e\prime}$ be a finite extension of $R$ of ramification index $e$ and associated residue field extension $k^{\e\prime}/k\subseteq \kbar/k$. Let $\m S^{\prime}$ be given by \eqref{hat3} and let $\m X^{\e\prime}=\varinjlim \m X^{\e\prime}_{n}$ be an adic formal $\m S^{\prime}$-scheme such that $\m X^{\e\prime}_{ne}$ is admissible relative to $S^{\e\prime}_{ne}\to S_{n}$ for every $n\geq 1$ (see Definition {\rm \ref{adm}}). Then $\Re_{\e\m S^{\e\prime}\be/\m S}\big(\m X^{\e\prime}\le\big)$  and $\Re_{\e k^{\e\prime}\be/k}\big(\mr{Gr}^{R^{\e\prime}}\!\big(\m X^{\e\prime}\le\big)\big)$ exist and
\[
\mr{Gr}^{R}\be\big(\Re_{\e\m S^{\e\prime}\be/\m S}\big(\m X^{\e\prime}\e\big)\big)=\Re_{\e k^{\e\prime}\be/k}\big(\mr{Gr}^{R^{\e\prime}}\!\big(\m X^{\e\prime}\le\big)\big).
\]
\end{proposition}
\begin{proof} By Theorem \ref{wr-rep}, $\Re_{S^{\lle\prime}_{ne}/S_{\le n}}(\m X^{\e\prime}_{ne})$ exists for every $n\in\N$. Further, by \eqref{wrbc}, \eqref{ram-ten} and \eqref{cart0}, for every pair of positive integers $r,n$ such that $2\leq r\leq n$, there exists a canonical isomorphism of $S_{r-1}$-schemes
\[
\Re_{\le S^{\le\prime}_{re}/S_{\le r}}(\m X^{\e\prime}_{re})=\Re_{\le S^{\lle\prime}_{ne}/S_{\le n}}(\m X^{\e\prime}_{ne})\times_{S_{\le n}} S_{\le r}.
\]
Thus $(\Re_{S^{\prime}_{ne}/S_{\le n}}(\m X^{\e\prime}_{ne}))$ is an adic inductive $(S_{n})$-system and we write 
\begin{equation}\label{for-wr}
\Re_{\e\m S^{\e\prime}\be/\m S}\big(\m X^{\e\prime}\e\big)\overset{\rm def.}{=}\varinjlim\Re_{\le S^{\lle\prime}_{ne}/S_{\le n}}(\m X^{\e\prime}_{ne})
\end{equation}
for the corresponding adic formal $\m S$-scheme. Let $\m T=\varinjlim_{n\in \N} \m T_{n}$ be an arbitrary adic formal $\m S$-scheme. Then, by \eqref{for-hom}, \eqref{wr} and \cite[Corollary 1.3.5, p.~267]{fk}, we have
\begin{align*}
\Hom_{\le\m S}(\m T,\Re_{\e\m S^{\e\prime}\be/\m S}(\m X^{\prime}\e))&=\varprojlim\Hom_{\le S_{n}}(\m T_{n},\Re_{S^{\le\prime}_{ne}/S_{\le n}}(\m X^{\e\prime}_{ne}))\\
&=
\varprojlim
\Hom_{\le S^{\lle\prime}_{ne}}\be(\m T_{n}\times_{S_{n}}S^{\e\prime}_{ne},\m X_{ne}^{\e\prime}\e)\\
&=\Hom_{\e \m S^{\prime}}(\m T\times_{\m S}\m S^{\e\prime},\m X^{\e\prime}\e),
\end{align*}
i.e., $\Re_{\,\m S^{\e\prime}\be/\m S}\big(\m X^{\e\prime}\e\big)$ exists (see Definition \ref{wr-for}). Now, by \eqref{for-def}, Theorem \ref{wr-gr} and Proposition \ref{w-lim}, we have
\[
\mr{Gr}^{R}\be\big(\Re_{\e\m S^{\e\prime}\be/\m S}\big(\m X^{\e\prime}\e\big)\big)=
\varprojlim_{n\in\N}\Re_{\le k^{\e\prime}/k}
(\mr{Gr}^{R^{\e\prime}}_{ne}(\m X^{\e\prime}_{ne}))=\Re_{\e k^{\e\prime}\be/k}\big(\mr{Gr}^{R^{\e\prime}}\!\big(\m X^{\e\prime}\e\big)\big),
\]
as claimed.
\end{proof}
								
\begin{remark} As noted in Remark \ref{ns-cat}, $(\text{Ad-For}/\m S)$ contains the category of formal schemes considered in \cite{bert}. Thus the fact that \eqref{for-wr} represents the formal Weil restriction functor on $(\text{Ad-For}/\m S)$ generalizes \cite[Theorem 1.4]{bert}.
\end{remark}

\begin{corollary}
Let $R^{\e\prime}$ be a finite and totally ramified extension of $R$ and let $\m X^{\e\prime}$ be an arbitrary adic formal $\m S^{\prime}$-scheme. Then $\Re_{\e\m S^{\e\prime}\!/\m S}\big(\m X^{\e\prime}\le\big)$ exists and
\[
\mr{Gr}^{R^{\e\prime}}\!\big(\m X^{\e\prime}\le\big)=\mr{Gr}^{R}\be\big(\Re_{\e\m S^{\e\prime}\!/\m S}\big(\m X^{\e\prime}\e\big)\big).
\]
\end{corollary}
\begin{proof} If $\m X^{\e\prime}=\varinjlim_{n\in \N} \m X^{\e\prime}_{n}$ then, by Remark \ref{tram}, $\m X^{\e\prime}_{ne}$ is admissible relative to $S^{\e\prime}_{ne}\to S_{n}$ for every integer $n\geq 1$, where $e$ is the degree of $R^{\e\prime}$ over $R$. Further, $k^{\e\prime}=k$. The corollary is now immediate from the proposition.
\end{proof}

\begin{remark} Recall that, if $R^{\e\prime}\be/R$ is a finite and totally ramified extension of degree $e$ and $\m X^{\e\prime}=\varinjlim \m X^{\e\prime}_{n}$ is an adic formal $\m S^{\e\prime}$-scheme, then Proposition \ref{tot-gr} yields a formula
\[
\mr{Gr}_{ne}^{R^{\le\prime}}\big(\m X^{\e\prime}\le\big)=\mr{Gr}_{n}^{R}\big(\Re_{\le S_{ne}^{\le\prime}/S_{n}}\big(\m X^{\e\prime}\e\big)\big)
\]
for every integer $n\geq 1$, where $\Re_{S_{ne}^{\le\prime}/S_{n}}\big(\m X^{\e\prime}\e\big)\overset{\rm def.}{=}\Re_{S_{ne}^{\le\prime}/S_{n}}(\m X^{\e\prime}_{ne})$. In particular, if $n=1$ above, then $\mr{Gr}_{e-1}^{R^{\le\prime}}\big(\m X^{\e\prime}\le\big)=\Re_{R_{e-1}^{\le\prime}/k}\big(\m X^{\e\prime}\e\big)$, which generalizes \cite[Theorem 4.1]{ns} (see Remark \ref{ns-cat}). Note that the hypothesis ``nice" (i.e., admissible) in the statement of \cite[Theorem 4.1]{ns} is unnecessary.
\end{remark}

\begin{proposition}\label{for-unr3}
Let $k^{\e\prime}/k$ be a subextension of $\e\kbar/k$ and let $R^{\e\prime}$ be the extension of $R$ of ramification index $1$ which corresponds to $k^{\e\prime}/k$. Then, for every adic formal $\m S$-scheme $\m X$, there exists a canonical isomorphism of $k^{\e\prime}$-schemes
\[
\gr(\m X\le)\times_{\spec k}\spec
k^{\le\prime}=\mathrm{Gr}^{R^{\e\prime}}\!\big(\m X\times_{\m S}\m S^{\prime}\e\big),
\]
where $\m S^{\prime}$ is given by \eqref{hat5}.
\end{proposition}
\begin{proof} Write $\m X=\varinjlim \m X_{n}$ as above.
Since $\m X\times_{\m S}\m S^{\prime}=\varinjlim \,(\m X_{n}\times_{S_{n}}{S}^{\e\prime}_{n})$ by \cite[Corollary 1.3.5, p.~267]{fk}, \eqref{grnx} yields $\mathrm{Gr}_{\lbe n}^{R^{\e\prime}}\!\big(\m X\times_{\m S}\m S^{\prime}\big)=\mathrm{Gr}_{\lbe n}^{R^{\e\prime}}\!\big(\m X_{n}\times_{S_{n}}{S}^{\e\prime}_{n}\big)$. Thus, since $\grn(\m X\le)=\grn(\m X_{n}\le)$, Proposition \ref{unr3} yields, for every $n\in\N$, a canonical isomorphism of $k^{\e\prime}$-schemes
\[
\grn(\m X\le)\times_{\spec k}\spec
k^{\e\prime}=\mathrm{Gr}_{\lbe n}^{R^{\e\prime}}\!\big(\m X\times_{\m S}\m S^{\prime}\e\big).
\]
The proposition now follows from \eqref{for-def2} noting that projective limits of schemes commute with base extension.
\end{proof}

The following proposition generalizes \cite[Theorem 3.8]{ns} (see Remark \ref{ns-cat}).
							
\begin{proposition}\label{for-bc} Let $\m X$ be an adic formal $\m S$-scheme and let $R^{\e\prime}$ be a finite extension of $R$ with associated residue field extension $k^{\e\prime}/k\subseteq \kbar/k$. Then there exists a canonical closed immersion of $k^{\e\prime}$-schemes
\[
\gr(\m X\le)\times_{\spec k}\spec
k^{\e\prime}\hookrightarrow\mathrm{Gr}^{R^{\e\prime}}\!\big(\m X\times_{\m S}\m S^{\prime}\e\big).
\]
If $R^{\e\prime}/R$ has ramification index $1$, then the preceding map is an isomorphism.
\end{proposition}
\begin{proof} The second assertion is a particular case of Proposition \ref{for-unr3}. Let $e$ be the ramification index of $R^{\e\prime}$ over $R$ and write $\m X=\varinjlim_{n\in\N} \m X_{n}$. By \cite[Corollary 1.3.5, p.~267]{fk} and \eqref{cart0}, we have
\[
(\m X\times_{\m S}\m S^{\prime}\e)_{ne}=\m X_{ne}\times_{S_{ne}} S^{\e\prime}_{ne}
=\m X_{n}\times_{S_{n}} S^{\e\prime}_{ne}
\]
for every $n\in\N$. Thus, by \eqref{grnx}, Proposition \ref{b-c} yields, for every $n\in\N$, a canonical closed immersion of $k^{\e\prime}$-schemes
\begin{equation}\label{cl-imm}
\mathrm{Gr}_{\lbe n}^{R}(\m X)\times_{\spec k}\spec
k^{\e\prime}\hookrightarrow \mathrm{Gr}_{\lbe ne}^{R^{\le\prime}}\be\big(\m X\times_{\m S}\m S^{\prime}\e\big).
\end{equation}
Now observe that $\big(\mathrm{Gr}_{\lbe n}^{R}(\m X)\times_{\spec k}\spec k^{\e\prime}\big)$ and $\big(\mathrm{Gr}_{\lbe ne-1}^{R^{\le\prime}}\be\big(\m X\times_{\m S}\m S^{\prime}\big)\big)$ are projective systems of $k^{\e\prime}$-schemes with affine transition morphisms, as follows from \cite[II, Proposition 1.6.2(iii)]{ega} and the proof of Proposition-Definition \ref{for-def}. Thus, by \cite[Proposition 3.2(v)]{bga}, we may take projective limits in \eqref{cl-imm}, which yields the proposition.
\end{proof}

If $k$ is perfect of positive characteristic, the composition of \eqref{for-fun2} with the perfection functor \eqref{rapf} yields a functor
\begin{equation}\label{for-pfgr}
\pgr\colon(\text{Ad-For}/\m S)\to (\mathrm{Perf}/k), \quad \m X\mapsto\gr(\m X)^{\pf}.
\end{equation}
Note that, by \eqref{for-def2}, \cite[Proposition 5.21]{bga} and \eqref{grnx}, we have
\begin{equation}\label{pf-proj}
\pgr\big(\m X\e\big)=\varprojlim\pgrn(\m X)
\end{equation}
where $\pgrn$ is the perfect Greenberg functor of level $n$ \eqref{pf-grf} and $\pgrn(\m X)\overset{\rm def.}{=}\pgrn(\m X_{n})$ for every $n\in\N$.

\begin{remark} 
Statements \ref{wr-gr2} to \ref{for-bc} remain valid when $\gr$ is replaced by $\pgr$, provided $\Re_{\e k^{\le\prime}\be/k}$ is replaced by $\Re_{\e k^{\le\prime}\be/k}^{\pf}$ in Proposition \ref{wr-gr2}. The corresponding proofs use \eqref{pf-proj} in place of \eqref{for-def} as well as \cite[Lemma 5.24, Remark 5.18(d) and Proposition 5.17(iv)]{bga}, as in the proofs of Propositions \ref{p-wrgr} to \ref{pf-bc}.
\end{remark}

\section{The Greenberg realization of an $R$-scheme}\label{grsh} 
Let $X$ be an $R$-scheme and let $\widehat{X}=\varinjlim\,(X\!\times_{S}\!S_{n})$ be the formal completion of $X$ along $X\!\times_{S}\!\spec k$. Recall that $\widehat{X}$ is an object of $(\text{Ad-For}/\m S)$. The {\it Greenberg realization of $X$} is the $k$-scheme
\begin{equation}\label{grx2}
\gr(X)\overset{\rm def.}{=}\gr\lbe\big(\widehat X\e\big)=\varprojlim\grn(X),
\end{equation}
where $\grn(X)=\grn(X\!\times_{S}\!S_{n})$ and the transition morphisms of the limit are the change of level morphisms $\varrho_{n,\e X}^{\e i}\colon\grni(X)\to\grn(X)$.

The resulting functor 
\begin{equation*} 
\gr\colon(\mathrm{Sch}/R)\to (\mathrm{Sch}/k), \quad X\mapsto \gr(X),
\end{equation*}
satisfies, by \eqref{for-grs},
\begin{equation*} 
\gr(S)=\gr\lbe\big(\e\widehat{S}\,\big)=\gr\lbe\big(\m S)=\spec k.
\end{equation*}
Note that, in general, $\gr(X)$ is not locally of finite type over $k$, even if $X$ is of finite type over $R$. For example, by \eqref{grx2}, $\gr(\A_{\lbe R}^{\be 1})=\varprojlim\grn(\A_{\lbe R_{n}}^{\be 1})=\varprojlim\s R_{n}=\s R\simeq\A^{\!(\N)}_{\le k}$, which is not locally of finite type.

\smallskip

The following lemma is an analog of Lemma \ref{rat-pts}(i).

\begin{proposition}\label{gr-pts2} Let $X$ be an $R$-scheme and let $A$ be a $k$-algebra which is assumed to be perfect if $R$ is a ring of unequal characteristics. Then $\gr(X)(A)=X(\widetilde{\s R}(A))$.
\end{proposition}
\begin{proof}
This is an instance of the Bhatt-Gabber Algebraization Theorem \cite[Theorem 4.1 and Remarks 4.3 and 4.6]{bha}. \end{proof}

\begin{remark} In the previous (preprint) version of this paper, the proof (but not the statement) of the above proposition contains an error. This error also appears in the proof of the corresponding result of the published version (where it carries the label Proposition 14.2). We claim, at the beginning of the proof, that one can reduce to the affine case using Proposition 3.16 (=Proposition 2.4 of the published version). However, Proposition 3.16 applies to an artinian local ring $R$ and fails in general for discrete valuation rings, e.g, $W(A_f)\neq W(A)_{[f]}$ (indeed,
$(1,1/f,1/f^2,1/f^3,...)$ is in $W(A_f)$ but not in $W(A)_{[f]}$). We thank Takashi Suzuki for pointing out this error.	Fortunately, the statement of the above proposition is correct and is, in fact, a particular case of the Bhatt-Gabber Algebraization Theorem, as pointed out in the new proof above.
\end{remark}

\begin{corollary}\label{r-kbar} Let $X$ be an $R$-scheme which is separated and  locally of finite type. Then $\gr(X)(\e\kbar\e)=X\be(\widehat{R}^{\le{\rm nr}})$.
\end{corollary}
\begin{proof} This follows from \eqref{rk} and the proposition. 
\end{proof}

\begin{proposition} \label{r-prop} Consider, for a morphism of schemes, the property of being:
\begin{enumerate}
\item[(i)] quasi-compact;
\item[(ii)] quasi-separated;
\item[(iii)] separated;
\item[(iv)]  affine; 
\item[(v)] a closed immersion;
\item[(vi)]  an open immersion;
\item[(vii)] formally \'etale.
\end{enumerate}
If $f\colon X\to Y$ is a morphism of $R$-schemes with property $\bm{P}$, then the morphism of $k$-schemes $\gr(\le f\le)\colon\gr(X)\to \gr(Y)$ has property $\bm{P}$ as well.
\end{proposition}
\begin{proof} Each of the properties listed above is stable under base extension. It follows that the morphism of $S_{n}$-schemes $f\times_{S}S_{n}\colon X\times_{S}S_{n}\to Y\times_{S}S_{n}$ has property $\bm{P}$ for every $n\in\N$. 
Now, if $\bm{P}$ denotes one of properties (i)-(v) then, by Remark \ref{rems1}(b) and Proposition \ref{gr-prop}, $\grn(\e f\times_{S}S_{n})\colon
\grn(X)\to\grn(Y)$ has property $\bm{P}$ for every $n\in\N$ and the proposition follows from \eqref{grx2} and \cite[Proposition 3.2]{bga}. If $\bm{P}$ denotes one of properties (vi) or (vii), then $\widehat{f}=\varinjlim\,(f\times_{S}S_{n})\colon\widehat X\to\widehat Y$ has property $\bm{P}$ by Lemmas \ref{for-prop}(iv) and \ref{for-et}. In this case the proposition follows from Proposition \ref{pro-gr}.
\end{proof}

\begin{lemma}\label{wr-fcomp} Let $R^{\e\prime}$ be a finite extension of $R$ and let $X^{\prime}$ be an $R^{\e\prime}$-scheme which is admissible relative to $R^{\e\prime}/R$ (see Definition {\rm \ref{adm}}). Then $\Re_{\e\m S^{\e\prime}/\m S}\big(\e\widehat{X^{\prime}}\e\big)$ and $\Re_{\le R^{\e\prime}\be/R}\big(\lbe X^{\prime}\big)$ exist and
\[
\Re_{\e\m S^{\e\prime}/\m S}\big(\e\widehat{X^{\prime}}\e\big)=\widehat{\Re_{\le R^{\e\prime}\be/R}\big(\lbe X^{\prime}\big)}.
\]
\end{lemma}
\begin{proof} The $R$-scheme $\Re_{\le R^{\e\prime}\be/R}\big(\lbe X^{\prime}\big)$ exists by Theorem \ref{wr-rep}. Now let $e$ be the ramification index of $R^{\e\prime}/R$. Since
$\widehat{X^{\prime}}=\varinjlim\,(X^{\prime}\times_{S^{\e\prime}}S^{\e\prime}_{ne})$ by \eqref{hat4} and $X^{\prime}\times_{S^{\e\prime}}S^{\e\prime}_{ne}$ is admissible relative to $S^{\e\prime}_{ne}\to S_{n}$ for every $n\in\N$ by Lemma \ref{adm2}, $\Re_{\e\m S^{\e\prime}/\m S}\big(\e\widehat{X^{\prime}}\e\big)$ exists by Proposition \ref{wr-gr2}. Further, \eqref{for-wr}, \eqref{wrbc} and \eqref{ram-ten} yield
\begin{align*}
\Re_{\e\m S^{\e\prime}/\m S}\big(\e\widehat{X^{\prime}}\e\big)&=\varinjlim\,\Re_{S^{\e\prime}_{ne}/S_{\le n}}(X^{\prime}\times_{S^{\e\prime}}S^{\e\prime}_{ne})=\varinjlim\,
\Re_{(S^{\e\prime}\times_{S}\e S_{n}\le)/S_{\le n}}(X^{\prime}\times_{S}S_{n})\\
&=\varinjlim\,(\Re_{\le S^{\e\prime}\be/S}\big(\lbe X^{\prime}\big)\times_{S}S_{n})=\widehat{\Re_{\le R^{\e\prime}\be/R}\big(\lbe X^{\prime}\big)},
\end{align*}
as claimed.
\end{proof}

\begin{proposition}\label{r-wr}
Let $R^{\e\prime}$ be a finite extension of $R$ with associated residue field extension $k^{\e\prime}\be/k\subseteq \kbar/k$ and let $X^{\prime}$ be an $R^{\e\prime}$-scheme which is admissible relative to $R^{\e\prime}\be/\lbe R$. Then $\Re_{\le R^{\e\prime}\be/R}\big(X^{\prime}\le\big)$  and $\Re_{\e k^{\e\prime}\be/k}\big(\mr{Gr}^{R^{\e\prime}}\!\big(X^{\prime}\le\big)\big)$ exist and
\[
\mr{Gr}^{R}\be\big(\Re_{\le R^{\e\prime}\be/R}\big( X^{\prime}\e\big)\big)=\Re_{\e k^{\e\prime}\be/k}\big(\mr{Gr}^{R^{\e\prime}}\!\big( X^{\prime}\e\big)\big).
\]
\end{proposition}
\begin{proof} The $R$-scheme $\Re_{\le R^{\e\prime}\be/R}\big(X^{\prime}\le\big)$ exists by Theorem \ref{wr-rep}. Now, as noted in the proof of Lemma \ref{wr-fcomp}, 
$\widehat{X^{\prime}}=\varinjlim\,(X^{\prime}\times_{S^{\e\prime}}S^{\e\prime}_{ne})$ and each $X^{\prime}\times_{S^{\e\prime}}S^{\e\prime}_{ne}$ is admissible relative to $S^{\e\prime}_{ne}\to S_{n}$. Thus
$\Re_{\e k^{\e\prime}\be/k}\big(\mr{Gr}^{R^{\e\prime}}\!\big(X^{\prime}\le\big)\big)=\Re_{\e k^{\e\prime}\be/k}\big(\mr{Gr}^{R^{\e\prime}}\!\big(\e\widehat{X^{\prime}}\e\big)\big)$ exists and 
\[
\Re_{\e k^{\e\prime}\be/k}\big(\mr{Gr}^{R^{\e\prime}}\!\big(X^{\prime}\le\big)\big)=\mr{Gr}^{R}\be\big(\Re_{\e\m S^{\e\prime}\be/\m S}\big(\e\widehat{X^{\prime}}\e\big)\big)
\]
by Proposition \ref{wr-gr2}. The result now follows from \eqref{grx2} and Lemma \ref{wr-fcomp}.
\end{proof}

\begin{proposition}\label{r-unr3}
Let $k^{\e\prime}/k$ be a subextension of $\e\kbar/k$ and let
$R^{\e\prime}$ be the extension of $R$ of ramification index $1$ which corresponds to
$k^{\e\prime}/k$. For every $R$-scheme $X$, there exists a canonical isomorphism of $k^{\e\prime}$-schemes
\[
\gr(X\le)\times_{\spec k}\spec
k^{\e\prime}=\mathrm{Gr}^{R^{\e\prime}}\!\big(X\times_{S}S^{\e\prime}\le\big).
\]
\end{proposition}
\begin{proof} Since $\widehat{X}\times_{\m S}\m S^{\e\prime}=\widehat{X\!\times_{S}\! S^{\e\prime}}$ by
\cite[Corollary 10.9.9, p.~426]{ega1}, this is immediate from Proposition \ref{for-unr3}.
\end{proof}

\begin{proposition}\label{r-bc}
Let $X$ be an $R$-scheme and let $R^{\e\prime}$ be a finite extension of $R$ with associated residue field extension $k^{\e\prime}/k\subseteq \kbar/k$. Then there exists a canonical closed immersion of $k^{\e\prime}$-schemes
\[
\gr(X\le)\times_{\spec k}\spec
k^{\e\prime}\hookrightarrow\mathrm{Gr}^{R^{\e\prime}}\!\big(X\times_{S} S^{\e\prime}\le\big).
\]
If $R^{\e\prime}/R$ has ramification index $1$, then the preceding map is an isomorphism.
\end{proposition}
\begin{proof} The second assertion is a particular case of Proposition \ref{r-unr3}. The first assertion is immediate from Proposition \ref{for-bc} since $\widehat{X}\times_{\m S}\m S^{\e\prime}=\widehat{X\!\times_{S}\! S^{\e\prime}}$ by \cite[Corollary 10.9.9, p.~426]{ega1}.
\end{proof}

The next result applies to commutative $R$-group schemes.

\begin{proposition}\label{ex3} Let  $0\to F\to G\stackrel{\!q}{\to} H \to 0$ be a sequence of commutative $R$-group schemes locally of finite type, where $q$ is smooth and quasi-compact. Assume that the preceding sequence is exact for the fpqc topology on $(\mr{Sch}/R)$. Then the induced sequence of commutative $k$-group schemes
\[
0\to \gr(F)\longrightarrow\gr(G)\to\gr(H)\to 0
\]
is exact for the fpqc topology on $(\mr{Sch}/k)$.
\end{proposition}
\begin{proof} Let $n\in\N$. By \cite[Lemma 2.2]{bga}, $F\simeq\krn q$ and $q$ is surjective. In particular, $q$ is faithfully flat and quasi-compact. Now, by \cite[Corollary 1.3.5, p.~33]{ega1}, $F\times_{S}S_{n}\simeq\krn (q\times_{\be S}\lbe S_{n})$. Thus, since $q\times_{\be S}\lbe S_{n}$ is faithfully flat and quasi-compact, \cite[Lemma 2.3]{bga} shows that the sequence 
\[
0\longrightarrow F\times_{S}S_{n}\longrightarrow G\times_{S}S_{n}\stackrel{\!q\times_{\be S}\lbe S_{n}}{\lra} H\times_{S}S_{n}\longrightarrow 0
\]
is exact for the fpqc topology on $(\mr{Sch}/R_{\le n})$. Note that $F\times_{S}S_{n},G\times_{S}S_{n}$ and $H\times_{S}S_{n}$ are locally of finite type over $S_{n}$ and $q\times_{S}S_{n}$ is smooth. Consequently, by Proposition \ref{ex-green}, the induced sequence of commutative $k$-group schemes locally of finite type
\[
0\longrightarrow \grn(F)\longrightarrow\grn(G)\stackrel{\!\grn\be(q)}\longrightarrow\grn(H)\longrightarrow 0,
\]
where $\grn(q)\overset{\rm def.}{=}\grn(q\times_{S}S_{n})$, is exact for the fpqc topology on $(\mr{Sch}/k)$. Now observe that, since $\grn(q)$ is smooth, quasi-compact and surjective by Propositions \ref{gr-prop}, \ref{gr-sm} and \cite[Lemma 2.2]{bga}, $\grn(q)$ is faithfully flat and quasi-compact. On the other hand, since $F=G\times_{H}S$ is smooth over $S$, the transition morphisms of the system $(\grn(F))$ are surjective by Proposition \ref{sm-surj}. We may now apply \cite[Proposition 3.8]{bga} to complete the proof.
\end{proof}

\begin{lemma}\label{sm-red} If $X$ is a smooth $R$-scheme, then $\gr(X)$ is a reduced $k$-scheme.
\end{lemma}
\begin{proof} Since $X\!\times_{S}\!S_{n}$ is smooth over $S_{n}$ for every $n$, $\grn(X)$ is smooth over $k$ for every $n$ by Corollary \ref{gr-sm2}. Consequently, each 
$\grn(X)$ is reduced and therefore $\gr(X)=\varprojlim \grn(X)$ is reduced as well by \cite[$\text{IV}_{3}$, Proposition 8.7.1]{ega}.
\end{proof}
\smallskip

If $k$ is perfect of positive characteristic and $X$ is an $R$-scheme, the {\it perfect Greenberg realization of $X$} is the perfect $k$-scheme 
\begin{equation}\label{fpr3}
\pgr(X)\overset{\rm def.}{=}\pgr\lbe\big(\widehat X\e\big),
\end{equation}
where $\pgr$ is the functor \eqref{for-pfgr}. Thus, by definitions \eqref{for-pfgr} and \eqref{grx2},
\begin{equation}\label{pgr-pf}
\pgr(X)=\gr(X)^{\pf}.
\end{equation}

\begin{remark} 
Statements \ref{r-wr} to \ref{r-bc} remain valid if $\gr$ is replaced by $\pgr$, provided $\Re_{\e k^{\le\prime}\be/k}$ is replaced by $\Re_{\e k^{\le\prime}\be/k}^{\pf}$ in Proposition \ref{r-wr}. The corresponding proofs make use of \eqref{pgr-pf} and \cite[Lemma 5.24, Proposition 5.17 and Remark 5.18(d)]{bga}.
\end{remark}

For any $R$-scheme $X$, let $X_{\red}$ be the reduced scheme associated to $X$. Thus $X_{\red}$ is a closed subscheme of $X$ and the canonical injection $\iota_{X}\colon X_{\red}\to X$ is a nilimmersion (i.e., a surjective closed immersion \cite[(4.5.16), p.~273]{ega1}). 

\begin{proposition}\label{gr-red} Let $X$ be a separated $R$-scheme locally of finite type. Then the canonical morphism $\gr(\iota_{X})\colon \gr(X_{\red})\to\gr(X)$ is a nilimmersion of $k$-schemes.
\end{proposition}
\begin{proof} By Proposition \ref{r-prop}(v), $\gr(\iota_{X})\colon\gr(X_{\red})\to\gr(X)$ is a closed immersion and we are reduced to checking that $\gr(\iota_{X})$ is surjective. By \cite[$\text{IV}_{1}$, Proposition 1.3.7]{ega}, it suffices to check that $\gr(\iota_{X})(L)\colon\gr(X_{\red})(L)\to\gr(X)(L)$ is surjective for every algebraically closed field $L$ containing $k$. By Lemma \ref{gr-pts2},  the latter map corresponds to the canonical map $X_{\red}(\widetilde{\s R}(L))\to X(\widetilde{\s R}(L))$, which is indeed surjective since $\widetilde{\s R}(L)$ is reduced by Remark \ref{rm-ns}(c).
\end{proof}

\begin{corollary} 
If $k$ is perfect of positive characteristic and $X$ is a separated $R$-scheme locally of finite type, then the canonical morphism $\pgr(\iota_{X})\colon\pgr(X_{\red})\to\pgr(X)$ is an isomorphism of perfect $k$-schemes.
\end{corollary}
\begin{proof} By the proposition,  $\gr(\iota_{X})_{\red}\colon\gr(X_{\red})_{\red}\to\gr(X)_{\red}$ is an isomorphism of $k$-schemes. Consequently, the induced morphism of perfect $k$-schemes $(\gr(X_{\red})_{\red})^{\pf}\to(\gr(X)_{\red})^{\pf}$ is an isomorphism. Since the latter morphism can be identified with $\pgr(\iota_{X})$ by \cite[(5.7)]{bga} and \eqref{pgr-pf}, the corollary follows.
\end{proof}

\begin{corollary} 
If $f\colon X\to Y$ is a nilimmersion of $R$-schemes, where $Y$ is separated and locally of finite type, then the induced morphism $\gr(\le f\le)\colon\gr(X)\to\gr(Y)$ is an nilimmersion of $k$-schemes.
\end{corollary}
\begin{proof} By Proposition \ref{r-prop}(v), it suffices to check that $\gr(\le f\le)$ is surjective. Since $f_{\red}$ is an isomorphism, the composite morphism $X_{\red}\overset{\!\!\iota_{X}}{\to} X\overset{f}{\to} Y$ can be identified with $\iota_{Y}\colon Y_{\red}\to Y$. Thus, by Proposition \ref{gr-red}, the composite $k$-morphism
\[
\gr(X_{\red})\overset{\!\gr(\iota_{X})}{\lra} \gr(X)\overset{\!\gr(\lle f\lle)}{\lra} \gr(Y)
\]
is surjective and therefore so also is $\gr(\le f\le)$.
\end{proof}

\begin{remark} 
Assume that $R$ is a ring of unequal characteristics and recall the ramification index $\ari$ of $R/W(k)$. Let $X$ be an $R$-scheme such that $\Re_{R\le/W(k)}(X)$ exists and let $n\in\N$. In \cite[\S4.1, p.~36]{beg} the author defined the Greenberg realization of level $n$ of $X$ to be
\[
\grbe{}_{n}(X)=\mr{Gr}^{W(k)}_{n}(\Re_{R/\le W(k)}(X)\times_{W(k)} \spec {W_{n}\lbe(k)}).
\]
See also \cite[p.~259, line 5]{adt} (where $\grbe{}_{n}(X)$ is denoted by $\mathcal G_{n}(X)$). By \eqref{wrbc} and \eqref{ram-ten}, we have 
\[
\Re_{R/W(k)}(X)\times_{ W(k)} \spec {W_{n}\lbe(k)}=\Re_{R_{\e
n\lbe\ari}/\e W_{n}(k)}(X\times_{S} S_{ n\lbe\ari}),
\]
whence
\begin{equation}\label{new}
\grbe{}_{\le n}(X)=\mr{Gr}^{W(k)}_{n}(\Re_{R_{\e n\lbe\ari}/\e W_{n}(k)}(X\times_S S_{ n\lbe\ari   } )).
\end{equation}
Note that, since $R/W(k)$ is totally ramified, $\Re_{R_{\e n\lbe\ari   }/W_{n}(k)}(X_{\e n\lbe\ari})$ exists for {\it any} $R$-scheme $X$ by Remark \ref{tram}. Thus \eqref{new} may be taken to be the {\it definition} of $\grbe{}_{ n}(X)$ when $\Re_{R/W(k)}(X)$ fails to exist. Now observe that, if $A$ is an {\it arbitray} $k$-algebra, then 
\[
\grbe{}_{\le n}(X)(A)=X(R\otimes_{\e W(k)}\be W_{\be n}(A)).
\]
Indeed, since $R_{\le n\ari}=R\otimes_{\e W(k)}W_{\be n}(k)$, \eqref{new} and Lemma \ref{rat-pts}(i) show that
\[
\begin{array}{rcl}
\grbe{}_{\le n}(X)(A)=\Re_{R_{\e n\lbe\ari}/\e W_{\lbe n}(k)}(X\times_{S} S_{ n\lbe\ari   })(W_{\be n}(A))&=& X(R_{\le n\ari }\otimes_{\e W_{\lbe n}(k)} W_{\be n}(A))\\
&=& X(R\otimes_{\e W(k)}W_{\be n}(A)),
\end{array}
\]
as claimed. Next, by Proposition \ref{tot-gr}, \eqref{new} may be written as
$\grbe{}_{\le n}(X)=\mr{Gr}^{R}_{\lbe n\lbe\ari }(X)$. It follows that, if 
$\grbe(X)=\varprojlim \grbe{}_{\,n}(X)$ is the object introduced in \cite[\S4.1, p.~36]{beg}, then $\grbe(X)=\gr(X)$, where $\gr(X)$ is the $k$-scheme \eqref{grx2}. Further, if $\uuline{\mr{G}}(X)\overset{\rm def.}{=}\grbe(X)^{\pf}$ is the perfect $k$-scheme considered in [loc.cit.] and $\pgr(X)$ is the object \eqref{fpr3}, then
$\uuline{\mr{G}}(X)=\pgr\lbe(X)$. Regarding the latter functor, [loc.cit., p.~36, line $-11$] contains the  (unproven) claim that, for every perfect $k$-algebra $A$,
\begin{equation*} 
\pgr(X)(A)=X(\e R\be\otimes_{\e W(k)}\! W(A)).
\end{equation*}
The latter is indeed valid and follows from \eqref{pgr-pf}, Proposition \ref{gr-pts2} and \eqref{pf}.  
\end{remark}

\section{Commutative group schemes}\label{comgr}

In this Section $R$ is a complete discrete valuation ring with maximal ideal $\mm$ and residue field $k$ which is assumed to be perfect in the unequal characteristics case. Recall $S=\spec R$ and, for each integer $n\geq 1$, $S_{n}=\spec R_{\le n}$, where $R_{\le n}=R/\mm^{\lbe n}$. 
\medskip

If $G$ is an $R$-group scheme and $n\in \N$, we will write
\[
G_{\lbe(n)}=G\!\times_{S}\!S_{n}.
\]
Further, if $f\colon G\to H$ is a morphism of $R$-group schemes, we will write $f_{(n)}=f\!\times_{S}\be S_{n}\colon G_{(n)}\to H_{(n)}$. If $G$ is commutative and quasi-compact (respectively, of finite type), then by part (i) (respectively, (v)) of Proposition \ref{gr-prop}, $\grn(G\le)=\grn(G_{\lbe(n)})$ is an object of the abelian category $\s C_{\qc}$ (respectively, $\s C_{\alg}$) whose objects are quasi-compact and commutative $k$-group schemes (respectively, commutative $k$-group schemes of finite type). Consequently, by Proposition \ref{r-prop}(i), $\gr(G\le)=\varprojlim \grn(G\le)$ is an object of $\s C_{\qc}$ (in both cases). Recall that the transition morphims in the preceding projective limit are the change of level morphisms  \eqref{clm0}
\begin{equation}\label{rni}
\varrho_{n,\le F}^{\e i}\colon \grni(G\e)\to\grn(G\e) \qquad(n,i\in \N\e).
\end{equation}
Recall also the canonical morphism $\Phi_{\be n,\e G}^{\le\lle i}\colon \mr{Gr}_{i}^{R}\lbe(\e\mathbb V\lbe(\omega_{\le G/R}^{1}))\!\to\! \krn\varrho_{n,\le G}^{\e i}$ \eqref{wj-1}.

Let $F$ be a flat, commutative and separated $R$-group scheme of finite type which has a {\it smooth resolution}, i.e., there exists a sequence of flat, commutative and separated $R$-group schemes of finite type
\begin{equation}\label{sm-res}
0\to F\overset{\!j}{\to} G\overset{\!q}{\to}H\to 0,
\end{equation}
where $G$ and $H$ are smooth, $q$ is faithfully flat and $j$ is a closed immersion which identifies $F$ with the scheme-theoretic kernel of $q$. Note that $q$ is an fppf morphism and, by \cite[Lemma 2.3]{bga}, the sequence \eqref{sm-res} is an exact sequence of sheaves for the fppf topology on $(\mr{Sch}/S\e)$. If $F$ is {\it finite} over $S$, then $F$ has a smooth resolution \eqref{sm-res} by \cite[Proposition 5.1(i) and its proof, pp.~217-218]{mr}. See also \cite[\S 2.2, pp.~25-27]{beg}, where a standard smooth resolution of such an $F$ is constructed. 

\smallskip

\begin{proposition}\label{vker2} Let $n\geq 1$ and $i\geq 1$ be integers. Then
\begin{enumerate}
\item[(i)] $\krn\varrho_{n,\e F }^{\e i}$ is a unipotent $k$-group scheme of finite type.
\item[(ii)] If $1\leq i\leq n$, $\Phi_{\be n,\e F}^{\le\lle i}\colon \mr{Gr}_{i}^{R}\lbe(\le\mathbb V\lbe(\omega_{\le F/R}^{1}))\!\to\! \krn\varrho_{n,\le F}^{\le i}$ \eqref{wj-1} is a morphism of unipotent $k$-group schemes of finite type whose kernel and cokernel are unipotent and infinitesimal.
\item[(iii)] The morphism $\Phi_{\be n,\e F}^{\le\lle i}$ \eqref{wj-1} is an isomorphism if $R$ is an equal characteristic ring or if $R$ is a ring of unequal characteristics and $n+i\leq \ari=v(p)$.
\end{enumerate}
\end{proposition}

\begin{proof}  Since $\mr{Gr}_{(-)}^{R}$ is a left-exact functor, \eqref{sm-res} induces an exact and commutative diagram in $\s C_{\alg}$
\[
\xymatrix{
0\ar[r]&\mr{Gr}_{n+i}^{R}(F\le)\ar[d]^{\varrho_{n\lbe,\le F}^{\le \le i}}\ar[rr]^{\mr{Gr}_{n+i}^{R}(\e j)}&&\mr{Gr}_{n+i}^{R}(G\le)\ar[d]^{\varrho_{n\lbe,\le G}^{\le \le i}}\ar[rr]^{\mr{Gr}_{n+i}^{R}(q)}&&\mr{Gr}_{n+i}^{R}(H\le)\ar[d]^{\varrho_{n\lbe,\le H}^{\le \le i}}\\
0\ar[r]&\mr{Gr}_{\lbe n}^{R}(F\le)\ar[rr]^{\mr{Gr}_{n}^{R}(\e j)}&&\mr{Gr}_{\lbe n}^{R}(G\le)\ar[rr]^{\mr{Gr}_{n}^{R}(q)}&&\mr{Gr}_{\lbe n}^{R}(H\le).
}
\]
The above diagram yields an exact sequence in $\s C_{\alg}$
\begin{equation}\label{vs}
0\to \krn\varrho_{n,\e F }^{\e i}\overset{\!\be\alpha}{\to} \krn\varrho_{n,\e G }^{\e i}\overset{\!\beta}{\to} \krn\varrho_{n,\e H }^{\e i},
\end{equation}
where we have written $\alpha$ and $\beta$ for the restrictions of $\mr{Gr}_{n+i}^{R}(\e j)$ and $\mr{Gr}_{n+i}^{R}(q)$ to $\krn\varrho_{n,\e F }^{\e i}$ and $\krn\varrho_{n,\e G }^{\e i}$, respectively. Since $\krn\varrho_{n,\e G }^{\e i}$ and $\krn\varrho_{n,\e H }^{\e i}$
are unipotent and of finite type by Proposition \ref{sm-vker}, assertion (i) is clear. Recall now from \S \ref{vb} that, for every $S$-group scheme $T$, the $S$-scheme $\mathbb V(\omega_{T\lbe/\lbe S}^{1})$  represents the functor $\underline{\rm{Lie}}(T/S)$. Thus, by \cite[Proposition 1.1(a), p.~459]{llr} and the left exactness of the functor $\mr{Gr}_{i}^{R}$, the sequence \eqref{sm-res} induces a sequence  
\begin{equation}\label{vs2}
0\to \mr{Gr}_{i}^{R}(\le\mathbb V(\omega_{\le F /R }^{1})) \to \mr{Gr}_{i}^{R}(\le\mathbb V(\omega_{\le G /R }^{1}))\to \mr{Gr}_{i}^{R}(\le\mathbb V(\omega_{\le H /R }^{1}))
\end{equation}
which is exact in $\s C_{\alg}$. 
Now observe that, since $G$ and $H$ are smooth, \eqref{obv} yields isomorphisms of $k$-group schemes $\mr{Gr}_{i}^{R}(\le\mathbb V(\omega_{\le G/R }^{1}))\simeq \G_{\lbe a,\le k}^{i\le d}$ and $\mr{Gr}_{i}^{R}(\le\mathbb V(\omega_{\le H/R }^{1}))\simeq \G_{\lbe a,\le k}^{i\le f}$, where $d=\dim G_{\mr s}$ and $f=\dim H_{\mr s}$. In particular, \eqref{vs2} implies that $\mr{Gr}_{i}^{R}(\le\mathbb V(\omega_{\le F /R }^{1}))$ is a unipotent $k$-group scheme. We now assume that $1\leq i\leq n$ and consider the exact and commutative diagram in $\s C_{\alg}$
\[
\xymatrix{0\ar[r]& \mr{Gr}_{i}^{R}\lbe(\le\mathbb V(\omega_{\le F /R }^{1}))\ar[r]\ar[d]^{\Phi_{\!n,\e F}^{\le i}} & \mr{Gr}_{i}^{R}(\le\mathbb V(\omega_{\le G /R }^{1}))\ar[r]\ar@{->>}[d]^{\Phi_{\!n,\e G}^{\le i}}& \mr{Gr}_{i}^{R}(\le\mathbb V(\omega_{\le H /R }^{1}))\ar@{->>}[d]^{\Phi_{\!n,\e H}^{\le i}}\\
0\ar[r]& \krn\varrho_{n,\e F }^{\e i}\ar[r] & \krn\varrho_{n,\e G }^{\e i} \ar[r]& \krn\varrho_{n,\e H }^{\e i},
}
\]
whose top (respectively, bottom) row is the sequence \eqref{vs2} (respectively, \eqref{vs}).
An application of the snake lemma to the preceding diagram yields an exact sequence in $\s C_{\alg}$
\begin{equation*} 
0\to \krn\e\Phi_{\!n,\e F}^{\le i}\to \krn\e\Phi_{\!n,\e G}^{\le i}\to\left(\krn\e\Phi_{\!n,\lle H}^{\le i}\lbe\right)^{\prime}\to \cok \Phi_{\!n,\e F}^{\le i}\to 0
\end{equation*}
for some $k$-subgroup scheme $\left(\krn\e\Phi_{\!n,\lle H}^{\le i}\lbe\right)^{\prime}$ of $\krn\e\Phi_{\!n,\e H}^{\le i}$. 
Since $\krn\e\Phi_{\!n,\e G}^{\le i}$ and $\krn\e\Phi_{\!n,\e H}^{\le i}$ are unipotent and infinitesimal $k$-group schemes by Propositions \ref{vker1} and \ref{vker}, the preceding sequence and Remark \ref{inf}(a) show that $\krn\e\Phi_{\!n,\e F}^{\le i}$ and $\cok \Phi_{\!n,\e F}^{\le i}$ are unipotent and infinitesimal as well, and trivial when the hypotheses of (iii) hold.
\end{proof}

Now, if $n\in\N$, the sequence induced by \eqref{sm-res}
\begin{equation}\label{sm-resn}
0\lra F_{(n)}\overset{\!j_{(n)}}{\lra} G_{(n)}\overset{\!q_{(n)}}{\lra} H_{(n)}\lra 0
\end{equation}
is a smooth resolution of the $S_{n}$-group scheme $F_{(n)}$. In particular,
\begin{equation}\label{sm-res0}
0\to F_{\be\rm s}\to G_{\lbe\rm s}\to H_{\lbe\rm s}\to 0
\end{equation}
is a smooth resolution of the special fiber $F_{\be\rm s}$ of $F$. Note that, since $G_{(n)}$ is smooth over $S_{n}$, $\HH^{1}_{\fppf}(\Rnnr,G\le)\overset{{\rm def.}}{=}\HH^{1}_{\fppf}(\Rnnr\be,G_{(n)}\!\times_{S_{n}}\!\Snnr)=0$ by Lemma \ref{coh=0}. Thus \eqref{sm-resn} induces an exact sequence of abelian groups
\begin{equation}\label{c-seq}
0\lra F(\Rnnr)\to G(\Rnnr)\to H(\Rnnr)\to \HH^{1}_{\fppf}(\Rnnr,F\le)\to 0.
\end{equation}

Let
\begin{equation}\label{sm-res2}
0\to F\overset{\be j^{\prime}}{\to} G^{\e\prime}\overset{\!q^{\e\prime}}{\to}H^{\e\prime}\to 0
\end{equation}
be a second smooth resolution of $F$ and let
\begin{equation}\label{sm-resn-p}
0\lra F_{(n)}\overset{\! j^{\e\prime}_{(n)}}{\lra} G^{\e\prime}_{(n)}\overset{\!q^{\e\prime}_{(n)}}{\lra} H^{\e\prime}_{(n)}\lra 0
\end{equation}
be the induced smooth resolution of $F_{(n)}$, where $n\in\N$. Since $j_{(n)}$ \eqref{sm-resn} and $j_{(n)}^{\e\prime}$ are closed immersions,
\[
j_{(n)}\be\times_{\be S_{n}}\!\!\big(\!-j_{(n)}^{\e\prime}\big)\colon F_{(n)}\times_{S_{n}}F_{(n)}\to G_{(n)}\times_{S_{n}}G_{(n)}^{\e\prime}
\]
is a closed immersion as well. Now, since $F_{(n)}$ is separated over $S_{(n)}$, the diagonal morphism $\Delta\colon F_{(n)}\to F_{(n)}\times_{S_{n}}F_{(n)}$ is also a closed immersion. We conclude that the composite morphism
\[
F_{(n)}\overset{\Delta}{\lra} F_{(n)}\times_{S_{n}}F_{(n)}\to G_{(n)}\times_{S_{n}}G_{\lbe(n)}^{\e\prime}
\]
is a closed immersion. We will regard $F_{(n)}$ as a closed $S_{n}$-subgroup scheme of $G_{(n)}\times_{S_{n}}G_{(n)}^{\e\prime}$ via the above morphism. Since $F_{(n)}$ is flat and $G_{(n)}\times_{S_{n}}G_{(n)}^{\e\prime}$ is of finite type over $R_{\le n}$, \cite[${\rm{VI}}_{\rm A}$, Theorem 3.3.2]{sga3} shows that the pushout of $j_{(n)}$ and $j^{\e\prime}_{(n)}$ in the abelian category $\mr{Ab}(\mr{Sch}/\be R_{\le n}\le)_{\fppf}$, i.e., 
\[
P\overset{\rm def.}{=}\big(G_{(n)}\times_{S_{n}}G_{(n)}^{\e\prime}\e\big)/F_{(n)},
\]
is (represented by) a separated, smooth and commutative $R_{\le n}$-group scheme of finite type. By the universal property of the pushout \cite[pp.~65-66]{mac} and the fact that the functor \eqref{fufa} for $\mathcal C=(\mr{Sch}/\be R_{\le n}\le)$ and $\tau=\fppf$ is fully faithful, there exist morphisms of $R_{\le n}$-group schemes $v\colon P\to H_{(n)}^{\le\prime}$ and $u\colon P\to H_{(n)}$ such that the following diagram in $\mr{Ab}(\mr{Sch}/R_{\le n}\e)_{\fppf}$ (which consists of flat, separated and commutative $R_{\le n}$-group schemes of finite type) is exact and commutative
\[
\xymatrix{&0\ar[d]&0\ar[d]&&\\
0\ar[r]& F_{(n)}\ar[d]^(.48){j_{(n)}^{\lle\prime}}\ar[r]^(.48){j_{(n)}}& G_{\lbe (n)}\ar[r]^{q_{(n)}}\ar[d]&
H_{(n)}\ar[r]\ar@{=}[d]& 0\\
0\ar[r]& G_{\lbe(n)}^{\e\prime}\ar[d]^(.45){q_{(n)}^{\e\prime}}\ar[r]& P\ar[r]^(.43){u}\ar[d]^(.45){v}& H_{(n)}\ar[r]& 0\\
& H_{(n)}^{\e\prime}\ar[d]\ar@{=}[r]&H_{(n)}^{\le\prime}\ar[d]&&\\
&0&0.&&
}
\]
Note that, since $G_{\lbe(n)}^{\e\prime}$ and $G_{\lbe(n)}$ are smooth over $S_{n}$, $u$ and $v$ are smooth morphisms. Now, by Proposition \ref{ex-green}, the middle column of the preceding diagram induces an exact sequence in $\mr{Ab}(\mr{Sch}/k\e)_{\fppf}$
\begin{equation*} 
0\lra\grn(G)\lra\grn(P)\overset{\!\!\grn\lbe\lbe(v\le)}{\lra}\grn(H^{\le\prime}\le)\lra 0.
\end{equation*}
Thus the bottom half of the above diagram induces the following exact and commutative diagram in $\mr{Ab}(\mr{Sch}/k\e)_{\fppf}\e$:
\begin{equation}\label{sdiag}
\xymatrix{
0\ar[r]&\grn(G^{\e\prime})\ar[d]^{\grn\lbe\lbe(q^{\le\prime}\le)}\ar[r]&\grn(P)\ar[r]^(.45){\grn\lbe\lbe(u\le)}
\ar@{->>}[d]^(.45){\grn\lbe\lbe(v\le)}&\grn(H)\ar[d]\ar[r]& 0\\
0\ar[r]&\grn( H^{\e\prime})\ar@{=}[r]&\grn( H^{\e\prime})\ar[r]&0\ar[r]&0\e,
}
\end{equation}
where the top row is exact by Proposition \ref{ex-green}, $\grn\lbe\lbe(q^{\e\prime}\le)=\grn\be\big(q_{(n)}^{\e\prime}\le\big)$, $\krn\e\grn\lbe\lbe(q^{\e\prime}\le)=\grn(F)$ and $\krn\e\grn\lbe\lbe(v\le)=\grn(G)$. Now an application of the snake lemma to \eqref{sdiag} yields the following exact sequence in $\mr{Ab}(\mr{Sch}/k\e)_{\fppf}\e$:
\begin{equation}\label{fgh}
0\lra\grn(F)\lra\grn(G)\overset{\!\!\grn\lbe\lbe(q\le)}{\lra}\grn(H\le)\lra \cok \grn\lbe\lbe(q^{\e\prime}\le)\lra 0.
\end{equation}
We conclude that there exists an isomorphism in $\s C_{\alg}$:
\[
\cok \grn\lbe\lbe(q\le)\overset{\!\sim}{\to}\cok \grn\lbe\lbe(q^{\e\prime}\le).
\]
Thus the commutative $k$-group scheme of finite type 
\begin{equation}\label{h1n}
\mathcal H^{1}\be(R_{\le n},F\le)\overset{\rm def.}{=}\cok\grn\lbe\lbe(q\le)
\end{equation}
is independent, up to isomorphism, of the choice of smooth resolution \eqref{sm-res}.  We will show in Lemma \ref{ucalg}(i) below that $\mathcal H^{1}\be(R_{\le n},F\le)(\e\kbar\e)=\HH^{\le 1}_{\fppf}(\Rnnr,F\le)$, which explains our choice of notation in \eqref{h1n}. Note however that, in general, $\mathcal H^{\le 1}\be(R_{\le n},F\le)(k)\neq \HH^{\le 1}_{\fppf}(R_{\le n},F\le)$, as Remark \ref{count-h} below shows.

\smallskip

\begin{remark} 
Using, respectively, \eqref{sm-res}, \eqref{sm-res2}, Proposition \ref{ex3} and \cite[Theorem 4.C, p.~53]{an} in place of \eqref{sm-resn}, \eqref{sm-resn-p}, Proposition \ref{ex-green} and \cite[${\rm{VI}}_{\rm A}$, Theorem 3.3.2]{sga3}, we derive the existence of an isomorphism in $\s C_{\qc}$:
\[
\cok\gr\lbe\lbe(q\le)\overset{\!\sim}{\to}\cok \gr\lbe\lbe(q^{\e\prime}\le).
\]
Consequently, the commutative and quasi-compact $k$-group scheme $\cok\gr\lbe\lbe(q\le)$ is independent, up to isomorphism, of the choice of smooth resolution \eqref{sm-res}.
\end{remark}

\smallskip
								
Now observe that, since the morphism $\grn(H\le)\to \cok \grn\lbe\lbe(q^{\e\prime}\le)$ in \eqref{fgh} is faithfully flat \cite[${\rm{VI}}_{\rm A}$, Proposition 5.4.1]{sga3} and $\grn(H\le)$ is smooth, $\mathcal H^{1}\be(R_{\le n},F\e)\simeq\cok \grn\lbe\lbe(q^{\e\prime}\le)$ is smooth as well by Lemma \ref{con-sm}. Thus Proposition \ref{ex-ex} shows that \eqref{fgh} is also exact in $\s C_{\alg}$. Consequently, \eqref{fgh} induces an exact sequence in $\s C_{\alg}$
\begin{equation}\label{fghh}
0\to\grn(F)\to\grn(G)\to\grn(H\le)\to \mathcal H^{1}\be(R_{\le n},F\le)\to 0.
\end{equation}
In particular, the morphism $\grn(H\le)\to \mathcal H^{1}\be(R_{\le n},F\le)$ is faithfully flat. Further, the preceding sequence, together with \eqref{sm-res0}, yields
\begin{equation}\label{h1=0}
\mathcal H^{\le 1}\be(R_{\e 1},F\le)=\mathcal H^{1}\be(k,F\le)=0.
\end{equation}
Now, by the exactness of \eqref{rho-sm} and Proposition \ref{ex-ex}, the following sequences are exact in $\s C_{\alg}$ for every pair of integers $r,i\geq 1$: 
\begin{equation}\label{rho-sm-g}
0\lra\krn\varrho_{r,\le G}^{\e i}\lra\mr{Gr}_{r+i}^{R}(G\le)\overset{\!\varrho_{r,\le G}^{\le i}}{\lra} \mr{Gr}_{r}^{R}(G\le)\lra 0
\end{equation}
and
\begin{equation}\label{rho-sm-h}
0\lra\krn\varrho_{r,\le H}^{\e i}\lra\mr{Gr}_{r+i}^{R}(H\le)\overset{\!\varrho_{r,\le H}^{\le i}}{\lra} \mr{Gr}_{r}^{R}(H\le)\lra 0.
\end{equation}
Let
\begin{equation*} 
\overline{\varrho}_{r,\le G}^{\, i}\colon \mr{Gr}_{r+i}^{R}(G\le)/\mr{Gr}_{r+i}^{R}(F\le)\to \mr{Gr}_{r}^{R}(G\le)/\mr{Gr}_{r}^{R}(F\le)
\end{equation*}
be the morphism induced by $\varrho_{r,\le G}^{\e i}$ and consider the following exact and commutative diagrams in $\s C_{\alg}$:
\[
\xymatrix{
0\ar[r]&\mr{Gr}_{r+i}^{R}(F\le)\ar[d]^{\varrho_{r,F}^{\le \le i}}\ar[r]&\mr{Gr}_{r+i}^{R}(G\le)\ar[r]\ar@{->>}[d]^{\varrho_{r,G}^{i}}&\mr{Gr}_{r+i}^{R}(G\le)/\mr{Gr}_{r+i}^{R}(F\le)\ar@{->>}[d]^(.48){\overline{\varrho}_{r,\le G}^{\, i}}\ar[r]& 0\\
0\ar[r]&\mr{Gr}_{r}^{R}(F\le)\ar[r]&\mr{Gr}_{r}^{R}(G\le)\ar[r]&\mr{Gr}_{r}^{R}(G\le)/\mr{Gr}_{r}^{R}(F\le)\ar[r]&0
}
\]
and
\[
\xymatrix{
0\ar[r]&\mr{Gr}_{r+i}^{R}(G\le)/\mr{Gr}_{r+i}^{R}(F\le)\ar@{->>}[d]^(.48){\overline{\varrho}_{r,\le G}^{\, i}}\ar[r]&\mr{Gr}_{r+i}^{R}(H\le)\ar[r]\ar@{->>}[d]^{\varrho_{r,H}^{\le i}}&\mathcal H^{\le 1}\be(R_{\e r+i},F\e)\ar[d]\ar[r]& 0\\
0\ar[r]&\mr{Gr}_{r}^{R}(G\le)/\mr{Gr}_{r}^{R}(F\le)\ar[r]&\mr{Gr}_{r}^{R}(H\le)\ar[r]&\mathcal H^{1}\be(R_{\e r},F\e)\ar[r]& 0\e,
}
\]
where the rows of the second diagram come from \eqref{fghh} and the vertical morphisms $\varrho_{r,G}^{\le i}$ and $\varrho_{r,H}^{\le i}$ have trivial cokernel by the exactness of \eqref{rho-sm-g} and \eqref{rho-sm-h}. From the first diagram, $\overline{\varrho}_{r,\le G}^{\, i}$ has trivial cokernel as well. Now an application of the snake lemma to the preceding diagrams yields the following exact sequences in $\s C_{\alg}$:
\begin{equation}\label{kseqm}
0\to\krn\varrho_{r,\le F}^{\e i}\to \krn\varrho_{r,\le G}^{\e i}\to \krn\e\overline{\varrho}_{r,\le G}^{\e i}\to\cok\lbe\varrho_{r,\le F}^{\e i}\to 0
\end{equation}
and
\begin{equation}\label{hmi}
0\to\krn\e\overline{\varrho}_{r,\le G}^{\e i}\to \krn\varrho_{r,\le H}^{\e i}\to \mathcal H^{1}\be(R_{\e r+i},F\le)\to\mathcal H^{\le 1}\be(R_{\e r},F\le)\to 0.
\end{equation}
Note that the last nontrivial morphisms in the preceding sequences are faithfully flat.

\begin{lemma}\label{u-cok} For every $n\in\N$, $\mathcal H^{\le 1}\be(R_{\le n},F\le)$ is a smooth, commutative, connected and unipotent $k$-group schemes.
\end{lemma}
\begin{proof} Smoothness was shown above and commutativity is obvious. Now set $i=n$ and $r=1$ in \eqref{hmi} and use \eqref{h1=0} to obtain the following exact sequence in $\s C_{\alg}$:
\[
0\to\krn\e\overline{\varrho}_{\le 1,\le G}^{\e n}\to \krn\varrho_{\le 1,\le H}^{\e n}\to \mathcal H^{1}\be(R_{\le n},F\le)\to 0.
\]
Since the right-hand nontrivial morphism above is faithfully flat and $\krn\varrho_{\le 1,\le H}^{\e n}$ is connected and unipotent by Proposition \ref{sm-vker}, $\mathcal H^{\le 1}\be(R_{\le n},F\le)$ is connected and unipotent as well by Lemma \ref{con-sm} and \cite[${\rm{XVII}}$, Proposition 2.2(iii)]{sga3}.
\end{proof}

\begin{lemma}\label{ucalg} Let $n\geq 1$ be an integer.
\begin{enumerate}
\item[(i)] $\mathcal H^{\le 1}\be(R_{\le n},F\le)(\e\kbar\e)=\HH^{1}_{\fppf}(\Rnnr,F\le)$.
\item[(ii)] Let $k^{\e\prime}\be/k$ be a subextension of $\e\kbar/k$ and let $R^{\e\prime}$ be the extension of $R$ of ramification index $1$ which corresponds to $k^{\e\prime}\be/k$. Then there exists a canonical isomorphism of $k^{\e\prime}$-group schemes
\[
\mathcal H^{\le 1}\be(R_{\le n},F\le)\times_{\spec k}\spec k^{\e\prime}= 
\mathcal H^{\le 1}\be(R_{\le n}^{\e\prime},F\be\times_{S}\be S^{\e\prime}\e).
\]
\end{enumerate} 
\end{lemma}
\begin{proof} Since the morphism $\grn(H\le)\to \mathcal H^{\le 1}\be(R_{\le n},F\le)$ appearing in \eqref{fghh} is surjective and of finite type, the induced morphism $\grn(G)\lbe\big(\e\kbar\e\big)\to \mathcal H^{\le 1}\be(R_{\le n},F\le)(\e\kbar\e)$ is surjective by \cite[$\text{IV}_{1}$, Proposition 1.3.7]{ega}. Thus \eqref{fghh} and Lemma \ref{rat-pts}(ii) yield an exact sequence of abelian groups
\[
0\to F(\Rnnr)\to G(\Rnnr)\to H(\Rnnr)\to \mathcal H^{\le 1}\be(R_{\le n},F\le)(\e\kbar\e)\to 0.
\]
Assertion (i) now follows from \eqref{c-seq}. The smooth resolution
\eqref{sm-resn}$\be\times_{S_{n}}\be S_{n}^{\e\prime}$ induces the following exact sequence of commutative $k^{\e\prime}$-group schemes of finite type (which is similar to \eqref{fghh}):
\[
0\to\mr{Gr}_{n}^{R^{\le\prime}}\!(F\!\times_{S_{n}}\be\!S_{n}^{\e\prime})\to\mr{Gr}_{n}^{R^{\le\prime}}\!(G\!\times_{S_{n}}\!\be S_{n}^{\e\prime})\to\mr{Gr}_{n}^{R^{\le\prime}}\!(H\!\times_{S_{n}}\!\be S_{n}^{\e\prime})\to \mathcal H^{\le 1}\lbe\lbe(R_{\le n}^{\e\prime},F\be\times_{S}\be S^{\e\prime}\e)\to 0.
\]
Now, by Proposition \ref{unr3}, $\mr{Gr}_{n}^{R^{\le\prime}}\!(F\!\times_{S_{n}}\!S_{n}^{\e\prime})=\grn(F\le)\times_{\spec k}\spec k^{\e\prime}=\grn(F\le)_{k^{\e\prime}}$ for every $S_{n}$-scheme $F$, whence the preceding sequence may be written as
\[
0\to\grn(F\le)_{k^{\e\prime}}\to\grn(G\le)_{k^{\e\prime}}\to\grn(H\le)_{k^{\e\prime}}\to \mathcal H^{1}\lbe\lbe(R_{\le n}^{\e\prime},F\be\times_{S}\be S^{\e\prime}\e)\to 0.
\]
Assertion (ii) of the lemma now follows by comparing the above sequence with the sequence \eqref{fghh}$\times_{\spec k}\,\spec k^{\e\prime}$.
\end{proof}

\begin{remark}\label{count-h} The formula in part (i) of the lemma fails if $\kbar$ and $\Rnnr$ are replaced with $k$ and $R_{\le n}$ (respectively) and $k$ is not algebraically closed. Indeed, if $n=1$, then $\mathcal H^{\le 1}\be(R_{\le 1},F\le)(k)=0$ by \eqref{h1=0}, whereas $\HH^{1}_{\fppf}(R_{\le 1},F\le)=\HH^{\le 1}_{\fppf}(k,F\le)$ is not zero in general if $k$ is not algebraically closed.
\end{remark}

Recall that every unipotent $k$-group scheme of finite type is affine over $k$ \cite[${\rm{XVII}}$, Proposition 2.1]{sga3}. As noted above, the transition morphisms of the projective system of affine $k$-schemes $(\mathcal H^{1}\be(R_{\le n},F\e))$ are faithfully flat. Consequently, by \cite[$\text{IV}_{3}$, Propositions 8.2.3 and 8.3.8(ii)]{ega},
\begin{equation*} 
\mathcal H^{\le 1}\lbe\lbe(R,F\e)\overset{\rm def.}{=}\varprojlim \mathcal H^{\le 1}\lbe\lbe(R_{\le n},F\e)
\end{equation*}
is an object of $\s C_{\qc}$ and every projection morphism $\mathcal H^{1}\be(R,F\le)\to \mathcal H^{\le 1}\lbe\lbe(R_{\le n},F\e)$ is faithfully flat. Since projective limits commute with base extension, Lemma \ref{ucalg}(ii) shows that, if $k^{\e\prime}\be/k$ is a subextension of $\e\kbar/k$ and $R^{\e\prime}/R$ is the extension of ramification index 1 which corresponds to $k^{\e\prime}/k$, then there exists a canonical isomorphism of $k^{\e\prime}$-group schemes
\begin{equation*} 
\mathcal H^{1}\be(R,F\e)\times_{\spec k}\spec k^{\e\prime}= 
\mathcal H^{1}\be\big(\widehat{R}^{\e\prime},F\be\times_{S}\be \widehat{S}^{\e\prime}\e\big).
\end{equation*}

\begin{lemma}\label{ucok-i} The $k$-group scheme $\mathcal H^{\le 1}\lbe\lbe(R,F\le)$ is affine, commutative, reduced and connected.
\end{lemma}
\begin{proof} Since each $k$-group scheme $\mathcal H^{\le 1}\lbe\lbe(R_{\le n},F\e)$ is reduced and connected by Lemma \ref{u-cok}, $\mathcal H^{\le 1}\lbe\lbe(R,F\le)$ is reduced and connected as well by \cite[$\text{IV}_{3}$, Propositions 8.4.1(ii) and 8.7.1]{ega}. Now, since $\mathcal H^{\le 1}\lbe\lbe(R_{\le n},F\e)$ is an affine scheme for every $n$ and the projection morphism $\mathcal H^{\le 1}\lbe\lbe(R,F\le)\to\mathcal H^{\le 1}\lbe\lbe(R_{\le n},F\e)$ is affine by \cite[$\text{IV}_{3}$, (8.2.2)]{ega}, $\mathcal H^{\le 1}\lbe\lbe(R,F\le)$ is also an affine scheme by \cite[II, Corollary 1.3.4]{ega}.
\end{proof}

For every $n\in\N$, \eqref{fghh} induces the following exact sequences in $\s C_{\alg}$:
\begin{equation}\label{sh1}
0\lra\grn(F)\lra\grn(G\le)\overset{\! f_{n}}{\lra}\grn(G\le)/\grn(F\le)\lra 0
\end{equation}
and
\begin{equation}\label{sh2}
0\lra\grn(G\le)/\grn(F\le)\overset{\! j_{\le n}}{\lra}\grn(H)\lra \mathcal H^{\le 1}\lbe\lbe(R_{\le n},F\e)\lra 0,
\end{equation}
where the canonical projection morphism $f_{n}$ is faithfully flat and $j_{n}$ is a closed immersion \cite[${\rm{VI}}_{\rm A}$, Proposition 5.4.1]{sga3}. Note that $j_{n}\be\circ\! f_{n}=\grn(q)$.

\begin{lemma}\label{aff-im} The transition morphisms of the projective system $(\grn(G\le)/\grn(F\le))$ are affine and surjective.
\end{lemma}
\begin{proof} Note that each morphism $\grnl(G\le)/\grnl(F\le)\to \grn(G\le)/\grn(F\le)$ is
induced by the change of level morphism $\varrho_{n,\le G}^{\e 1}$ \eqref{rni}. The sequence \eqref{sh1} induces a commutative diagram in $\s C_{\alg}$
\[
\xymatrix{\grnl(G\le)\ar[d]^{\varrho_{n,\le G}^{\e 1}}\ar[r]^(.37){f_{n+1}}& \grnl(G\le)/\grnl(F\le)\ar[d]\\
\grn(G\le)\ar[r]^(.45){f_{n}}&\grn(G\le)/\grn(F\le),}
\]
where the horizontal morphisms are faithfully flat and the left-hand vertical morphism is surjective by Proposition \ref{sm-surj}. It is now clear that the right-hand vertical morphism in the above diagram is surjective. On the other hand, \eqref{sh2} induces the following commutative diagram in $\s C_{\alg}$:
\[
\xymatrix{\grnl(G\le)/\grnl(F\le)\ar[d]\ar[r]^(.63){j_{\le n+1}}& \grnl(H\le)\ar[d]^{\varrho_{n,\le H}^{\e 1}}\\
\grn(G\le)/\grn(F\le)\ar[r]^(.6){j_{n}}& \grn(H\le),}
\]
where the horizontal morphisms are closed immersions and the right-hand vertical morphism is affine by Proposition \ref{aff}. Since closed immersions are affine \cite[II, Proposition 1.6.2(i)]{ega}, the left-hand vertical morphism in the above diagram is affine as well, by a combination of \cite[II, Proposition 1.6.2, (ii) and (v)]{ega} and \cite[Proposition 9.1.3, p.~354]{ega1}. This completes the proof.
\end{proof}

The lemma shows that

\begin{equation}\label{grp}
\gr(G)^{\prime}\overset{\rm def.}{=}\varprojlim \grn(G\le)/\grn(F\le)
\end{equation}
is an object of $\s C_{\qc}$. By \eqref{sh1}, there exists an exact sequence in $\s C_{\qc}$
\begin{equation}\label{lex}
0\to\gr\lbe(F\le)\to\gr(G\le)\overset{\! f}{\to}\gr(G)^{\prime},
\end{equation}
where $f=\varprojlim f_{n}$. 

\begin{proposition} 
There exists an exact sequence in $\s C_{\qc}$
\[
0\to\gr(G)^{\prime}\to\gr(H\le)\to \mathcal H^{\le 1}\lbe\lbe(R,F\e)\to 0,
\]
where $\gr(G)^{\e\prime}$ is given by \eqref{grp}.
\end{proposition}
\begin{proof} By Lemmas \ref{aff-im} and \ref{qc-alg}, the sequence of
projective systems
\[
0\to(\grn(G\le)/\grn(F\le))\to(\grn(H\le))\to (\mathcal H^{\le 1}\lbe\lbe(R_{\le n},F\e))\to 0
\]
satifies the hypotheses of \cite[Proposition 3.8]{bga}. Thus the sequence of the proposition is exact for the fpqc topology $\s C_{\qc}$. Since $\gr\lbe(H\le)$ and $\mathcal H^{\le 1}\lbe\lbe(R,F\e)$ are reduced by Lemmas \ref{sm-red} and \ref{ucok-i}, respectively,  Proposition \ref{ex-ex} shows that the given sequence is also exact in $\s C_{\qc}$, and this completes the proof.
\end{proof}

\begin{remark} 
It should be noted that, if $R$ is a ring of unequal characteristics and $k$ is algebraically closed, then the results of \cite[\S 4]{beg} on the cokernel of $\gr(q)$ differ from the results discussed above. Indeed, the statements in [op.cit.] alluded to above are valid in the category of {\it quasi-algebraic $k$-groups}. In particular, the $k$-group $\uwave{\mathrm H}^{\le 1}\lbe\lbe(R,F\e)$ considered in \cite[p.~41]{beg} should not be confused with the group $\mathcal H^{\le 1}\lbe\lbe(R,F\e)$ discussed above. In the context of this Section, the following is true for any $R$. Since $\grn(q)=j_{\le n}\be\circ\! f_{n}$ for every $n$, $\gr(q)$ factors as 
\[
\gr(G\le)\overset{\! f}{\to}\gr(G)^{\prime}\overset{ j}{\hookrightarrow}\gr(H\le),
\]
where $j\overset{\rm def.}{=}\varprojlim j_{n}$ has trivial kernel by the proposition. An application of Lemma \ref{ker-cok} to the complex $\gr(G\le)\to \gr(H\le)\to \mathcal H^{1}\be(R,F\e)$, together with \eqref{lex}, produces the 5-term sequence 
\begin{equation}\label{5-term}
0\to\gr\lbe(F\le)\to\gr(G\le)\overset{\! f}{\to} \gr(G)^{\prime}\to\cok \gr(q)\to \mathcal H^{1}\be(R,F\e)\to 0
\end{equation}
which is exact for the fpqc topology on $k$. It is now clear that the map $\cok \gr(q)\to \mathcal H^{1}\be(R,F\e)$ appearing above is an isomorphism if, and only if, $\cok f=0\,$\, \footnote{We do not know examples where $\cok f$ is not zero.}. Regarding the group $\mathcal H^{\le 1}\lbe\lbe(R,F\e)(\e\kbar\e)$, the following holds. Since $\widehat{R}^{\le\nr}$ is local and henselian with residue field $\kbar$ and $G$ is smooth over $\widehat{R}^{\le\nr}$, we have $\HH^{1}_{\fppf}(\widehat{R}^{\le\nr}\be, G\!\times_{S}\!\widehat{S}^{\nr})=0$ (see the proof of Lemma \ref{coh=0}). Thus \eqref{sm-res} induces an exact sequence
\[
0\to F(\widehat{R}^{\le\nr})\to G(\widehat{R}^{\le\nr})\to H(\widehat{R}^{\le\nr})\to \HH^{1}_{\fppf}(\widehat{R}^{\le\nr}\be, F\!\times_{S}\!\widehat{S}^{\le\nr})\to 0.
\]
By Corollary \ref{r-kbar}, the preceding sequence must agree with the sequence
\[
0\to \gr(F\le)(\e\kbar\e)\to \gr(G\le)(\e\kbar\e)\to \gr(H\le)(\e\kbar\e)\to (\cok \gr(q))(\e\kbar\e)\to 0,
\]
i.e., $(\cok \gr(q))(\e\kbar\e)=\HH^{1}_{\fppf}(\widehat{R}^{\le\nr}\be, F\lbe\times_{S}\lbe\widehat{S}^{\le\nr})$. Now \eqref{5-term} yields the exact sequence of abelian groups 
\[
0\to F(\widehat{R}^{\le\nr})\to G(\widehat{R}^{\le\nr})\to \gr(G)^{\prime}(\e\kbar\e)\to
\HH^{1}_{\fppf}(\widehat{R}^{\le\nr}\be, F\!\times_{S}\!\widehat{S}^{\le\nr})\to\mathcal H^{1}\be(R,F\e)\be(\e\kbar\e)\to 0,
\]
whence the canonical map $\HH^{1}_{\fppf}(\widehat{R}^{\le\nr}, F\!\times_{S}\!\widehat{S}^{\le\nr})\to\mathcal H^{1}\be(R,F\e)\be(\e\kbar\e)$ is an isomorphism if, and only if, $(\cok f)(\e\kbar\e)=0$, i.e., if $\cok f$ is infinitesimal. See Lemma \ref{red=1} and Remark \ref{inf}(b). 
\end{remark}

\begin{theorem}\label{thm} Assume that $F$ is generically smooth. Then there exists an integer $i_{\le 0}\in\N$ such that, for every integer $n\geq i_{\le 0}$, the transition morphism $\mathcal H^{\le 1}\lbe\lbe(R_{\le n+1},F\le)\to\mathcal H^{\le 1}\lbe\lbe(R_{\le n},F\le)$ is an isomorphism of $k$-group schemes. 
\end{theorem}
\begin{proof} By Lemma \ref{ucalg}(ii) and faithfully flat and quasi-compact descent \cite[$\mr{IV}_{2}$, Proposition 2.7.1(viii)]{ega}, we may assume that $k=\kbar$.
It is shown in \cite[p.~465]{llr} (with $G^{\e\prime}=F$, $G^{\e\prime\prime}=H$, $u=q$ and $g^{\e\prime\prime}=h$ in the notation of that paper) that there exists a commutative diagram of flat and commutative $R$-group schemes of finite type
\begin{equation*} 
\xymatrix{0\ar[r]&\widetilde{F}\ar[d]\ar[r]& \widetilde{G}\ar[d]^{g}\ar[r]^(.45){\widetilde{q}}&
\widetilde{H}\ar[d]^{h}\ar[r]& 0\\
0\ar[r]& F\ar[r]&G\ar[r]&H\ar[r]& 0,}
\end{equation*}
where $\widetilde{q}$ is smooth, faithfully flat and of finite presentation, and the bottom row is the sequence \eqref{sm-res}. 
For every integer $n\geq 1$, the preceding diagram induces an exact and commutative diagram in $\mr{Ab}(\mr{Sch}/k)_{\fppf}$
\[
\xymatrix{0\ar[r]&\grn(\widetilde{F})\ar[d]\ar[r]& \grn(\widetilde{G})\ar[d]^{\grn(g)}\ar[r]&
\grn(\widetilde{H})\ar[d]^{\grn(h)}\ar[r]&0\\
0\ar[r]&\grn(F)\ar[r]&\grn(G)\ar[r]&\grn(H)\ar[r]&\mathcal H^{\le 1}\lbe\lbe(R_{\le n},F\le)\ar[r]&0,}
\]
where the top row is exact by Proposition \eqref{ex-green} and the bottom row is the sequence \eqref{fghh}. We conclude that there exists an exact and commutative diagram in $\mr{Ab}(\mr{Sch}/k)_{\fppf}$
\begin{equation*} 
\xymatrix{\cok\mr{Gr}_{n+1}^{R}(g)\ar[d]^{\alpha_{n}}\ar[r]& \cok\mr{Gr}_{n+1}^{R}(h)\ar[d]^{\beta_{n}}\ar[r]&
\mathcal H^{\le 1}\lbe\lbe(R_{\le n+1},F\le)\ar[d]\ar[r]&0\\
\cok\grn(g)\ar[r]& \cok\grn(h)\ar[r]&
\mathcal H^{\le 1}\lbe\lbe(R_{\le n},F\le)\ar[r]&0.}
\end{equation*}
Now it is shown in \cite[p.~471]{llr}  (set $g_{i}=\alpha_{n+1}$ and $g_{i}^{\le\prime\prime}=\beta_{n+1}$ in [loc.cit.]) that there exists an integer $i_{\le 0}\in\N$  such that the maps $\alpha_{n}$ and $\beta_{n}$ appearing above are isomorphisms of smooth $k$-group schemes for every integer $n\geq i_{\le 0}$. The theorem is now clear.
\end{proof}

\begin{corollary}\label{cor1} Assume that $F$ is generically smooth and let $i_{\le 0}\in\N$ be as in the theorem. Then, for every integer $n\geq i_{\le 0}$, we have:
\begin{enumerate}
\item[(i)] The canonical projection 
$\mathcal H^{\le 1}\lbe\lbe(R,F\e)\to\mathcal H^{\le 1}\lbe\lbe(R_{\le n},F\e)$ is an isomorphism of $k$-group schemes.
\item[(ii)] There exists an isomorphism in $\s C_{\alg}$
\[
\cok\lbe\varrho_{n,\le F}^{\le 1}\simeq \mathbb G_{a,\e k}^{r},
\]
where $\varrho_{n,\le F}^{\le 1}$ is the change of level morphism \eqref{rni} and
\[
r=\dimn{\rm Lie}(F_{\mr s})-\dim F_{\mr s}.
\]
\item[(iii)] $\dim\grn(F)=(n-i_{\le 0})\dim F_{\mr s}+\dim\mr{Gr}_{i_{\le 0}}^{R}(F)$.
\item[(iv)] $\dim \mathcal H^{\le 1}\lbe\lbe(R_{\le n},F\e)=\dim\mr{Gr}_{i_{\le 0}}^{R}(F\le)-i_{\le 0}\dim F_{\mr s}$.
\end{enumerate}
\end{corollary}
\begin{proof} Assertion (i) is immediate from the theorem. Now, by \eqref{kseqm}, \eqref{hmi} and the theorem, the smooth resolution \eqref{sm-res} induces an exact sequence in $\s C_{\alg}$
\[
0\to\krn\varrho_{n,\le F}^{\e i}\to \krn\varrho_{n,\le G}^{\e i}\to \krn\e\varrho_{n,\le H}^{\e i}\to\cok\lbe\varrho_{n,\le F}^{\e i}\to 0,
\]
where $n\geq i_{\le 0}$ and $i\geq 1$. If $n\geq 2$, then 
\begin{equation}\label{kdim}
\dim \krn\varrho_{n,\le F}^{\lle 1}=\dimn\e\mr{Lie} (F_{\mr s})
\end{equation} 
by \eqref{pc}, \eqref{d} and Proposition \ref{vker2}(ii). Further, by Propositions \ref{vker1} and \ref{vker}, $\krn\varrho_{n,\le G}^{\e 1}$ is a smooth and unipotent $k$-group scheme which is isomorphic to $\G_{a,k}^{d}$, where $d=\dim G_{\mr s}$, and similarly with $H$ in place of $G$.  
In particular, since $\dim G_{\mr s}=\dim F_{\mr s}+\dim H_{\mr s}$, we conclude that $\cok\lbe\varrho_{n,\le F}^{\le 1}$ has dimension $r=\dimn\e{\rm Lie}(F_{\mr s})-\dim F_{\mr s}$. Further, by \cite[IV, \S 3, Corollary 6.8, p.~ 523]{dg}, there exists an isomorphism of $k$-group schemes $\cok\lbe\varrho_{n,\le F}^{\le 1}\simeq\mathbb G_{a,\e k}^{r}$. This completes the proof of (ii). Now, by (ii), there exists an exact sequence in $\s C_{\alg}$ 
\[
0\lra\krn\varrho_{n,\le F}^{\e 1} \lra\mr{Gr}_{n+1}^{R}(F)\overset{\varrho_{n,\le F}^{\le 1}}{\lra}\grn(F)\lra \mathbb G_{a,\e k}^{r}\lra 0.
\]  
Thus, by the definition of $r$ and \eqref{kdim}, we have 
$\dim \mr{Gr}_{n+1}^{R}(F)=\dim\grn(F)+\dim F_{\mr s}$. Assertion (iii) now follows by induction. Assertion (iv) follows by combining (iii), \eqref{sm-res0}, \eqref{fghh} and Corollary \ref{id-comp}(i). 
\end{proof}

\begin{remark} 
\indent
Corollary \ref{cor1} (and therefore Theorem \ref{thm}) fails if $F$ is not generically smooth. Indeed, let $F=\alpha_{\le p}$ be the kernel of the Frobenius endomorphism $\G_{a,k}\to\G_{a,k}$ in the equal positive characteristic $p$ case. It follows from Example \ref{ex.alpha} that
\begin{equation*}
\grn(\alpha_{\le p})\simeq\spec (k[\le x_{\le 0},\dots, x_{n-1}\le]/(x_{i}^{\e p}, i\leq (n-1)/p)),
\end{equation*}  
whence $\dim\grn(\alpha_{\le p})\geq (p-1) \lfloor
(n-1)/p \rfloor$, which is unbounded as $n\to\infty$. Thus the conclusion of Corollary \ref{cor1}(iii) fails if $F=\alpha_{\le p}$. 
\end{remark}

\begin{lemma} 
Let $n$ and $r$ be integers such that $1\leq r<n$. Then
\[
\dim\grn(\e\mathbb V( R_{\e r}))=r.
\]   
\end{lemma}
\begin{proof}
We begin by constructing a morphism of $k$-group schemes $\gamma\colon \s M_{\lbe n}^{\le n-r}\to \grn(\e\mathbb V( R_{\e r}))$. Let $A$ be any $k$-algebra. By \eqref{vb2} and Lemma \ref{rat-pts}(i),  
\[
\grn(\e\mathbb V( R_{\e r}))(A)=\mathbb V(R_{\e r})(\s R_{\le n}(A))=\mr{Hom}_{R_{\le n}\mr{\text{-}mod}}(R_{\e r}, \s R_{\le n}(A))=\s R_{\e n}(A)_{\pi^{\le r}_{n}\text{-\e tors}} .
\]
Further, by \eqref{rnms1}, the inclusion $\overbarr{\s M_{\lbe n}^{\le n-r}}(A)\subseteq \s R_{\e n}(A)$ factors through $\s R_{\e n}(A)_{\pi^{\le r}_{n}\text{-\e tors}}$. Let $\gamma(A)$ be the composition of the canonical map $\Theta_{n,n-r}(A)\colon \s M_{\lbe n}^{\le n-r}(A)\to \overbarr{\s M_{\lbe n}^{\le n-r}}(A)$ \eqref{ttr} and the inclusion $\overbarr{\s M_{\lbe n}^{\e n-r}}(A)\subseteq \s R_{\e n}(A)_{\pi^{\le r}_{n}\text{-\e tors}}$. The preceding construction is functorial in $A$ and defines the required morphism $\gamma\colon \s M_{\lbe n}^{\le n-r}\to \grn(\e\mathbb V( R_{\e r}))$. If $R$ is an equal characteristic ring, then $\gamma$ is, in fact, an isomorphism. In effect
\begin{equation}\label{vtor}
\s R_{\le n}(A)_{\pi^{\le r}\text{-\e tors}}=\s M_{\lbe n}^{\le n-r}\be(A)=\overbarr{\s M_{\lbe n}^{n-r}}(A) 
\end{equation}
by Remark \ref{power}(d), \eqref{wem} and the flatness of $A$ over $k$.
Therefore
\[
\dim \grn(\e\mathbb V( R_{\e r}))=\dim\e\s M_{\lbe n}^{\le n-r}=\dimn\e M_{\lbe n}^{\le n-r}=\dimn\e R_{\e n-r}=r.
\]
See \eqref{vid} and the beginning of Subsection \ref{sec-k}.
	
Now let $R$ be a ring of unequal characteristics. Then, by Remark \ref{power}(c), the equality \eqref{vtor} holds if $A$ is perfect. Consequently $\gamma^{\le\pf}\colon  (\s M_{\lbe n}^{\le n-r})^{\pf}\simeq \grn(\e\mathbb V( R_{\e r}))^{\pf}$ by \cite[Remark 5.18(a)]{bga}. On the other hand, by \cite[Remark 5.18(b)]{bga}, the perfection functor preserves dimensions. It now follows from \eqref{vid}, \eqref{lgth} and the description of Greenberg modules in Subsection \ref{sec-w} that $\dim \grn(\e\mathbb V( R_{\e r}))= \dim\e \s M_{\lbe n}^{\le n-r}=\mr{length}_{\e W(k)}\e M_{\lbe n}^{\le n-r}=\mr{length}_{\e W(k)}\e R_{\e r}=r$, as claimed.
\end{proof}

\begin{proposition} 
 Assume that $F$ is finite and generically \'etale. Then
\[
\dim\mathcal H^{\le 1}\lbe\lbe(R,F\e)=\delta(F),
\]
where $\delta(F)$ is the defect of smoothness of $F$ \eqref{df}.
\end{proposition}
\begin{proof} By Corollary \ref{cor1}, (i), (iii) and (iv), we have
\[
\dim\mathcal H^{\le 1}\lbe\lbe(R,F\e)=\dim\mr{Gr}^{\lbe R}_{\lbe r}(F\le)
\]
for every integer $r\geq i_{0}$. On the other hand, by Lemma \ref{gr-3c}, 
$\varrho_{\le n,F}^{\e n}$ factors through a finite $k$-subgroup scheme of $\mr{Gr}^{\lbe R}_{n}(F)$ if $n\geq\delta(F)+2$. Thus, by  Proposition \ref{vker2}(ii) and \cite[${\rm{VI}}_{\rm A}$, Proposition 2.5.2(b)]{sga3}, we have $\dim\mr{Gr}^{\lbe R}_{2n}(F\le)=\dim\mr{Gr}_{n}^{R}(\e\mathbb V(\omega_{\le F/R}^{1}))$ if $n\geq r=\max\{i_{0},\delta(F)+2\}$. Therefore
$\dim\mathcal H^{\le 1}\lbe\lbe(R,F\e)=\dim\mr{Gr}_{n}^{R}(\e\mathbb V(\omega_{\le F/R}^{1}))$ if $n\geq r$. Now, by the structure theorem for torsion $R$-modules, there exists an isomorphism of $R$-modules $\omega_{F/R}^{1}\simeq\oplus_{i=1}^{t} R/(\pi^{n_{i}})$, where $\sum n_{i}=\mr{length}_{\lbe R}(\omega_{F/R}^{1}\le)=\delta(F)$ (see Remark \ref{rm.df}).  Thus, by \cite[Exercise 1.22 and Lemma 1.23, p.~258]{liu}, Remark \ref{rems1}(d) and \cite[II, Proposition 1.7.11(iii)]{ega}, we are reduced to checking that
$\dim\mr{Gr}_{n}^{R}(\e\mathbb V(R/(\pi^{n_{i}})))=\dim\mr{Gr}_{n}^{R}(\e\mathbb V(R_{\le n_{i}})))=n_{i}$. This follows from the previous lemma.
\end{proof}

\section{A generalization of the equal characteristic case}\label{eqch}
Let $k$ be any field and let $B$ be a noetherian local $k$-algebra with maximal ideal $\mm$ and residue field $k$. For every integer $n\geq 1$, set $B_{\le n} =B/\mm^{n}$. Then $B_{\le n}$ is a finite local $k$-algebra with residue field $k$ \cite[Proposition 8.6(ii), p.~90, and Exercise 3, p.~92]{am}. Note that, if $B$ is artinian, then $B_{\le n}=B$ for all sufficiently large values of $n$ and the results in Sections \ref{gr-art}, \ref{s-cr} and \ref{bas} are valid with $\mathfrak R=B$. Further, by Lemma \ref{uh0} (with $k^{\e\prime}=k$), $\spec B_{\le n} \to\spec k$ is a finite and locally free universal homeomorphism. Consequently, by Corollary \ref{wr-uh}, $\Re_{B_{\le n}\lbe/\le k}(Z)$ exists for every $B_{\le n}$-scheme $Z$ and $\mr{Gr}^{\lbe B_{\le n}}$ agrees with $\Re_{B_{\le n}\lbe/\le k}$ by Remark \ref{rems1}(c). The following statement is analogous to the equal characteristic case of Corollary \ref{wr-gr1} (see Remark \ref{wgeq}).

\begin{proposition}\label{wr-gr-b} Let $B$ be as above and let $B\to B^{\e\prime}$ be a finite and flat homomorphism of local rings, where $B^{\e\prime}$ is a noetherian local $k^{\e\prime}$-algebra with residue field $k^{\e\prime}$. Let $n\geq 1$ be an
integer and set $C_{\le n}=B_{\le n}\!\otimes_{\lbe B}\be B^{\e\prime}$. If $Z$ is a quasi-projective $C_{\le n}$-scheme, then
$\Re_{\e k^{\prime}\be/k}\big(\Re_{\e C_{\le n}\lbe/k^{\prime}}\lbe(Z)\big)$
and $\Re_{\e C_{\le n}\lbe/\lbe B_{\le n}}\be(Z)$
exist and
\[
\Re_{\e k^{\prime}\be/k}(\Re_{\e C_{\le n}\lbe/k^{\prime}}(Z))=
\Re_{\e B_{\le n}\lbe/k}(\Re_{\le C_{\le n}\lbe/B_{\le n} }\be(Z)).
\]
\end{proposition}
\begin{proof} Note that $C_{\le n}$ is a $B_{\le n}$-algebra as well as a $k^{\e\prime}$-algebra via its $B^{\le\prime}$-algebra structure. By Theorem \ref{wr-rep} and Remark \ref{rems-adm}(a), $\Re_{\e C_{\le n}\lbe/k^{\prime}}\lbe(Z)$ and $\Re_{\e C_{\le n}\lbe/\lbe B_{\le n} }\lbe(Z)$ exist. On the other hand, 
\cite[Proposition A.5.8]{cgp} shows that $\Re_{\e C_{\le n}\lbe/k^{\prime}}\lbe(Z)$ is quasi-projective over $k^{\e\prime}$. Thus $\Re_{\e k^{\prime}\be/k}(\Re_{\e C_{\le n}\lbe/k^{\prime}}(Z))$ also exists. The formula of the theorem is now immediate from \eqref{wrcomp}.
\end{proof}

Regarding the contents of Section \ref{gp-sch}, Lemma \ref{coh=0} (with $\Rnnr$ replaced by $B_{\le n} \otimes_{k}\kbar\e$) and Proposition \ref{ex-green} extend to the present context without difficulty. The proofs of Propositions \ref{vker1} and \ref{sm-vker}, however, rely on Proposition \ref{rnm-2},  which does not extend to the setting of this Section. Nevertheless, the following analog of Proposition \ref{sm-vker} holds. See also Proposition \ref{sth}.

\begin{proposition} 
Let $G$ be a smooth $B$-group scheme. Then, for every pair of positive integers $r,i$, the change of level morphism
\[
\varrho_{r,\le G}^{\le i}\colon \Re_{\e B_{\le r+i}\lle/\lle k}(G_{\be B_{\le r+i} }\lbe)\to\Re_{\e B_{\le r}\lbe/k}(G_{\be B_{\le r}}\lbe)
\]
is smooth and surjective and its kernel is a smooth, connected and unipotent $k$-group scheme.
\end{proposition}
\begin{proof} We fix an integer $j\geq 1$ and apply \cite[Proposition A.3.5]{oes} (see also \cite[Proposition A.5.12]{cgp}) to the finite and local $k$-algebra $B_{j}$. We conclude that, if $1\leq r<j$, then $\varrho_{r,\le G}^{\le 1}$ is smooth and surjective and $\krn\varrho_{r,\le G}^{\e 1}\simeq \mathbb G_{a,\e k}^{\le d(r)}$, where
\[
d(r)=\dimn\lbe(\le\mr{Lie}(G_{\lbe \rm s}\lbe) \be\otimes_{k}\be\mm^{r}\!/\mm^{r+1}\le)\e
\]
(Note that, if $\mm$ is principal, then $\mm^{r}/\mm^{r+1}\simeq k$ and $d(r)=d$ is the integer \eqref{d}.) Since $j$ is arbitrary, the preceding conclusions hold for every integer $r\geq 1$. The rest of the proof is by induction, using an obvious analog of the exact sequence \eqref{3u}.
\end{proof}

The above proposition can be used to define a filtration similar to the equal characteristic case of \eqref{fil}. In effect, let $n\geq 1$ be an integer,  let $G$ be a smooth and commutative $B$-group scheme and set $H=\Re_{B_{\le n}\lbe/k}(G_{\be B_{\le n}}\lbe)$. Then $F^{\e i}\lbe H=\krn\varrho_{i,\le G}^{\le n-i}$, where $1\leq i\leq n$,  defines a filtration of length $n$ on $H$:
\begin{equation}\label{fil2}
H\supseteq F^{\e 1}\lbe H\supseteq\cdots \supseteq F^{\e n}\lbe H=0.
\end{equation}
Note that $H/F^{\e 1}\be H=G_{\mr s}$. Further, if $1\leq i\leq n-1$, then
\[
F^{\e i}\lbe H/F^{\e i+1}\lbe H\simeq\krn\varrho_{i,\le G}^{\le 1}\simeq \mathbb G_{a,\e k}^{\le d(i)},
\]
where $d(i)=\dimn\lbe(\le\mr{Lie}(G_{\lbe \rm s}\lbe)\be\otimes_{k}\be\mm^{i}\be/\mm^{i+1}\le)\e$.

\begin{example} A particular case of \eqref{fil2} appeared in \cite[proof of Theorem 1]{ell}, as we now explain. Let $D$ be a henselian discrete valuation ring with residue field $k$ and field of fractions $K$. Let $K^{\le\prime}\be/K$ be a finite and separable extension, let $D^{\e\prime}$ be the integral closure of $D$ in $K^{\prime}$ and let $k^{\e\prime}$ be the residue field of $D^{\e\prime}$. Assume that $k^{\e\prime}/k$ is purely inseparable. Let $A^{\prime}$ be an abelian variety over $K^{\le \prime}$, $\mathcal A^{\prime}$ its N\'eron model over $D^{\e\prime}$ and $\mathcal B=\Re_{D^{\e\prime}\be/\lbe D}(\mathcal  A^{\prime})$. Set $B=D^{\e\prime}\otimes_{D} k^{\e \prime}\e$ (this is denoted by $R$ in \cite[proof of Theorem 1]{ell}), which is a finite and local $k^{\e\prime}$-algebra with residue field $k^{\e\prime}$, and set $n=\mr{dim}_{\e k^{\le\prime}}\lbe B \geq 1$. Then $\mm^{n}=0$, where $\mm$ is the maximal ideal of $B$, whence $B_{\le n}=B$ for every $m\geq n$. The filtration of length $n$ of $H=\mathcal B_{\e k^{\le\prime}}$ considered in \cite[proof of Theorem 1]{ell} is the filtration \eqref{fil2} for $G=\mathcal A^{\e\prime}_{B}$.
\end{example}

\medskip

Finally, if $k$ is a perfect field of positive characteristic, several of the results in Section \ref{pfgf} remain valid for the functor  
$(\e\mathrm{Sch}/B_{\le n})\to
(\mathrm{Perf}/k), Z\mapsto \Re_{\e B_{\le n}\lbe/k}(Z)^{\pf}$. For example, Proposition \ref{wr-gr-b} above yields a formula similar to that of Proposition \ref{p-wrgr}(i) under an appropriate quasi-projectivity hypothesis.


\bigskip



\begin{thebibliography}{[28]}

\bibitem[Ab]{ab} Abbes, A.: \'El\'ements de G\'eom\'etrie Rigide. Volume I. Construction et \'etude g\'eom\'etrique des espaces rigides. Progress in Math. {\bf{286}}, Birkh\"auser, 2010.

\bibitem[An]{an} Anantharaman, S.: Sch\'emas en groupes, espaces homog\`enes et espaces alg\'ebriques sur une base de dimension 1. Bull. Soc. Math. France {\bf{33}} (1973). 

\bibitem[AM]{am} Atiyah, M. and MacDonald, I.: Introduction to commutative algebra,  Addison-Wesley Publishing Company, Inc., Reading, MA (1969).

\bibitem[B\'eg]{beg} B\'egueri, L.: Dualit\'e sur un corps local \`a corps r\'esiduel alg\'ebriquement clos.  M\'em. Soc. Math. France {\bf{4}}, (1980).

\bibitem[Bert]{bert} Bertapelle, A.: Formal N\'eron models and Weil restriction. Math. Ann. {\bf{316}} (2000), 437--463.



\bibitem[BGA]{bga} Bertapelle, A. and Gonz\'alez-Aviles, C.: On the perfection of schemes. Available at \url{http://arxiv.org/abs/1611.02060}.

\bibitem[BT]{bt} Bertapelle, A. and Tong, J.: On torsors under elliptic curves and Serre's pro-algebraic structures.  Math. Z. {\bf{277}} (2014), 91--147. 

\bibitem[Bha]{bha}
Bhatt, B.:  Algebraization and {T}annaka duality.
 Camb. J. Math. {\bf{4}} (2016), 403--461.

\bibitem[BLR]{blr} Bosch, S., L\"utkebohmert, W. and Raynaud, M.: N\'eron models. Erg. der Math. Grenz. {\bf{21}}, Springer-Verlag,  Berlin, 1990.

\bibitem[Bou]{bou} Bourbaki, N.: Commutative Algebra. Chapters 1--7. Softcover edition of the 2nd printing 1989, Springer-Verlag, 1998. ISBN 3-540-64239-0.

\bibitem[Bou2]{bou2} Bourbaki, N.: Algebra II. Chapters 4--7. Softcover printing of the first English edition of 1990, Springer-Verlag, 2003. ISBN 3-540-00706-7.

\bibitem[Bou3]{bou3} Bourbaki, N.: Theory of Sets. Addison-Wesley, 1968.


\bibitem[BP]{bp} Brinkmann, H.-B. and Puppe, D.: Abelsche und exakte Kategorien, Korrespondenzen. Lect. Notes in Math. {\bf{96}}, Springer-Verlag, 1969.
 

\bibitem[Chev]{chev}  
S\'eminaire Claude Chevalley, 1, 1956--1958.
Classification des groupes de Lie algébriques.
 

\bibitem[CGP]{cgp} Conrad, B., Gabber, O. and Prasad, G.:
Pseudo-reductive groups. New mathematical monographs {\bf{17}}, Cambridge
Univ. Press 2010.

\bibitem[CR]{cr} Cunningham, C. and Roe, D.: A function-sheaf dictionary for algebraic tori over local fields. arXiv:1310.2988v2 [math.AG], J. Inst. Math. Jussieu, to appear.

\bibitem[$\text{SGA3}_{\le\text{new}}$]{sga3}  Demazure, M. and
Grothendieck, A. (Eds.): Sch\'emas en groupes. S\'eminaire de
G\'eom\'etrie Alg\'ebrique du Bois Marie 1962-64 (SGA 3). Augmented and
corrected 2008-2011 re-edition of the original by P.Gille and P.Polo.
Available at \url{http://www.math.jussieu.fr/~polo/SGA3}. Reviewed
at \url{http://www.jmilne.org/math/xnotes/SGA3r.pdf}.

\bibitem[DG]{dg} Demazure, M. and Gabriel, P.:  Groupes
alg\'ebriques. Tome I: G\'eom\'etrie alg\'ebrique,
g\'en\'eralit\'es, groupes commutatifs. Masson \& Cie, \'Editeur,
Paris, 1970. (with an Appendix by M. Hazewinkel: Corps de classes
local). ISBN 7204-2034-2.

\bibitem[Ed]{ed} Edixhoven, B.: N\'eron models and tame ramification. Comp. Math. {\bf{81}} (1992) 291--306.  


\bibitem[ELL]{ell} Edixhoven, B., Liu, Q. and Lorenzini, D.: The $p$-part of the group of components of a N\'eron model. J. Algebraic Geometry {\bf{5}} (1996)  801--813 .
 

\bibitem[FK]{fk} Fujiwara, K. and Kato, F.: Foundations of rigid geometry.
arXiv:1308.4734v3. 

\bibitem[Gi]{gi}  Giraud, J.: Cohomologie non ab\'elienne.  Grund. der Math. Wiss. {{179}}, Springer-Verlag, Berlin-New York, 1971.


\bibitem[Gre1]{gre1} Greenberg, M. J.: Schemata over local rings.
Ann. of Math. (2) {\bf{73}} (1961), 624--648.

\bibitem[Gre2]{gre2} Greenberg, M. J.: Schemata over local rings:
II. Ann.of Math. (2) {\bf{78}} (1963), 256--266.
 

\bibitem[Gre3]{gre3} Greenberg, M. J.: Rational points in henselian discrete valuation rings. Publ. Math. IHES {\bf{31}} (1966), 59--64.



\bibitem[$\text{EGA I}_{\le\text{new}}$\e]{ega1} Grothendieck, A. and
Dieudonn\'e, J.:\'El\'ements de g\'eom\'etrie alg\'ebrique I. Le langage des
sch\'emas. Grund. der Math. Wiss. {\bf{166}} (1971).


\bibitem[EGA]{ega} Grothendieck, A. and Dieudonn\'e, J.: \'El\'ements de g\'eom\'etrie alg\'ebrique. Publ. Math. IHES {\bf{8}} ($=\text{EGA II}$) (1961), {\bf{11}} ($=\text{EGA III}_{1}$) (1961), {\bf{20}} ($=\text{EGA IV}_{1}$) (1964), {\bf{24}} ($=\text{EGA IV}_{2}$) (1965), {\bf{28}} ($=\text{EGA IV}_{3}$) (1966), {\bf{32}} ($=\text{EGA IV}_{4}$) (1967).

\bibitem[SGA1]{sga1} Grothendieck, A.: Rev\^etements \'etales etgroupe fondamental (SGA 1). S\'eminaire de g\'eom\'etrie 			alg\'ebrique du Bois Marie 1960--61. Lecture Notes in Math. {\bf{224}}, pringer-Verlag 1971.
	 

\bibitem[Dix]{dix}  Grothendieck, A.: Le groupe de Brauer III. In:
Dix expos\'es sur la cohomologie de sch\'emas, 
North-Holland Pub. Co.,  Amsterdam 1968.
						
\bibitem[Ill]{ill} Illusie, L.:
Complexe de de Rahm-Witt et cohomologie cristalline,  Ann. Sci. \'Ecole Norm. Sup.  {\bf{12}}, no.4 (1979), 501--661.

\bibitem[Liu]{liu} Liu, Q.: 
Algebraic Geometry and Arithmetic Curves,  Oxford Graduate Texts in Mathematics {\bf{6}}, 2002.

\bibitem[LLR]{llr} Liu, Q., Lorenzini, D. and Raynaud, M.: N\'eron models,    Lie algebras, and reduction of curves of genus one. Invent. Math. {\bf{157}} (2004), 455-518.

\bibitem[Lip]{lip} Lipman, J.:
The Picard group of a scheme over an Artin ring,  Publ. Math. IHES {\bf{46}} (1976), 15--86.

\bibitem[LS]{ls} Loeser, F. and Sebag, J.: Motivic integration on
smooth rigid varieties and invariants of degenerations. Duke Math.
J. {\bf{119}}, no. 2 (2003), 315--344.

\bibitem[Mac]{mac}  Mac Lane, S.: Categories for the working mathematician. Springer Verlag, 1971.
 

\bibitem[Mat]{mat}  Matsumura, H.: Commutative Algebra. Second edition. Math. Lect. Notes Series {\bf{56}}, The Benjamin/Cummings Publishing Company, Inc., 1980. ISBN 0-8053-7024-2. 

\bibitem[MR]{mr} Mazur, B. and Roberts, L.: Local Euler characteristics. Invent. Math. {\bf{9}} (1969/1970), 201--234.
 

\bibitem[ADT]{adt} Milne, J.S.: 
Arithmetic Duality Theorems, Second Edition (electronic version), 2006.
 .

\bibitem[NS]{ns} Nicaise, J. and Sebag, J.:  
Motivic Serre invariants and Weil restriction, J. Algebra {\bf{319}} (2008) 1585--1610.

\bibitem[NS2]{ns2} Nicaise, J. and Sebag, J.:  
Motivic invariants of rigid varieties, and applications to complex singularities, in: Motivic integration and its interactions with model theory and non-archimedean geometry,  R. Cluckers, J. Nicaise and J. Sebag (eds.), London Mathematical Society Lecture Notes Series, vol. {\bf{383}}, Cambridge University Press, 2011, 244--304.

\bibitem[Oes]{oes} Oesterl\'e, J.: Nombres de Tamagawa et groupes unipotentes en
caract\'eristique $p$. Invent. Math. {\bf{78}} (1984), 13--88.

\bibitem[Per]{per} Perrin, D.: Sch\'emas en groupes quasi-compacts sur un corps. Publ. Math. Orsay {\bf{165-75.46}} (1\'ere partie). \url{http://sites.mathdoc.fr/PMO/afficher_{n}otice.php?id=PMO_1975_A26}.


\bibitem[Ray]{ray} Raynaud, M.: Anneaux Locaux Hens\'eliens. Lect. Notes in Math. {\bf{169}}, Springer-Verlag 1970.

\bibitem[Sal]{sal} Salmon, P.: Sur les series formelles restreintes. Bull. Soc. Math. France  {\bf{92}} (1964), 385--410.

\bibitem[Seb]{seb} Sebag, J.: Int\'egration motivique sur les sch\'emas formels. Bull. Soc. Math. France {\bf{132}}, no. 1, (2004), 1--54. 

\bibitem[SeCFT]{secft} Serre, J.-P.: Sur les corps locaux \`a corps r\'esiduel alg\'ebriquement clos. Bull. Soc. Math. France {\bf{89}} (1961) 105--154.

\bibitem[SeLF]{self} Serre,  J.-P.: Local Fields,  Grad. Texts in Math. {\bf{67}} (second corrected printing, 1995), Springer-Verlag 1979.  

\bibitem[Ta]{ta} Tate, J.: Finite group schemes. In: Modular Forms and 
Fermat's last theorem. Cornell, G., Silverman, J. and Stevens, G., eds. Springer-Verlag 1997, pages 121--154.

\bibitem[Vis]{v} Vistoli, A.: 
Grothendieck topologies, fibered categories and descent theory, in: Fundamental Algebraic Geometry. Mathematical Surveys and Monographs, 123, American Mathematical Society, Providence, 2005, 7--104. 

\bibitem[Wa]{wa} Waterhouse, W.: Introduction to affine group schemes. Grad. Texts in Math. {\bf{66}}, Springer-Verlag 1979.

\end{thebibliography}
\end{document}